\documentclass[a4paper,reqno,11pt]{amsart}
\usepackage{t1enc}
\usepackage[margin=1.0in]{geometry}
\usepackage{ifthen}
\usepackage{graphicx}
\usepackage{amsmath,amstext,amssymb,amsopn,amsthm,color}
\usepackage{enumitem}

\usepackage{commands}

\numberwithin{equation}{section}

\title [Generalized Bessel-Dunkl diffusions]{Generalized Bessel-Dunkl diffusions}
\thanks{}


\author{Jacek Małecki}
\address{Department of Mathematics \\ Wroc{\l}aw University of Science and Technology \\ ul. Wybrze{\.z}e Wyspia{\'n}\-skiego 27 \\ 50-370 Wroc{\l}aw, Poland}
\email{jacek.malecki@pwr.edu.pl}

\date{\today}

\begin{document}

\begin{abstract}
We develop a general theory of Bessel-Dunkl type diffusions in Weyl chambers associated with classical root systems.  The class considered here allows time-dependent and configuration-dependent diffusion and drift coefficients, as well as state-dependent singular repulsion coefficients which may vanish on the walls of the chamber.  This includes, in a unified framework, Dyson-type logarithmic particle systems, radial Dunkl processes, squared Bessel and Wishart particle systems, and non-colliding diffusion models.

Using the geometry of the underlying root system and a symmetric-polynomial approach, we establish weak existence in both the strictly positive repulsion and the degenerate regimes.  In the strictly positive regime, we prove that positive repulsion prevents positive-time multiple collisions.  In the degenerate case, we identify geometric conditions which exclude boundary sticking and allow one to recover the genuine singular equation from its interior form.  We also prove non-explosion under a radial linear-growth condition, obtain pathwise uniqueness and strong existence in the non-sticky class under natural Yamada-Watanabe type or locally Lipschitz assumptions, and establish a mean-field convergence theorem for systems of types $\AN$, $\BN$, and $\DN$.

Together, these results build the structural foundations for a systematic theory of generalized Bessel-Dunkl diffusions with non-constant and possibly degenerate coefficients.
\end{abstract}

\maketitle


\section{Introduction}
\label{sec:introduction}

We consider the Euclidean space $\R^\no$, where $\no\in \N_+$, equipped with
the usual inner product and the corresponding Euclidean norm
    \begin{equation*}
        \inner{x}{y} = \sum_{k=1}^{\no} x_k y_k, \qquad \norm{x}^2=\inner{x}{x}, \qquad x,y\in\R^\no .
    \end{equation*}
Let $R$ denote the irreducible finite root system in $\R^\no$, with chosen subsystem of positive roots $R_+$, positive Weyl chamber $\C$, and closed Weyl chamber $\Cc$.  The paper is devoted to the study of existence, uniqueness, boundary behaviour, non-explosion, and mean-field limits for chamber-valued solutions $X=(X_1,\ldots,X_\no)$ of the singular stochastic differential system
    \begin{equation}
        \label{eq:main:SDE:cord}
        dX_k(t) = \sigma_k(t,X(t))\,dB_k(t) + b_k(t,X(t))\,dt + \sum_{\alpha\in R_+} \frac{k_\alpha(t,X(t))\alpha_k}{\langle X(t),\alpha\rangle}\,dt, \qquad k=1,\ldots,\no ,
    \end{equation}
such that $X(t)\in \Cc$, $0\le t<\life{X}$. Here $\life{X}$ denotes the lifetime of the process, which may be a fixed deterministic time horizon or the explosion time.  The driving process $B=(B_1,\ldots,B_\no)$ is an $\no$-dimensional Brownian motion with independent one-dimensional components.  The diffusion matrix is diagonal, $\sigma(t,x)= \operatorname{diag}\bigl(\sigma_1(t,x),\ldots,\sigma_\no(t,x)\bigr)$,  the regular drift is $b=(b_1,\ldots,b_\no)$, and the repulsion coefficients $k_\alpha$ are non-negative functions on $[0,\infty)\times\Cc$.  Equivalently, \eqref{eq:main:SDE:cord} can be written in vector form as
    \begin{equation}
        \label{eq:main:SDE}
        dX(t) = \sigma(t,X(t))\,dB(t) + b(t,X(t))\,dt + \sum_{\alpha\in R_+} \frac{k_\alpha(t,X(t))\,\alpha}{\langle X(t),\alpha\rangle}\,dt .
    \end{equation}
The singular drift in \eqref{eq:main:SDE} acts in the normal directions to the root hyperplanes $\{x\in\R^\no:\langle x,\alpha\rangle=0\}$, $\alpha\in R_+$. Thus the boundary of the Weyl chamber plays the role of a collision boundary. In type $\AN$, this boundary is formed by the hyperplanes $x_i=x_j$.  In types $\BN$ and $\DN$, one additionally encounters walls corresponding to collisions with mirror images and, in type $\BN$, the wall $x_i=0$.  The term ${k_\alpha(t,X(t))\,\alpha}/{\langle X(t),\alpha\rangle}$ is therefore a Bessel-type repulsion from the corresponding wall.  The rank-one case makes this interpretation explicit. Indeed, if $\no=1$ and $R_+=\{1\}$, then $\C=(0,\infty)$ and $\Cc=[0,\infty)$, and \eqref{eq:main:SDE} reduces to
    \begin{equation*}
        dX(t) = \sigma(t,X(t))\,dB(t) + b(t,X(t))\,dt + \frac{k(t,X(t))}{X(t)}\,dt .
    \end{equation*}
In the special case $\sigma\equiv1$, $b\equiv0$, and $k(t,x)\equiv k_0>0$, this is the stochastic differential equation for a
Bessel process of dimension $2k_0+1$ or, equivalently, of index $\nu=k_0-\frac12$ (see \cite{bib:RevuzYor:1999}). Thus the singular terms in \eqref{eq:main:SDE} are Bessel-type repulsions from the walls of the chamber.  The multidimensional root-system case replaces the single boundary point $0$ by the family of root hyperplanes $\{\langle x,\alpha\rangle=0\}$, $\alpha\in R_+$.

\medskip

We shall call solutions of \eqref{eq:main:SDE} generalized Bessel-Dunkl diffusions.  The word ``Bessel'' reflects the rank-one reduction above, while ``Dunkl'' reflects the root-system geometry of the singular drift.  In the classical specialization $\sigma\equiv I$, $b\equiv0$, $k_\alpha(t,x)\equiv k(\alpha)$, where $k$ is a Weyl-invariant multiplicity function, one recovers the usual radial Dunkl processes, also known as multivariate Bessel processes.

\medskip

To the best of our knowledge, equation \eqref{eq:main:SDE} has not been studied before at this level of generality.  The diffusion coefficients may depend on time and on the whole configuration, the regular drift may be time-inhomogeneous and configuration-dependent, and the singular coefficients $k_\alpha$ are not assumed to be constant root multiplicities.  Moreover, the coefficients $k_\alpha$ may vanish on the corresponding wall $\{\langle x,\alpha\rangle=0\}$.  This last feature is essential for degenerate models, including squared Bessel and Wishart-type particle systems, but it also creates the main analytic difficulty: the singular drift may no longer prevent the process from interacting nontrivially with the boundary.

\medskip

In this sense, equation \eqref{eq:main:SDE} provides a natural starting point for a systematic theory of generalized Bessel-Dunkl diffusions.  The classical radial Dunkl processes correspond to the special case $\sigma\equiv I$, $b\equiv0$, and $k_\alpha(t,x)\equiv k(\alpha)$, where $k$ is a Weyl-invariant multiplicity function. The present framework allows one to move beyond that setting by treating time-dependent, state-dependent, and possibly degenerate singular coefficients in the same root-geometric formulation.

\medskip

A key feature of our approach is that it uses the geometry of the underlying root system in an essential way.  The chamber $\Cc$ is not treated merely as an ordered state space with singular pairwise interactions, but as a fundamental domain for the action of the corresponding reflection group.  The boundary faces, their normal directions, and the singular drift are all encoded by the root configuration.  This geometric structure allows us to choose adapted invariant coordinates, to desingularise the equation, and to analyse boundary faces simultaneously rather than one collision hyperplane at a time.  This is one of the reasons why the results can be obtained in the present generality with time-dependent and configuration-dependent coefficients and with repulsion coefficients that may vanish on the walls.  For background on reflection groups, Weyl chambers and invariant polynomial coordinates, see \cite{bib:Humphreys:1990,bib:Chevalley:1955}.

\medskip

The main problem is to construct a chamber-valued process satisfying the singular equation and to understand what happens on the boundary. 
When $k_\alpha$ is strictly positive on the wall $\{\langle x,\alpha\rangle=0\}$, the repulsion is Bessel-type and forces the corresponding singular drift to be integrable along the trajectory.  In particular, the process cannot spend positive Lebesgue time on that wall. Moreover, in the strictly positive regime, the geometry of the root system prevents positive-time multiple collisions, that is, simultaneous hits of non-orthogonal walls. When $k_\alpha$ is allowed to vanish at the wall, this automatic non-sticking property may fail.  We therefore distinguish between a strictly positive repulsion regime and a degenerate regime.  In the degenerate case we first construct solutions of the natural interior equation, and then give a geometric criterion which excludes hidden boundary measures and allows us to recover the genuine singular equation.

\medskip

The method is based on invariant coordinates associated with the root system. The singular equation in the Weyl chamber is transformed into a non-singular system for a suitable family of invariant polynomials.  After solving this regularized system, we return to the closed Weyl chamber through the inverse invariant map and prove that the obtained process satisfies the original singular SDE.  In the degenerate case, the boundary analysis is carried out by means of face detectors, which measure the distance to boundary faces in admissible normal directions.  These detectors allow us to control local times and to decide whether additional finite-variation terms can be supported on the boundary.

\medskip

The main results of the paper are the following.  First, we prove weak existence in the strictly positive repulsion regime and show that positive repulsion prevents positive-time multiple collisions.  Second, we prove weak existence in the degenerate regime, initially for an interior equation in which the singular terms are switched off on the corresponding walls.  Third, under a facewise non-sticking condition, we show that the constructed solution spends zero Lebesgue time on the relevant boundary faces and therefore solves the genuine singular equation \eqref{eq:main:SDE}.  Fourth, we establish a radial linear-growth condition which excludes explosion.  Fifth, we prove pathwise uniqueness and strong existence in the non-sticky class, under either Yamada-Watanabe type assumptions or locally Lipschitz Euclidean assumptions. Finally, we prove a mean-field convergence theorem for systems of types $\AN$, $\BN$, and $\DN$.  This last result connects the present work with the literature on mean-field limits and large-dimensional eigenvalue dynamics, including \cite{bib:RogersShi:1993,bib:BenArousGuionnet:1997,bib:CabanalDuvillardGuionnet:2001,bib:MaleckiPerez:2022,bib:SongYaoYuan:2020,bib:SongYaoYuan:2021}.

\medskip

We now recall the main sources of motivation for such systems.  The first classical example is Dyson's Brownian motion \cite{bib:Dyson:1962}.  Dyson showed that the eigenvalues of Brownian motion on spaces of random matrices evolve as interacting one-dimensional particles with logarithmic electrostatic repulsion.  For real symmetric, complex Hermitian, and quaternionic self-dual matrix Brownian motions one obtains the classical parameters $\beta=1,2,4$.  The same equation, however, is meaningful for arbitrary positive values of the repulsion parameter. This model is fundamental in random matrix theory and statistical physics, where it appears as the dynamical version of the logarithmic Coulomb gas or log-gas; see \cite{bib:Mehta:2004,bib:AndersonGuionnetZeitouni:2010,bib:Forrester:2010}. For general beta ensembles and related diffusion limits, see also \cite{bib:DumitriuEdelman:2002,bib:RamirezRiderVirag:2011}. The connection between eigenvalue diffusions and limiting spectral laws was also studied in \cite{bib:Chan:1992,bib:RogersShi:1993}. The same inverse-square and logarithmic interaction structures are closely related to Calogero-Moser-Sutherland systems; see \cite{bib:Calogero:1971,bib:Sutherland:1971,bib:Sutherland:1972, bib:Moser:1975,bib:OlshanetskyPerelomov:1976,bib:OlshanetskyPerelomov:1981}.

\medskip

After Dyson, singularly interacting Brownian particles were studied from several complementary perspectives.  Rogers and Shi \cite{bib:RogersShi:1993} analysed interacting Brownian particles and their connection with the Wigner law.  C\'e{}pa and L\'e{}pingle \cite{bib:CepaLepingle:1997} developed a general framework for diffusing particles with electrostatic repulsion, using stochastic equations with multivalued drift.  Their work is one of the foundational references for one-dimensional ordered particle systems in which the singular repulsion is understood as a boundary mechanism preventing or controlling collisions.  The method of multivalued stochastic differential equations is very effective in that setting, but it does not seem to extend in a direct way to the present root-system framework, where the singularities are organized by all positive roots, the coefficients may depend on the full configuration, and the repulsion may degenerate on boundary faces.  In contrast, the approach of the present paper is intrinsically geometric: it uses the Weyl chamber, reflection symmetries, invariant polynomials, and face structure associated with the root system. Later developments include, among others, the work of Graczyk and Małecki \cite{bib:GraczykMalecki:2014,bib:GraczykMalecki:2019}, where strong solutions and squared Bessel particle systems were studied under assumptions allowing weak or degenerate repulsion.

\medskip

A second important origin of these processes comes from conditioning particles not to collide.  The Karlin-McGregor formula \cite{bib:KarlinMcGregor:1959} gives determinantal transition densities for one-dimensional Markov processes killed at the first collision time.  Applying a Doob $h$-transform \cite{bib:Doob:1984} with the Vandermonde determinant produces non-colliding Brownian motion and recovers the same logarithmic drift as in Dyson's model.  Grabiner \cite{bib:Grabiner:1999} extended this viewpoint to Brownian motion in Weyl chambers and related it to random matrices.  This conditioning interpretation is one of the main reasons why Weyl chambers and root systems are natural state spaces for such diffusions. Related constructions and path-transform interpretations appear in \cite{bib:BianeBougerolOConnell:2005,bib:Warren:2007, bib:KatoriTanemura:2007}.

\medskip

Root systems also enter through Dunkl theory. Dunkl operators were introduced in \cite{bib:Dunkl:1989}; see also the survey
\cite{bib:Roesler:2003}.  The associated Markov processes were studied by
Rösler and Voit \cite{bib:RoeslerVoit:1998}; for further developments on
Dunkl processes and their radial parts, see
\cite{bib:ChybiryakovGallardoYor:2008,bib:Chybiryakov:2008}. The radial parts of Dunkl processes are diffusions in Weyl chambers with drift of the form
    \begin{equation*}
        \sum_{\alpha\in R_+} k(\alpha)\, \frac{\alpha}{\langle x,\alpha\rangle},
    \end{equation*}
where the multiplicity function $k$ is constant on Weyl-group orbits.  Radial Dunkl processes and their boundary hitting properties were further investigated, for example, by Demni \cite{bib:Demni:2009}.  These processes form the classical non-degenerate Bessel-Dunkl class.  The present paper keeps the same geometric singular structure, but allows the coefficients to be time-dependent, configuration-dependent, and possibly degenerate.

\medskip

Matrix-valued stochastic processes provide another major family of examples. Bru introduced Wishart processes \cite{bib:Bru:1991}, which are matrix-valued analogues of squared Bessel processes.  The eigenvalues of such processes satisfy singular SDEs of Laguerre or Wishart type.  In the complex case, K\"o{}nig and O'Connell \cite{bib:KonigOConnell:2001} identified the eigenvalues of the Laguerre process with non-colliding squared Bessel processes.  More recent $\beta$-extensions, such as beta-Wishart particle systems \cite{bib:JourdainKahn:2022}, further illustrate the need for a theory that does not rely on a fixed matrix representation. Further developments on Wishart processes and their eigenvalues include \cite{bib:DonatiMartinDoumercMatsumotoYor:2004, bib:GraczykMaleckiMayerhofer:2018,bib:SongYaoYuan:2020}.

\medskip

The importance of studying \eqref{eq:main:SDE} in this generality is therefore twofold.  On the one hand, it gives a unified stochastic framework for models coming from random matrix theory, statistical physics, Dunkl theory, and conditioned non-colliding particle systems.  On the other hand, it isolates the analytic mechanisms responsible for existence, uniqueness, non-explosion, and boundary non-sticking in singular multidimensional diffusions.  This is especially relevant for models in which the repulsion is weak or degenerate, because in such cases the behaviour at the boundary cannot be inferred from the classical non-colliding theory.

\medskip

The paper is organized as follows.  Section~\ref{sec:assumptions-results} states the assumptions and the main results.  Section~\ref{sec:preliminaries} recalls the necessary notation concerning root systems, Weyl chambers, symmetric polynomials, and collision faces.  The next sections develop the invariant-polynomial construction, analyse collisions and integrability of the singular drift, and prove existence in the non-degenerate and degenerate regimes.  The uniqueness results are proved in Section~\ref{sec:uniqueness}. Section~\ref{sec:mean-field-proof} contains the proof of the mean-field convergence theorem.  The final section discusses examples, including Dyson-C\'e{}pa-L\'e{}pingle logarithmic systems, generalized particle systems introduced in \cite{bib:GraczykMalecki:2014}, squared Bessel and Wishart particle systems, and non-colliding Brownian bridge models.


\section{Assumptions and Results}
\label{sec:assumptions-results}

This section is devoted to presenting the main results of the paper together with an explanation of the relevant assumptions. The proofs, as well as a more detailed description of the notation, are postponed to later sections, and here we introduce only the objects necessary to state the results. The assumptions are presented in groups, and this division is closely related to the theorem or proposition to which they apply.

Throughout this section, as mentioned in the Introduction, $R$ denotes one of the classical root systems considered in the paper, $R_+$ is the chosen set of positive roots, $\C$ is the open Weyl chamber, and $\Cc$ is its closure. The boundary is denoted by $\Cb$. The coefficients in \eqref{eq:main:SDE} are always evaluated on $[0,\infty)\times\Cc$. See Section~\ref{sec:preliminaries} for a more detailed introduction to and description of these objects.

\medskip

We begin with two structural assumptions that will be used throughout the paper. The first concerns continuity of the non-singular coefficients and the repulsion coefficients, while the second concerns compatibility of the diffusion coefficient with the boundary geometry of the Weyl chamber.

\medskip

\begin{enumerate}[label=(C\arabic*), ref=(C\arabic*)]
    \item \label{ass:cont}
    For every $i=1,\ldots,N$ and every $\alpha\in R_+$, the functions
    \begin{align}
        \sigma_i &: [0,\infty)\times\Cc\longrightarrow \R,
        &
        b_i &: [0,\infty)\times\Cc\longrightarrow \R,
        &
        k_\alpha &: [0,\infty)\times\Cc\longrightarrow [0,\infty)
        \label{eq:ass:continuity-functions}
    \end{align}
    are continuous.

    \item \label{ass:sigma}
    For every $t\ge0$ and every $\alpha\in R_+$ such that $|\alpha|^2=2$, write
    \begin{align*}
        \operatorname{supp}(\alpha)
            &= \{r\in\{1,\ldots,N\}:\alpha_r\neq0\} \\
            &= \{i,j\}.
    \end{align*}
    We assume that
    \begin{align}
        \sigma_i^2(t,x)-\sigma_j^2(t,x)
            &\longrightarrow 0
            \qquad\text{as }x\in\Cc,\ \langle x,\alpha\rangle\downarrow0.
            \label{eq:ass:sigma}
    \end{align}
\end{enumerate}

Assumption~\ref{ass:cont} is standard. It is imposed on the non-singular coefficients $\sigma_i$, $b_i$, and $k_\alpha$, and the same type of continuity hypothesis already appears in the work of Graczyk and Małecki \cite{bib:GraczykMalecki:2014}. In their setting, the diffusion coefficient is scalar and depends on a single space variable, so continuity of $\sigma:\R\to\R$ is sufficient.

Assumption~\ref{ass:sigma} is specific to, and natural in, the present level of generality. Here the diffusion coefficient is diagonal with components $\sigma=(\sigma_1,\ldots,\sigma_N)$ and depends on the whole configuration $x\in\Cc$, so one needs an additional continuity-type compatibility condition at long-root walls. For $\alpha=e_i-e_j$ or $\alpha=e_i+e_j$, it requires that $\sigma_i^2(t,x)$ and $\sigma_j^2(t,x)$ become asymptotically equal as $\langle x,\alpha\rangle\downarrow0$ within $\Cc$. In particular, it is not an ellipticity assumption. In type $A_{N-1}$ it reduces to the usual requirement that colliding particles have the same variance. In type $B_N$ it is imposed only on long roots, so it does not force $\sigma_i$ to vanish at the short-root wall $x_i=0$. This condition is precisely what prevents the It\^o correction terms in the invariant-coordinate equations from creating singular boundary contributions.

\medskip


\subsection{Strictly positive repulsion}
\label{subsec:positive-repulsion}

The first existence regime we consider is the one in which each repulsion coefficient is strictly positive on its corresponding wall. This is the natural analogue of the one-dimensional Bessel-type situation: the singular drift is strong enough to prevent the process from spending positive Lebesgue time on the boundary, although instantaneous visits to the boundary may still occur.

\medskip

\begin{enumerate}[label=(A\arabic*), ref=(A\arabic*)]
    \item \label{ass:k:positive}
    For every $\alpha\in R_+$, every $t\ge0$, and every $x\in\Cc$ satisfying
    $\langle x,\alpha\rangle=0$, we assume that
    \begin{align}
        k_\alpha(t,x) &> 0.
        \label{eq:k:positive}
    \end{align}
\end{enumerate}

Assumption~\ref{ass:k:positive} is local to the wall associated with the root $\alpha$. Away from this wall, the quotient $k_\alpha(t,x)/\langle x,\alpha\rangle$ is regular, so no lower bound on $k_\alpha$ is needed there.

\begin{theorem}[Existence in the strictly positive regime]
    \label{thm:existence}
    Let $x_0\in\Cc$. Assume \ref{ass:cont}, \ref{ass:sigma}, and \ref{ass:k:positive}. Then there exists a weak solution $X=(X_1,\ldots,X_N)$ of \eqref{eq:main:SDE}, started from $x_0$, which takes values in $\Cc$ up to its lifetime $\life{X}>0$. More precisely, there exist a filtered probability space carrying an $N$-dimensional Brownian motion $B$ and a continuous $\Cc$-valued semimartingale $X=(X_1,\ldots,X_N)$ such that
    \begin{align}
        X(t)
            & = x_0 + \int_0^t \sigma(s,X(s))\,dB(s) + \int_0^t b(s,X(s))\,ds  +\sum_{\alpha\in R_+}\alpha \int_0^t \frac{k_\alpha(s,X(s))}{\langle X(s),\alpha\rangle}\,ds
            \label{eq:thm:existence-sde}
    \end{align}
    for every $t\in[0,\life{X})$. Moreover, the singular integrals in \eqref{eq:thm:existence-sde} are finite on compact subintervals of $[0,\life{X})$, with the usual improper interpretation at the initial time if $x_0\in\Cb$.
\end{theorem}

By continuity of $k_\alpha$ and Assumption~\ref{ass:k:positive}, the quotient $k_\alpha(t,x)/\langle x,\alpha\rangle$ extends continuously to the wall $\{\langle x,\alpha\rangle=0\}$ as the value $+\infty$ in the extended real sense. Consequently, the finiteness of the singular drift integral in Theorem~\ref{thm:existence} implies that the process cannot spend positive Lebesgue time on that wall. 

The proof of Theorem~\ref{thm:existence} is based on the
invariant-coordinate map.  One first transforms the singular equation in the
chamber into a non-singular system for the basic invariant polynomials.  After
solving that system, one returns to the closed chamber via the inverse
invariant map and verifies that the recovered process satisfies
\eqref{eq:main:SDE}.  The zero-occupation property discussed above then
follows from the wall positivity of $k_\alpha$ and the finiteness of the
corresponding singular drift integrals.

The proof also yields an additional geometric conclusion.  The crucial point
is that, in the strictly positive regime, the process cannot hit multiple
collision strata after time zero.  This property is used in the identification
of the singular drift, and, once Theorem~\ref{thm:existence} is proved, it
becomes a property of the constructed solution.  Due to its importance, we
state it separately.

\begin{theorem}[Absence of positive-time multiple collisions]
    \label{thm:no-multiple-collisions-positive}
    Let $x_0\in\Cc$.  Assume \ref{ass:cont}, \ref{ass:sigma}, and \ref{ass:k:positive}.  Then the weak solution $X$ constructed in Theorem~\ref{thm:existence} has no multiple collisions at positive times.  More precisely, for every two different roots $\alpha,\beta \in R_+$ such that $\sum_{i=1}^\no \alpha_i^2\beta_i^2>0$ we have 
        \begin{equation*}
            \inner{X(t)}{\alpha}+\inner{X(t)}{\beta}>0, \qquad 0<t<\life{X}\qquad \textrm{a.s.} 
        \end{equation*}
\end{theorem}

Since $X(t)\in\Cc$, all quantities $\inner{X(t)}{\alpha}$ are non-negative. The condition $\sum_{i=1}^{\no}\alpha_i^2\beta_i^2>0$
means that the two roots $\alpha$ and $\beta$ involve at least one common coordinate. Thus the conclusion says that two such walls cannot be hit simultaneously after time zero. This includes, in particular, all pairs of non-orthogonal positive roots. In type $\AN$, it implies that three or more particles cannot collide at a positive time. In types $\BN$ and $\DN$, it also includes the orthogonal zero-pair $\alpha=e_i-e_j$, $\beta=e_i+e_j$ and therefore excludes the simultaneous vanishing of $X_i-X_j$ and $X_i+X_j$
at a positive time.

This should be distinguished from simultaneous collisions of genuinely orthogonal type with disjoint coordinate supports. It may still be possible that two distinct roots $\alpha,\beta\in R_+$ with $\sum_{i=1}^{\no}\alpha_i^2\beta_i^2=0$, satisfy
\[
    \inner{X(t)}{\alpha}=\inner{X(t)}{\beta}=0
\]
at the same time. Related results excluding such simultaneous collisions for a different class of
interacting particle systems were obtained in
\cite{bib:AndrausHufnagelMalecki:2025}.


\subsection{Degenerate repulsion}
\label{subsec:degenerate-repulsion}

We next allow the singular repulsion to vanish on the walls. This generalization lets us treat a substantially wider class of particle systems, including squared Bessel particle systems and related Wishart-type examples, in which the repulsion coefficients may vanish at specific boundary points.

\medskip

We begin by introducing two assumptions that replace Assumption~\ref{ass:k:positive}. Both are natural: their role is to ensure that there is still enough repulsion at the boundary even when the coefficient of the singular term vanishes. The formulation of the second assumption requires the introduction of the collision faces and the associated detectors. More precisely, for $y\in\Cb$, set
    \begin{align*}
        S(y) &:= \{\alpha\in R_+:\langle y,\alpha\rangle=0\}.
    \end{align*}
Let $\mathfrak S$ be the family of all non-empty sets of the form $S(y)$, and for $S\in\mathfrak S$ define
    \begin{align}
        F_S &:= \{x\in\Cc: \langle x,\alpha\rangle=0 \text{ for }\alpha\in S,\  \langle x,\beta\rangle>0 \text{ for }\beta\in R_+\setminus S\}.
    \end{align}
Thus $F_S$ is the relatively open part of the boundary on which exactly the roots in $S$ vanish. A vector $u\in\operatorname{span}S$ is called an admissible detector for $F_S$ if $\tau_u(x):=\langle u,x\rangle$ satisfies
    \begin{align*}
        \tau_u(x)&\ge0, &&x\in\Cc, \\
        \tau_u(x)&=0, &&x\in F_S,\\
        \langle u,\alpha\rangle&\ge0, &&\alpha\in S.
    \end{align*}
The roots of the face crossed by $u$ are collected in
    \begin{align*}
        \Gamma(S,u) &:= \{\alpha\in S:\langle u,\alpha\rangle>0\}.
    \end{align*}
For every $S\in\mathfrak S$, we fix a finite family $\mathfrak D(S)$ of admissible detectors. We choose this family so that it contains the canonical detector $u_S:=\sum_{\beta\in S}\beta$, and such that each $\alpha\in S$ belongs to $\Gamma(S,u)$ for at least one $u\in\mathfrak D(S)$. For an explicit example, consider type $A_{N-1}$ with $\C=\{x\in\R^N:x_1>\cdots>x_N\}$, and the face
    \begin{align*}
        F_S = \{x\in\Cc:x_1>x_2=x_3=x_4>x_5>\cdots>x_N\}
    \end{align*}
corresponding to
    \begin{align*}
        S = \{e_2-e_3,\ e_2-e_4,\ e_3-e_4\}.
    \end{align*}
An admissible detector for this face is $\tau_u(x)=x_2-x_4=\langle e_2-e_4,x\rangle$: it is nonnegative on $\Cc$, vanishes on $F_S$, and is strictly sensitive to each root in $S$. One may therefore take $\mathfrak D(S) = \{e_2-e_4\}$, because then $\Gamma(S,e_2-e_4)=S$. Finally, for $u\in\mathfrak D(S)$, define the regular detector drift by
    \begin{align}
        B_{S,u}(t,x) &:= \langle u,b(t,x)\rangle + \sum_{\beta\in R_+\setminus S}\frac{k_\beta(t,x)\langle u,\beta\rangle}{\langle x,\beta\rangle}.
        \label{eq:def:B-S-u-degenerate}
    \end{align}
Only roots not vanishing on the face appear in $B_{S,u}$. Hence, after localizing away from neighboring faces, the denominators in \eqref{eq:def:B-S-u-degenerate} stay bounded away from zero, and $B_{S,u}$ admits a continuous trace on $F_S$. Intuitively, $B_{S,u}$ is the regular part of the drift seen in the detector direction.

\medskip

\begin{enumerate}[label=(D\arabic*), ref=(D\arabic*)]
    \item \label{ass:k:dom}
    \emph{Normal dominance near the wall.}
    For every $T,R>0$ there exist constants
    $\varepsilon=\varepsilon(T,R)>0$ and $\gamma=\gamma(T,R)>0$ such that, for every $t\in[0,T]$, every $x\in\Cc$ with $|x|\le R$, and every $\alpha\in R_+$,
    \begin{align}
        0\le\langle x,\alpha\rangle\le\varepsilon
            &\Longrightarrow
            k_\alpha(t,x)\ge \gamma  \sum_{i=1}^N \alpha_i^2 \sigma_i^2(t,x).
        \label{eq:ass:k-dom}
    \end{align}

    \item \label{ass:face:sign}
    \emph{Face sign condition.} For every $S\in\mathfrak S$ and every $u\in\mathfrak D(S)$, the continuous trace of the regular detector drift satisfies
        \begin{align}
            B_{S,u}(t,x) &\ge0, \qquad t\ge0, \qquad x\in F_S.
            \label{eq:ass:face-sign}
        \end{align}
\end{enumerate}

Since $\sum_{i=1}^N \alpha_i^2 \sigma_i^2(t,x)$ is the quadratic variation density of the martingale part in the normal direction $\alpha$, Assumption~\ref{ass:k:dom} says that when $x$ approaches the wall $\{\langle x,\alpha\rangle=0\}$, the repulsion coefficient $k_\alpha(t,x)$ may vanish, but not faster than the corresponding normal martingale variation. In this sense, the possible loss of repulsion is compensated by a comparable loss of noise in the same direction. This is the Bessel-type balance that is later used to show that the relevant detector local times vanish. It also makes the framework flexible enough to cover, in particular, $\beta$-versions of known particle systems also those with vanishing coefficients.
 
Assumption~\ref{ass:face:sign} is a weak inward-pointing condition on the regular part of the detector drift. It is used to eliminate finite-variation measures which could otherwise be supported on the boundary. Together, these two assumptions yield the general degenerate existence result below. By themselves, however, they do not force the process to spend zero Lebesgue time on the boundary.

\begin{theorem}[Existence in the degenerate case]
    \label{thm:degenerate-existence}
    Let $x_0\in\Cc$. Assume \ref{ass:cont}, \ref{ass:sigma}, \ref{ass:k:dom}, and \ref{ass:face:sign}. Then there exists a filtered probability space carrying an $N$-dimensional Brownian motion $B$ and a continuous $\Cc$-valued semimartingale $X=(X_1,\ldots,X_N)$, defined up to its lifetime $\life{X}>0$, such that 
        \begin{align}
            X(t) &= x_0 + \int_0^t \sigma(s,X(s))\,dB(s) + \int_0^t b(s,X(s))\,ds + \sum_{\alpha\in R_+}\alpha
            \int_0^t \frac{k_\alpha(s,X(s))}{\langle X(s),\alpha\rangle} \mathbf 1_{\{\langle X(s),\alpha\rangle>0\}}\,ds\/,
        \label{eq:degenerate-existence-indicator}
    \end{align}
    for every $t\in[0,\life{X})$. In particular, for every $\alpha\in R_+$ the singular integrals in \eqref{eq:degenerate-existence-indicator} are finite on compact subintervals of $[0,\life{X})$.
\end{theorem}

Theorem~\ref{thm:degenerate-existence} is the basic existence statement in the degenerate regime. In contrast to the strictly positive case, it does not exclude the possibility that the process spends a positive amount of Lebesgue time on the boundary. In the present framework, such behavior may occur only at genuinely degenerate boundary points, that is, at points where the corresponding coefficient $k_\alpha$ vanishes. By Assumption~\ref{ass:k:dom}, this then forces the normal quadratic variation coefficient $a_\alpha(t,x)$ to vanish as well. Thus any possible sticking can only occur at boundary points where both the singular repulsion and the normal martingale variation degenerate simultaneously.

The indicator in \eqref{eq:degenerate-existence-indicator} means that the singular drift is interpreted through its interior density. This is natural in the degenerate case, because when both $k_\alpha(t,x)$ and $\langle x,\alpha\rangle$ vanish, the ratio $k_\alpha(t,x)/\langle x,\alpha\rangle$ need not have a canonical boundary value and, in general, need not admit a continuous extension to the wall. For this reason, the interior formulation with an indicator is the natural object to consider. A formulation of this type already appears in \cite{bib:GraczykMalecki:2019}, where squared Bessel particle systems are considered.

Finally, any genuine solution of the equation without the indicator is automatically a solution of \eqref{eq:degenerate-existence-indicator} whenever it spends zero Lebesgue time on every root wall, because the two drift densities then differ only on a Lebesgue-null set of times. The converse need not hold in general: the indicator formulation may admit sticky behavior at degenerate walls. This is why removing the indicator requires an additional non-sticking input, which we isolate in the next subsection.


\subsection{A non-sticking criterion for degenerate walls}
\label{subsec:removing-indicator}

Theorem~\ref{thm:degenerate-existence} gives a solution of the degenerate equation in its natural interior form. To pass to the formulation without the indicator, one has to exclude sticking on the boundary. We impose a structural face-by-face escape condition expressed in terms of admissible detectors and the quadratic-variation densities of the corresponding detector martingales. This is the form that seems most useful in applications.

\medskip

\begin{enumerate}[label=(D\arabic*), ref=(D\arabic*)]
    \setcounter{enumi}{2}
    \item \label{ass:nonsticky}
    \emph{Facewise escape at fully degenerate points.}
    For every collision face $F_S$, every $t\ge0$, and every $x\in F_S$,
    \begin{align}
        \max_{u\in\mathfrak D(S)} \sum_{i=1}^N u_i^2 \sigma_i^2(t,x)=0
            &\Longrightarrow
            \max_{u\in\mathfrak D(S)} B_{S,u}(t,x)>0.
        \label{eq:ass:nonsticky}
    \end{align}
\end{enumerate}

Assumption~\ref{ass:nonsticky} concerns only the genuinely degenerate situation in which all admissible detectors have zero quadratic variation at the face point. If for some admissible detector $u$ the quantity
\begin{align*}
    \sum_{i=1}^N u_i^2 \sigma_i^2(t,x)
\end{align*}
is positive, then that detector already provides the stochastic escape mechanism needed in the local argument. Assumption~\ref{ass:nonsticky} says that in the complementary case, when all detector martingale variations vanish, at least one detector must still have strictly positive regular drift.

This condition is closely related in spirit to the non-degeneration condition used by Graczyk and Małecki \cite{bib:GraczykMalecki:2014}. In their one-dimensional ordered-particle setting, when the pairwise martingale part and the pairwise repulsion vanish at a degenerate collision point, one assumes that the remaining drift of the colliding block does not vanish. Here the same idea is formulated geometrically on each collision face: if every admissible detector loses its martingale variation, then some detector must still point strictly away from the face. In the Wishart-type examples, this role is played by a block-center detector on a zero block.

\begin{theorem}[Degenerate existence with non-sticking boundary behavior]
    \label{thm:degenerate-remove-indicator}
    Let $x_0\in\Cc$. Assume \ref{ass:cont}, \ref{ass:sigma}, \ref{ass:k:dom}, \ref{ass:face:sign}, and \ref{ass:nonsticky}. Then there exists a filtered probability space carrying an $N$-dimensional Brownian motion $B$ and a continuous $\Cc$-valued semimartingale $X=(X_1,\ldots,X_N)$, defined up to its lifetime $\life{X}>0$, such that $X(0)=x_0$ and, for every $t\in[0,\life{X})$,
    \begin{align}
        X(t) &= x_0 + \int_0^t \sigma(s,X(s))\,dB(s) + \int_0^t b(s,X(s))\,ds + \sum_{\alpha\in R_+}\alpha \int_0^t \frac{k_\alpha(s,X(s))}{\langle X(s),\alpha\rangle}\,ds.
        \label{eq:degenerate-existence-without-indicator}
    \end{align}
    The singular integrals in \eqref{eq:degenerate-existence-without-indicator} are finite on compact subintervals of $[0,\life{X})$, with the usual improper interpretation at the initial time if $x_0\in\Cb$.
\end{theorem}

By the discussion above, at each collision face there is an escape mechanism: either some admissible detector retains non-trivial quadratic variation, or, in the fully degenerate case, Assumption~\ref{ass:nonsticky} provides a strictly positive detector drift. Consequently, the solution from Theorem~\ref{thm:degenerate-existence} spends zero Lebesgue time on every collision face, hence also on every root wall $\{x\in\Cc:\langle x,\alpha\rangle=0\}$, $\alpha\in R_+$ and the indicator in \eqref{eq:degenerate-existence-indicator} becomes irrelevant.

The indicator formulation in Theorem~\ref{thm:degenerate-existence} is nevertheless the natural starting point in the general degenerate setting. Indeed, when both $k_\alpha(t,x)$ and $\langle x,\alpha\rangle$ vanish, the ratio ${k_\alpha(t,x)}/{\langle x,\alpha\rangle}$ need not have a canonical boundary value and, in general, need not admit a continuous extension to the wall. Theorem~\ref{thm:degenerate-remove-indicator} shows that under the additional facewise escape condition this ambiguity is harmless, because the boundary is then visited only on a Lebesgue-null set of times.


\subsection{Non-explosion under radial linear growth}
\label{subsec:non-explosion}

The lifetime in the preceding results is included because, up to this point, no growth condition has been imposed on the coefficients. A standard radial linear-growth assumption removes the possibility of an explosion. 

\medskip

\begin{enumerate}[label=(G\arabic*), ref=(G\arabic*)]
    \item \label{ass:growth:p2}
    For every $T>0$ there exists a constant $C_T<\infty$ such that, for every
    $t\in[0,T]$ and every $x\in\Cc$,
    \begin{align}
        \sum_{i=1}^N x_i b_i(t,x)
        +
        \frac12\sum_{i=1}^N \sigma_i^2(t,x)
        +
        \sum_{\alpha\in R_+} k_\alpha(t,x)
            &\le
        C_T\bigl(1+|x|^2\bigr).
        \label{eq:ass:growth:p2}
    \end{align}
\end{enumerate}

Assumption~\ref{ass:growth:p2} is the usual radial linear-growth condition for the Lyapunov function $|x|^2$. It controls at once the contribution of the drift, the It\^o correction coming from the diffusion part, and the radial contribution of the singular repulsion. This is the analogue of the non-explosion condition used by Graczyk and Małecki \cite{bib:GraczykMalecki:2014}. We use here the radial bound in the form involving $|x|^2$; in the formulation of condition (C2) in \cite{bib:GraczykMalecki:2014}, the upper bound for the interaction function $H$ is written with a term of the form $|xy|$, but the estimate needed in the non-explosion argument is the radial one, namely a bound by a constant times $1+|x|^2+|y|^2$. These two bounds are not equivalent in the required direction.

\begin{theorem}[Non-explosion]
    \label{thm:non-explosion}
    Assume \ref{ass:growth:p2}. Let $X$ be any $\Cc$-valued weak solution, defined up to its explosion lifetime, satisfying either the singular equation \eqref{eq:thm:existence-sde} or the degenerate interior equation \eqref{eq:degenerate-existence-indicator}. Then
    \begin{align}
        \life{X}
            &=
        \infty,
        \qquad\text{a.s.}
        \label{eq:non-explosion-conclusion}
    \end{align}
    Consequently, under Assumption~\ref{ass:growth:p2}, the weak and strong solutions provided by Theorems~\ref{thm:existence}, \ref{thm:degenerate-existence}, \ref{thm:degenerate-remove-indicator}, and \ref{thm:strong:existence:uniqueness} are global in time.
\end{theorem}


\subsection{Uniqueness}
\label{subsec:uniqueness}

The assumptions used in the existence theorems are not sufficient to imply uniqueness. We therefore introduce a separate set of hypotheses tailored to pathwise uniqueness. These assumptions are natural from two complementary viewpoints. On the one hand, they extend the standard one-dimensional uniqueness assumptions from the literature. On the other hand, they are adapted to the root-system geometry underlying the particle representation of the Bessel-Dunkl diffusions considered here.

\medskip

\begin{enumerate}[label=(U\arabic*), ref=(U\arabic*)]
    \item \label{ass:uniq:sigma}
    For every $T,r>0$, at least one of the following two alternatives holds:
    \begin{enumerate}[label=(\alph*), ref=(\alph*)]
        \item there exists an increasing function
        $\rho=\rho_{T,r}:\R_+\to\R_+$ such that
        \begin{align}
            \rho(0)&=0,
            &
            \rho(u)&>0 \quad\text{for }u>0,
            &
            \int_{0+}\frac{du}{\rho(u)}&=\infty,
            \label{eq:uniq:rho}
        \end{align}
        and, for every $i=1,\ldots,N$,
        \begin{align}
            |\sigma_i(t,x)-\sigma_i(t,y)|^2
                &\le
                \rho(|x_i-y_i|)
            \label{eq:uniq:sigma:yw}
        \end{align}
        whenever $t\in[0,T]$, $x,y\in\Cc$, and $|x|\vee |y|\le r$;

        \item there exists a constant $L_{T,r}<\infty$ such that, for every $i=1,\ldots,N$,
        \begin{align}
            |\sigma_i(t,x)-\sigma_i(t,y)|
                &\le
                L_{T,r}|x-y|
            \label{eq:uniq:sigma:lip}
        \end{align}
        whenever $t\in[0,T]$, $x,y\in\Cc$, and $|x|\vee |y|\le r$.
    \end{enumerate}

    \item \label{ass:uniq:b}
    For every $T,r>0$ there exists a constant $L_{T,r}<\infty$ such that
    \begin{align}
        \sum_{i=1}^N
        \operatorname{sgn}(x_i-y_i)
        \bigl(b_i(t,x)-b_i(t,y)\bigr)
            &\le
            L_{T,r}\sum_{i=1}^N |x_i-y_i|
        \label{eq:uniq:b}
    \end{align}
    whenever $t\in[0,T]$, $x,y\in\Cc$, and $|x|\vee |y|\le r$.


    \item \label{ass:uniq:k}
For $x\in\C$, define the singular force
\begin{equation}
    \label{eq:uniq:singular-force-G}
    G(t,x)
    :=
    \sum_{\alpha\in R_+}
    \alpha\,
    \frac{k_\alpha(t,x)}{\langle x,\alpha\rangle}.
\end{equation}
For every $T,r>0$, every $t\in[0,T]$, and all $x,y\in\C$ with
$|x|\vee |y|\le r$,
\begin{equation}
    \label{eq:uniq:k:l1-dissipative}
    \sum_{i=1}^N
    \operatorname{sgn}(x_i-y_i)
    \bigl(G_i(t,x)-G_i(t,y)\bigr)
    \le0.
\end{equation}
\end{enumerate}
We shall also use the following Euclidean alternatives in the locally Lipschitz
diffusion regime.

\medskip
\begin{enumerate}[label=(U\arabic*'), ref=(U\arabic*')]
\setcounter{enumi}{1}
\item \label{ass:uniq:b:euclidean}
For every $T,r>0$ there exists a constant $L_{T,r}<\infty$ such that
\begin{equation}
    \langle x-y,b(t,x)-b(t,y)\rangle
    \le
    L_{T,r}|x-y|^2
\end{equation}
whenever $t\in[0,T]$, $x,y\in\Cc$, and $|x|\vee |y|\le r$.

\item \label{ass:uniq:k:dissipative}
For every $T,r>0$, every $t\in[0,T]$, and all $x,y\in\C$ with
$|x|\vee |y|\le r$,
\begin{equation}
    \langle x-y,G(t,x)-G(t,y)\rangle
    \le
    0.
\end{equation}
\end{enumerate}

The assumptions above split the uniqueness theory into two regimes. Assumption~\ref{ass:uniq:sigma}\textup{(a)} is the coordinatewise Yamada-Watanabe modulus condition and, in one dimension, reduces to the classical criterion allowing less than Lipschitz regularity in the martingale part.  It is the natural framework for square-root-type examples.  In this regime the proof is based on Tanaka's formula in the $\ell^1$-geometry.  Therefore it is paired with Assumption~\ref{ass:uniq:b}, the corresponding one-sided Lipschitz condition on the regular drift, and with the rootwise monotonicity condition Assumption~\ref{ass:uniq:k}.

Assumption~\ref{ass:uniq:sigma}\textup{(b)} is the standard local Lipschitz condition.  In this regime the proof is based on It\^o's formula for $|X-\widetilde X|^2$.  Therefore the appropriate drift condition is the Euclidean one-sided Lipschitz condition Assumption~\ref{ass:uniq:b:euclidean}.  The singular part may then be controlled either by the rootwise condition Assumption~\ref{ass:uniq:k}, or more generally by the global dissipativity condition Assumption~\ref{ass:uniq:k:dissipative}.

In the ordered-particle setting of Graczyk and Małecki \cite{bib:GraczykMalecki:2014}, alternative~\textup{(a)} in Assumption~\ref{ass:uniq:sigma} together with Assumption~\ref{ass:uniq:b} plays the role of their condition~\textup{(C1)}, while Assumption~\ref{ass:uniq:k} is the analogue of their condition~\textup{(A1)}.

\begin{remark}[Logarithmic potentials]
\label{rem:uniq:logarithmic-potentials}
A useful way to verify Assumption~\ref{ass:uniq:k:dissipative} is through a logarithmic potential.  Suppose that, for each $t$, there exists a positive function $h_t$ on $\C$ such that $\log h_t$ is concave and
    \begin{equation*}
        G(t,x)=\nabla\log h_t(x), \qquad x\in\C.
    \end{equation*}
Then Assumption~\ref{ass:uniq:k:dissipative} holds, because the gradient of a concave function is dissipative:
    \begin{equation*}
        \langle x-y,\nabla\log h_t(x)-\nabla\log h_t(y)\rangle\le0.
    \end{equation*}
Equivalently, if $G(t,x)=-\nabla\Phi_t(x)$ for a convex function $\Phi_t$, then Assumption~\ref{ass:uniq:k:dissipative} holds. In particular, the constant-multiplicity root-barrier force is obtained from
    \begin{equation*}
        h_t(x) = \prod_{\alpha\in R_+} \langle x,\alpha\rangle^{m_\alpha(t)}, \qquad m_\alpha(t)>0.
    \end{equation*}
Then
    \begin{equation*}
        \nabla\log h_t(x) = \sum_{\alpha\in R_+} m_\alpha(t) \frac{\alpha}{\langle x,\alpha\rangle},
    \end{equation*}
which corresponds to $k_\alpha(t,x)=m_\alpha(t)$.  This is the root-system analogue of the logarithmic electrostatic repulsion appearing in the Cépa-Lépingle framework.
\end{remark}

As already observed in \cite{bib:GraczykMalecki:2019}, equations with an indicator in the singular term may admit solutions of two genuinely different types when degenerate boundary points are present: sticky solutions, which may spend positive Lebesgue time on the boundary, and non-sticky solutions, for which the time spent on the boundary has Lebesgue measure zero. One therefore cannot expect pathwise uniqueness in the full class of solutions of \eqref{eq:degenerate-existence-indicator}. For this reason, we restrict attention to non-sticky solutions, that is, solutions satisfying
\begin{align}
    \int_0^t
    \mathbf 1_{\{\langle X(s),\alpha\rangle=0\}}\,ds
        &=
        0,
        \qquad
        \alpha\in R_+,
        \qquad
        t<\life{X},
        \qquad
        \text{a.s.}
    \label{eq:zero:leb}
\end{align}
In this class, the indicator is irrelevant, and a non-sticky solution of \eqref{eq:degenerate-existence-indicator} is equivalently a solution of \eqref{eq:main:SDE}.

\begin{theorem}[Pathwise uniqueness in the non-sticky class]
    \label{thm:pathwise-uniq-indicator-nonsticky}
    Assume that one of the following two uniqueness regimes holds.

    \begin{enumerate}[label=\textup{(\roman*)}]
        \item The Yamada-Watanabe alternative \textup{(a)} in
        Assumption~\ref{ass:uniq:sigma} holds, together with
        Assumptions~\ref{ass:uniq:b} and \ref{ass:uniq:k}.

        \item The locally Lipschitz alternative \textup{(b)} in
        Assumption~\ref{ass:uniq:sigma} holds, together with
        \ref{ass:uniq:b:euclidean} and Assumption~\ref{ass:uniq:k:dissipative}.
    \end{enumerate}
    Let $(X,B)$ and $(\widetilde X,B)$ be two solutions of
    \eqref{eq:degenerate-existence-indicator}, defined on the same filtered
    probability space, driven by the same Brownian motion, and with the same
    initial condition. Assume moreover that \eqref{eq:zero:leb} holds for both
    $X$ and $\widetilde X$. Then
    \begin{align}
        X(t) &= \widetilde X(t), \qquad 0\le t<\life{X}\wedge\life{\widetilde X}, \qquad \text{a.s.}
        \label{eq:pathwise-uniq-conclusion}
    \end{align}
    In particular,
    \begin{align*}
        \life{X} &= \life{\widetilde X}, \qquad\text{a.s.}
    \end{align*}
\end{theorem}

Under alternative~(a) in Assumption~\ref{ass:uniq:sigma}, the proof follows the classical Tanaka-Gronwall strategy: the coordinatewise local times vanish, Assumption~\ref{ass:uniq:b} controls the contribution of the regular drift, and Assumption~\ref{ass:uniq:k} makes the singular term non-positive in the $\ell^1$ estimate. The treatment of alternative~(b) is different and will be handled separately.

Again, every solution of \eqref{eq:main:SDE} satisfying \eqref{eq:zero:leb} is also a solution of \eqref{eq:degenerate-existence-indicator}. Therefore, combining the weak existence results with Theorem~\ref{thm:pathwise-uniq-indicator-nonsticky} yields strong existence and uniqueness by the Yamada-Watanabe principle, after localization by the explosion time. In the degenerate regime this is done only after Theorem~\ref{thm:degenerate-remove-indicator}, because only then do we have a weak solution of \eqref{eq:main:SDE} in the non-sticky class.

\begin{theorem}[Strong existence and uniqueness]
    \label{thm:strong:existence:uniqueness}
    Let $\nu$ be a probability measure on $\Cc$.  Assume \ref{ass:cont} and
    \ref{ass:sigma}.  Assume moreover that one of the two uniqueness regimes in
    Theorem~\ref{thm:pathwise-uniq-indicator-nonsticky} holds.  Suppose also that
    one of the following alternatives holds:
    \begin{enumerate}[label=(S\arabic*), ref=(S\arabic*)]
        \item \label{ass:exist:positive}
        Assumption~\ref{ass:k:positive} holds.

        \item \label{ass:exist:degenerate}
        Assumptions~\ref{ass:k:dom}, \ref{ass:face:sign}, and
        \ref{ass:nonsticky} hold.
    \end{enumerate}
    If $\xi$ is a $\Cc$-valued random variable with law $\nu$, independent of an
    $N$-dimensional Brownian motion $B$, then \eqref{eq:main:SDE} has a strong
    solution $X$ with $X(0)=\xi$, adapted to the completed filtration generated by
    $\xi$ and $B$, up to its lifetime.  This solution is pathwise unique in the
    class of solutions satisfying \eqref{eq:zero:leb}, and uniqueness in law holds
    in the same class.
\end{theorem}


\subsection{Mean-field convergence}
\label{subsec:mean-field-convergence}

One of the most significant results in the theory of particle systems concerns the convergence of systems in the case where the number of particles tends to infinity. The starting point for studying phenomena of this type is always provided by theorems concerning the existence and uniqueness of solutions to stochastic differential equations. Therefore, the previously presented results, as well as their generality, make it possible to explore this part of the theory in a broad context.

Below we present one such result concerning mean-field convergence. In this case, however, we restrict the class of particles under consideration to those for which the martingale and regular drift coefficients depend only on one coordinate, while the repulsion coefficient attached to a pair root depends only on the two particles involved in this interaction. This is precisely the structure which makes the limiting integro-differential equation for the normalized empirical measures closed. Nevertheless, we retain the dependence of all coefficients on the time variable, and the model still covers most examples considered in the literature.

Fix a type $\mathfrak r\in\{A,B,D\}$ and set $E_A = \R$, $E_B=E_D=[0,\infty)$ together with $\eta_A^0 =\eta_D^0=0$, $\eta_B^0  =1$. We will drop the subscript related to the root type in the notation below and write simply $E$ in all cases. In the $\DN$ case, under our chamber convention, only the last coordinate may be negative. We assume that the martingale and drift coefficients are functions of the form
    \begin{equation*}
        \sigma_i(t,x) = \sigma(t,x_i)\/,\qquad b_i(t,x) = b(t,x_i)
    \end{equation*}
with a fairly clear abuse of notation, and $\sigma,b:[0,\infty)\times E \to \R$. The coefficient describing the singular behavior is of the form
    \begin{equation}
        k_\alpha(t,x) = k(t,x_i,x_j)\/,\quad \textrm{for }\alpha  = e_i\pm e_j\/,\qquad k_\alpha(t,x) = k(t,x_i)\/,\quad \textrm{for }\alpha = e_i\/.
        \label{eq:kalpha:special}
    \end{equation}
Thus, whenever we write $k_\alpha$ below we understand it to be in the form given in \eqref{eq:kalpha:special}, with the simplified convention that whenever $k$ has two spatial arguments, then it corresponds to the pair roots, and with only one spatial argument, whenever we are dealing with short roots. It is clear that $k(t,x,y)=k(t,y,x)$ in this setting. In the $\DN$ case, if the last coordinate is involved in the evaluation of a coefficient, we use its absolute value in the argument of $\sigma$, $b$, and $k$.

Additionally, we consider the set of test functions $D = C_b^2(\R)$, whenever we work in the $\AN$ case and we get $D = \{f\in C_b^2([0,\infty)):\ f'(0)=0\}$ in the cases $\BN$ and $\DN$. Note that although for $\DN$ one of the particle might be negative, since we consider the limit with number of particles growing to infinity, it is irrelevant and the limiting measures have supports on non-negative half-line. Finally, we define
    \begin{align*}
            D_f^{-}(x,y) &:=
                \begin{cases}
                    \dfrac{f'(x)-f'(y)}{x-y}, & x\neq y,\\[1.2ex]
                    f''(x), & x=y,
                \end{cases}
    \end{align*}
and additionally for $x,y\geq 0$ we set
    \begin{align*}
        D_f^{+}(x,y) &:=
            \begin{cases}
                \dfrac{f'(x)+f'(y)}{x+y}, & x+y>0,\\[1.2ex]
                f''(0), & x=y=0.
            \end{cases}
    \end{align*}
Using those two we can cover all three considered case by writing $D_f(x,y) = D_f^{-}(x,y)$ in $\AN$ case and $D_f(x,y) = D_f^{-}(x,y)+D_f^{+}(x,y)$ in the remaining two cases. 

\medskip

For each $N$, let $X^{(N)}=(X_1^{(N)},\ldots,X_N^{(N)})$ be a continuous solution, taking values in the appropriate closed chamber, of the system
\begin{align}
    dX_i^{(N)}(t)
        &=
        \frac{1}{\sqrt N}\sigma\bigl(t,X_i^{(N)}(t)\bigr)\,dB_i(t)
        +
        b\bigl(t,X_i^{(N)}(t)\bigr)\,dt
        +
        \sum_{\alpha\in R_+}
        \alpha_i
        \frac{
            k_{\alpha}^{(N)}\bigl(t,X^{(N)}(t)\bigr)
        }{
            \langle X^{(N)}(t),\alpha\rangle
        }\,dt,
    \label{eq:mf:root-sde}
\end{align}
for $i=1,\ldots,N$
where $k_{e_i\pm e_j}^{(N)}(t,x) = \frac1N k(t,x_i,x_j)$ and $k_{e_i}^{(N)}(t,x) = k(t,x_i)$, so the additional scaling is not present when we are considering the short roots. The corresponding empirical measure is defined as
    \begin{align*}
        \mu_t^{(N)} &:= \frac1N \sum_{i=1}^N \delta_{X_i^{(N)}(t)}
    \end{align*}
in the $\AN$ and $\BN$ cases. In the $\DN$ case, since only the last coordinate may be negative, we define instead
    \begin{align*}
        \mu_t^{(N)}
        &:=
        \frac1{N-1}
        \sum_{i=1}^{N-1}
        \delta_{X_i^{(N)}(t)}.
\end{align*}
This modification has no effect on the limiting measure as $N\to\infty$, and it makes $\mu_t^{(N)}$ supported on $E_D=[0,\infty)$.

Since in our framework the coefficients $\sigma$ and $k$ are not tied together by any specific matrix structure, one may also study different scalings. For instance, one may replace $N^{-1/2}$ by $\sqrt{\varepsilon_N}$ in the martingale part and $N^{-1}$ by $q_N/N$ in the pair interaction, where $\varepsilon_N\ge0$ and $q_N\ge0$ are given convergent sequences. The classical regime considered above corresponds to $\varepsilon_N=1/N$ and $q_N\equiv1$. More general choices lead to different limiting equations and can be treated in exactly the same way.

\medskip

We impose the following assumptions on every finite time interval.

\begin{enumerate}[label=(MF\arabic*), ref=(MF\arabic*)]
    \item \label{ass:mf:exist}
    For every $N$, the system \eqref{eq:mf:root-sde} has a continuous solution on compact time intervals in the appropriate closed chamber, and all singular drift terms in \eqref{eq:mf:root-sde} are integrable on compact time intervals.

        \item \label{ass:mf:coeff}
    The functions $\sigma$, $b$, and the one- and two-variable versions of $k$ are continuous on their respective domains. Moreover, the two-variable function $k$ is symmetric, non-negative, and for every $T>0$ there exists $C_T<\infty$ such that, for all $t\in[0,T]$ and all admissible $x,y$,
    \begin{align}
        |\sigma(t,x)|^2 + |b(t,x)|^2 + k(t,x)
            &\le
            C_T(1+|x|^2), \qquad x\in E\/,
        \label{eq:mf:growth-sigma}
        \\
        k(t,x,y)
            &\le
            C_T(1+|x|)(1+|y|)\/,\qquad x,y \in E\/.
        \label{eq:mf:growth-k}
    \end{align}
    The one-variable function $k(t,x)$ is present only in the $\BN$ case.

    \item \label{ass:mf:initial}
    The initial empirical measures converge in probability to a deterministic probability measure $\mu_0\in\mathcal P(E)$, and
    \begin{align*}
        \sup_{N\ge1}
        \E \int_E |x|^8 \mu_0^{(N)}(dx)<\infty.
    \end{align*}
\end{enumerate}

\begin{theorem}[Mean-field convergence]
    \label{thm:mf:classical}
    Assume \ref{ass:mf:exist}-\ref{ass:mf:initial}. Then, for every $T>0$, the laws of $(\mu_t^{(N)})_{0\le t\le T}$ are tight in $C([0,T],\mathcal P(E_{\mathfrak r}))$. Moreover, every weak limit point $(\mu_t)_{0\le t\le T}$ satisfies, for every $f\in D$ and every $t\in[0,T]$,
    \begin{align}
        \left\langle\mu_t,f\right\rangle &= \left\langle\mu_0,f\right\rangle +\int_0^t \int_E b(s,x)f'(x)\mu_s(dx)\,ds +  \eta_{\mathfrak r}^0 \int_0^t \int_E \frac{k(s,x) f'(x)}{x}\mu_s(dx)\,ds \nonumber \\
        &\qquad + \frac12 \int_0^t \int_{E^2} D_f(x,y)\, k(s,x,y)\,\mu_s(dx)\mu_s(dy)\,ds
        \label{eq:mf:limit-equation}
    \end{align}
    The quotient $k(s,x)f'(x)/x$ is understood at $x=0$ as $k(s,0)f''(0)$.  In particular, every subsequential limit solves \eqref{eq:mf:limit-equation}. If \eqref{eq:mf:limit-equation} has a unique solution in $C([0,T],\mathcal P(E))$, then
    \begin{align*}
        \mu^{(N)}
            \Longrightarrow
            \mu
            \qquad\text{in }C([0,T],\mathcal P(E)).
    \end{align*}
\end{theorem}

The same argument extends immediately to the case in which the one-particle and two-particle coefficients depend on $N$, provided they satisfy the same growth bounds uniformly in $N$ and converge locally uniformly on compact sets. This is exactly the pattern of the general convergence theorem in Małecki and Pérez \cite{bib:MaleckiPerez:2022}. Since no new idea is needed here, we omit the proof and leave it to the reader.

\begin{corollary}[Mean-field convergence with converging coefficients]
    \label{cor:mf:coeff-convergence}
    Let the finite systems be given by \eqref{eq:mf:root-sde}, with the same $1/N$ mean-field scaling, but with $\sigma,b,k$ replaced by coefficients $\sigma_N,b_N,k_N$. Assume \ref{ass:mf:exist} and \ref{ass:mf:initial}. Assume also that $k_N$ is symmetric, that all coefficients satisfy the growth bounds \eqref{eq:mf:growth-sigma}-\eqref{eq:mf:growth-k} uniformly in $N$, and that, for every $T,R>0$,
    \begin{align}
        \lim_{N\to\infty}
        \sup_{\substack{0\le t\le T\\ |x|\le R}}
        \left|\sigma_N^2(t,x)-\sigma^2(t,x)\right|
            &=
            0,
        \label{eq:mf:sigmaN-convergence}\\
        \lim_{N\to\infty}
        \sup_{\substack{0\le t\le T\\ |x|\le R}}
        |b_N(t,x)-b(t,x)|
            &=
            0,
        \label{eq:mf:bN-convergence}\\
        \lim_{N\to\infty}
        \sup_{\substack{0\le t\le T\\ |x|\le R,\ |y|\le R}}
        |k_N(t,x,y)-k(t,x,y)|
            &=
            0,
        \label{eq:mf:kN-convergence}\\
        \lim_{N\to\infty}
        \sup_{\substack{0\le t\le T\\ 0\le x\le R}}
        |k_N(t,x)-k(t,x)|
            &=
            0.
        \label{eq:mf:cN-convergence}
    \end{align}
    Then the conclusions of Theorem~\ref{thm:mf:classical} remain valid, with the same limiting equation \eqref{eq:mf:limit-equation}.
\end{corollary}


\section{Preliminaries and notations}
\label{sec:preliminaries}

\subsection{Root systems}

Let $E$ be the Euclidean state space. We shall use finite reduced root systems, possibly in a non-essential ambient realization. More precisely, for a finite set $R\subset E\setminus\{0\}$ we put $E_R:=\operatorname{span}R$.We assume that $R$, viewed as a subset of the Euclidean space $E_R$, is a reduced root system, that is, it satisfies the following properties:
    \begin{enumerate}
        \item $R$ is finite and spans $E_R$.
        \item If $\alpha \in R$, then the only multiples of $\alpha$ in $R$ are $\alpha$ and $-\alpha$.
        \item For every $\alpha \in R$, the reflection $s_{\alpha}$ preserves $R$.
        \item For any two roots $\alpha, \beta \in R$, the quantity $\frac{2\inner{\beta}{\alpha}}{\inner{\alpha}{ \alpha}}$ is an integer.
    \end{enumerate}
The case $E_R=E$ is the essential realization. If $E_R\neq E$, the roots are regarded as vectors in the larger state space $E$, and the orthogonal complement $E_R^\perp$ is left unconstrained by the root system. The reflections are then understood as the corresponding orthogonal reflections in $E$, acting trivially on $E_R^\perp$.

The reflection $s_{\alpha}$ across the hyperplane perpendicular to a root $\alpha \in R$ is given by the formula: 
        \begin{equation}
            \label{eq:defn:reflection}
            s_{\alpha}(x) = x - \frac{2\inner{x}{\alpha}}{|\alpha|^2} \alpha\/,\quad x\in E\/.
        \end{equation}
We denote by $R_+$ the collection of positive roots. We specify those roots in the following with regard to the individual root systems we are considering. In general, the set of roots is divided into positive and negative roots in such a way that for given root $\alpha\in R$ it is either positive or negative, and then $-\alpha$ is then negative or positive, respectively. However, when fixed $\alpha,\beta \in R$ there exists a unique choice of sign such that $\pm s_\alpha(\beta)$ is in $R_+$. We denote the resulting root ($\alpha$ reflected in $\beta$ then a chosen sign) as $\prf{\beta}{\alpha}$. Then we immediately obtain 
        \begin{equation*}
            \ainvbb = \mp\ab\/,\quad \xainvb = \pm\left(\xa-\frac{2\ab\xb}{|\beta|^2}\right)\/,
        \end{equation*}
    and consequently, we arrive at one of the most crucial relations, which we will exploit very often in the sequel
        \begin{equation}
            \label{eq:inv:master}
            \ab\xb+\ainvbb\xainvb  = \frac{2\ab^2\xb}{|\beta|^2}\/.
        \end{equation}

    The \textbf{positive Weyl chamber}, denoted $\C$, is the region of $E$ where vectors have a positive inner product with all simple roots. It is given by:
        \begin{equation*}
            \C = \{ x \in E : \inner{x}{\alpha} > 0 \text{ for all } \alpha \in R_+\}\/.
        \end{equation*}

A root system is said to be \textbf{irreducible} if it cannot be partitioned into two nonempty, mutually orthogonal subsets. We will consider the finite classical root systems $\AN$, $\BN$, and $\DN$. Given the general form of the root systems we are considering, types $\BN$ and $\CN$ lead to exactly the same systems of stochastic differential equations \eqref{eq:main:SDE:cord}, and therefore we do not consider type $\CN$ separately, since it is fully covered by the $\BN$ case.

Throughout the paper the stochastic process is $N$-dimensional, and we realize all three systems in the ambient space $E=\R^N$, with standard basis $e_1,\ldots,e_N$. In types $\BN$ and $\DN$ the roots span all of $\R^N$. In type $\AN$, however, the roots $\pm(e_i-e_j)$ span only the hyperplane $H_N:=\left\{x\in\R^N:\sum_{i=1}^N x_i=0\right\}$. Thus type $\AN$ is used in its standard non-essential realization in $\R^N$: the root system acts on the difference coordinates, while the direction $\R(1,\ldots,1)$ is the unconstrained center-of-mass direction. Accordingly, the type $\AN$ chamber below is the inverse image in $\R^N$ of the essential $A_{N-1}$ Weyl chamber in $H_N$.

The specified positive root systems and positive Weyl chambers in the considered
three types are given below.

    \begin{itemize}
        \item[1.] \textbf{Root system of type $\AN$:}
            \begin{eqnarray*}
                \textrm{{Roots:}} && R = \{ \pm(e_i - e_j) : 1 \le i<j \le \no \}\\
                \textrm{Positive Roots:} && R_+ = \{ e_i - e_j : 1 \le i < j \le \no \}\\
                \textrm{Positive Weyl Chamber:} && \C = \{ x \in \R^\no : x_1 > x_2 > \dots > x_{\no} \}
            \end{eqnarray*}

        \item[2.] \textbf{Root system of type $\BN$:}
            \begin{eqnarray*}
                \textrm{{Roots:}} && R = \{ \pm e_i : 1 \le i \le \no \} \cup\{ \pm(e_i-e_j), \pm(e_i + e_j) : 1 \le i < j \le \no\}\\
                \textrm{Positive Roots:} && R_+ = \{ e_i : 1 \le i \le N \} \cup \{ e_i \pm e_j : 1 \le i < j \le N \}\\
                \textrm{Positive Weyl Chamber:} && \C = \{ x \in \R^\no : x_1 > x_2 > \dots > x_N > 0 \}
            \end{eqnarray*}

        \item[3.] \textbf{Root system of type $\DN$:}
             \begin{eqnarray*}
                \textrm{{Roots:}} && R = \{ \pm(e_i-e_j), \pm(e_i + e_j) : 1 \le i < j \le \no\}\\
                \textrm{Positive Roots:} && R_+ = \{ e_i \pm e_j : 1 \le i < j \le N \}\\
                \textrm{Positive Weyl Chamber:} && \C = \{ x \in \R^\no : x_1 > x_2 > \dots> x_{N-1}> |x_N|\}
            \end{eqnarray*}

    \end{itemize}
            For every fixed $A\subseteq R_+$ and every fixed $\alpha\in R_+$ we define the following sets of positive roots
        \begin{eqnarray}
            \label{defn:N1}
            N_\alpha^{(1)}(A) &=& \{\beta \in A: \beta \neq \alpha, \ab\neq 0\}\/,\\
            \label{defn:N2}
            N_\alpha^{(2)}(A) &=& \{(\beta,\gamma): \beta,\gamma\in A: \beta \neq \gamma, \ab\neq 0, \inv{\beta}{\gamma} = \alpha\}
        \end{eqnarray}

\subsection{Symmetric polynomials}
For a given set $A=\{a_1,\ldots,a_p\}$, $p\in\N$, we denote by $e_n(A)$, $n=1,\ldots,p$ the basic symmetric polynomials in variables $a_1,\ldots,a_p$, i.e. 
\begin{equation*}
  e_n(A) = \sum_{i_1<\ldots<i_n} a_{i_1}\cdot\ldots\cdot a_{i_n}\/,\quad n=1,\ldots,p\/.
\end{equation*}
We also define the corresponding power sums $p_k(A)$, $k=1,2,3,\ldots$ by
\begin{equation*}
	p_k(A) = \frac{1}{k}\sum_{i=1}^p a_i^k\/.
\end{equation*}
We introduce the normalizing factor $1/k$ in the definition to make some of the calculations appearing in the next sections shorter and more transparent. 

\medskip

\newcommand{\SR}{\Delta}
\newcommand{\Cbn}[1]{\Cb^{(#1)}}

\newcommand{\ClusterSet}{\Phi}
\newcommand{\ClusterSetMax}{\ClusterSet^{max}}
\newcommand{\ClusterSetn}[1]{\ClusterSet^{(#1)}}
\newcommand{\Cluster}{\varphi}
\newcommand{\SimpleNo}{s}
\newcommand{\ClusterGen}[1]{R_+(#1)}

\subsection{Clusters and collision points}
    \label{subsec:clusters}
    We begin the section with an introduction of clusters, which is a concept helping to deal with collisions of the process with the Weyl chamber walls. 

    \begin{definition}
        A subset $\Cluster$ of $R_+$ is called \textbf{a cluster}, if there exists $k\in \{1, \ldots, \no\}$ such that $\alpha_k\neq 0$ for every $\alpha\in \Cluster$.
    \end{definition}
    Based on the above definition, the clusters in $\AN$
  are 
    \begin{equation*}
        \varphi_1 = \{e_2-e_3, e_3-e_4\}\qquad 
        \varphi_2 = \{e_1-e_2, e_1-e_6, e_1-e_8\}
    \end{equation*}
    as well as the following are the clusters in $\BN$
    \begin{equation*}
        \varphi_3 = \{e_1,e_1-e_3\}\qquad 
        \varphi_4 = \{e_2, e_2+e_3, e_2-e_4\}
    \end{equation*} 
    
    We will denote by $\ClusterSet$ a family of all clusters and $|\Cluster|$ the length of a cluster,  which is just the number of elements in the cluster. A cluster with length $1$ will be called \textbf{singular cluster} and \textbf{multiple cluster} is a cluster with length greater than $1$. Finally, for each $m\in \N$ we introduce the set of clusters with fixed length
        \begin{equation*}
            \ClusterSetn{m} = \{\Cluster \in \ClusterSet: |\Cluster| = m\}\/.
        \end{equation*}
    For a given cluster $\Cluster \in \ClusterSet$ we denote by $\ClusterGen{\Cluster}$ the set of all positive roots, which can be obtained as a reflection of the cluster elements in the cluster elements (with suitable change of sign if necessary).
        \begin{equation*}
            R_+(\Cluster) = \Cluster \cup \{\prf{\alpha}{\beta}\in R_+: \alpha,\beta \in \Cluster\}\/,
        \end{equation*}
    where we recall that $\prf{\alpha}{\beta} = \pm s_\alpha(\beta)$ and the  sign is chosen in a way to ensure that we obtain a positive root. For example we get
    \begin{eqnarray*}
        R_+(\varphi_1) &=& \{e_2-e_3, e_3-e_4, e_2-e_4\}\\
        R_+(\varphi_2) &=& \{e_1-e_2, e_1-e_6, e_1-e_8, e_2-e_6, e_2-e_8, e_6-e_8\}\\
        R_+(\varphi_4) &=& \{e_2, e_3, e_4, e_2\pm e_3, e_2\pm e_4, e_3\pm e_4\}
    \end{eqnarray*}
    
For every boundary point $y\in \Cb$, we define clusters related to a given boundary point as 
\begin{equation*}
    \ClusterSet(y) = \{\Cluster\in \ClusterSet: \ya = 0\textrm{ for all }\alpha\in \Cluster\}.
\end{equation*}
Therefore, we can decompose the boundary of the Weyl chamber according to the size of the longest cluster in $\ClusterSet(y)$. For every $m= 1,\ldots, \SimpleNo$, we define
\begin{equation*}
    \Cbn{m} = \{y\in \Cb: \max_{\Cluster \in \ClusterSet(y)}|\Cluster| = m\}\/.
\end{equation*}
Then the boundary $\Cb$ is a disjoint union of the above-given sets.
\begin{equation*}
    \Cb = \bigcup_{m=1}^{\SimpleNo} \Cbn{m}\/.
\end{equation*}
The size of the longest possible cluster $\SimpleNo$ depends on $N$ and the kind of a root system we are dealing with, but its specific value is not important for our consideration. We call the boundary points in $\Cbn{m}$ \textbf{a multiple collision point of order $m$}, whenever $m\geq 2$. If $y\in \Cbn{1}$, then it is called \textbf{a single collision point}. Recall that at a single collision point we can see many $\alpha\in R_+$ such that $\ya=0$, but none of them form a cluster of length greater than or equal to $2$. In other words, there might be simultaneous singular collisions, but not multiple ones. 

\subsection{Notation}
For a given diffusion $Z = (Z(t))$ we define its life-time by
\begin{equation*}
    \life{Z} = \inf\{t: |Z(t)| = \infty\} = \lim_{R\to \infty} \inf\{t: |Z(t)|\geq R\}\/.
\end{equation*}
If additionally $Z$ is a continuous semi-martingale, then we denote by $\mart{Z}$ its local martingale part and $\drift{Z}$ will stand for the finite variation part and consequently
\begin{equation*}
    Z(t) = Z(0) + \mart{Z}(t)+\drift{Z}(t)
\end{equation*}
for every $t<\life{Z}$. Moreover, we denote the first hitting time of a set $D\subseteq \R^\no$ by $Z$ as
    \begin{equation*}
        \htime{D}{Z} = \inf\{t: Z(t)\in D\}\/.
    \end{equation*}
Whenever we consider hitting a single point $z$, we will simplify the notation by writing $\htime{z}{Z}$.


\section{Invariant polynomials and construction of a candidate solution}
The core idea of our construction is to leverage the similarities between multivariate Bessel-Dunkl diffusions on $\Cc$ and Bessel processes on the half-line $[0,\infty)$.  Recall that the classical definition of Bessel processes involves starting with a stochastic differential equation (SDE) for squared Bessel processes, where the drift term is non-singular. Subsequently, by applying the inverse function $x\to \sqrt{x}$, Bessel processes are defined, and it is then shown that they satisfy a corresponding SDE. We adopt a similar approach, beginning with an SDE for invariant polynomials, which serve as the multidimensional analogue of the square function used in the classical Bessel process derivation.

\subsection{Stochastic description of symmetric polynomials}
Recall that the idea of constructing a solution based on invariant polynomials was introduced in \cite{bib:GraczykMalecki:2014} for the root system $A_{N-1}$. Since we want to cover also remaining root systems, we have to expand this approach and choose family of polynomials suitable for every considered root system. However, to unify the notation, we define
\begin{align*}
	\Fak{k}{x} = 
\begin{cases}
    \sum\limits_{j=0}^{k-2}x_n^{j}x_m^{k-2-j}, & \text{if $\alpha=r_n-r_m$ and $k=1,2,\ldots$}\/,\\[1.9ex]
		\sum\limits_{j=0}^{k-2}x_n^{j}(-x_m)^{k-2-j}, & \text{if $\alpha=r_n+r_m$ and $k=2,4,6,\ldots$}\/,\\[1.9ex]
    x_n^{k-2}, & \text{if $\alpha=r_n$ and $k=2,4,6,\ldots$}\/,		
  \end{cases}
\end{align*}
which is obviously continuous in $x$ for every $\alpha$ and $k$. 

\medskip

Going back to the stochastic description of the invariant polynomials, we take the starting point $x(0)$ in the interior of the Weyl chamber  $C_+$ and we consider $X(t)=(X_1(t),\ldots,X_N(t))$ as a solution to \eqref{eq:main:SDE} up to the first hitting time $T_C=\inf\{t: X(t)\in\partial C\}$. Consequently, using the It\^o formula, we can write SDEs for the related polynomials $p_k(x)$ for $t<T_C$ as follows
\begin{align*}
   dp_k\left(X\right) & = \sum_{i=1}^N X_i^{k-1}dX_i+\frac{k-1}{2} \sum_{i=1}^N X_i^{k-2}dX_idX_i\\
					 & = \sum_{i=1}^N X_i^{k-1}\sigma_i\left(t, X\right)dB_i+\frac{k-1}{2} X_i^{k-2}\sigma_i^2\left(t, X\right) dt\\
                     &\quad +\sum_{i=1}^N \left[X_i^{k-1}b_i\left(t, X\right)+X_i^{k-1}\sum_{\aR}\dfrac{\kaX \alpha_i}{\Xa}\right]dt.
\end{align*}
The crucial observation is that the singular expression $\kax/\xa$ disappear for the carefully chosen $k$. Indeed, for every $\alpha=r_n-r_m$ and every $k=1,2,\ldots$ we have
    \begin{equation*}
	   \sum_{i=1}^N x_i^{k-1}\alpha_i = x_n^{k-1}-x_m^{k-1} = (x_n-x_m)\sum_{j=0}^{k-2}x_n^{j}x_m^{k-2-j} = \xa \Fak{k}{x}\/.
    \end{equation*}
Moreover, for even $k$ and $\alpha=r_n+r_m$ we obtain
    \begin{equation*}
	   \sum_{i=1}^N x_i^{k-1}\alpha_i = x_n^{k-1}+x_m^{k-1} =  x_n^{k-1}-(-x_m)^{k-1}=(x_n+x_m)\sum_{j=0}^{k-2}x_n^{j}(-x_m)^{k-2-j} = \xa \Fak{k}{x}\/.
    \end{equation*}
Finally, for $\alpha=r_n$ and $k\geq 2$ we simply get
    \begin{equation*}
	   \sum_{i=1}^N x_i^{k-1}\alpha_i = x_n^{k-1} = x_n \cdot x_n^{k-2} = \xa \Fak{k}{x}\/.
    \end{equation*}
Collecting all together we arrive at
    \begin{align}
        dp_k\left(X\right) &= \sum_{i=1}^N X_i^{k-1}\sigma_i\left(t, X\right)dB_i+\sum_{i=1}^N \left(X_i^{k-1}b_i\left(t,X\right)+\frac{k-1}{2} X_i^{k-2}\sigma_i^2\left(t, X\right)\right)dt \nonumber\\
        &\quad +\sum_{\aR}F_{\alpha,k}\left(X\right)\kaX dt\/,
        	   \label{eq:pk:SDE}
    \end{align}
where all the functions appearing above are continuous. 

\medskip

Although the symmetric polynomials $p_k(x)$ are sufficient to define a solution in $A_{N-1}$ and $B_N$ cases, we need to consider $e_N(x)$, which is the product $x_1\cdot,\ldots\cdot x_N$, in the $D_N$ case. Proceeding as previously we apply the It\^o formula to arrive at the following SDEs for $e_N(x)$ assuming the start from the interior of the Weyl chamber and working up to the first hitting time of its boundary
    \begin{equation*}
        de_N\left(X\right)  = \sum_{i=1}^N e_{N-1}^{\overline{x_i}}(X)dX_i = \sum_{i=1}^N e_{N-1}^{\overline{x_i}}\left(X\right)\left(\sigma_i\left(t, X\right)dB_i+b_i\left(t, X\right)dt+\sum_{\aR}\frac{\alpha_i \kaX}{\Xa}dt\right)
    \end{equation*}
Once again the singularity can be removed. For every $\alpha=r_n\pm r_m$ using the relation $e_{N-1}^{\overline{x_n}}(x)=x_m e_{N-2}^{\overline{x_n}, \overline{x_m}}(x)$ we obtain
\begin{align*}
   \sum_{i=1}^N \alpha_i e_{N-1}^{\overline{x_i}}(x) = \alpha_n e_{N-1}^{\overline{x_n}}(x)+\alpha_m e_{N-1}^{\overline{x_m}}(x) = (\alpha_n x_n+\alpha_m x_m) e_{N-2}^{\overline{x_n}, \overline{x_m}}(x) = \xa e_{N-2}^{\overline{x_n}, \overline{x_m}}(x)
\end{align*}
and consequently setting $G_{\alpha,N}(x)= e_{N-2}^{\overline{x_n}, \overline{x_m}}(x)$ for $\alpha=r_n\pm r_m$ we arrive at
\begin{equation}
	\label{eq:eN:SDE}
	de_N\left(X\right) = \sum_{i=1}^N e_{N-1}^{\overline{x_i}}\left(X\right)\sigma_i\left(t, X\right)dB_i+\sum_{i=1}^N e_{N-1}^{\overline{x_i}}\left(X\right) b_i\left(t, X\right)dt+\sum_{\aR} G_{\alpha,N}\left(X\right)\kaX dt\/.
\end{equation}

\begin{remark} We can also remove the singularity from the SDEs for every symmetric polynomial $e_n(x)$, $n=1\ldots, N$ in the $A_{N-1}$ case. Indeed, the similar computations as above but involving the general formula $e_{n-1}^{\overline{x_n}}(x)=x_m e_{n-2}^{\overline{x_n}, \overline{x_m}}(x)+e_{n-1}^{\overline{x_n}, \overline{x_m}}(x)$ (note that for $n=N$ the last expression vanishes) lead to 
\begin{equation*}
  \sum_{i=1}^N \alpha_i e_{n-1}^{\overline{x_i}}(x) = \xa e_{n-2}^{\overline{x_n}, \overline{x_m}}(x) + e_{n-1}^{\overline{x_n}, \overline{x_m}}(x) \sum_{i=1}^N \alpha_i  = \xa e_{n-2}^{\overline{x_n}, \overline{x_m}}(x)\/,
\end{equation*}
where the last equations is a consequence of the fact that $\sum_{i=1}^N\alpha_i=0$ for every $\alpha=r_n-r_m$. Consequently, we have
\begin{equation*}
	de_n\left(X\right) = \sum_{i=1}^N e_{n-1}^{\overline{x_i}}\left(X\right)\left(\sigma_i\left(t, X\right)dB_i+b_i\left(t, X\right)dt\right)+\sum_{\aR} G_{\alpha,n}\left(X\right)\kax dt\/,
\end{equation*}
for every $n=1,\ldots,N$, $G_{\alpha,n}(x) = e_{n-2}^{\overline{x_n}, \overline{x_m}}(x)$ for $\alpha=r_n-r_m\in A_{N-1}^+$.
\end{remark}


\subsection{Basic invariant polynomials}
\label{subsec:basic_polynomials}
For every considered root system, we choose the corresponding basic invariant polynomials $w=(w_1,\ldots,w_N)$ indicated in \cite{bib:Lee:1974}. Since for the $A_N$ root system the basic invariant polynomials are just $e_1,\ldots,e_N$ we will take $w_k=p_k$ for $k=1,\ldots,N$ since it is more convenient for further calculations to work with the power sums than with $e_k$ polynomials. The set of basic invariant polynomials in the $B_N$ case can be chosen as power sums of squares of $x_1,\ldots,x_N$, i.e. we define $w_k=p_{2k}$ for $k=1,\ldots, N$ in this case. Finally, for $D_N$ root system we set $w_k=p_{2k}$ for $k=1,\ldots, N-1$ as in the $B_N$ case, but the last polynomial must be defined as $w_N=e_N$ to have a control on the sign of the smallest particle $x_N$.

\medskip

Such choice of $w=(w_1,\ldots,w_N)$ ensure that there is one-to-one correspondence between $x=(x_1,\ldots,x_N)$ and $(w_1,\ldots,w_N)$ whenever we restrict our consideration to the positive Weyl chamber $C$, i.e. the smooth function 
\begin{equation*}
  w=(w_1,\ldots,w_N):C\longrightarrow w[C]
\end{equation*}
is one-to-one, where $w[C]$ stands for the image of $C$ by $w$. To see this note that in the $A_{N-1}$ case it just follows from the fact that there is one-to-one relation between $(p_1,\ldots,p_N)$ and $(e_1,\ldots,e_N)$, which (with suitable sign) are coefficients of the characteristic polynomial $P(x)=\prod_{i=1}^N(x-x_i)$, i.e. there is one-to one relation between coefficients of the polynomial and its ordered roots. In the $B_N$ case the same argument shows that there is one-to-one relation between $(w_1,\ldots,w_N)$ and $(x_1^2,\ldots,x_N^2)$ and since here we are restricted to the positive $x_i$'s we get the claim. Finally, in the $D_N$ case, knowing $(w_1,\ldots,w_N)$ we can determine $(p_2,p_4, \ldots, p_{2N})$. From $p_2,\ldots,p_{2N-2}$ we determine the first $N-1$ elementary symmetric polynomials of $y_i=x_i^2$, and from $w_N=e_N(x)$ we determine the last one, namely $\prod_i y_i = e_N(x)^2=w_N^2$. Hence we determine the unordered set $\{x_1^2,\ldots,x_N^2\}$. It gives $(x_1,x_2,x_3,\ldots,|x_N|)$, but using $w_N=x_1\cdot\ldots\cdot x_N$ we can determine the sign of $x_N$.

\medskip

Summing these up we can define the inverse diffeomorphism between $C$ and $w[C]$
\begin{equation*}
	f:w[C_+] \stackrel{1-1}{\longrightarrow} C_+\/,
\end{equation*}
which is smooth on $w[C_+]$can be continuously extended to
\begin{equation*}
	f:\overline{w[C_+]} \stackrel{1-1}{\longrightarrow} \overline{C_+}\/.
\end{equation*}
Moreover, there exists extension of $f$ on $\R^N$, which is obviously not $1-1$, but remains continuous
\begin{equation}
 \label{eq:f:ext}
 \tilde{f}:\R^{N}\longrightarrow\overline{C_+}
\end{equation}
and such that $\tilde{f}[(w[C_+])^c]=\partial C_+$. Starting with $A_{N-1}$ recall that the function $g:\C^N\to\C^N$ giving the relation between coefficients of a polynomial and its (suitable ordered) complex roots is continuous (see \cite{bib:Marden:1949}) and in particular the function
$$
	\R^N\ni a=(a_1,\ldots,a_N) \longrightarrow \tilde{f}(a) = H(\Re(g_1(a)),\ldots\Re(g_N(a)))\/,
$$
is continuous, where $H(b_1,\ldots,b_N)=(b_{\pi(1)},\ldots,b_{\pi(N)})$ is the ordering permutation, i.e. we have $b_{\pi(i)}\geq b_{\pi(i+1)}$, $i=1,\ldots,N-1$. Observe also that if $a\in {w[C]}$, then the polynomial with coefficients $a_1,\ldots,a_N$ has all real and distinct roots and $\tilde{f}(a)=f(a)$. If $a\in \partial w[C]$, then we still have all real roots of the polynomial, but some of them are multiple roots, i.e. $\tilde{f}(a)\in\partial C$. Finally, if $a\notin \overline{w[C]}$, then we have a polynomial with real coefficients, but some of its roots are complex. However, since the coefficients are real, for every complex root its conjugate is also a root, which means that some of $\Re(g_i(a)$ are equal, i.e. $\tilde{f}(a)\in\partial C$. Thus we have just constructed the desired extension \eqref{eq:f:ext} for $A_{N-1}$. For $B_N$ root system we just have to compose this with square root function $x\to \sqrt{\textrm{max}(x,0)}$ on every coordinate and $D_N$ case requires additionally to observe that $x_\no=\textrm{sgn}(w_N)|x_\no|$.

\bigskip

To unify all the considered cases we introduce for $R=A_{N-1}$ and every $k=1,\ldots,N$ the following notation
\begin{equation*}
   a_{i,k}^R(t, x) = x_i^{k-1}\sigma_i(t, x)\/,\quad h_{i,k}^R(t, x) = x_i^{k-1}b_i(t, x)+\frac{k-1}{2} x_i^{k-2}\sigma_i^2(t, x)\/,\quad H_{\alpha,k}^R(x) = F_{\alpha,k}(x)\/.
\end{equation*}
Since for $R=B_N$ we just have $u_{k}=p_{2k}$, for every $k=1,\ldots,N$ and the same polynomials are chosen for $R=D_N$ for $k=1,\ldots,N-1$, so in those cases we set
\begin{equation*}
  a_{i,k}^R(t, x) = x_i^{2k-1}\sigma_i(t, x)\/,\quad h_{i,k}^R(t, x) = x_i^{2k-1}b_i(t, x)+\frac{2k-1}{2} x_i^{2k-2}\sigma_i^2(t, x)\/,\quad H_{\alpha,k}^R(x)=F_{\alpha,2k}(x)\/.
\end{equation*}
Finally, for $R=D_N$ and $k=N$ we have
\begin{equation*}
	a_{i,N}^R(t, x) = e_{N-1}^{\overline{x_i}}(x)\sigma_i(t, x)\/,\quad h_{i,N}^R(t, x) = e_{N-1}^{\overline{x_i}}(x)b_i(t, x)\/,\quad H_{\alpha,N}^R(x)=G_{\alpha,N}(x)\/.
\end{equation*}
Now we can define $U=(U_1,\ldots,U_N)$ as solution to appropriate SDEs of the form \eqref{eq:pk:SDE} and/or \eqref{eq:eN:SDE} with $X$ replaced by $\tilde{f}(U)$, which, according to the previous considerations, are SDEs with continuous coefficients. 

\begin{definition}
\label{defn:u}
For a given collection of independent one-dimensional Brownian motions
$(B_1,\ldots,B_N)$ we define $U=(U_1,\ldots,U_N)$ as a solution to the system
\begin{equation}
    \label{eq:u:SDE}
    dU_k
    =
    \sum_{i=1}^N a_{i,k}^{R}(t, \tilde f(U))\,dB_i
    + \sum_{i=1}^N h_{i,k}^{R}(t, \tilde f(U))\,dt
    + \sum_{\alpha\in R_+}
        H_{\alpha,k}^R(\tilde f(U))\,k_\alpha(t,\tilde f(U))\,dt,
\end{equation}
for every $k=1,\ldots,N$, with initial condition $U(0)=u_0\in\R^N$.
\end{definition}

\begin{remark}
The existence of a solution of \eqref{eq:u:SDE} follows from the continuity of $\tilde f$ on $\R^N$ and consequently from the continuity of the coefficients. Similar arguments were used in the construction presented in \cite{bib:GraczykMalecki:2014}. However two arguments were missing in the proof presented there, which was pointed out to the author by Prof. Michael Voit. First is the definition of the extension of $f$. Second is showing that in fact we can choose $U$ to stay in the set $\overline{w[C_+]}$. In the present paper both problems are solved in much more general setting, which keeps all the results from \cite{bib:GraczykMalecki:2014} staying in charge. 
\end{remark}

Let $\life{U}$ denote the lifetime of the solution $U$ of \eqref{eq:u:SDE}.  We set
\begin{equation}
    \label{eq:K-and-tauK-defn}
    \mathcal K:=\overline{w[\C]}\subset\R^N,
    \qquad
    \tau_{\mathcal K}:=
        \inf\{0\le t<\life{U}:U(t)\notin\mathcal K\},
\end{equation}
with the convention that the infimum of the empty set is $\life{U}$.

\begin{definition}
    \label{defn:x}
    Let $x_0\in\Cc$ and put $u_0=w(x_0)$.  Let $U$ be a solution of the invariant-coordinate equation from Definition~\ref{defn:u}, started from $U(0)=u_0$, and let $\life{U}$ be its lifetime.  If among such solutions there exists one satisfying $\tau_{\mathcal K}=\life{U}$ a.s.,  then, in the construction below, we choose such a solution $U$.
    We define the candidate chamber-valued process by
    \begin{equation}
        \label{eq:x:defn}
        X(t)=\tilde f(U(t)),\qquad 0\le t<\life{U}.
    \end{equation}
    We also put $\life{X}:=\life{U}$.
\end{definition}
On the stochastic interval $[0,\tau_{\mathcal K})$ we have
$U(t)\in\mathcal K$, and therefore
\begin{equation}
    \label{eq:K-identification-before-exit}
    X(t)=\tilde f(U(t))=f(U(t)), \qquad U(t)=w(X(t)), \qquad 0\le t<\tau_{\mathcal K}.
\end{equation}
Thus, before the no-exit result is proved, all particle computations are understood with the additional stopping at $\tau_{\mathcal K}$.

\begin{remark}
    \label{rem:direct:Ito}
    Let $p$ be a polynomial invariant under the Weyl group corresponding to the root system $R$.  By the choice of the basic invariant polynomials $w_1,\ldots,w_N$, there exists a polynomial $\widehat p\in\R[u_1,\ldots,u_N]$ such that
    \begin{equation*}
        p(x)=\widehat p(w_1(x),\ldots,w_N(x)),\qquad x\in\C.
    \end{equation*}
    Hence, on $[0,\tau_{\mathcal K})$,
    \begin{equation*}
        p(X(t))=\widehat p(U(t)).
    \end{equation*}
    The rigorous way to compute $p(X)$ on this stopped interval is therefore to apply It\^o's formula to the smooth polynomial $\widehat p(U)$.  In the sequel we shall often use the shorter symbolic notation obtained by applying It\^o's formula formally in the particle variables $x$.  This is only a shorthand on $[0,\tau_{\mathcal K})$: on the open chamber the map $w:\C\to w[\C]$ is a smooth diffeomorphism with inverse $f$, and the identity $p\circ f=\widehat p$ shows, by the chain rule, that the first- and second-order terms coincide with those obtained from It\^o's formula for $\widehat p(U)$.  After Proposition~\ref{prop:no-exit-K-classical} below is proved, $\tau_{\mathcal K}=\life{X}$ and the same shorthand is valid for all $t<\life{X}$.
\end{remark}


\section{Collisions and preservation of the invariant chamber image}

The main difficulty in showing that the process constructed through the invariant-coordinate equation is a genuine solution of the particle system is the behavior at the boundary of the Weyl chamber.  In this section we prove that the invariant-coordinate process does not leave the closed invariant chamber image $\mathcal K=\overline{w[\C]}$

The proof is organized as follows.  We first prove a stopped no-multiple-collision result, valid for interior initial points up to the first possible exit time $\tau_{\mathcal K}$.  This is then used to prove that $\tau_{\mathcal K}=\life{X}$ whenever the process starts from the interior $\C$.  Boundary initial points are handled by approximation from the interior.  Finally, we prove instantaneous diffraction from the boundary and deduce the full no-multiple-collision result after the initial time.

\subsection{Multiple collision points}

    Recall the definitions related to clusters and multiply collision points presented in Subsection \ref{subsec:clusters}. We begin with introducing some additional functions and symmetric polynomials, which enable us to control visits to multiple collision points. First, for given cluster $\Cluster\in\ClusterSet$, we define
        \begin{equation*}
            \rcx = \sum_{\alpha\in \ClusterGen{\Cluster}}\xa^2\/,\quad x \in \Cc\/.
        \end{equation*}
    We also introduce the following function, which up to the constant $2$, is the derivative  of $\rcx$ with respect to $x_k$
        \begin{equation*}
            \hcx{k} = \sum_{\alpha\in \ClusterGen{\Cluster}}\alpha_k\xa\/,\quad x \in \Cc\/, \quad k=1,\ldots,\no\/.
        \end{equation*}
    together with 
        \begin{equation}
            \label{eq:hcxdelta:defn}
            \hcx{\delta} = \sum_{k=1}^\no \delta_k \hcx{k} = \sum_{\alpha\in R_+(\Cluster)}\ad\xa \/.
        \end{equation}
    Please note that the sum in the definition of $\rcy$ is over all elements generated by the given cluster. In particular, if $x$ is an element of the boundary and $\Cluster$ is a cluster of $\ClusterSet(x)$, then the function $\Ss{\Cluster}(x)$ is zero. Conversely, if $\Ss{\Cluster}(x)=0$ for some cluster $\Cluster\in\ClusterSet$ and $x\in \Cc$, then $x\in \Cb$ and $\Cluster \in \ClusterSet(x)$. Consequently, defining
        \begin{equation}
            \label{eq:q:defn}
            \qm(x) = \prod_{\ccm} \rcx\/,\quad m=1,\ldots, \SimpleNo
        \end{equation}
    we obtain a symmetric polynomial, which controls if the point $x$ is an element of $\Cbn{m}$. In fact, if $x\in \Cbn{m}$, then there exists a cluster $\Cluster$ of length $m$ in $\ClusterSet(x)$, which implies $\Ss{\Cluster}(x) = 0$  and consequently $\q{m}(x)=0$. Conversely, if $\qm(x)=0$, then there exists $\ccm$ such that $\rcx=0$, which implies $x\in \Cbn{j}$ with $j\geq m$.

    \medskip

    In what follows, we focus on the stochastic description and the properties of the symmetric polynomial process defined as
        \begin{equation}
            \label{eq:qproc:defn}
            \qm(X) = \prod_{\Cluster\in \ClusterSetn{m}}\Ss{\Cluster}(X)\/.
        \end{equation}
    As we have discussed previously, the process $\q{m}(X)$ controls, in some way, the collisions between particles, which makes it natural to study its first hitting time of zero
        \begin{equation*}
            \taumz := \htime{0}{\q{m}(X)} = \inf\{t\geq 0: \q{m}(X(t))=0\}\/.
        \end{equation*}
    We will use the classical approach to study the hitting time considering its logarithm $\Q{m} = -\ln(\q{m}(X))$ and its explosion time $\htime{\infty}{\Q{m}}$.  

    \begin{proposition}
        \label{prop:qm:form}
        Let $\qm(X)$ be defined in \eqref{eq:qproc:defn} and assume that $X(0) = x_0\in \Cc$ is such that $\qm(x_0)>0$. Then, for $t<  \taumz\wedge \tau_{\mathcal K}$, the process $\Qm = -\ln \qm(X)$ is a continuous semi-martingale with the local martingale part 
            \begin{equation}
                \label{eq:lam:mart}
                d\mart{\Qm} = -2\sum_{\ccm} \frac{1}{\rcx} \sum_{k=1}^N \hcX{k} \sigma_k(X)dB_k
            \end{equation}
        and the drift giving as a sum $d\drift{\Qm} = h_1(X)dt+h_2(X)dt+h_3(X)dt$, where
            \begin{eqnarray}
                \label{eq:lam:drift:1}
	           h_1(x) &=& \sum_{\ccm} \frac{1}{(\rcx)^2}\sum_{k=1}^N\left(2(\hcx{k})^2-\rcx \sum_{\alpha\in R_+(\Cluster)}\alpha_k^2\right)\sigma_k^2(x),\\
                \label{eq:lam:drift:2}
	           h_2(x) &=& -2\sum_{\ccm} \frac{1}{\rcx} \sum_{k=1}^N \hcx{k} b_k(x),\\
                \label{eq:lam:drift:3}
	        h_3(x) &=& -2\sum_{\ccm} \frac{1}{\rcx} \sum_{\dR} \frac{\kdx \hcx{\delta} }{\xd} \/,
            \end{eqnarray}

    \end{proposition}
    \begin{proof}

        \medskip
    
        Since we simply have $\ln\qm(x) = \sum_{\ccm}\ln\rcx$, the It\^o formula implies
            \begin{equation*}
	           d\Qm = -\sum_{\ccm} d\ln \rcx = -\sum_{\ccm}\left(\frac{1}{\rcx} d\rcx-\dfrac12 \frac{d\rcx d\rcx}{(\rcx)^2}\right)\/,
            \end{equation*}	
        which gives the local martingale part $d\mart{\Qm}$ of the form \eqref{eq:lam:mart} and the drift part is represented as a sum $d\drift{\Qm} = h_1(X)dt+h_2(X)dt+h_3(X)dt$, where $h_1$ and $h_2$ are given in \eqref{eq:lam:drift:1} and \eqref{eq:lam:drift:2} respectively. For the last part we just have
            \begin{eqnarray*}
                \nonumber
                h_3(x) &=& -2\sum_{\ccm} \frac{1}{\rcx} \sum_{k=1}^N \hcx{k}\sum_{\dR} \frac{\kdx \delta_k}{\xd}
                = -2\sum_{\ccm} \frac{1}{\rcx} \sum_{\dR}  \frac{\kdx \hcx{\delta}}{\xd}\/.
            \end{eqnarray*}
    \end{proof}

    Note that for fixed cluster $\ccm$, if $\sum_{\alpha\in\Cluster}\ad^2=0$, then $\sum_{\alpha\in R_+(\Cluster)}\ad^2=0$ and consequently $\hcx{\delta} = 0$. Thus it is enough to consider those $\delta \in R_+$ such that $\sum_{\alpha\in\Cluster}\ad^2>0$ or conversely, for fixed $\delta\in R_+$ we can consider only clusters from
        \begin{equation*}
            \ClusterSetn{m}_\delta = \{\ccm:\sum_{\alpha\in\Cluster}\ad^2>0 \}\/.
        \end{equation*}
    Moreover, for fixed $\delta \in R_+$ and $\Cluster \in \ClusterSetn{m}_\delta$ we define
        \begin{equation*}
            \InCluster = \{\prf{\alpha}{\delta}\in R_+:\alpha\in \Cluster\}\/.
        \end{equation*}
    It is important to see that if $\delta\in R_+(\Cluster)$, directly from the definition of $R_+(\Cluster)$ together with $\prf{\prf{\alpha}{\delta}}{\delta} = \alpha$ we have $R_+(\Cluster) = R_+(\InCluster)$. Then $\hcx{\delta} = \hincx{\delta}$ and $\rcx = \sincx$. 
        
    Note that, since at least one of the inner products $\ad$ is non-zero and obviously $\InCluster\in \ClusterSetn{m}_\delta$, as $\ad = \pm \inner{\prf{\alpha}{\delta}}{\delta}$. We can thus rewrite $h_3(x)$ as
        \begin{equation*}
                h_3(x) = -\sum_{\delta \in R_+} \dfrac{\kdx }{\xd} \sum_{\ccmd}\left(\frac{\hcx{\delta}}{\rcx}+\frac{\hincx{\delta}}{\sincx}\right)
        \end{equation*}
    Recall that 
        \begin{equation*}
            \inv{\alpha}{\delta} = \pm\left(\alpha-\frac{2\ad \delta}{|\delta|^2}\right)\/, \ainvdd = \mp\ad\/,\quad \xainvd = \pm\left(\xa-\frac{2\ad\xd}{|\delta|^2}\right),
        \end{equation*}
    and 
        \begin{equation}
            \label{eq:inv:master:2}
            \ad\xa+\ainvdd\xainvd  = \frac{2\ad^2\xd}{|\delta|^2}\/.
        \end{equation}
    We begin with using \eqref{eq:inv:master:2} to get
    
        \begin{equation*}
            \hcx{\delta} + \hincx{\delta} = \sum_{\alpha\in R_+(\Cluster)} \ad\xa + \ainvdd \xainvd = \frac{2\xd}{|\delta|^2} \sum_{\alpha\in R_+(\Cluster)} \ad^2.
        \end{equation*}
    We have
        \begin{equation*}
            \xa^2-\xainvd^2 = \frac{2\xd}{|\delta|^2}\left(\ad\xa-\ainvdd\xainvd \right).
        \end{equation*}
    Consequently,
        \begin{align*}
            \rcx-\sincx &= \sum_{\alpha\in R_+(\Cluster)} \left(\xa^2-\xainvd^2\right) \\
                        &= \frac{2\xd}{|\delta|^2} \sum_{\alpha\in R_+(\Cluster)} \left( \ad\xa-\ainvdd\xainvd\right) \\
                        &= \frac{2\xd}{|\delta|^2} \left(\hcx{\delta}-\hincx{\delta}\right).
        \end{align*}
    where we have also used the fact that if $\ad=0$ then $\inv{\alpha}{\delta} = \alpha$. Combining all together we get
        \begin{eqnarray*}
            \frac{\hcx{\delta}}{\rcx}+\frac{\hincx{\delta}}{\sincx} &=& 
            \frac{(\rcx+\sincx)(\hcx{\delta}+\hincx{\delta})}{2\rcx\sincx} 
            - \frac{(\rcx-\sincx)(\hcx{\delta}-\hincx{\delta})}{2\rcx\sincx}\\
            &=& \frac{\xd}{|\delta|^2} \left[ \left(\frac{1}{\rcx}+\frac{1}{\sincx}\right)\sum_{\alpha\in R_+(\Cluster)} \ad^2 - \frac{(\hcx{\delta}-\hincx{\delta})^2}{\rcx\sincx}\right]
        \end{eqnarray*}
    which allow us to remove the singularity related to $\ad^{-1}$ in the formula for $h_3(x)$ and arrive at the following representation of $h_3(x)$.
        \begin{equation}
            \label{eq:h3:final}
             -\sum_{\delta \in R_+} \dfrac{\kdx }{|\delta|^2} \sum_{\ccmd}\left[\left(\frac{1}{\rcx}+\frac{1}{\sincx}\right)\sum_{\alpha\in R_+(\Cluster)} \ad^2 - \frac{(\hcx{\delta}-\hincx{\delta})^2}{\rcx\sincx}\right]\/.
        \end{equation}
    Examine the behavior of $h_3(x)$ in the neighborhood of the boundary point $y\in \Cbm$. Recall that for such a point, the largest clusters in $\ClusterSet(y)$ are of length $m$. The function $h_3(x)$ might explode only if $\rcx=0$ for some $\ccmd$ and this might only occur in the small neighborhood of $y$ if additionally $\Cluster \in \ClusterSet(y)$. Fix this $\Cluster \in \ClusterSet(y)\cap \ClusterSetn{m}_\delta$. If $\delta\in R_+(\Cluster)$, then, as we have previously seen, $R_+(\Cluster) = R_+(\InCluster)$ and $\hcx{\delta} = \hincx{\delta}$, $\rcx = \sincx$, which reduce the corresponding expression in \eqref{eq:h3:final} to 
        \begin{equation*}
            -\dfrac{2\kdx }{|\delta|^2} \sum_{\ccmd}\frac{1}{\rcx}\sum_{\alpha\in R_+(\Cluster)} \ad^2\/.
        \end{equation*}

     We use those observations to prove the next proposition, which is one of the most crucial results and shows that whenever the repulsive function $\ka(t,x)$ does not vanish at the multiple boundary points, then those points are not visited by the process.

\begin{proposition}
    \label{prop:no_multi_collisions:first}
    Assume that $X(0)=x_0\in\C$, and \ref{ass:k:positive} holds. Recall that $\tau_0^m = \inf\{t>0:X(t)\in\Cbn{m}\}$, for $m\in \{2,\ldots,\SimpleNo\}$.
    Then
    \begin{equation*}
        \pr\left(\tau_0^m\le \tau_{\mathcal K},\ \tau_0^m<\life{X}\right)=0,\qquad m=2,\ldots,\SimpleNo.
    \end{equation*}
    Equivalently, before and at the first possible exit from $\mathcal K$, the process cannot reach a multiple collision point.
\end{proposition}

    \begin{proof}
        We begin with fixing $m\in \{2,\ldots, \SimpleNo\}$ and restricting our consideration to the sets
        \begin{equation*}
            \{\tau_0^{m}<\tau_0^{m+1}\}\cap \{\tau_0^m\leq \tau_{\mathcal K}\}\cap \{\tau_0^m< \life{X}\}\/,
        \end{equation*}
        where we have $\tau_0^{\SimpleNo+1}= \life{X}$. We consider $\Qm = -\ln \qm(X)$, with $\qm(X)$ defined in \eqref{eq:qproc:defn}. By Proposition \ref{prop:qm:form}, we know the form of $\Qm(t)$ for $t<\tau_0^{m}$, which is our starting point for further consideration. We will show that under assumption on $\ka$, this process does not explode to $\infty$ in finite time and consequently $\tau_0^{m}=\tau_0^{m+1}$ on $\{\tau_0^m\leq \tau_{\mathcal K}\}\cap \{\tau_0^m< \life{X}\}$. Indeed, the martingale part $\mart{\Qm}$ given in \eqref{eq:lam:mart}, by the McKean argument, cannot explode in finite time, as it is always a Brownian motion with time changed. 
        
        To deal with the drift part, let us denote by $y=X(\tau_0^m)$. Since we are working on $\{\tau_0^{m}<\tau_0^{m+1}\}$, it follows that $y\in \Cbn{m}$. 
        Let us see that the blow-up of any of the parts $h_1$, $h_2$, or $h_3$ (in the form \eqref{eq:h3:final}) is only possible when $\rcy$ becomes zero, which might happen only if $\Cluster\in \ClusterSet(y)\cap \ClusterSetn{m}$. From now on we will fix $\Cluster$ in $\ClusterSet(y)\cap \ClusterSetn{m}$ and consider the corresponding element from $h_1(x)+h_2(x)+h_3(x)$, i.e. the expression
            \begin{eqnarray}
            \label{eq:sigma:main}
                &&\frac{1}{(\rcx)^2}\sum_{k=1}^N\left(2(\hcx{k})^2-\rcx \sum_{\alpha\in R_+(\Cluster)}\alpha_k^2\right)\sigma_k^2(t, x)\\
            \label{eq:b:main}
                &&- \frac{2}{\rcx} \sum_{k=1}^N \hcx{k} b_k(t, x)\\
            \label{eq:k:main}
                &&-\frac{2}{\rcx}\sum_{\delta \in R_+(\Cluster)} \dfrac{\kdx }{|\delta|^2} \sum_{\alpha\in R_+(\Cluster)} \ad^2\\
            \label{eq:k:rest}
                &&-\frac{1}{\rcx}\sum_{\delta \notin R_+(\Cluster)} \dfrac{\kdx }{|\delta|^2\sincx} \left[\left(\rcx+\sincx\right)\sum_{\alpha\in R_+(\Cluster)} \ad^2 -(\hcx{\delta}-\hincx{\delta})^2\right]
            \end{eqnarray}
        First we deal with the last part. Note that if $\delta\notin R_+(\Cluster)$, then $\InCluster$ cannot be a cluster from $\ClusterSet(y)$. Indeed, otherwise all roots from $\InCluster$ would vanish at $y$, and we would obtain a cluster of length larger than $m$, which is impossible since $y\in\Cbm$. Hence not all roots in  $\InCluster$ vanish at $y$, and therefore $\sincy>0$. Consequently, using $\ad^2=\ainvdd^2$, we get
            \begin{eqnarray*}
                \lim_{x\to y}\lefteqn{[(\rcx+\sincx)\sum_{\alpha\in R_+(\Cluster)}\ad^2 - (\hcx{\delta}-\hincx{\delta})^2] =}\\ 
                &=&\sincy \sum_{\alpha\in R_+(\Cluster)}\ainvdd^2 - (\hincy{\delta})^2\\
                &=&  \sum_{\alpha\in R_+(\Cluster)} \yainvd^2\sum_{\alpha\in R_+(\Cluster)}\ainvdd^2 - \left(\sum_{\alpha\in R_+(\Cluster)}\ainvdd \yainvd \right)^2\geq 0\/.
            \end{eqnarray*}
        Non-negativity of the last expressions follows from Cauchy-Schwarz inequality.
        It is also clear that $\hcx{k} b_k(x)$ goes to zero as $x\to y$.
        These two observations together with positivity of $k_\delta(t,y)$, $\sincy$ and $\sum_{\alpha\in R_+(\Cluster)} \ad^2$ and continuity argument show that there exists a neighborhood $U(y)$ of $y$ such that for every $x\in U(y)$ the sum of \eqref{eq:sigma:main}, \eqref{eq:b:main}, \eqref{eq:k:main} and \eqref{eq:k:rest} is bounded from above by 
        \begin{equation*}
            \frac{1}{(\rcx)^2}\sum_{k=1}^N\left(2(\hcx{k})^2-\rcx \sum_{\alpha\in R_+(\Cluster)}\alpha_k^2\right)\sigma_k^2(t, x)-\frac{1}{\rcx}\sum_{\delta \in R_+(\Cluster)} \dfrac{\kdx }{|\delta|^2} \sum_{\alpha\in R_+(\Cluster)} \ad^2
        \end{equation*}
        Using the Cauchy-Schwarz inequality we get
        \begin{equation*}
            2(\hcx{k})^2-\rcx \sum_{\alpha\in R_+(\Cluster)}\alpha_k^2 \leq \rcx \sum_{\alpha\in R_+(\Cluster)}\alpha_k^2\/.
        \end{equation*}
        Continuity of $\sigma_k(t,x)$ gives $(\sigma_k(t,x)-\sigma_k(t,y)) \to 0$ as $x$ approaches $y$ and the same argument as above allows us to replace the expression by the following  
        \begin{equation*}
            \frac{1}{(\rcx)^2}\sum_{k=1}^N\left(2(\hcx{k})^2-\rcx \sum_{\alpha\in R_+(\Cluster)}\alpha_k^2\right)\sigma_k^2(t, y)-\frac{1}{2\rcx}\sum_{\delta \in R_+(\Cluster)} \dfrac{\kdx }{|\delta|^2} \sum_{\alpha\in R_+(\Cluster)} \ad^2
        \end{equation*}
        perhaps by shrinking the neighborhoods of $U(y)$, if needed. Note that the only problematic element of the expression is $2(\hcx{k})^2$, since the rest is obviously non-positive and thus can not explode to $\infty$. This is way we can restrict the sum over $k$ to the set of indices, which appear in roots from $R_+(\Cluster)$, i.e.
            \begin{equation*}
                I(\Cluster) = \left\{k\in \{1,\ldots, \no\}: \sum_{\alpha\in R_+(\Cluster)} \alpha_k^2>0\right\}
            \end{equation*}
        because only for those $k$ we might have $\hcx{k}\neq 0$.    
        By Assumption~\ref{ass:sigma}, the diffusion variances are constant on the coordinate block generated by the long roots in $R_+(\Cluster)$.  Hence there exists $c_\Cluster(t,y)\ge0$ such that 
            \begin{eqnarray}
                \sigma_k^2(t,y)=c_\Cluster(t,y), \qquad k\in I(\Cluster)\/.
            \end{eqnarray}
        These observations enable us to write
            \begin{eqnarray*}
                \sum_{k=1}^N\left[2(\hcx{k})^2-\rcx \sum_{\alpha\in R_+(\Cluster)}\alpha_k^2\right]\sigma_k^2(t, y) &=& c_\Cluster(t,y)\sum_{k\in I(\Cluster)}\left[2(\hcx{k})^2-\rcx\sum_{\alpha\in R_+(\Cluster)}\alpha_k^2\right]\\
                &=& c_\Cluster(t,y)\sum_{k=1}^\no\left[2(\hcx{k})^2-\rcx\sum_{\alpha\in R_+(\Cluster)}\alpha_k^2\right]\\
                &=& c_\Cluster(t,y)\left[2\sum_{k=1}^\no(\hcx{k})^2-\rcx\sum_{\alpha\in R_+(\Cluster)}|\alpha|^2\right]\/.
            \end{eqnarray*}
    It remains to estimate
        \begin{equation*}
            2\sum_{k=1}^\no(\hcx{k})^2 - \rcx\sum_{\alpha\in R_+(\Cluster)}|\alpha|^2 .
        \end{equation*}
    Put $q:=|I(\Cluster)|$. Since $y\in\Cbn{m}$ and $\Cluster\in\ClusterSet(y)\cap\ClusterSetn{m}$, the cluster $\Cluster$ is maximal at $y$.  Therefore, up to relabelling of the coordinates in $I(\Cluster)$, the generated set $R_+(\Cluster)$ has one of the following forms. First, we may have an $A$-type cluster,
        \begin{equation*}
            R_+(\Cluster) = \{e_i-e_j:\ i,j\in I(\Cluster),\ i<j\}.
        \end{equation*}
    This case may occur in all three classical root systems. Second, in type $\DN$ we may have a mixed nonzero mirror cluster.  In this
    case
        \begin{equation*}
            I(\Cluster)=J\cup\{N\}, \qquad J\subset\{1,\ldots,N-1\},
        \end{equation*}
    and
        \begin{equation*}
            R_+(\Cluster) = \{e_i-e_j:\ i,j\in J,\ i<j\} \cup \{e_i+e_N:\ i\in J\}.
        \end{equation*}
    This corresponds to nonzero collisions of the form $x_i=-x_N\neq0$ together with ordinary equalities among the coordinates
    indexed by $J$. Third, in type $\BN$ we may have a zero cluster,
        \begin{equation*}
            R_+(\Cluster) = \{e_i:\ i\in I(\Cluster)\} \cup \{e_i-e_j,\ e_i+e_j:\ i,j\in I(\Cluster),\ i<j\}.
        \end{equation*}
    Finally, in type $\DN$ we may have a zero cluster,
        \begin{equation*}
            R_+(\Cluster) = \{e_i-e_j,\ e_i+e_j:\ i,j\in I(\Cluster),\ i<j\}.
        \end{equation*}
    Notice that in the last case with $q=2$ this includes the zero-pair cluster $\{e_i-e_j,\ e_i+e_j\}$. We now compute the above expression in these cases.
    
    \medskip

    We now compute the above expression in these cases. In the type $\AN$ case, writing
        \begin{equation*}
            \bar x_I:=\frac1q\sum_{i\in I(\Cluster)}x_i,
        \end{equation*}
    we have
        \begin{equation*}
            \rcx = \sum_{i<j,\ i,j\in I(\Cluster)}(x_i-x_j)^2 = q\sum_{i\in I(\Cluster)}(x_i-\bar x_I)^2,
        \end{equation*}
    and, for $i\in I(\Cluster)$,
        \begin{equation*}
            \hcx{i} = \sum_{j\in I(\Cluster),\,j\neq i}(x_i-x_j) = q(x_i-\bar x_I).
        \end{equation*}
    Hence
        \begin{equation*}
            \sum_{k=1}^{\no}(\hcx{k})^2 = q\,\rcx .
        \end{equation*}
    Since $|e_i-e_j|^2=2$ and there are $q(q-1)/2$ roots in this generated subsystem,
        \begin{equation*}
            \sum_{\alpha\in R_+(\Cluster)}|\alpha|^2 = q(q-1).
        \end{equation*}
    Therefore
        \begin{equation*}
            2\sum_{k=1}^{\no}(\hcx{k})^2 - \rcx\sum_{\alpha\in R_+(\Cluster)}|\alpha|^2 = q(3-q)\rcx \le 0,
        \end{equation*}
    because in the present multiple-collision case $q\ge3$. In the mixed type $\DN$ case, put
        \begin{equation*}
            z_i=x_i,\quad i\in J, \qquad z_N=-x_N, \qquad \bar z_I:=\frac1q\left(\sum_{i\in J}z_i+z_N\right).
        \end{equation*}
    Then
        \begin{equation*}
            \rcx = \sum_{i<j,\ i,j\in J}(x_i-x_j)^2 + \sum_{i\in J}(x_i+x_N)^2 = \sum_{i<j,\ i,j\in I(\Cluster)}(z_i-z_j)^2
                 = q\sum_{i\in I(\Cluster)}(z_i-\bar z_I)^2.
        \end{equation*}
    Moreover, for $i\in J$,
        \begin{equation*}
            \hcx{i} = \sum_{j\in J,\,j\neq i}(x_i-x_j)+(x_i+x_N) = q(z_i-\bar z_I),
        \end{equation*}
    while
        \begin{equation*}
            \hcx{N} = \sum_{i\in J}(x_i+x_N) = -q(z_N-\bar z_I).
        \end{equation*}
    Hence
        \begin{equation*}
            \sum_{k=1}^{\no}(\hcx{k})^2 = q\,\rcx .
        \end{equation*}
    The generated subsystem contains
        \begin{equation*}
            \binom{q-1}{2}+(q-1)=\frac{q(q-1)}2
        \end{equation*}
    positive roots, all of length squared equal to $2$. Therefore
        \begin{equation*}
            \sum_{\alpha\in R_+(\Cluster)}|\alpha|^2=q(q-1).
        \end{equation*}
    Consequently,
        \begin{equation*}
            2\sum_{k=1}^{\no}(\hcx{k})^2 - \rcx\sum_{\alpha\in R_+(\Cluster)}|\alpha|^2 = q(3-q)\rcx \le 0.
        \end{equation*}
    Here $q=2$ would give only the single root $e_i+e_N$, so in the present multiple-collision case we again have $q\ge3$. In the type $\BN$ case we have
        \begin{equation*}
            \rcx = \sum_{i\in I(\Cluster)}x_i^2 + \sum_{i<j,\ i,j\in I(\Cluster)} \left[(x_i-x_j)^2+(x_i+x_j)^2\right] = (2q-1)\sum_{i\in I(\Cluster)}x_i^2.
        \end{equation*}
    Moreover, for $i\in I(\Cluster)$,
        \begin{equation*}
            \hcx{i} = x_i + \sum_{j\in I(\Cluster),\,j\neq i} \left[(x_i-x_j)+(x_i+x_j)\right] = (2q-1)x_i.
        \end{equation*}
    Thus
        \begin{equation*}
            \sum_{k=1}^{\no}(\hcx{k})^2 = (2q-1)\rcx .
        \end{equation*}
    The generated subsystem contains $q$ short roots and $2\binom q2=q(q-1)$ long roots.  Consequently,
        \begin{equation*}
            \sum_{\alpha\in R_+(\Cluster)}|\alpha|^2 = q+2q(q-1) = q(2q-1).
        \end{equation*}
    Hence
        \begin{equation*}
            2\sum_{k=1}^{\no}(\hcx{k})^2 - \rcx\sum_{\alpha\in R_+(\Cluster)}|\alpha|^2 = (2-q)(2q-1)\rcx \le 0,
        \end{equation*}
    because here $q\ge2$ in the multiple-collision case. Finally, in the type $\DN$ case we have
        \begin{equation*}
            \rcx = \sum_{i<j,\ i,j\in I(\Cluster)} \left[(x_i-x_j)^2+(x_i+x_j)^2\right] = 2(q-1)\sum_{i\in I(\Cluster)}x_i^2.
        \end{equation*}
    Moreover, for $i\in I(\Cluster)$,
        \begin{equation*}
            \hcx{i} = \sum_{j\in I(\Cluster),\,j\neq i} \left[(x_i-x_j)+(x_i+x_j)\right] = 2(q-1)x_i.
        \end{equation*}
    Thus
        \begin{equation*}
            \sum_{k=1}^{\no}(\hcx{k})^2 = 2(q-1)\rcx .
        \end{equation*}
    The generated subsystem contains $2\binom q2=q(q-1)$ positive roots, all of length squared equal to $2$. Therefore
        \begin{equation*}
            \sum_{\alpha\in R_+(\Cluster)}|\alpha|^2 = 2q(q-1),
        \end{equation*}
    and
        \begin{equation*}
            2\sum_{k=1}^{\no}(\hcx{k})^2 - \rcx\sum_{\alpha\in R_+(\Cluster)}|\alpha|^2 = -2(q-1)(q-2)\rcx \le 0.
        \end{equation*}
    In particular, the type $\DN$ zero-pair case corresponds to $q=2$, where the last expression is equal to zero.  This is still sufficient, because the active singular term \eqref{eq:k:main} is strictly negative near the zero-pair wall by Assumption~\ref{ass:k:positive}.

    We can now finish the proof, since we get to the conclusion, that the whole expression related to $\rcx$, on sufficiently small neighborhood of $y$, is non-positive and thus can not explode in finite time. It means that the process does not reach $y$ in finite time almost surely, which means that $\tau_0^{m}=\tau_0^{m+1}$ for every $m\geq 2$ on $\{\tau_0^m\leq \tau_{\mathcal K}\}\cap \{\tau_0^m< \life{X}\}$. Since $\tau_0^{\SimpleNo+1} = \life{X}$, this ends the proof. 
    
\end{proof}

    We next prove tha whenever \ref{ass:k:positive} holds, the process $U$, which is the baseline for our construction of a solution $X$ does not leave the set $\mathcal{K}$. It is crucial to remove the restriction related to $\tau_{\mathcal K}$ from our results and consequently provide the general result considering the multi-collisions in the positive repulsion regime.  


\subsection{No exit from the invariant chamber image}
We begin with showing that $U$ stays in $\mathcal{K}$, whenever we start our diffusion from inside of the set. Then, in particular, the positive repulsion assumed in \ref{ass:k:positive} pushed the process away from the multiple collision points, as we have seen in Proposition~\ref{prop:no_multi_collisions:first}. This is one of the most important observation used in the following. 

\begin{proposition}
    \label{prop:no-exit-K-classical}
    Assume \ref{ass:cont}, \ref{ass:sigma}, and \ref{ass:k:positive}.  If $U(0)=w(x_0)$ with $x_0\in\C$, then
        \begin{equation}
            \label{eq:no-exit-K-classical}
            U(t)\in\mathcal K=\overline{w[\C]}, \qquad 0\le t<\life{U},
        \end{equation}
    almost surely.  Consequently $\tau_{\mathcal K}=\life{U}$ a.s. and $X(t)=f(U(t))$ for all $t<\life{X}=\life{U}$.
\end{proposition}

\begin{proof}
    We write the proof in a form which covers the three classical root systems at once.  The first step is to record the algebraic equations defining the closed chamber image.

    Let $M_R:=|R_+|$.  In type $\AN$, let $P_u^A$ be the monic polynomial whose normalized power sums are $u_1,\ldots,u_N$, and let $z_1(u),\ldots,z_N(u)$ be its complex roots.  Define $\Phi_1^A,\ldots,\Phi_{M_R}^A$ by
        \begin{equation}
            \label{eq:Phi-A-defn}
            \prod_{1\le i<j\le N} \left(\lambda+(z_i(u)-z_j(u))^2\right) = \lambda^{M_R}+\Phi_1^A(u)\lambda^{M_R-1} +\cdots+\Phi_{M_R}^A(u).
        \end{equation}
    In types $\BN$ and $\DN$ we use the squared variables $y_i=x_i^2$.  For $R=\BN$, let $P_u^B$ be the monic polynomial whose normalized power sums are
        \begin{equation*}
            \frac1k\sum_{i=1}^N y_i^k=2u_k, \qquad k=1,\ldots,N.
        \end{equation*}
    Let $y_1(u),\ldots,y_N(u)$ be the complex roots of $P_u^B$ and define $\Phi_1^B,\ldots,\Phi_{N^2}^B$ by
        \begin{equation}
            \label{eq:Phi-B-defn}
            \prod_{i=1}^N(\lambda+y_i(u)) \prod_{1\le i<j\le N} \left((\lambda+y_i(u)+y_j(u))^2-4y_i(u)y_j(u)\right) = \lambda^{N^2}+\Phi_1^B(u)\lambda^{N^2-1} +\cdots+\Phi_{N^2}^B(u).
        \end{equation}
    For $R=\DN$, let $P_u^D$ be the monic polynomial in the variables $y_i=x_i^2$ whose normalized power sums are
        \begin{equation*}
            \frac1k\sum_{i=1}^N y_i^k=2u_k, \qquad k=1,\ldots,N-1,
        \end{equation*}
    and whose top elementary symmetric polynomial is
        \begin{equation*}
            e_N(y_1,\ldots,y_N)=u_N^2.
        \end{equation*}
    Let $y_1(u),\ldots,y_N(u)$ be the complex roots of $P_u^D$ and define $\Phi_1^D,\ldots,\Phi_{N(N-1)}^D$ by
        \begin{equation}
            \label{eq:Phi-D-defn}
            \prod_{1\le i<j\le N} \left((\lambda+y_i(u)+y_j(u))^2-4y_i(u)y_j(u)\right) = \lambda^{N(N -1)}+\Phi_1^D(u)\lambda^{N(N-1)-1} +\cdots+\Phi_{N(N-1)}^D(u).
        \end{equation}
    Note that there is no factor $\prod_i(\lambda+y_i)$ in type $\DN$, because type $\DN$ has no short roots.

    \medskip

    The coefficients in \eqref{eq:Phi-A-defn}-\eqref{eq:Phi-D-defn} are genuine real polynomials in the invariant coordinates $u$.  Indeed, they are symmetric polynomials in the roots of the corresponding polynomial $P_u^R$, and hence are polynomials in its coefficients, equivalently in the basic invariant coordinates chosen above.  Moreover,
        \begin{equation}
            \label{eq:K-Phi-all-types}
            \mathcal K = \{u:\Phi_m^R(u)\ge0,\ m=1,\ldots,M_R\}.
        \end{equation}
    To justify this, observe first that if $u=w(x)$ with $x\in\Cc$, then the factors in \eqref{eq:Phi-A-defn}-\eqref{eq:Phi-D-defn} are exactly $\lambda+\langle x,\alpha\rangle^2$, $\alpha\in R_+$.  Thus all coefficients $\Phi_m^R(u)$ are non-negative.  Conversely, if all these coefficients are non-negative, the polynomial in $\lambda$ on the left-hand side cannot vanish at any positive $\lambda$.  In type $\AN$, a non-real pair $a\pm ib$ among the roots would give the positive zero $\lambda=4b^2$.  In types $\BN$ and $\DN$, a non-real pair $a\pm ib$ among the squared roots $y_i$ gives a positive zero of one of the pair factors
        \begin{equation*}
            (\lambda+y_i+y_j)^2-4y_iy_j.
        \end{equation*}
    In type $\BN$, a negative real squared root gives a positive zero of the short root factor $\lambda+y_i$.  In type $\DN$, negative real squared roots are also impossible: since $\prod_i y_i=u_N^2\ge0$, either at least two negative roots occur, or one negative root occurs together with a zero root, and then one of the pair factors again has a positive zero.  Hence the squared roots are non-negative real numbers in types $\BN$ and $\DN$, and the roots are real in type $\AN$.  Ordering the roots, and in type $\DN$ choosing the sign of the last coordinate according to the sign of $u_N$, gives a point $x\in\Cc$ with $u=w(x)$.  This proves \eqref{eq:K-Phi-all-types}.

    \medskip

    We now prove that the solution cannot leave this set.  Suppose, to the contrary, that $\tau_{\mathcal K}<\life{U}$ with positive probability.  By continuity, $U(\tau_{\mathcal K})\in\partial\mathcal K$ on this event.  Put
        \begin{equation*}
            u_*:=U(\tau_{\mathcal K}), \qquad x_*:=f(u_*).
    \end{equation*}
    By Proposition~\ref{prop:no_multi_collisions:first}, no two non-orthogonal positive roots can vanish at $x_*$. Thus the vanishing roots at $x_*$ are pairwise orthogonal.  The only point which requires separate mention is type $\DN$: if a zero pair occurs, then both $e_i-e_j$ and $e_i+e_j$ vanish, and these two roots are orthogonal.

    \medskip

    We next describe local defining functions for $\mathcal K$ near $u_*$.  After shrinking a neighborhood $\mathcal O$ of $u_*$ if necessary, one can find real analytic functions $r_1,\ldots,r_q$ on $\mathcal O$ such that
        \begin{equation}
            \label{eq:local-K-r-all-types}
            \mathcal K\cap\mathcal O = \{u\in\mathcal O:r_1(u)\ge0,\ldots,r_q(u)\ge0\}.
        \end{equation}
    Before going to the construction given below note that formally we do not start with $u_*$, which is random and then consider the random set $\mathcal{O}$, but instead we cover all boundary points in $\partial \mathcal K$, which have property that the vanishing roots at corresponding boundary point $x$ are pairwise orthogonal by open cover. Since the space considered is separable, we can choose a countable subcover, and then we can construct family of functions $r_\ell$ for each element of this subcover, as it is described below, and finally work on the set $U(\tau_{\mathcal K}) \in \mathcal{O}_i$ for some $i\in \N$. We omit the subscript $i$ in the notation to make the proof easier to read. The construction of local defining functions is as follows. In type $\AN$, each active root corresponds to a separated double root of the local polynomial $P_u^A$.  We factor the corresponding pair as $z^2-s_\ell(u)z+p_\ell(u)$ and put
        \begin{equation*}
            r_\ell(u):=s_\ell(u)^2-4p_\ell(u).
        \end{equation*}
    On $\mathcal K$ this gives
        \begin{equation*}
            r_\ell(w(x))=\langle x,\alpha_\ell\rangle^2.
        \end{equation*}
    In type $\BN$, there are two possible local factors.  For a short-root wall $x_i=0$, the corresponding squared root $y_i$ is a separated simple root of $P_u^B$ and we take
        \begin{equation*}
            r_\ell(u):=y_i(u), \qquad r_\ell(w(x))=x_i^2.
        \end{equation*}
    For a non-zero pair collision $x_i=x_j>0$, the two squared roots form a local quadratic factor $z^2-s_\ell(u)z + p_\ell(u)$ and we take
        \begin{equation*}
            r_\ell(u):=s_\ell(u)^2-4p_\ell(u), \qquad r_\ell(w(x))=(x_i^2-x_j^2)^2.
        \end{equation*}
    The root $e_i+e_j$ cannot occur as an isolated simple wall in type $\BN$ inside $\Cc$; if $x_i+x_j=0$, then both coordinates are zero and non-orthogonal roots also vanish, a case excluded by the stopped no-multiple-collision result. In type $\DN$, a non-zero collision $x_i=x_j\ne0$ or $x_i=-x_j\ne0$ is treated in the same squared-variable way, with
        \begin{equation*}
            r_\ell(w(x))=(x_i^2-x_j^2)^2.
        \end{equation*}
    There is also the zero-pair case.  Then the two small squared roots form a local quadratic factor $z^2-s(u)z+p(u)$ with $s=x_i^2+x_j^2$ on $\mathcal K$.  Since the inactive coordinates stay away from zero, the product coordinate $u_N=e_N(x)$ gives an analytic local function $m(u)$ satisfying $m(w(x))=x_i x_j$. We then  use the two local variables
        \begin{equation*}
            r_-(u):=s(u)-2m(u), \qquad r_+(u):=s(u)+2m(u),  
        \end{equation*}
    so that, on $\mathcal K$,
        \begin{equation*}
            r_-(w(x))=(x_i-x_j)^2, \qquad r_+(w(x))=(x_i+x_j)^2.
        \end{equation*}
    Together with the variables coming from the remaining active, mutually orthogonal collisions, these functions give the local description \eqref{eq:local-K-r-all-types}.

    \medskip

    Fix one of the local functions $r_\ell$ and put $R_\ell(t)=r_\ell(U(t))$. On the stochastic interval on which $U(t)\in\mathcal O$, It\^o's formula applied to the smooth function $r_\ell(U)$ gives
        \begin{equation}
            \label{eq:r-local-SDE-all-types}
            dR_\ell(t) = \sum_{i=1}^N\theta_{\ell i}(t,U(t))\,dB_i(t) +\gamma_\ell(t,U(t))\,dt.
        \end{equation}
    We claim that, after shrinking $\mathcal O$ and localizing to a compact time interval $[0,T]$, there are constants $c,C>0$ such that
        \begin{align}
            \label{eq:theta-local-estimate-all-types}
            \sum_{i=1}^N\theta_{\ell i}^2(t,u) &\le C |r_\ell(u)|, &&0\le t\le T, \ u\in\mathcal O, \\
            \label{eq:gamma-local-estimate-all-types}
            \gamma_\ell(t,u) &\ge c, &&0\le t\le T, \ u\in\mathcal O, \ r_\ell(u)\le0.
        \end{align}
    On the hypersurface $\{r_\ell=0\}$, the two local roots corresponding to the active factor coincide.  Hence the first-order variation of the local discriminant $r_\ell$ in every Brownian direction of the invariant-coordinate equation vanishes:
        \begin{equation*}
            \theta_{\ell i}(t,u)=0, \qquad u\in\mathcal O,\ r_\ell(u)=0,\quad i=1,\ldots,N.
        \end{equation*}
    The estimate \eqref{eq:theta-local-estimate-all-types} is obtained directly on a local branch.  Let $z_\ell(x)=\langle x,\alpha_\ell\rangle$ be the active root coordinate.  Locally, the corresponding boundary coordinate has the form
        \begin{equation*}
            r_\ell(w(x))=q_\ell(x)z_\ell(x)^2,
        \end{equation*}
    where $q_\ell$ is positive and continuous.  Applying It\^o's formula on the interior branch shows that the martingale coefficient of $R_\ell=r_\ell(U)$ contains the factor $z_\ell(x)$.  Hence, on compact time intervals and after shrinking the neighbourhood,
        \begin{equation*}
             \sum_{i=1}^N\theta_{\ell i}^2(t,u) \le C z_\ell(x)^2\le C' |r_\ell(u)|.
        \end{equation*}
    The same argument applies to the walls $x_i-x_j=0$, $x_i+x_j=0$, and, in type $\BN$, to the short-root wall $x_i=0$. This gives \eqref{eq:theta-local-estimate-all-types}.  The drift at the wall is
        \begin{equation}
            \label{eq:gamma-wall-value-all-types}
            \gamma_\ell(t,w(x_*)) = q_\ell(x_*) \left( \sum_{i=1}^N \alpha_{\ell,i}^2\sigma_i^2(t,x_*)+2|\alpha_\ell|^2 k_{\alpha_\ell}(t,x_*) \right)>0.
        \end{equation}
    The remaining singular-looking terms are harmless: inactive denominators stay bounded away from zero, while active roots different from $\alpha_\ell$ are orthogonal to $\alpha_\ell$.  Assumption~\ref{ass:sigma} guarantees the required continuity at long-root walls.  By Assumption~\ref{ass:k:positive} and continuity, \eqref{eq:gamma-local-estimate-all-types} follows after shrinking $\mathcal O$ and localizing in time.

    \medskip
    
    We now apply Tanaka's formula.  Let $s<\tau_{\mathcal K}$ be such that $U(s)\in\mathcal O$ and stop the process at the first time after $s$ when it leaves $\mathcal O$, and also at the fixed time $T$.  Since $R_\ell(s)\ge0$, Tanaka's formula for $R_\ell^-:=\max\{-R_\ell,0\}$ gives on the localized interval
        \begin{align*}
            R_\ell^-(t) &=-\sum_{i=1}^N\int_s^t \mathbf 1_{\{R_\ell(r)<0\}} \theta_{\ell i}(r,U(r))\,dB_i(r) \\
                &\quad -\int_s^t\mathbf 1_{\{R_\ell(r)<0\}}\gamma_\ell(r,U(r))\,dr+\frac12\left(L_t^0(R_\ell)-L_s^0(R_\ell)\right).
        \end{align*}
    By \eqref{eq:theta-local-estimate-all-types},
        \begin{equation*}
            d\langle R_\ell\rangle_t = \sum_{i=1}^N\theta_{\ell i}^2(t,U(t))\,dt \le C|R_\ell(t)|\,dt.
        \end{equation*}
    Moreover,
        \begin{equation*}
            \int_s^t \mathbf 1_{\{R_\ell(r)\ne0\}} \frac{1}{|R_\ell(r)|}\, d\langle R_\ell\rangle_r \le C(t-s)<\infty.
        \end{equation*}
    By the occupation-time formula, this implies
        \begin{equation*}
            L_t^0(R_\ell)-L_s^0(R_\ell)=0.
    \end{equation*}
     The stochastic integral is a true martingale after the localization.  Taking expectations and using \eqref{eq:gamma-local-estimate-all-types}, we obtain
        \begin{equation*}
            \ex R_\ell^-(t) \le -c\,\ex\int_s^t\mathbf 1_{\{R_\ell(r)<0\}}\,dr \le0.
        \end{equation*}
    Thus $R_\ell^-(t)=0$ throughout the localized interval.  No local defining function $r_\ell$ can become negative while the path remains in $\mathcal O$.

    Returning to the supposed first exit, choose $s<\tau_{\mathcal K}$ so close to $\tau_{\mathcal K}$ that $U(s)\in\mathcal O$.  On $[s,\tau_{\mathcal K})$ all $r_\ell(U(t))$ are non-negative.  The previous paragraph shows that they remain non-negative as long as the path stays in $\mathcal O$.  But by \eqref{eq:local-K-r-all-types}, leaving $\mathcal K$ inside $\mathcal O$ is exactly the event that at least one of the functions $r_\ell$ becomes negative.  This contradicts the definition of $\tau_{\mathcal K}$.  Hence $\tau_{\mathcal K}=\life{U}$ almost surely.
\end{proof}

The next step is to generalize the result we have just obtained to all possible starting points including boundary points in $\Cb$. This can be easily achieved by limiting procedure. 

\begin{proposition}
    \label{prop:boundary-starts-by-approximation}
    Assume \ref{ass:cont}, \ref{ass:sigma}, and \ref{ass:k:positive}.  Then, for every $x_0\in\Cc$, there exists a weak solution $U$ of the invariant-coordinate equation \eqref{eq:u:SDE}, started from $U(0)=w(x_0)$, such that
        \begin{equation*}
            U(t)\in\mathcal K:=\overline{w[\C]}, \qquad 0\le t<\life{U},
        \end{equation*}
    almost surely.  Consequently, for
        \begin{equation*}
            X(t):=f(U(t)),\qquad 0\le t<\life{X},
        \end{equation*}
    the process $X$ is well-defined and takes values in $\Cc$.
\end{proposition}

\begin{proof}
    If $x_0\in\C$, the assertion follows from Proposition~\ref{prop:no-exit-K-classical}.  Let now $x_0\in\Cb$. Choose a fixed vector $\rho\in\C$, and put
        \begin{equation*}
            x_0^{(n)}:=x_0+\frac1n\rho, \qquad u_0^{(n)}:=w(x_0^{(n)}).
        \end{equation*}
    Then $x_0^{(n)}\in\C$ and $u_0^{(n)}\to u_0:=w(x_0)$. For every $n$, let $U^{(n)}$ be the solution of \eqref{eq:u:SDE} started from $u_0^{(n)}$.  By Proposition~\ref{prop:no-exit-K-classical},
        \begin{equation*}
            U^{(n)}(t)\in\mathcal K, \qquad 0\le t<\zeta^{(n)},
        \end{equation*}
    almost surely. We now use a standard localization and compactness argument.  Fix $T,r>0$ and define
        \begin{equation*}
            \zeta_r^{(n)} := \inf\{t\ge0: |U^{(n)}(t)|\ge r\}.
        \end{equation*}
    The coefficients of \eqref{eq:u:SDE} are continuous, hence locally bounded. Therefore the laws of the stopped processes
        \begin{equation*}
            \bigl(U^{(n)}(t\wedge\zeta_r^{(n)})\bigr)_{0\le t\le T}
        \end{equation*}
    are tight in $C([0,T],\R^N)$.  Along a subsequence they converge weakly to a continuous process $U^r$.  Since $\mathcal K$ is closed and $U^{(n)}(t\wedge\zeta_r^{(n)})\in\mathcal K$ for all $t\le T$, we also have
        \begin{equation*}
            U^r(t)\in\mathcal K, \qquad 0\le t\le T,
        \end{equation*}
    almost surely. The usual stability of the localized martingale problem under locally uniform convergence of the coefficients shows that $U^r$ solves the stopped invariant-coordinate equation up to the exit time from the ball of radius $r$. Letting $r\uparrow\infty$ and using a diagonal argument gives a weak solution $U$ of \eqref{eq:u:SDE}, started from $u_0=w(x_0)$, up to its explosion lifetime $\life{U}$.  Since the stopped limits take values in the closed set $\mathcal K$, the unstopped solution satisfies
        \begin{equation*}
            U(t)\in\mathcal K, \qquad 0\le t<\zeta,
        \end{equation*}
    almost surely. Finally, because $U(t)\in\mathcal K=\overline{w[\C]}$, the inverse map $f:\mathcal K\to\Cc$ is well-defined, and therefore
        \begin{equation*}
            X(t):=f(U(t))
        \end{equation*}
    is a continuous $\Cc$-valued process.
\end{proof}

\subsection{Instant diffraction}
We have collected all ingredients to show that positive repulsion \ref{ass:k:positive} not only keeps the process away from the multiple collisions points after the start, but also immediately push the process away from the boundary. 
\begin{proposition}[Instantaneous diffraction]
    \label{prop:diffraction}
    Assume \ref{ass:cont}, \ref{ass:sigma}, and \ref{ass:k:positive}.  Let $U$ be the $\mathcal K$-valued solution constructed in Proposition~\ref{prop:boundary-starts-by-approximation}, and put $X=f(U)$.  Then the instant diffraction of the particles appears, i.e.
        \begin{equation*}
            \inf\{t>0:X(t)\in\C\}=0 \qquad\text{a.s.}
        \end{equation*}
\end{proposition}

\begin{proof}
    Let us consider the basic symmetric polynomials $\eA{n}$, where 
        \begin{equation*}
            \Aa = \left\{\xa^2: \alpha\in R_+\right\}\/,\qquad n = 1, \ldots, \mo\/,
        \end{equation*}
    where $\mo$ is the size of $\Aa$, i.e. $\mo = |R_+|$. In particular, in the $\AN$ case, we have $\mo = \no(\no-1)/2$, with $\mo = \no(\no-1)$ for $\DN$ and finally, $\mo = \no^2$, when we consider $\BN$. In this part, we will slightly break with the standard notation. When we write $\eA{n}$, we will not explicitly indicate the point $x$ for which we are calculating the value, but in fact $\Aa$ depends on $x$ and consequently $\eA{n}$ also does. Consequently, we will use the same notation to refer to both a function of the variable x and the stochastic process that we get by applying this function to $X = (X_1,\ldots, X_{\no})$. However, it will always be clear from the context which function we are referring to.

    Following the approach presented in Remark \ref{rem:direct:Ito} together with the fact that $U$ is a $\mathcal K$-valued solution, we get the stochastic description of $\eA{n}$ in the following form
    \begin{align*}
        d\eA{n} &= 
            2\sum_{k=1}^N \sum_{\aR} \alpha_k \Xa \eaA{n-1}\sigX{k}dB_k
            +2\sum_{k=1}^N \sum_{\aR} \alpha_k \Xa \eaA{n-1} b_k(t, X)dt\\
            &+2\sum_{\aR}   \eaA{n-1}|\alpha|^2 \kaX\,dt + 2\sum_{\beta\in R_+} \frac{\kbX}{\Xb} \mathop{\sum_{\aR}}_{\alpha \neq \beta , \ab\neq 0}\ab \Xa \eaA{n-1} dt\\
        &  +2 \sum_{k=1}^N\anbR\alpha_k\beta_k\Xa\Xb\eabA{n-2}\siggX{k}dt \\
        &  + \sum_{k=1}^N \sum_{\aR}\alpha_k^2 \eaA{n-1}\siggX{k}dt\/.
    \end{align*}
    Using \eqref{eq:inv:master} and the fact that if $\ab\neq 0$ then $ \ainvbb = \pm \ab \neq 0$, we can remove the apparent singularity in the following way, where to shorten the notation we write $\gamma = s_\beta^+(\alpha)$. Then, by using the definition of $s_\beta^+(\alpha) = \pm(\alpha - 2\ab\beta/|\beta|^2)$ and noticing that the given below equalities are invariant under the choice of sign in $s_\beta^+(\alpha)$, we obtain
    	\begin{equation}
        \label{eq:ac:form1}
            \ab\xa+\bc\xc = \frac{2\xb}{|\beta|^2}
    	\end{equation}
    and
        \begin{equation}
        \label{eq:ac:form2}
            \ab\xc+\bc\xa = \frac{2\ab\bc\xb}{|\beta|^2}\/.
	\end{equation}
    This follows also from the fact that for the classical root systems considered here, if $\alpha\neq\beta$ and $\langle\alpha,\beta\rangle\neq0$, then $\langle\alpha,\beta\rangle^2=1$. Using those we get
    \begin{align*}
        2 \mathop{\sum_{\aR}}_{\alpha \neq \beta , \ab\neq 0}&\ab \xa \eaA{n-1} =  \sum_{(\alpha,\gamma)\in R_+(\beta)}(\ab\xa+\cb\xc)\eacA{n-1}\\
        &+\sum_{(\alpha,\gamma)\in R_+(\beta)}\xa\xc(\ab\xc+\cb\xa)\eacA{n-2}\\
        &\underset{\eqref{eq:ac:form2}}{\overset{\eqref{eq:ac:form1}}{=}} \frac{2\xb}{|\beta|^2}\sum_{(\alpha,\gamma)\in R_+(\beta)} (\eacA{n-1}+\ab\bc\xa\xc \eacA{n-2}).
    \end{align*}
    As usual, elementary symmetric polynomials of negative degree are understood to be zero. Let us assume that one of the processes $\eA{n}$, with $n=1,\ldots, \mo$, stays to be zero for $t\in [0,T_n]$, where $T_n$ is positive with positive probability. It means that the drift part of $\eA{n}=0$ on this time interval as well as all the products of $\Xa$ of length $n$ are zero there. In particular $\Xa \eaA{n-1} = 0$, $\alpha_k\beta_k\Xa\Xb\eabA{n-2} = 0$ and $\Xa\Xc \eacA{n-2} = 0$. Using the non-negativity of the remain components of the drift term we arrive at
        \begin{equation*}
            \sum_{\aR}   \eaA{n-1} |\alpha|^2 \kaX = 0\/,\quad t\in [0,T_n]\/.
        \end{equation*}
    Now \ref{ass:k:positive} implies that $\eaA{n-1} = 0$ for every $\alpha\in R_+$, which is equivalent to $\eA{n-1}=0$ on $[0,T_n]$. Inductively we obtain 
        \begin{equation*}
            \eA{\mo} = \ldots = \eA{1} = \eA{0} = 0
        \end{equation*}
    for every $t\in [0,T_n]$ with positive probability. Since $\eA{0}\equiv 1$ we get a contradiction, which ends the proof. 
    \end{proof}

\subsection{Properties of $X$}
Using the results from the previous section we can finally state the general no multiple collision result. 

\begin{proposition}[No multiple collisions]
    \label{prop:no_multi_collisions}
     Assume \ref{ass:cont}, \ref{ass:sigma}, and \ref{ass:k:positive}.  Let $U$ be the $\mathcal K$-valued solution constructed in Proposition~\ref{prop:boundary-starts-by-approximation}, and put $X=f(U)$. Then for every $X(0)=x_0\in \Cc$, the process $X$ does not visit multiple collision points after the start. 
\end{proposition}

\begin{proof}
If $x_0\in\C$, then Proposition~\ref{prop:no_multi_collisions:first} excludes visits to $\Cbn m$, $m\ge2$, up to $\tau_{\mathcal K}$. By Proposition~\ref{prop:no-exit-K-classical}, $\tau_{\mathcal K}=\life X$, and hence no multiple collision occurs before the lifetime. Let now $x_0\in\Cc$. By Proposition~\ref{prop:diffraction}, there exist times $t_n\downarrow0$ such that $X(t_n)\in \C$ a.s. Applying the preceding interior-start argument to the shifted process $(X(t_n+s))_{s\ge0}$ gives absence of multiple collisions on $[t_n,\life X)$ for each $n$. Since $t_n\downarrow0$, the assertion follows for every positive time. 
\end{proof}

    \begin{remark}
        Notice that the proof of the above Proposition shows a slightly more general fact. Namely, we obtain that when starting from a point $X(0)=x_0\in \Cc$, which is not a multiple collision point, the process never visits multiple collision boundary points $y\in \Cb$ for which we have $\ka(t,y)>0$.
    \end{remark}

\medskip

Finally, we end this section with the proof of the non-explosion result stated in Theorem~\ref{thm:non-explosion}. The argument is standard, but we include the details for completeness. Note that 

\begin{proof}[Proof of Theorem~\ref{thm:non-explosion}]
Let $X$ be a solution to \eqref{eq:thm:existence-sde} or the degenerate interior equation \eqref{eq:degenerate-existence-indicator} and define
    \begin{align*}
        V(t) := 1+ |X(t)|^2, \qquad 0\le t<\life{X}.
    \end{align*}
For $\alpha\in R_+$, set
    \begin{align*}
        \vartheta_\alpha(s) &:=
            \begin{cases}
                1, &\text{for the equations without the indicator,}\\[2mm]
                \mathbf 1_{\{\langle X(s),\alpha\rangle>0\}}, &\text{for the equation with the indicator.}
            \end{cases}
    \end{align*}
In particular, in all cases, $0\le \vartheta_\alpha(s)\le 1$. Applying It\^o's formula gives, for every $t<\life{X}$,
    \begin{align}
        V(t) &= V(0) + M(t) + \int_0^t \left( 2\sum_{i=1}^N X_i(s)b_i(s,X(s))+\sum_{i=1}^N \sigma_i^2(s,X(s))\right)\,ds \nonumber \\
            &\quad + 2\sum_{\alpha\in R_+} \int_0^t k_\alpha(s,X(s))\vartheta_\alpha(s)\,ds,
        \label{eq:radial-nonexplosion-sde}
    \end{align}
where
    \begin{align*}
        M(t) &:= 2\sum_{i=1}^N \int_0^t X_i(s)\sigma_i(s,X(s))\,dB_i(s).
    \end{align*}
Fix $T>0$ and define
    \begin{align*}
        \tau_r &:= \inf\{0\le t<\life{X}:\ V(t)\ge r\}, \qquad r>1,
    \end{align*}
with the convention that the infimum of the empty set is $\life{X}$.
By the definition of the explosion lifetime, $\tau_r \uparrow \life{X}$ as $r\to\infty$.  Let
    \begin{align*}
        K_r := \{x\in\Cc:\ 1+|x|^2\le r\}.
    \end{align*}
On the stochastic interval $[0,T\wedge\tau_r]$, apart from the trivial case $\tau_r=0$, the process remains in the compact set $K_r$. Since the coefficients are continuous, the functions $x_i\sigma_i(t,x)$ are bounded on $[0,T]\times K_r$. Hence the stopped martingale $M(t\wedge\tau_r)$ is square-integrable on $[0,T]$. Let
    \begin{align*}
        A_m &:= \{V(0)\le m\}, \qquad m\in\N.
    \end{align*}
Since $A_m\in\mathcal F_0$ and $M(t\wedge\tau_r)$ is a square-integrable martingale, we have
    \begin{align*}
        \E\left[\mathbf 1_{A_m}M(t\wedge\tau_r)\right] &= 0, \qquad 0\le t\le T.
    \end{align*}
Moreover, since $k_\alpha\ge0$ and $0\le\vartheta_\alpha(s)\le1$, Assumption~\ref{ass:growth:p2} yields, for $0\le s\le T$ and $s<\life{X}$,
\begin{align}
    &2\sum_{i=1}^N X_i(s)b_i(s,X(s))
    +
    \sum_{i=1}^N \sigma_i^2(s,X(s))
    +
    2\sum_{\alpha\in R_+}k_\alpha(s,X(s))\vartheta_\alpha(s)
    \nonumber
    \\
    &\qquad
    \le
    2\left(
        \sum_{i=1}^N X_i(s)b_i(s,X(s))
        +
        \frac12\sum_{i=1}^N \sigma_i^2(s,X(s))
        +
        \sum_{\alpha\in R_+}k_\alpha(s,X(s))
    \right)
    \nonumber
    \\
    &\qquad
    \le
    2C_T\bigl(1+|X(s)|^2\bigr)
    =
    2C_T V(s).
    \label{eq:radial-nonexplosion-drift-bound}
\end{align}
Stopping the identity for $V$ at $\tau_r$, multiplying by $\mathbf 1_{A_m}$, taking expectations, and using
\eqref{eq:radial-nonexplosion-drift-bound}, we obtain, for $0\le t\le T$,
    \begin{align*}
        \E\left[\mathbf 1_{A_m}V(t\wedge\tau_r)\right] &\le \E\left[\mathbf 1_{A_m}V(0)\right] + 2C_T \int_0^t \E\left[\mathbf 1_{A_m}V(s\wedge\tau_r)\right]\,ds \\
             &\le  m + 2C_T \int_0^t \E\left[\mathbf 1_{A_m}V(s\wedge\tau_r)\right]\,ds .
    \end{align*}
Gronwall's lemma gives
    \begin{align}
        \E\left[\mathbf 1_{A_m}V(t\wedge\tau_r)\right] &\le m e^{2C_T t} \le m e^{2C_TT}, \qquad 0\le t\le T.
        \label{eq:radial-nonexplosion-gronwall-after}
    \end{align}
On the event $A_m\cap\{\tau_r\le t\}$ we have $V(t\wedge\tau_r) = V(\tau_r) \ge r$. Therefore \eqref{eq:radial-nonexplosion-gronwall-after} implies
    \begin{align}
        r\, \pr\bigl(A_m\cap\{\tau_r\le t\}\bigr) &\le m e^{2C_TT}, \qquad 0\le t\le T.
        \label{eq:radial-nonexplosion-prob-bound}
    \end{align}
Since $\tau_r\uparrow\life{X}$ as $r\to\infty$, we have $\{\life{X}\le t\} \subseteq \{\tau_r\le t\}$, $r>0$. Hence, by \eqref{eq:radial-nonexplosion-prob-bound},
    \begin{align*}
        \pr\bigl(A_m\cap\{\life{X}\le t\}\bigr) &\le \pr\bigl(A_m\cap\{\tau_r\le t\}\bigr) \le \frac{m e^{2C_TT}}{r}.
    \end{align*}
Letting $r\to\infty$ gives
    \begin{align}
        \pr\bigl(A_m\cap\{\life{X}\le t\}\bigr) &= 0, \qquad 0\le t\le T.
        \label{eq:radial-nonexplosion-fixed-m}
    \end{align}
Since $V(0)<\infty$ almost surely, $\pr\left(\bigcup_{m=1}^\infty A_m\right)=1$. Letting $m\to\infty$ in \eqref{eq:radial-nonexplosion-fixed-m}, we obtain
    \begin{align}
        \pr(\life{X}\le t) &= 0, \qquad 0\le t\le T.
    \end{align}
Finally, since $T>0$ was arbitrary, $\life{X}=\infty$ almost surely.
\end{proof}


\section{Integrability of the singular drift part}
The most delicate part of the proof is related to show that for $x=(x_1,\ldots,x_N)$ defined in \eqref{eq:x:defn}, the singular part of the drift from \eqref{eq:main:SDE} is integrable, i.e. for every $i=1,\ldots,N$ we have
\begin{equation}
    \label{eq:singular:integr}
        \int_0^t \dfrac{k_{\alpha}(s, X(s))}{\left<X_s,\alpha\right>}ds < \infty\qquad \alpha\in R_+,\quad  t<\life{x}\/,
\end{equation}
almost surely. Essentially, the above condition is almost sufficient for $x$ to be a solution, as we demonstrate in the following proposition.

Recall that for every $i=1,2,\ldots,N$ we introduced $\Ri = \{\aR: \alpha_i\neq 0\}$ and consider the corresponding sets 
\begin{equation*}
	\Ai=\left\{\xa^2:\aRi\right\}\/,\quad \Bi=\left\{\xa:\aRi\right\}\/,
\end{equation*}	
where we collect all the expressions $\xa^2$ (or $\xa$) in which $x_i$ appears. Note that for fixed $N$ the number of elements in $\Ri$ is the same for every $i$ and let us denote it by $M=\#\Ri$. In the $A$ case we simply have $M=N-1$.

\medskip

Consider $e_n(\Ai)$, $n=1,\ldots, M$, which control the collisions between $x_i$ and other particles (or the wall of the Weyl chamber). These are definitely continuous processes. We can also apply the It\^o formula, for every starting point in the interior of the Weyl chamber, to get (at least up to the first hitting time of $\partial C_+$) that
\begin{align}
	\label{eq:eAi:SDE2}
	\nonumber
	d\eAi{n} &= 2\sum_{k=1}^N \sigX{k} \sum_{\aRi}\alpha_k\Xa \eaAi{n-1} dB_k + 2\sum_{k=1}^N \bX{k} \sum_{\aRi}\alpha_k\Xa \eaAi{n-1}dt\\ \nonumber
	&+\sum_{k=1}^N \siggX{k}\sum_{\aRi}\alpha_k^2 \eaAi{n-1}dt+2\sum_{\aRi}\sum_{\bR}\dfrac{\ab\Xa}{\Xb} \eaAi{n-1}\kbX \\
  &+2\sum_{k=1}^N \siggX{k}\anbRi\Xa\Xb\alpha_k\beta_k\eabAi{n-2}dt\/.
\end{align}

At the first glance the last term above is singular and it is in fact true for every $n=1,\ldots,M-1$. However, the first crucial observation in this construction is that it is not the case for $n=M$, i.e. the function appearing in this expression can be extended to a continuous function on $\overline{C_+}$. Indeed, note that for $\alpha,\beta\in\Ri$ such that $\alpha\neq \beta$ we have $\eaAi{M-1}=\xb^2\eabAi{M-2}$, which is exceptional in the case $n=M$, since in general we have $\eaAi{M-1}=\xb^2\eabAi{M-2}+\eabAi{M-1}$. 

For $\alpha\in\Ri$ and $\beta\notin \Ri$ such that $\ab\neq 0$ let us denote $\gamma=\prf{\beta}{\alpha}$, where $\prf{\beta}{\alpha}$ was  defined in Section~\ref{sec:preliminaries}. Thus
\begin{align*}
 \sum_{\aRi}\sum_{\bR}&\dfrac{\ab\xa}{\xb} \eaAi{M-1}k_\beta(x) = \anbRi\ab\xa\xb\eabAi{M-2}k_\beta(x)\\
&+ \sum_{\aRi}|\alpha|^2\eaAi{M-1}\ka(x)+\sum_{\beta\notin\Ri}k_\beta(x) \sum_{\aRi}\frac{\ab\xa\xc^2}{\xb}\eacAi{M-2}\/.
\end{align*}
To deal with the last expression observe that $\cb=-\ab$ and consequently, $\gamma(\gamma(\alpha,\beta),\beta)=\alpha$. Since
\begin{align*}
\ab\left<x,\gamma(\alpha,\beta)\right>+\left<\gamma(\alpha,\beta),\beta\right>\xa &= -\frac{2\ab\xb}{|\beta|^2}\/,
\end{align*}
we obtain the following simplification
\begin{align*}
\sum_{\aRi}\frac{\ab\xa\xc^2}{\xb}\eacAi{M-2} &=\frac{1}{2}\sum_{\aRi}\frac{\xa\xc(\ab\xc+\cb\xa)}{\xb}\eacAi{M-2}\\
&=-\sum_{\aRi}\frac{\ab}{|\beta|^2}\xa\xc\eacAi{M-2}\/.
\end{align*}
In the next proposition we in fact show that $\eAi{M}$ are semimartingales and the above-given description holds without any additional assumption on the starting point and with no restriction on time.


\begin{proposition}
    \label{prop:vandermond:i}
	For every $i=1,\ldots,N$ the processes $\eAi{M}$ are semi-martingales described by
	\begin{equation}
		\label{eq:eAi:SDE}
		d\eAi{M} = 2\eBi{M}\sum_{k=1}^N\sigX{k}\sum_{\aRi}\alpha_k \eaBi{M-1} dB_k + (V^{(1)}_{M}(X)+\eBi{M}V_{M}^{(2)}(X))dt\/,
	\end{equation}
	where
\begin{align*}
	V^{(1)}_M(X) &= \sum_{k=1}^N \siggX{k}\left(\sum_{\aRi}\alpha_k \eaBi{M-1}\right)^2+2\sum_{\aRi}|\alpha|^2\eaAi{M-1}\kaX
\end{align*}
and
\begin{align*}	
	V^{(2)}_M(X) &= 2\sum_{k=1}^N\siggX{k}\anbRi\alpha_k\beta_k\eabBi{M-2}+2\sum_{k=1}^N \bX{k}\sum_{\aRi}\alpha_k \eaBi{M-1}\\
  &+2\anbRi\ab\eabBi{M-2}k_\beta(x)-2\sum_{\beta\notin\Ri}\sum_{\aRi}\frac{\left<\alpha,\sigma_\beta(\alpha)\right>}{|\beta|^2}\easBi{M-2}{\sigma_\beta(\alpha)}k_\beta(x)\/.
\end{align*}
\end{proposition}
\begin{proof}
	We start with the following calculation in $A_{N-1}$ case, where we exploit the fact that $\eBi{M}$ is just characteristic polynomial with roots $\{x_j:j\neq i\}$ evaluated at point $x_i$, which immediately gives
	\begin{align*}
		\eBi{M}  & = \prod_{\aRi}\xa = (-1)^{i+1}\sum_{k=0}^{N-1} x_i^{N-1-k}(-1)^{k}e_{k}^{\overline{x_i}}(x)\\
        & = (-1)^{i+1}\sum_{k=0}^{N-1} (N-k)\,x_i^{N-1-k}(-1)^{k}e_{k}(x)\/.
	\end{align*}
	It is now easy to see that for $R=B_N$ we have
	\begin{equation*}
		\eBi{M} = x_i\sum_{k=0}^{N-1} (N-k)\,(x_i^2)^{N-1-k}(-1)^{k}e_{k}(x^2)\/,
	\end{equation*}
	where $x^2=(x_1^2,\ldots,x_N^2)$ and removing the first factor $x_i$ we get the representation for $R=D_N$. We can thus conclude that in general we have
	\begin{equation*}
		\eBi{M} = \sum_{k=0}^{N-1}g_k^R(x_i)H_k^R(u_1,\ldots,u_N)\/,
	\end{equation*}
	where $g_k^R$ is suitable power function and, which is more important, $H_k^R$ is smooth (polynomial in fact) function of $u=(u_1,\ldots,u_N)$. This together with Proposition \ref{prop:xipk} shows that $\eAi{M}$ is a smooth function of $u=(u_1,\ldots,u_N)$ whenever we are away from this part of the boundary which corresponds to collisions involving $x_i$ with its neighbors. Indeed, for $R=A_{N-1}$ we have
	\begin{eqnarray*}
		\dfrac{\partial}{\partial p_k}(\eAi{M}) &=& 2\eBi{M} \dfrac{\partial}{\partial p_k}\sum_{l=0}^{N-1}g_l^{A}(x_i)H_l^{A}(p_1,\ldots,p_N)
	\end{eqnarray*}
		which by \eqref{eq:xipk} is equal to
	$$	
	 2\sum_{l=0}^{N-1}(g_l^{A})'(x_i)(-1)^{i+k}{e_{N-k+1}^{\overline{x_i}}(x)}H_l^{A}(p_1,\ldots,p_N)+2\eBi{M}\sum_{l=0}^{N-1}g_l^{A}(x_i)\dfrac{\partial}{\partial p_k}H_l^{A}(p_1,\ldots,p_N)\/.
	$$
  Obviously it is a smooth function. Moreover, we can exploit the relation ${e_{l}^{\overline{x_i}}(x)}=e_{l}(x)-x_ie_{l-1}^{\overline{x_i}}(x)$ to represents each of ${e_{N-k+1}^{\overline{x_i}}(x)}$ as a polynomial in $x_i$ with coefficients being smooth functions of $(p_1,\ldots,p_N)$. Another usage of \eqref{eq:xipk} will show that the second derivative with respect to $p_k$ and $p_l$ of $\eAi{M}$ will be a smooth function whenever $\eAi{M}>0$ (More precisely on the corresponding image of $\{x:\eAi{M}>0\}$ by $p=(p_1,\ldots,p_N)$). It is because the singular expression $1/\eBi{M}$ will appear and we do not have the factor $\eBi{M}$ anymore to remove this singularity as it has happened in the first derivative. Note that in the $B_N$ case we consider polynomials $(p_1,\ldots,p_N)$ of $(x_1^2,\ldots,x_N^2)$ and $\eAi{M}$ is just the same function as in $A_{N-1}$ case with $(x_1,\ldots,x_N)$ replaced by $(x_1^2,\ldots,x_N^2)$ and multiplied by $x_i^2$. Consequently, the similar arguments lead to the corresponding claim. Finally, in $D_N$ case, in $\eAi{M}$ we just replace particles by its squares and the result follows directly from the consideration in $A_{N-1}$ case.

	\medskip
	
	Now let us consider the family of functions $\psi_\varepsilon:\R\to\R$ for $\varepsilon>0$ defined by $\psi_\varepsilon(z)=z\varphi(z/\varepsilon)$, where $\varphi$ is a fix smooth function on $\R$ with values in $[0,1]$ such that $\varphi(z)=0$ for $|z|<1$ and $\varphi(z)=1$ for $|z|\geq 2$. It is then clear that we have the following point wise convergence
	\begin{eqnarray}
        \label{eq:psi:limit}
		\psi_\varepsilon(z) & \stackrel{\varepsilon\to 0}{\longrightarrow}& z\/,\quad z\in\R\/,\\
        \label{eq:psi:limit:first}
		\psi'_\varepsilon(z)  = \varphi(z/\varepsilon)+\frac{z}{\varepsilon}\varphi'(z/\varepsilon) & \stackrel{\varepsilon\to 0}{\longrightarrow}& \ind_{\R\setminus\{0\}}(z)\/,\quad z\in\R\/,\\
        \label{eq:psi:limit:double}
		\psi''_\varepsilon(z) = \frac{2}{\varepsilon}\varphi'(z/\varepsilon) + \frac{z}{\varepsilon^2}\varphi''(z/\varepsilon) & \stackrel{\varepsilon\to 0}{\longrightarrow}& 0\/,\quad z\in\R\/,
	\end{eqnarray}
	since we have $\varphi'(z/\varepsilon)=\varphi''(z/\varepsilon)=0$ for every $|z|>2\varepsilon$ as well as for $z=0$. Moreover, we have the following uniform estimates
	\begin{align}
		\label{eq:psiderv:bounds}
		|\psi'_\varepsilon(z)| \leq ||\varphi||_\infty+||\varphi'||_\infty\/,\quad |z \psi''(z)|\leq 2||\varphi'||_\infty+||\varphi''||_\infty\/,
	\end{align}
	which hold for every $z\in\R$ and $\varepsilon>0$. Using $\psi_\varepsilon$ we consider the family of diffusions $\psi_\varepsilon(\eAi{M})$, which are well-defined semi-martingales for every $t<T$. Indeed, as we have already seen, the function 
	\begin{equation*}
		u\longrightarrow \psi_\varepsilon\left(\left(\sum_{k=0}^{N-1}g_k^R(f_i(u))H_k^R(u_1,\ldots,u_N)\right)^2\right)
	\end{equation*}
	is smooth, because the function $\psi_\varepsilon$ removes the problem with singularity of the derivatives of $f_i(u)$, which appears when $\xa=0$ for $\aRi$ or equivalently $\eAi{M}=0$. The It\^o formula gives that
	\begin{align}
			\label{eq:eAi:first}
			d\psi_\varepsilon(\eAi{M}) 
			& = 2\psi_\varepsilon'(\eAi{M}) \eBi{M}\sum_{k=1}^N\sigma_k(x)\sum_{\aRi}\alpha_k \eaBi{M-1} dB_k\\
			&+ \psi_\varepsilon'(\eAi{M})(V_M^{(1)}(X)+V_M^{(2)}(X)\eBi{M})dt\\
			\label{eq:eAi:last}
			&+2\psi_\varepsilon''(\eAi{M}) \eAi{M}\sum_{k=1}^N\sigma_k^2(x)\left(\sum_{\aRi}\alpha_k \eaBi{M-1}\right)^2dt
	\end{align}
	Standard localization argument, i.e. considering our processes for $t<T_K=\inf\{x: |x_i|\leq K, i=1,\ldots,N\}$ and then taking $K\to \infty$, which gives $T_K\to T$, allows us to consider all the above-given coefficients to by uniformly (for $t<T_K$) bounded. 
	Then using the bounds in \eqref{eq:psiderv:bounds} we can easily show that \eqref{eq:eAi:last} vanishes when $\varepsilon\to 0$. Since $\psi_\varepsilon'(z) \to  \ind_{\R\setminus\{0\}}(z)$ and $\eBi{M}$ vanishes iff $\eAi{M}$ vanishes we get that 
	$$
		\psi_\varepsilon'(\eAi{M}) \eBi{M} \longrightarrow \eBi{M}
	$$
	and consequently the right-hand side of \eqref{eq:eAi:first} goes to the 
	$$
		2\eBi{M}\sum_{k=1}^N\sigma_k(x)\sum_{\aRi}\alpha_k \eaBi{M-1} dB_k\/,
	$$
	which follows from the $L_2$ property of the Ito integral. Finally, 
	$$
		\psi_\varepsilon'(\eAi{M})(V_M^{(1)}+V_M^{(2)}\eBi{M})dt \longrightarrow \ind_{\R\setminus\{0\}}(\eAi{M})(V_M^{(1)}(X)+V_M^{(2)}(X)\eBi{M})dt\/.
	$$
	Since $\psi_\varepsilon(\eAi{M}) \to \eAi{M}$ we have just proved that $\eAi{M}$ is a semi-martingale for every $t<T$ given by \eqref{eq:eAi:SDE} but with additional indicator $\ind_{\R\setminus\{0\}}(\eAi{M})$ in the drift part. To finish the proof it is enough to show that we can replace this indicator simply by $1$. Indeed, taking the semi-martingale $Y_t=\eAi{M}_t$, since the process is non-negative, $L_t^{0-}[Y]_t\equiv 0$, where $L_t^a[Y]$ denotes the local time of $Y$. From the other-side, the occupation times formula gives
	\begin{align*}
		\int_{[0,\infty)}\ind_{\R\setminus \{0\}}(a)\frac{1}{a}L_t^{a}[Y]da &= \int_0^t \ind_{\R\setminus \{0\}}(Y_s)\frac{1}{Y_s}dY_sdY_s \\
		& = 4\int_0^t \ind_{\R\setminus \{0\}}(\eAi{M}_s) \sum_{k=1}^N \sigma_k^2(s, X_s)\left(\sum_{\aRi}\alpha_k \eaBi{M-1}_s\right)^2ds\/,
	\end{align*}
	where the last expression is always finite. It implies that $L_t^{0}[Y] = \lim_{a\to 0^+}L_t^{a}[Y] =0$. This together with the formula (see \cite{bib:RevuzYor:1999} Theorem 1.7 page 225)
	$$
		0=L_t^{0}[Y]-L_t^{0-}[Y] = \int_0^t \ind_{\{0\}}(Y_s)dV_s
	$$
	with $V = V_M^{(1)}+V_M^{(2)}\eBi{M}$ implies the desired result.
\end{proof}

\begin{proposition}
        \label{prop:singular:integrability}
        Assume that $X=(X_1,\ldots,X_N)$ is a diffusion process defined in \eqref{defn:x} such that $X(0)\in \C$. If the continuity assumption \ref{ass:cont}, \ref{ass:sigma} together with \ref{ass:k:positive} hold, then for every $\alpha\in R_+$
            \begin{equation}
        \label{eq:singular:integrability}
        \int_0^t \frac{k_\alpha(s,X(s))}{\left<X(s),\alpha\right>}ds<\infty\/,\quad 
    \end{equation}
    for every $t<\life{X}$.
\end{proposition}

\begin{proof}
    We start by localizing our consideration and define $T_r = \inf\{t: \norm{X_t}>r\}\wedge \inf\{t>0: \eAi{M-1}<1/r\}\wedge r$ for given and fixed $r>0$ such that $\norm{X(0)}<r$. We will consider $t<T_r$. We begin with the following simple computation
    \begin{equation*}
         \int_{0}^{t} \sum_{\aRi}\dfrac{k_\alpha(s,X(s)) |\alpha|^2}{\left<X(s),\alpha\right>}ds = \int_{0}^{t} \sum_{\aRi}|\alpha|^2\eaBi{M-1} \frac{k_\alpha(s,X(s))}{\eBi{M}}ds\/.
    \end{equation*}
    Recall that Proposition \ref{prop:no_multi_collisions} implies the existence of the positive $c_3=c_3(\omega)>0$ such that $\eBi{M-1}_t>c_3$ for every $t<T_r$, which gives that the last integral can be bounded from above by
    \begin{eqnarray*}
        \frac{1}{c_3} \int_{0}^{t} \sum_{\aRi}|\alpha|^2\eaBi{M-1}\eBi{M-1} \frac{ k_\alpha(s,X(s))}{\eBi{M}_s}ds\/,
    \end{eqnarray*}
    where we also have
    \begin{eqnarray*}
        \sum_{\aRi}|\alpha|^2\eaBi{M-1}\eBi{M-1} &=& \sum_{\aRi}|\alpha|^2\eaAi{M-1} +\anbRi|\alpha|^2\eaBi{M-1}\ebBi{M-1}\\
        &=&\sum_{\aRi}|\alpha|^2\eaAi{M-1} +\anbRi|\alpha|^2\eBi{M}\eabBi{M-2}\/.
    \end{eqnarray*}
    Since the factor $\eBi{M}$ in the last sum cancels the singularity in the integral, we can just reduce our consideration to show the finiteness of the integral
    \begin{equation*}
         \int_{0}^{t}  \sum_{\aRi}|\alpha|^2\eaAi{M-1} \frac{k_\alpha(s,X(s))}{\eBi{M}_s}ds\/.
    \end{equation*}

    Note that since $\eAi{M-1}$ stays strictly positive, we will get $\lim_{r\to \infty} T_r = \life{X}$. Due to the continuity of the coefficient and the lack of singularities in \eqref{eq:eAi:SDE}, we can find a constant $c_1=c_1(r)>0$ such that $\norm{V_M^{(2)}(x)}\leq c_1(r)$ for every $x\in\R^N$ such that $\norm{x}\leq r$. Moreover, by  \ref{ass:k:positive}, the functions $(t, x) \to \sigma_k(t,x)$ and $(t, x) \to k_\alpha(t,x)$ are continuous, and the latter is additionally strictly positive on the compact set $\{x\in \R^N: \norm{x}\leq r\}\cap \Cb$. It means that there exist $\varepsilon_0  =\varepsilon_0(r) > 0$, $\gamma = \gamma(r)>0$, and $c_2 = c_2(r)>0$ such that for every $t<T_r$ and $x$, for which $\prod_{\alpha\in R_+^i}\xa^2\leq \varepsilon_0$, we have that
    \begin{equation}
        \sum_{\aRi} \eaAi{M-1} \sum_{k=1}^N  \alpha_k^2\siggX{k} \leq \frac{1}{M\gamma} \sum_{\aRi}|\alpha|^2\eaAi{M-1}\kaX
    \end{equation}
    and 
    \begin{equation}
        c_2 \sum_{\aRi} |\alpha|^2 \eaAi{M-1} \leq \sum_{\aRi} |\alpha|^2 \eaAi{M-1} \kaX
    \end{equation}
    Indeed, if we take $\alpha\in R_+$ such that $\xa$ is sufficiently small \ref{ass:k:positive} gives
    \begin{equation}
        \label{eq:sigma_k:ineq}
        \sum_{k=1}^N \alpha_k^2 \siggx{k}\leq \frac{|\alpha|^2}{2M\gamma} k_\alpha(t,x)\quad \textrm{ and } \quad k_\alpha(t,x)\geq c_2/2\/,
    \end{equation}
    For those $\alpha\in R_+$ which are not making $\xa$ small, the small factors $\xb$ appear in $\eaAi{M-1}$ and those elements of the sums can be consumed by the previous ones again due to the positivity of $k_\alpha$. We set $\varepsilon = \varepsilon_0/2 \wedge c_2^2/(2r^2c_1^2) \wedge 1 > 0$ and define the following recurrent sequence of stopping times starting with $\tau_0^{\uparrow} = 0$ together with
    \begin{eqnarray*}
        \tau_{k}^{\downarrow} &=&  \inf\{t\geq  \tau_{k-1}^{\uparrow}: \eAi{M}_t\leq \varepsilon\}\/,\quad k=1, 2,3,\ldots\\
        \tau_{k}^{\uparrow} &=& \inf\{t\geq  \tau_{k}^{\downarrow}: \eAi{M}_t\geq 2\varepsilon\}\/,\quad k= 1,2,3,\ldots
    \end{eqnarray*}
    By definition, we have $0\leq \tau_1^{\downarrow}< \tau_1^{\uparrow}<\tau_2^{\downarrow}<\tau_2^{\uparrow}<\ldots$. It is also obvious that there exists a finite number $n=n(\omega)$ such that only $\tau_n^{\uparrow} < T_r \leq \tau_{n+1}^{\downarrow}$. Otherwise, these moments would accumulate over a bounded period of time, which would contradict the continuity of the process. In general, the idea is to notice that in every time interval of the form $[\tau_k^{\uparrow}, \tau_{k+1}^{\downarrow}]$, $k=0,1, \ldots, n$, the process $ \eAi{M}$ stays above $\varepsilon$ (away from zero), and for $[\tau_k^{\downarrow}, \tau_{k}^{\uparrow}]$ with $k=1,\ldots, n$ we have $ \eAi{M}\leq 2\varepsilon$ (near zero).  As a consequence, we get that 
    \begin{equation*}
        \int_{\tau_k^{\uparrow}}^{\tau_{k+1}^{\downarrow}}  \sum_{\aRi}|\alpha|^2\eaAi{M-1} \frac{k_\alpha(s,X(s))}{\eBi{M}_s}ds \leq \frac{1}{\sqrt{\varepsilon}} \int_{\tau_k^{\uparrow}}^{\tau_{k+1}^{\downarrow}}  \sum_{\aRi}|\alpha|^2\eaAi{M-1} k_\alpha(s,X(s)) ds<\infty\/,
    \end{equation*}
    and it implies that it is enough to show the finiteness of every integral of the form, 
    \begin{equation}
        \label{eq:integral:nearzero}
        \int_{\tau_{k}^{\downarrow}}^{\tau_{k}^{\uparrow}} \sum_{\aRi}|\alpha|^2\eaAi{M-1} \frac{k_\alpha(s,X(s))}{\eBi{M}_s}ds \/.
    \end{equation}
    We also recall that the SDE for $\eAi{M}$ can be rewritten as
    \begin{equation*}
        \label{eq:eAi:SDE:rewrite}
		d\eAi{M} = 2\sqrt{\eAi{M}}\sqrt{a(t)} dW + (V^{(1)}_{M}(X)+\eBi{M}V_{M}^{(2)}(X))dt,
    \end{equation*}
    where $W$ is a Wiener process and 
    \begin{equation*}
        A(t) = \int_0^t a(s)ds\,\qquad  a(t) = \sum_{k=1}^N\sigma_{k}^2(t, X)\left(\sum_{\aRi}\alpha_k \eaBi{M-1}\right)^2 \/.
    \end{equation*}
    For $t\in [\tau_k^{\downarrow},\tau_k^{\uparrow}]$ we have by H\"older's inequality that
    \begin{eqnarray*}
        \sum_{k=1}^N \siggX{k}\left(\sum_{\aRi}\alpha_k \eaBi{M-1}\right)^2 &\leq& M\sum_{\aRi} \eaAi{M-1} \sum_{k=1}^N  \alpha_k^2\siggX{k}\\
        &\leq& \frac{1}{\gamma} \sum_{\aRi}|\alpha|^2\eaAi{M-1}\kaX\/.
    \end{eqnarray*}
    Moreover, we also have the following
    \begin{eqnarray*}
        |\eBi{M}V_{M}^{(2)}(X)| &\leq& \sqrt{2\varepsilon} c_1 \leq \sqrt{2\varepsilon} c_1 r \, \eAi{M-1}\\
         &\leq&  c_2 \sum_{\aRi} |\alpha|^2 \eaAi{M-1} \leq \sum_{\aRi} |\alpha|^2 \eaAi{M-1} \kaX \/.
    \end{eqnarray*}
    Combining all together, we arrive at the following estimates of the drift part of $\eAi{M}$,
    \begin{eqnarray*}
        (V^{(1)}_{M}(X)+\eBi{M}V_{M}^{(2)}(X)) &\geq& (1+\gamma) \sum_{k=1}^N \siggX{k}\left(\sum_{\aRi}\alpha_k \eaBi{M-1}\right)^2 = (1+\gamma)a(t)\/,
    \end{eqnarray*}
    which is valid on the near-zero interval $[\tau_k^{\downarrow},\tau_k^{\uparrow}]$. In particular, by the comparison theorem (see Revuz Yor, Corollary 3.4 p.390 and Theorem 3.7 p.394), it implies that $\eAi{M}$ is bounded from below by process $Y$, which is a solution to 
    \begin{equation}
        \label{eq:QBES:SDE}
        dY = 2\sqrt{Y}\sqrt{a(t)}dW + (1+\gamma)a(t)dt\/,
    \end{equation}
    which is a squared Bessel process with dimension $(1+\gamma)$ with time changed by $A(t)$. This is a crucial observation in this proof. 
    
    First, we use it to show that for every $t\in [\tau_k^\downarrow, \tau_k^\uparrow]$ such that  $\eAi{M}_t = 0$ we have
    \begin{eqnarray}
        \label{eq:time_change_positivity}
        \sum_{k=1}^N \sigma_k^2(t,X_t)\sum_{\aRi} \alpha_k^2 \eaAi{M-1}>0\/,\quad a.s.
    \end{eqnarray}
    Indeed, if with positive probability the process $X$ hits boundary ($\eAi{M}_t=0$) in a point where $\sum_{k=1}^N \sigma_k^2(t,X_t)\sum_{\aRi} \alpha_k^2 \eaAi{M-1}=0$, then using continuity of the coefficients and the fact that $ \sum_{\aRi}|\alpha|^2\eaAi{M-1}\kaX$ is strictly positive (due to positivity of $k_\alpha$ and the result of Proposition \ref{prop:no_multi_collisions}) we can find (with positive probability) the small random time sub-interval of $[\tau_k^\downarrow, \tau_k^\uparrow]$, where we can make $\gamma$ in \eqref{eq:sigma_k:ineq} as big as we want. In particular, we see that there exists a time interval, where the process $\eAi{M}$ is bounded from below by the squared Bessel process of dimension $2$ with time change, but this process never visits $0$, which gives a contradiction. 
    
    Next, we define the random set
    \begin{equation*}
        Z_0 = \{t\in [\tau_k^\downarrow, \tau_k^\uparrow]: \eAi{M}_t = 0\}
    \end{equation*}
    It is definitely compact set as it is closed due to continuity of $t\to  \eAi{M}_t$ and bounded (finiteness of $\tau_k^\downarrow$ and $\tau_k^\uparrow$). As we have previously seen, for each $t\in Z_0$ we have \eqref{eq:time_change_positivity}. It means that for every $t^*\in Z_0$ there exists $\delta=\delta(t,\omega)>0$ such that for every $t\in (t^*-\delta, t^*+\delta)$ the inequality \eqref{eq:time_change_positivity} still holds and consequently there exists $L^* = L^*(t^*,\delta, \omega)>0$ such that 
    \begin{equation}
        \label{eq:sigma_ka:L}
        \sum_{\aRi}|\alpha|^2\eaAi{M-1}\kaX \leq L^* \sum_{k=1}^N \sigma_k^2(t,X_t)\sum_{\aRi} \alpha_k^2 \eaAi{M-1}
    \end{equation}
    for every $t\in (t^*-\delta, t^*+\delta)$. It provides an open (random) cover of the set $Z_0$ and compactness of $Z_0$ implies that we can find finite sub-cover. Then taking $\mathcal{U}$ as a sum of all elements from this finite sub-cover we end up with an open set such that $Z_0\subset \mathcal{U}$ and there exists $L>0$ (which is obviously a maximum of all $L^*$ found for each element of the finite sub-cover) such that \eqref{eq:sigma_ka:L} holds with $L^*$ replaced by $L$ for every $t\in \mathcal{U}$.

    Writing $1=\ind_\mathcal{U}(s)+\ind_{\mathcal{U}^c}(s)$, we split the near-zero integral
\eqref{eq:integral:nearzero} on the interval
$[\tau_k^\downarrow,\tau_k^\uparrow]$. Namely,
\begin{eqnarray*}
    I_1
    &=&
    \int_{\tau_k^\downarrow}^{\tau_k^\uparrow}
    \sum_{\aRi}|\alpha|^2\eaAi{M-1}
    \frac{k_\alpha(s,X_s)}{\eBi{M}_s}
    \ind_\mathcal{U}(s)\,ds \\
    &\leq&
    L
    \int_{\tau_k^\downarrow}^{\tau_k^\uparrow}
    \sum_{j=1}^N\sigma_j^2(s,X_s)
    \sum_{\aRi}\alpha_j^2\eaAi{M-1}
    \ind_\mathcal{U}(s)\frac{ds}{\eBi{M}_s} \\
    &\leq&
    L
    \int_{\tau_k^\downarrow}^{\tau_k^\uparrow}
    \sum_{j=1}^N\sigma_j^2(s,X_s)
    \sum_{\aRi}\alpha_j^2\eaAi{M-1}
    \frac{ds}{\eBi{M}_s}.
\end{eqnarray*}
Moreover,
\begin{eqnarray*}
    I_2
    &=&
    \int_{\tau_k^\downarrow}^{\tau_k^\uparrow}
    \sum_{\aRi}|\alpha|^2\eaAi{M-1}
    \frac{k_\alpha(s,X_s)}{\eBi{M}_s}
    \ind_{\mathcal{U}^c}(s)\,ds.
\end{eqnarray*}

    Note that $\mathcal{U}^c\cap [\tau_k^\downarrow, \tau_k^\uparrow]$ is again a compact set on which the continuous function $t \to \eBi{M}_t$ is positive, which means that there exists $c_4=c_4(\omega)>0$ such that $\eBi{M}_t>c_4$ on  $\mathcal{U}^c\cap [\tau_k^\downarrow, \tau_k^\uparrow]$ and consequently $I_2$ is just finite. Consequently, we can focus our attention on showing the finiteness of 
    \begin{equation}
         \int_{\tau_k^\downarrow}^{\tau_k^\uparrow}{\sum_{j=1}^N\sigma_{j}^2(s, X)\left(\sum_{\aRi}\alpha_j \eaBi{M-1}\right)^2} \frac{ds}{\eBi{M}_s} =  \int_{\tau_k^\downarrow}^{\tau_k^\uparrow} \frac{A'(s) ds}{\eBi{M}_s}\/,
    \end{equation}
    as again we have
    \begin{eqnarray*}
       \left(\sum_{\aRi}\alpha_j \eaBi{M-1}\right)^2 &=& \sum_{\aRi} \alpha_j^2 \eaAi{M-1}  +\anbRi\alpha_j \beta_j \eBi{M}\eabBi{M-2} 
    \end{eqnarray*}
    and we see as before a similar cancellation of the singularity for the second double sum. 

    Finally, we return to the fact that $\eBi{M}_s\geq \sqrt{Y(s)}$ for $s\in [{\tau_k^\downarrow}, {\tau_k^\uparrow}]$, where $Y$ is a squared Bessel process with dimension $1+\gamma>1$, starting from $0$, with time-changed by $A(t)$, i.e. we have $Y(s) = Z\bigl(A(s)-A(\tau_k^\downarrow)\bigr)$, where $Z$ is $\textrm{BESQ}^{(1+\gamma)}_0$. Using a similar approach as before, one can show that in fact $t\to A(t)$ is strictly increasing in the neighborhood of the boundary as we have \eqref{eq:time_change_positivity}. However, we can also provide the final claim noticing that since $H = \{s\in [{\tau_k^\downarrow}, {\tau_k^\uparrow}]:A'(s) = a(s)>0\}$ is relatively open in $[\tau_k^\downarrow,\tau_k^\uparrow]$, it can be obtained as countable union of disjoined intervals, i.e. $H = \bigcup_{i=1}^\infty (l_i,r_i)$ and use this random set representation in the following
    \begin{eqnarray*}
        \int_{\tau_k^\downarrow}^{\tau_k^\uparrow} \frac{a(s) ds}{\eBi{M}_s} &\leq& \int_{\tau_k^\downarrow}^{\tau_k^\uparrow} \frac{a(s) ds}{ \sqrt{Y(s)}} = \int_{\tau_k^\downarrow}^{\tau_k^\uparrow} \frac{a(s)}{ \sqrt{Y(s)}}\ind_{\{a(s)>0\}}ds = \sum_{i=1}^\infty \int_{l_i}^{r_i} \frac{a(s) ds}{ \sqrt{Z(A(s)-A(\tau_k^\downarrow))}}\\
        &=& \sum_{i=1}^\infty \int_{A(l_i)-A(\tau_k^\downarrow)}^{A(r_i)-A(\tau_k^\downarrow)} \frac{du}{ \sqrt{Z(u)}} \leq \int_0^{A(\tau_k^\uparrow)-A(\tau_k^\downarrow)}
    \frac{du}{\sqrt{Z(u)}}<\infty.
    \end{eqnarray*}
    The last integral is finite because $\sqrt Z$ is a Bessel process of dimension $1+\gamma>1$, and therefore its reciprocal is locally integrable in time.

\end{proof}

\begin{remark}
    \label{remark:integrability:sigma}
    Note that the proof given above shows in particular that whenever \ref{ass:k:positive} holds, then for every $\alpha\in R_+$ we have for every $t<\life{X}$ that 
        \begin{equation}
            \label{eq:integrability:sigma}
            \int_0^t \dfrac{\sum_{k=1}^N \alpha_k^2\sigma_k^2(s,X(s))}{\inner{X(s)}{\alpha}}ds <\infty\/,\quad a.s.
        \end{equation}
    This fact will be important in the final stage of the solution construction. 
\end{remark}


\section{Construction of a solution in non-degenerate case}

All the propositions provided in the previous sections lead to this point, where we are finally able to prove that our candidate process $X=(X_1,\ldots,X_N)$ is a semi-martingale, which solve \eqref{eq:main:SDE}. 

\begin{proposition}
    \label{prop:positive-construction-solves-sde}
    Assume that $X=(X_1,\ldots,X_N)$ is a diffusion process defined in \eqref{defn:x} such that $X(0)\in \Cc$. If \ref{ass:k:positive} holds, then $X$ is a semi-martingale solving \eqref{eq:main:SDE}.
\end{proposition}

\begin{proof}
    \newcommand{\st}{T}
    First, we assume that the starting point $X(0)$ is in the interior $\C$ and focus our attention to the case $\AN$. Then we define the following sequence of random times starting from $T_0 = 0$. Then for $n\geq 0$, if $T_n<\life{X}$, we define 
        \begin{align*}
            T_{n+1} &= \inf\{t\geq T_n: \textrm{there exists two different }\alpha,\beta \in\mathcal{H}_n(t)\textrm{ such that }\ab\neq 0\}\wedge \life{X}
        \end{align*}
    with the convention that the infimum of the empty set is $\life{X}$. Here
        \begin{equation*}
            \mathcal{H}_n(t) = \{\alpha\in R_+: \textrm{there exists } u\in [T_n,t] \textrm{ such that }\inner{X(u)}{\alpha}=  0\}.
        \end{equation*}
        
    In other words, after time $T_n$ we collect all new roots $\alpha$ for which $\inner{X(u)}{\alpha}$ vanishes for some $u\geq T_n$ and we stop this procedure whenever the constructed set $\mathcal{H}_n(t)$ contains two non-orthogonal roots. Continuity of $X$ and the lack of multiple collisions implies that 
        \begin{equation*}
            \st_0<\st_1<\ldots<\st_n<\ldots\/.
        \end{equation*}
    By Proposition~\ref{prop:no_multi_collisions}, we have $T_{n+1}>T_n$ on $\{T_n<\life{X}\}$. Indeed, if $T_{n+1}=T_n$, then two non-orthogonal roots would vanish at the same time $T_n$, which is excluded by absence of multiple collisions.

    Moreover, the sequence $(T_n)_n$ has no accumulation point before $\life{X}$. Suppose, to the contrary, that $T_n\uparrow T<\zeta$. By the definition of $T_{n+1}$, for each $n$ there exist non-orthogonal roots $\alpha_n,\beta_n\in R_+$ and times $s_n,u_n\in[T_n,T_{n+1}]$ such that
        \begin{equation*}
            \langle X(s_n),\alpha_n\rangle=0, \qquad \langle X(u_n),\beta_n\rangle=0, \qquad \langle\alpha_n,\beta_n\rangle\neq0.
        \end{equation*}
    Since $R_+$ is finite, passing to a subsequence if necessary, we may assume that $\alpha_n=\alpha$ and $\beta_n=\beta$ for all $n$, with $\langle\alpha,\beta\rangle\neq0$. Since $s_n,u_n\to T$, continuity of $X$ gives
        \begin{equation*}
            \langle X(T),\alpha\rangle = \langle X(T),\beta\rangle = 0,
        \end{equation*}
    which contradicts Proposition~\ref{prop:no_multi_collisions}. Hence $(T_n)_n$ is locally finite on $[0,\life{X})$, and therefore $T_n\uparrow\life{X}$.

    Let us now fix $n=0,1,2,\ldots$ and $t<T_{n+1}$. By construction, $\mathcal H_n(t)$ is a pairwise orthogonal subset of $R_+$, and every root which vanishes somewhere on $[T_n,t]$ belongs to $\mathcal H_n(t)$. In the type $\AN$ case, the roots in $\mathcal H_n(t)$ correspond to disjoint adjacent collision pairs. If $i$ is such that for every $\alpha\in \mathcal H_n(t)$ $\alpha_i=0$, then $X_i$ does not hit any of its neighbors on the interval $[T_n,t]$. Therefore, using Proposition~\ref{prop:root:smooth}(\ref{prop:xip:item_single}) and It\^o's formula, we obtain
        \begin{equation*}
            X_i(s) = X_i(\st_n)+\int_{\st_n}^s \sigsX{i}dB_u + \int_{\st_n}^s\bXs{i}du+\int_{\st_n}^s\sum_{\aR}\dfrac{\kasX \alpha_i}{\Xa}du\/,
        \end{equation*}
    for every $s\in[T_n,t]$. To deal with the remaining case, i.e. there exists $\alpha\in \mathcal H_n(t)$ such that $\alpha_i\neq 0$, then on the interval $[T_n,t]$, the particles $X_j$ and $X_{j+1}$ can collide only with each other, where $i\in \{j,j+1\}$. Thus, by Proposition~\ref{prop:root:smooth}(\ref{prop:xip:item_sum}), the function $(x_j+x_{j+1})(p_1,\ldots,p_N)$ is smooth on this localized branch, and It\^o's formula gives
        \begin{eqnarray}
            \nonumber
            X_j(s)+X_{j+1}(s) &=& X_j(\st_{n})+X_{j+1}(\st_{n})+\int_{\st_n}^t \sigsX{j}dB_j +\int_{\st_n}^t\sigsX{j+1}dB_{j+1}\\
            \nonumber
            &&+ \int_{\st_n}^t(\bXs{j}+\bXs{j+1})ds\\
            \label{eq:xjxj1:plus}
            &&+\int_{\st_n}^t\sum_{\aR}\dfrac{\kasX \alpha_j}{\Xa}ds+\int_{\st_n}^t\sum_{\aR}\dfrac{\kasX \alpha_{j+1}}{\Xa}ds\/.
        \end{eqnarray}
     To go from SDE for the sum $X_j+X_{j+1}$ to the SDEs for $X_j$ and $X_{j+1}$ separately we will explore the difference $X_j-X_{j+1}$. Here we can not claim the the mapping from the polynomials to the difference is smooth as we definitely might see collisions between $X_j$ and $X_{j+1}$. To deal with this we use again the smoothing function, which we have used before in Proposition~\ref{prop:vandermond:i}. Recall that we defined $\psi_\varepsilon:\R\to\R$ for $\varepsilon>0$ by $\psi_\varepsilon(z)=z\varphi(z/\varepsilon)$, where $\varphi$ is a fix smooth function on $\R$ with values in $[0,1]$ such that $\varphi(z)=0$ for $|z|<1$ and $\varphi(z)=1$ for $|z|\geq 2$. Proposition~\ref{prop:root:smooth} shows that $Z_j = \psi_\varepsilon(X_j-X_{j+1})$ is a semimartingle and the It\^o  formula implies the following representation on $[\st_n, \st_{n+1})$
        \begin{eqnarray*}
            d\psi_\varepsilon(X_j-X_{j+1}) &=& \psi'_\varepsilon(X_j-X_{j+1})(dX_j-dX_{j+1})\\
            &&+\dfrac12 \psi''_\varepsilon (X_j-X_{j+1})(dX_jdX_j+dX_{j+1}dX_{j+1})\/.
        \end{eqnarray*}
    for every $\varepsilon>0$. To deal with the limit as $\varepsilon$ goes to $0$, let us focus on the last part, which is given by 
        \begin{eqnarray*}
            \frac12\int_{\st_n}^t (X_j-X_{j+1})\psi''_\varepsilon(X_j-X_{j+1}) \dfrac{\sigma^2_j(s,X(s))+\sigma^2_{j+1}(s,X(s))}{X_j-X_{j+1}}ds \stackrel{\varepsilon \to 0}{\longrightarrow} 0 
        \end{eqnarray*}
    by \eqref{eq:psi:limit:double} and the Dominated Convergence Theorem (by \eqref{eq:psiderv:bounds} and integrability given in \eqref{eq:integrability:sigma}). Similarly, we proceed with the drift part in $\psi'_\varepsilon(X_j-X_{j+1})(dX_j-dX_{j+1})$, where we use the bounds for the first derivative of $\psi_\varepsilon$ given in \eqref{eq:psiderv:bounds} and the integrability of the drift terms given in \eqref{eq:singular:integrability}. Additionally, the integrability given in \eqref{eq:singular:integrability} implies that the set $\{t: X_j(t)=X_{j+1}(t)\}$ has Lebesgue measure zero a.s. Consequently, $\psi'_\varepsilon(X_j(t)-X_{j+1}(t)) \to 1$ as $\varepsilon \to 0$. Finally, the local-martingale part, after localization, can be treated in the same way using the $L_2$ property of the Ito integral. It allows us to arrive at
        \begin{eqnarray}
            \nonumber
            X_j(t)-X_{j+1}(t) &=& X_j(\st_n)-X_{j+1}(\st_n)+\int_{\st_n}^t \left(\sigma_j(s,X(s))dB_j-\sigma_{j+1}(s,X(s))dB_{j+1}\right) \\
            \nonumber
            &&+ \int_{\st_n}^t\left(b_j(s,X(s))-b_{j+1}(s,X(s))\right)ds\\
            \label{eq:xjxj1:minus}
            && +\int_{\st_n}^t\sum_{\aR}\dfrac{\kasX \alpha_j}{\Xa}ds-\int_{\st_n}^t\sum_{\aR}\dfrac{\kasX \alpha_{j+1}}{\Xa}ds
        \end{eqnarray}

    Obviously, using \eqref{eq:xjxj1:plus} and \eqref{eq:xjxj1:minus} we obtain the desired SDEs for $X_j$ and $X_{j+1}$. Gluing these local representations gives the desired equation in type $\AN$ for an initial point in the interior $\C$.

    
        \medskip

    Let us now briefly describe the changes in the above-given proof, which are necessary to deal with the remaining cases$\BN$ and $\DN$.  The argument is local and uses the same approximation procedure as above.  For $x\in\Cc$ put
        \begin{equation*}
            \mathcal A(x):=\{\alpha\in R_+:\langle x,\alpha\rangle=0\}.
        \end{equation*}
    Fix a positive time and localize to a sufficiently small random time interval on which no root outside the active set can vanish.  By Proposition~\ref{prop:no_multi_collisions}, all roots in the active set are pairwise orthogonal.  Hence, if $\beta$ is an active root and $\gamma$ is another active root, then the singular term generated by $\gamma$ does not contribute to the equation for $\langle X,\beta\rangle$, because $\langle \beta,\gamma\rangle=0$. Therefore the active walls can be treated one by one.  All quantities which are smooth functions of the invariant coordinates are handled by It\^o's formula exactly as in the type $\AN$ case.  Only the corresponding root distances $Z_\alpha(t):=\langle X(t),\alpha\rangle$, $\alpha\in R_+$ require the smoothing argument with $\psi_\varepsilon$.

    In type $\BN$, the positive roots are $e_i$, $e_i-e_j$ and $e_i+e_j$. First observe that $e_i+e_j$ does not produce any additional problems as $x_i+x_j=0$ forces $x_i=x_j=0$ and hence a non-orthogonal multiple collision, which is excluded by Proposition~\ref{prop:no_multi_collisions}. Moreover, the roots $\alpha = e_i-e_j$ can be treated exactly the same way as in the $\AN$ case with the smoothing functions $\psi_\varepsilon$, the It\^o formula application to $\psi_\varepsilon(Z_\alpha)$  and passing to the limit $\varepsilon\downarrow0$. Finally, the only remaining case to deal with is the last coordinate, namely for the root $\alpha=e_N$ and the wall $X_N=0$.  The square $Y_N:=X_N^2$
    is locally a smooth function of the invariant coordinates.  On the interior, It\^o's formula gives
        \begin{align*}
            dY_N(t) &=2X_N(t)\sigma_N(t,X(t))\,dB_N(t)\\
            &\quad+ \left[ \sigma_N^2(t,X(t)) +2X_N(t)b_N(t,X(t)) +2X_N(t)\sum_{\delta\in R_+} \frac{k_\delta(t,X(t))\delta_N}{\langle X(t),\delta\rangle}\right]dt.
        \end{align*}
    We then apply It\^o's formula to $h_\varepsilon(Y_N):=\sqrt{Y_N+\varepsilon}$. The martingale part converges in quadratic variation to
    $\int_a^t\sigma_N(s,X(s))\,dB_N(s)$, while the drift terms converge to the desired drift by dominated convergence, using \eqref{eq:integrability:sigma} and \eqref{eq:singular:integrability}. The only additional It\^o correction is
        \begin{equation*}
            \frac{\varepsilon\,\sigma_N^2(t,X(t))}
                 {2\left(X_N^2(t)+\varepsilon\right)^{3/2}}\,dt,
        \end{equation*}
    which converges to zero by dominated convergence, using again \eqref{eq:integrability:sigma}, which reads here as
        \begin{equation*}
            \int_0^t
            \frac{\sigma_N^2(s,X(s))}{X_N(s)}
            \,ds<\infty.
        \end{equation*}

    \medskip

    Finally, in the case $\DN$ there are no short roots.  The walls corresponding to roots $e_i-e_j$ and $e_i+e_j$ are treated by the same smoothing argument.  More precisely, for a root $\alpha=e_i-e_j$ we apply the argument to $Z_\alpha=X_i-X_j$,
    while $X_i+X_j$ stays locally separated from zero and is therefore a smooth function of the invariant coordinates on the localized branch. Similarly, for a root $\alpha=e_i+e_j$ we apply the argument to $Z_\alpha=X_i+X_j$, while $X_i-X_j$ stays locally separated from zero. There is no remaining zero-pair case to consider.  Indeed, if both $e_i-e_j$ and $e_i+e_j$ vanished at a positive time, then $X_i=X_j=0$, and the two roots would form a multiple cluster.  This is excluded by Proposition~\ref{prop:no_multi_collisions}.  Hence, after localization, each active type $\DN$ wall is treated as a single root wall, exactly as in the preceding smoothing argument.

    Since the preceding arguments are local, the resulting semimartingale representations are glued over the same type of localized intervals as in the type $\AN$ proof.  Thus, for an initial point in the interior $\C$, the process $X$ satisfies \eqref{eq:main:SDE} up to its lifetime in all three classical cases $\AN$, $\BN$, and $\DN$.


    \medskip

     Last part of the proof relates to the general starting point $X(0)\in\Cc$, which, in particular, might be on the boundary of the Weyl chamber and has multiple collisions.  Then, by Proposition~\ref{prop:diffraction}, there exists a sequence of times $t_n\downarrow0$ such that $X(t_n)\in\C$.  For every fixed $n$, applying the preceding interior-start argument to the shifted process gives the expected semi-martingale representation on $[t_n,t]$.  It remains only to justify that the singular integrals remain finite when $t_n\downarrow0$. This follows by the standard one-sided argument used for ordered particle systems.  With our convention for the Weyl chamber, $X_1$ is the largest particle.  After putting the singular part on one side and all non-singular terms on the other side, the right-hand side has a finite limit as $t_n\downarrow0$ by continuity of $X$, continuity of the stochastic integral, and local integrability of the regular drift.  Since the singular terms in the $X_1$-equation are non-negative, no cancellation is possible, and each of these integrals is finite on $[0,t]$. We then proceed inductively.  Suppose that all singular integrals involving particles with indices smaller than $i$ have already been shown to be finite. In the equation for $X_i$, we move those already controlled terms to the non-singular side.  The remaining singular terms, namely those involving $X_i$ and particles with larger indices, and in type $\BN$ also the term corresponding to the wall $X_i=0$, again have the same sign.  The same one-sided argument therefore proves their finiteness.  Repeating this for $i=1,\ldots,N$ gives the finiteness of all singular integrals on $[0,t]$. Consequently, taking the limit $t_n\downarrow0$ in the shifted semi-martingale representations yields the semi-martingale representation given in the main SDE~\eqref{eq:main:SDE}.
\end{proof}

Collecting all together we have completed the proof of Theorem~\ref{thm:existence} in the following way.

\begin{proof}[Proof of Theorem~\ref{thm:existence}]
By Proposition~\ref{prop:boundary-starts-by-approximation}, there exists a $\mathcal K$-valued solution $U$ of the invariant-coordinate equation started from $w(x_0)$. Put $X=f(U)$. Proposition~\ref{prop:positive-construction-solves-sde} shows that $X$ is a semimartingale satisfying \eqref{eq:main:SDE}. For $x_0\in\C$, the finiteness of the singular integrals follows from Proposition~\ref{prop:singular:integrability}.  For $x_0\in\Cc\setminus\C$, it follows from the one-sided limiting argument at the end of Proposition~\ref{prop:positive-construction-solves-sde}. This proves the theorem.
\end{proof}


\section{Construction of a solution in the degenerate case}
\label{sec:degenerate:weak}

In this section we prove the existence result in the degenerate boundary regime stated
in Theorem~\ref{thm:degenerate-existence}.  In contrast to the strictly positive
case, the coefficient $k_\alpha$ is now allowed to vanish on the wall
$\{\langle x,\alpha\rangle=0\}$.  The construction is based on approximation by
systems with strictly positive boundary repulsion.  Namely, for $n\in\N$ we replace $k_\alpha$ by
    \begin{equation}
        k_\alpha^{(n)}(t,x) := k_\alpha(t,x)+\frac1n, \qquad \alpha\in R_+.
        \label{eq:kn:defn}
    \end{equation}
For each fixed $n$, the coefficients $k_\alpha^{(n)}$ satisfy the strict positivity
condition \ref{ass:k:positive}.  Hence the corresponding
non-degenerate existence theorem yields a chamber-valued solution $X^{(n)}$.

The main difficulty is to identify the limit of the singular drift terms.  Along a
weakly convergent subsequence, the finite-variation terms
\begin{equation*}
    \int_0^t
    \frac{k_\alpha^{(n)}(s,X^{(n)}(s))}
         {\langle X^{(n)}(s),\alpha\rangle}
    \,ds
\end{equation*}
may converge to a measure which, a priori, contains an additional component supported on the set of times $\{t:\langle X(t),\alpha\rangle=0\}$.  The purpose of the argument below is to show that, under Assumptions~\ref{ass:k:dom} and~\ref{ass:face:sign}, such hidden boundary measures are in fact absent.

\medskip

We first recall the two degenerate assumptions introduced in Section~\ref{sec:assumptions-results}. They replace the strict positivity condition \ref{ass:k:positive}.  The first one is the normal dominance condition \ref{ass:k:dom}.  For $\alpha\in R_+$, set
    \begin{equation*}   
        a_\alpha(t,x) := \sum_{j=1}^N \alpha_j^2\,\sigma_j^2(t,x).
    \end{equation*}
This is the quadratic variation density of the martingale part of $\langle X,\alpha\rangle$. Assumption~\ref{ass:k:dom} says that, locally in time and space, for points close to the wall $\{\langle x,\alpha\rangle=0\}$, the repulsion coefficient $k_\alpha(t,x)$ dominates the normal martingale variation $a_\alpha(t,x)$.  Thus $k_\alpha$ may vanish at the wall, but only in a way compatible with the degeneration of the diffusion in the same normal direction. More precisely, for every $T,R>0$ there exist constants $\varepsilon=\varepsilon(T,R)>0$ and $\gamma=\gamma(T,R)>0$ such that, for every $t\in[0,T]$, every $x\in\Cc$ with $|x|\le R$, and every $\alpha\in R_+$,
\begin{equation*}
    0\le \langle x,\alpha\rangle\le \varepsilon
    \quad\Longrightarrow\quad
    k_\alpha(t,x)\ge \gamma\,a_\alpha(t,x).
\end{equation*}
Since $R_+$ is finite, the constants may be chosen uniformly over all positive roots
on each compact time-space set.  This condition is weaker than strict positivity of
$k_\alpha$ on the wall.  Indeed, if $k_\alpha$ is strictly positive on
\begin{equation*}
    \{0\le t\le T,\ |x|\le R,\ \langle x,\alpha\rangle=0\},
\end{equation*}
then, by continuity and compactness, $k_\alpha$ remains bounded from below in a
neighborhood of the wall, whereas $a_\alpha$ is bounded above on the same compact set.
Hence Assumption~\ref{ass:k:dom} follows after possibly shrinking the neighborhood.

The second assumption is the face sign condition \ref{ass:face:sign}.  Recall the notation introduced in Section~\ref{sec:assumptions-results} in this context. In particular, by our convention, the detector family $\mathfrak D(S)$
contains the canonical detector
\begin{equation*}
    u_S:=\sum_{\beta\in S}\beta
\end{equation*}
and, for $u\in\mathfrak D(S)$, the regular drift of the detector $\tau_u$ is
    \begin{equation*}
        B_{S,u}(t,x) = \langle u,b(t,x)\rangle + \sum_{\beta\in R_+\setminus S} \frac{k_\beta(t,x)\langle u,\beta\rangle}{\langle x,\beta\rangle}.
    \end{equation*}
This is the same quantity as in \eqref{eq:def:B-S-u-degenerate}.  On localized neighborhoods of $F_S$, all denominators with $\beta\in R_+\setminus S$ are bounded away from zero.  Consequently, $B_{S,u}$ admits a continuous trace on $F_S$. Indeed, while the process is localized near $F_S$, applying the linear functional $\tau_u$ to the SDE formally gives
    \begin{align*}
        d\tau_u(X(t)) &= dM^u(t) + B_{S,u}(t,X(t))\,dt + \sum_{\alpha\in S} \frac{k_\alpha(t,X(t))\langle u,\alpha\rangle}{\langle X(t),\alpha\rangle} \,dt,
    \end{align*}
where
    \begin{equation*}
        M^u(t) := \sum_{j=1}^N u_j\int_0^t\sigma_j(s,X(s))\,dB_j(s).
    \end{equation*}
Thus $B_{S,u}$ is precisely the regular part of the detector drift, after the singular terms corresponding to the roots vanishing on the face have been separated. Assumption~\ref{ass:face:sign} requires that, for every collision face $F_S$ and every admissible detector $u\in\mathfrak D(S)$, the continuous trace of $B_{S,u}$ on $F_S$ satisfies
    \begin{equation*}   
        B_{S,u}(t,x)\ge0, \qquad t\geq 0\/, x\in F_S.
    \end{equation*}
This is an inward-pointing condition for the regular part of the detector drift. Since $\tau_u$ is nonnegative on the chamber and vanishes on $F_S$, the regular drift should not push the detector into negative values at the boundary.  The condition is not needed for tightness or for the construction of a generalized weak limit.  It is used only in the final local-time argument, where the possible hidden boundary measures are eliminated. 


\subsection{Approximation and localized compactness estimates}
\label{subsec:deg:approx-compactness}

We continue with the approximating solutions $X^{(n)}$ associated with \eqref{eq:kn:defn}. Thus $X^{(n)}$ is a $\Cc$-valued weak solution, started from $x_0$, to \eqref{eq:main:SDE} with $k_\alpha$ replaced by $k_\alpha^{(n)}$. We denote its lifetime by $\zeta^{(n)}$. At this compactness stage only Assumptions~\ref{ass:cont}-\ref{ass:sigma} are used. The degenerate boundary assumptions \ref{ass:k:dom} and \ref{ass:face:sign} enter only later, after a generalized weak limit has been constructed, in the argument excluding boundary measures.

For $\alpha\in R_+$ and $0\le t<\zeta^{(n)}$, define the accumulated
$\alpha$-singular drift by
\begin{equation}
    A^{(n),\alpha}(t)
    :=
    \int_0^t
    \frac{k_\alpha^{(n)}(s,X^{(n)}(s))}
         {\langle X^{(n)}(s),\alpha\rangle}
    \,ds.
    \label{eq:def:Analpha}
\end{equation}
Since the approximating coefficients $k_\alpha^{(n)}$ satisfy
Assumption~\ref{ass:k:positive}, Theorem~\ref{thm:existence} ensures that these
increasing processes are well defined and finite on compact subintervals of
$[0,\zeta^{(n)})$.

The next proposition collects the localized estimates needed for the limiting
argument: tightness of the stopped paths, uniform control of the singular drift
masses, and compactness of the associated stopped drift measures.

\begin{proposition}[Localized compactness estimates]
    \label{prop:deg:localized-compactness}
    Assume \ref{ass:cont}-\ref{ass:sigma}. Let $X^{(n)}$ be the approximating solutions associated with \eqref{eq:kn:defn}. For every $T,r>0$, define
        \begin{equation}
            \chi_r^{(n)} := T\wedge \inf\{t\in[0,\zeta^{(n)}):\ |X^{(n)}(t)|\ge r\},
            \label{eq:def:chi:stoptime}
        \end{equation}
with the convention $\inf\varnothing=\infty$. Then the following hold.

\begin{enumerate}[label=\textup{(\roman*)}]
    \item The laws of the stopped processes
        \begin{equation*}
            \bigl(X^{(n)}(t\wedge\chi_r^{(n)})\bigr)_{0\le t\le T}, \qquad n\in\N,
        \end{equation*}
    form a tight family in $\mathcal C([0,T];\Cc)$.

    \item There exists a constant $C_{r,T}<\infty$, independent of $n$, such
    that
        \begin{equation*}
            \sup_{n\in\N}  \mathbb E \sum_{\alpha\in R_+} A^{(n),\alpha}(\chi_r^{(n)}) \le C_{r,T}.
        \end{equation*}
    In particular, for every $\alpha\in R_+$there exists $C_{\alpha,r,T}<\infty$, independent of $n$, such that
        \begin{equation*}
            \sup_{n\in\N} \mathbb E \int_0^{\chi_r^{(n)}} \frac{k_\alpha^{(n)}(s,X^{(n)}(s))}{\langle X^{(n)}(s),\alpha\rangle}\,ds\le C_{\alpha,r,T}.
        \end{equation*}

    \item For every $\alpha\in R_+$, define the stopped drift measure on $[0,T]$ by
        \begin{equation}
            \mu_{r,T}^{(n),\alpha}(dt) := \mathbf 1_{\{t\le \chi_r^{(n)}\}} \frac{k_\alpha^{(n)}(t,X^{(n)}(t))}{\langle X^{(n)}(t),\alpha\rangle}\,dt.
            \label{eq:def:aprox:measure}
        \end{equation}
    Then the joint laws of
        \begin{equation*}
            \left(X^{(n)}(\cdot\wedge\chi_r^{(n)}),\bigl(\mu_{r,T}^{(n),\alpha}\bigr)_{\alpha\in R_+}\right), \qquad n\in\N,
        \end{equation*}
    are tight in
        \begin{equation*}
            \mathcal C([0,T];\Cc) \times \prod_{\alpha\in R_+}\mathcal M_+([0,T]),
        \end{equation*}
    where $\mathcal M_+([0,T])$ is equipped with the weak topology.
\end{enumerate}

    Moreover, for every $n\in\N$,
        \begin{equation*}
             \chi_r^{(n)} \uparrow T\wedge\zeta^{(n)} \qquad\text{as } r\to\infty.
        \end{equation*}
\end{proposition}

\begin{proof}
Let $w=(w_1,\ldots,w_N)$ be the basic invariant-coordinate map introduced in Subsection~\ref{subsec:basic_polynomials}. For $t<\zeta^{(n)}$, set
    \begin{equation*}
        U^{(n)}(t):=w(X^{(n)}(t)).
    \end{equation*}
Then $U^{(n)}$ solves the invariant-coordinate equations from Definition~\ref{defn:u}, with $k_\alpha$ replaced by $k_\alpha^{(n)}$. We first prove the path tightness in $\mathcal C([0,T];\Cc)$. Set
    \begin{equation*}
        Y^{(n)}(t) := U^{(n)}(t\wedge\chi_r^{(n)}), \qquad 0\le t\le T.
    \end{equation*}
Let us denote by $K_{r_*}$ the closed ball with center at $0$ and radius $r^* = r\vee |x_0|$. For $0\le t\le\chi_r^{(n)}$, the process $X^{(n)}(t)$ stays in $K_{r_*}$. Since the invariant-coordinate equations are non-singular, their coefficients are continuous functions of $(t,X^{(n)}(t))$. Moreover,
    \begin{equation*}
        0\le k_\alpha^{(n)}(t,x) = k_\alpha(t,x)+\frac1n \le k_\alpha(t,x)+1, \qquad n\ge1.
    \end{equation*}
Thus, on $[0,T]\times K_{r_*}$, the stopped invariant-coordinate coefficients are bounded uniformly in $n$. Consequently, for each $k=1,\ldots,N$, the stopped process $Y_k^{(n)}$ admits a decomposition
    \begin{align*}
        Y_k^{(n)}(t) &= Y_k^{(n)}(0) + \sum_{i=1}^N \int_0^t  a_{i,k}^{(n)}(s)\,dB_i^{(n)}(s) + \int_0^t c_k^{(n)}(s)\,ds,
    \end{align*}
where the progressively measurable stopped coefficients satisfy
    \begin{equation*}
        | a_{i,k}^{(n)}(s)|+|c_k^{(n)}(s)| \le C_{r,T}, \qquad 0\le s\le T,
    \end{equation*}
with a constant independent of $n$. In what follows, $C_{r,T}$ may change from line to line, as long as it depends only on $r$ and $T$.

Let $0\le s\le t\le T$. By the Burkholder-Davis-Gundy inequality,
    \begin{align*}
        \E \left| \sum_{i=1}^N \int_s^t \bar a_{i,k}^{(n)}(u)\,dB_i^{(n)}(u) \right|^4 &\le C\, \E \left( \int_s^t \sum_{i=1}^N|\bar a_{i,k}^{(n)}(u)|^2\,du \right)^2  \le C_{r,T}|t-s|^2.
    \end{align*}
For the drift part,
    \begin{equation*}
        \E \left| \int_s^t \bar c_k^{(n)}(u)\,du \right|^4 \le C_{r,T}|t-s|^4.
    \end{equation*}
Hence
    \begin{equation*}
        \E |Y_k^{(n)}(t)-Y_k^{(n)}(s)|^4 \le C_{r,T}|t-s|^2.
    \end{equation*}
Summing over $k=1,\ldots,N$, and increasing the constant if necessary, we get
    \begin{equation*}
        \E |Y^{(n)}(t)-Y^{(n)}(s)|^4 \le C_{r,T}|t-s|^2, \qquad 0\le s\le t\le T,
    \end{equation*}
with $C_{r,T}$ independent of $n$. Since the initial values $Y^{(n)}(0)=w(x_0)$ are deterministic, Kolmogorov's tightness criterion implies that the laws of $Y^{(n)}$ are tight in $\mathcal C([0,T];\R^N)$. Finally, $X^{(n)}=\tilde f(U^{(n)})$, where $\tilde f:\R^N\to\Cc$ is the continuous extension introduced together with the invariant-coordinate map. Therefore
    \begin{equation*}
        X^{(n)}(t\wedge\chi_r^{(n)}) = \tilde f\bigl(Y^{(n)}(t)\bigr), \qquad 0\le t\le T.
    \end{equation*}
The map $\mathcal C([0,T];\R^N)\ni y \longmapsto \tilde f\circ y\in\mathcal C([0,T];\Cc)$ is continuous for the uniform topology. Hence the continuous mapping theorem implies the tightness of the laws of
    \begin{equation*}
        \bigl(X^{(n)}(t\wedge\chi_r^{(n)})\bigr)_{0\le t\le T}, \qquad n\in\N,
    \end{equation*}
in $\mathcal C([0,T];\Cc)$. This proves \textup{(i)}.

\medskip

We now prove the uniform singular-drift estimate. The idea is to test the approximating equation against a single vector $\rho\in\C$. Since $\rho$ lies in the open chamber, every positive root has strictly positive projection on $\rho$. Fix such a $\rho$, and set $c_\beta:=\langle\rho,\beta\rangle$, $\beta\in R_+$. Then $c_\beta>0$ for every $\beta\in R_+$. Applying the scalar product $\langle\rho,\cdot\rangle$ to the approximating SDE gives, for $t<\zeta^{(n)}$,
    \begin{align*}
        \langle\rho,X^{(n)}(t)\rangle &= \langle\rho,x_0\rangle + \sum_{j=1}^N \rho_j \int_0^t \sigma_j(s,X^{(n)}(s))\,dB_j^{(n)}(s)        \\
        &\quad + \int_0^t \langle\rho,b(s,X^{(n)}(s))\rangle\,ds + \sum_{\beta\in R_+} c_\beta A^{(n),\beta}(t).
    \end{align*}

We now localize once more in order to justify taking expectations.  Fix $m\in\N$ and define
    \begin{equation*}
       \eta_m^{(n)} := \inf\left\{ t\in[0,\zeta^{(n)}): \sum_{\beta\in R_+} c_\beta A^{(n),\beta}(t)\ge m \right\},
    \end{equation*}
We then set $\tau_{r,m}^{(n)} = \chi_r^{(n)}\wedge \eta_m^{(n)}$. The role of $\eta_m^{(n)}$ is only to cap the accumulated singular drift and thus make all terms manifestly integrable. Recall that by $K_{r_*}$ we have denoted the closed centered ball with radius $r_*:=r\vee |x_0|$. For $0\le t\le \chi_r^{(n)}$, the process $X^{(n)}(t)$ stays in $K_{r_*}$. Since $\sigma$ is continuous, it is bounded on $[0,T]\times K_{r_*}$. Hence the quadratic variation of the martingale part in the semi-martingale decomposition of $\langle \rho,X^{(n)}(t)\rangle$ stopped at $\tau_{r,m}^{(n)}$ is bounded by a deterministic constant depending only on $r$, $T$, and $\rho$. In particular, the stopped martingale is square-integrable, and therefore
    \begin{equation*}
        \E \left[\sum_{j=1}^N \rho_j\int_0^{\tau_{r,m}^{(n)}} \sigma_j(s,X^{(n)}(s))\,dB_j^{(n)}(s)\right]=0.
    \end{equation*}
Evaluating the scalar identity at $\tau_{r,m}^{(n)}$ and taking expectations, we obtain
    \begin{align*}
        \E\sum_{\beta\in R_+} c_\beta A^{(n),\beta}({\tau_{r,m}^{(n)}}) &= \E\langle \rho,X^{(n)}({\tau_{r,m}^{(n)}})\rangle - \langle \rho,x_0\rangle - \E\int_0^{\tau_{r,m}^{(n)}} \langle \rho,b(s,X^{(n)}(s))\rangle\,ds.
    \end{align*}
We now estimate the right-hand side term by term. Since $|X^{(n)}({\tau_{r,m}^{(n)}})|\le r_*$ and $b$ is continuous, we have
    \begin{equation*}
        \E\langle \rho,X^{(n)}({\tau_{r,m}^{(n)}})\rangle \le |\rho|\,r_*.
    \end{equation*}
and 
    \begin{equation*}
        \left| \E\int_0^{\tau_{r,m}^{(n)}} \langle \rho,b(s,X^{(n)}(s))\rangle\,ds \right| \le C_{\rho,r,T}.
\end{equation*}
Combining these bounds yields
    \begin{equation*}
        \E\sum_{\beta\in R_+} c_\beta A^{(n),\beta}({\tau_{r,m}^{(n)}}) \le |\rho|\,r_* + |\rho|\,|x_0| + C_{\rho,r,T}.
    \end{equation*}
We now let $m\to\infty$. Since each $A^{(n),\beta}$ is increasing,
    \begin{equation*}
        A^{(n),\beta} ({\tau_{r,m}^{(n)}}) \uparrow A^{(n),\beta}({\chi_r^{(n)}}) \qquad\text{as } m\to\infty,
    \end{equation*}
and monotone convergence gives
    \begin{equation*}
        \E\sum_{\beta\in R_+} c_\beta A^{(n),\beta}({\chi_r^{(n)}}) \le |\rho|\,r_* + |\rho|\,|x_0| + C_{\rho,r,T}.
    \end{equation*}
At this point the positivity of the coefficients $c_\beta$ becomes decisive: for the fixed root $\alpha\in R_+$,
    \begin{equation*}
        \E A^{(n),\alpha}({\chi_r^{(n)}}) \le \frac{|\rho|\,r_* + |\rho|\,|x_0| + C_{\rho,r,T}}{\langle \rho,\alpha\rangle}.
    \end{equation*}
The right-hand side does not depend on $n$. Since
    \begin{equation*}
        A^{(n),\alpha}({\chi_r^{(n)}}) = \int_0^{\chi_r^{(n)}} \frac{k_\alpha^{(n)}(s,X^{(n)}(s))}{\langle X^{(n)}(s),\alpha\rangle}\,ds,
    \end{equation*}
this proves the required uniform estimate. Finally, the monotonicity
\begin{equation*}
    \chi_r^{(n)} \uparrow T\wedge \zeta^{(n)}
    \qquad\text{as } r\to\infty
\end{equation*}
follows directly from the definition of the explosion lifetime through exit times from
compact balls.


\medskip

It remains to prove the compactness of the stopped drift measures. For
$\alpha\in R_+$, by \textup{(ii)},
\begin{equation*}
    \sup_{n\in\N}
    \E\mu_{r,T}^{(n),\alpha}([0,T])
    =
    \sup_{n\in\N}
    \E A^{(n),\alpha}(\chi_r^{(n)})
    \le
    C_{\alpha,r,T}.
\end{equation*}
Hence, by Markov's inequality, for every $L>0$,
\begin{equation*}
    \sup_{n\in\N}
    \mathbb P\bigl(\mu_{r,T}^{(n),\alpha}([0,T])>L\bigr)
    \le
    \frac{C_{\alpha,r,T}}{L}.
\end{equation*}
Since $[0,T]$ is compact, the set
\begin{equation*}
    \mathcal K_L
    :=
    \{\mu\in\mathcal M_+([0,T]):\ \mu([0,T])\le L\}
\end{equation*}
is compact in the weak topology on finite measures. Therefore the laws of
$\mu_{r,T}^{(n),\alpha}$, $n\in\N$, are tight in
$\mathcal M_+([0,T])$. Since $R_+$ is finite, the laws of
\begin{equation*}
    \bigl(\mu_{r,T}^{(n),\alpha}\bigr)_{\alpha\in R_+}
\end{equation*}
are tight in
\begin{equation*}
    \prod_{\alpha\in R_+}\mathcal M_+([0,T]).
\end{equation*}
Combining this measure tightness with the path tightness proved in \textup{(i)}
gives the tightness of the joint laws of
\begin{equation*}
    \left(
        X^{(n)}(\cdot\wedge\chi_r^{(n)}),
        \bigl(\mu_{r,T}^{(n),\alpha}\bigr)_{\alpha\in R_+}
    \right)
\end{equation*}
in
\begin{equation*}
    \mathcal C([0,T];\Cc)
    \times
    \prod_{\alpha\in R_+}\mathcal M_+([0,T]).
\end{equation*}
This proves \textup{(iii)} and completes the proof.
\end{proof}

We now pass to the limit. At this point it is convenient to keep track not only of the
stopped paths and the stopped drift measures, but also of the corresponding regular
drift, martingale part, and quadratic variation. This will allow us to identify the
structure of the limiting equation before we analyze the support of the limiting drift
measures.

For fixed $T,r>0$, define
\begin{align*}
    X^{(n),r}(t)
    &:=
    X^{(n)}(t\wedge \chi_r^{(n)}), \\
    D^{(n),r}(t)
    &:=
    \int_0^{t\wedge \chi_r^{(n)}}
        b(s,X^{(n)}(s))\,ds, \\
    M^{(n),r}(t)
    &:=
    \int_0^{t\wedge \chi_r^{(n)}}
        \sigma(s,X^{(n)}(s))\,dB^{(n)}(s), \\
    Q_i^{(n),r}(t)
    &:=
    \int_0^{t\wedge \chi_r^{(n)}}
        \sigma_i^2(s,X^{(n)}(s))\,ds,
    \qquad i=1,\ldots,N.
\end{align*}
Then, for every $0\le t\le T$,
\begin{equation*}
    X^{(n),r}(t)
    =
    x_0
    +
    M^{(n),r}(t)
    +
    D^{(n),r}(t)
    +
    \sum_{\alpha\in R_+}
        \alpha\,\mu_{r,T}^{(n),\alpha}([0,t]).
\end{equation*}

\begin{proposition}[Skorokhod representation and generalized weak limit]
\label{prop:deg:skorokhod}
Assume \ref{ass:cont}-\ref{ass:sigma}, and let $X^{(n)}$ be the approximating
sequence. Then, after passing to a subsequence, there exist a probability space, a
continuous $\Cc$-valued process $X$, a continuous local martingale
$M=(M_1,\ldots,M_N)$, and, for every $\alpha\in R_+$, a continuous increasing
process $A^\alpha$ on $[0,\zeta)$ with $A^\alpha(0)=0$, such that, for
\begin{equation*}
    \zeta
    :=
    \lim_{r\to\infty}
    \inf\{t\ge0:\ |X(t)|\ge r\},
\end{equation*}
we have, for every $t<\zeta$,
\begin{equation}
    \label{eq:deg:generalized-limit}
    X(t)
    =
    x_0
    +
    M(t)
    +
    \int_0^t b(s,X(s))\,ds
    +
    \sum_{\alpha\in R_+}\alpha\,A^\alpha(t),
\end{equation}
and
\begin{equation}
    \label{eq:deg:limit-bracket}
    \langle M_i,M_j\rangle(t)
    =
    \delta_{ij}
    \int_0^t \sigma_i^2(s,X(s))\,ds,
    \qquad i,j=1,\ldots,N.
\end{equation}
Consequently, possibly after enlarging the probability space, there exists an
$N$-dimensional Brownian motion $B$ such that
\begin{equation*}
    M(t)
    =
    \int_0^t \sigma(s,X(s))\,dB(s),
    \qquad t<\zeta.
\end{equation*}
In particular, $X$ is a generalized weak limit of the approximating sequence.
\end{proposition}

\begin{proof}
We begin with a fixed time horizon $T$ and a fixed localization radius $r$. The
idea is first to identify all possible limits on the compact interval $[0,T]$, and
only afterwards to remove the localization by a diagonal argument.

On the stopped interval $[0,\chi_r^{(n)}]$, the process $X^{(n)}$ remains in
$K_{r_*} = \{x\in\Cc:\ |x|\le r_*\}$, $r_*:=r\vee |x_0|$. Hence, by continuity of $b$ and $\sigma$, there exists a constant
$C_{r,T}<\infty$ such that
\begin{equation*}
    |b(t,x)|+\sum_{i=1}^N |\sigma_i(t,x)|^2
    \le
    C_{r,T},
\end{equation*}
for all $0\le t\le T$ and all $x\in K_{r_*}$.

As a consequence, the stopped drift processes $D^{(n),r}$ are uniformly Lipschitz:
\begin{equation*}
    |D^{(n),r}(t)-D^{(n),r}(s)|
    \le
    C_{r,T}|t-s|,
    \qquad 0\le s\le t\le T.
\end{equation*}
Similarly, each stopped quadratic variation process $Q_i^{(n),r}$ satisfies
\begin{equation*}
    |Q_i^{(n),r}(t)-Q_i^{(n),r}(s)|
    \le
    C_{r,T}|t-s|,
    \qquad 0\le s\le t\le T.
\end{equation*}
Thus the families $\{D^{(n),r}\}_{n\ge1}$ and $\{Q^{(n),r}\}_{n\ge1}$ are tight
in the corresponding spaces of continuous paths.

For the martingale part, we again use Kolmogorov's criterion. Let $0\le s\le t\le T$.
By the Burkholder-Davis-Gundy inequality,
\begin{align*}
    \E\left[
        |M^{(n),r}(t)-M^{(n),r}(s)|^4
    \right]
    &\le
    C
    \E\left[
        \left(
            \int_s^t
            \mathbf 1_{\{u\le \chi_r^{(n)}\}}
            \sum_{i=1}^N \sigma_i^2(u,X^{(n)}(u))
            \,du
        \right)^2
    \right] \\
    &\le
    C_{r,T}|t-s|^2,
\end{align*}
where $C$ is the Burkholder-Davis-Gundy constant for $p=4$, and the value of
$C_{r,T}$ may change from line to line. Therefore $\{M^{(n),r}\}_{n\ge1}$ is
tight in $\mathcal C([0,T];\R^N)$.

Combining these bounds with the path tightness and measure compactness from Proposition~\ref{prop:deg:localized-compactness}, we obtain tightness of the joint laws of
\begin{equation*}
    \left(
        X^{(n),r},
        D^{(n),r},
        M^{(n),r},
        Q^{(n),r},
        \bigl(\mu_{r,T}^{(n),\alpha}\bigr)_{\alpha\in R_+}
    \right)
\end{equation*}
in
\begin{equation*}
    \mathcal C([0,T];\Cc)
    \times
    \mathcal C([0,T];\R^N)
    \times
    \mathcal C([0,T];\R^N)
    \times
    \mathcal C([0,T];\R_+^N)
    \times
    \prod_{\alpha\in R_+}\mathcal M_+([0,T]).
\end{equation*}

By Prokhorov's theorem and the Skorokhod representation theorem, after passing to a
subsequence we may realize these variables on a common probability space so that they
converge almost surely. We denote the corresponding limits by
\begin{equation*}
    \left(
        X^r,
        D^r,
        M^r,
        Q^r,
        \bigl(\mu^{\alpha,r}\bigr)_{\alpha\in R_+}
    \right).
\end{equation*}
Thus $X^{(n),r} \longrightarrow X^r$, $D^{(n),r} \longrightarrow D^r$, $M^{(n),r} \longrightarrow M^r$, $Q^{(n),r} \longrightarrow Q^r$ and all of those convergence are uniform on $[0,T]$. Moreover, for every $\alpha\in R_+$,
\begin{equation*}
    \mu_{r,T}^{(n),\alpha}
    \Longrightarrow
    \mu^{\alpha,r}
\end{equation*}
weakly as finite measures on $[0,T]$.

The next point is to recover the martingale structure in the limit. For each $n$,
$M^{(n),r}$ is a continuous square-integrable martingale, and for every
$i,j\in\{1,\ldots,N\}$ the process
\begin{equation*}
    M_i^{(n),r}(t)M_j^{(n),r}(t)-\delta_{ij}Q_i^{(n),r}(t)
\end{equation*}
is a martingale. Moreover,
\begin{equation*}
    \sup_{n\in\N}\E Q_i^{(n),r}(T)\le C_{r,T},
\end{equation*}
so these martingales are uniformly integrable. Passing to the limit in the usual way
against bounded continuous functionals of the paths up to time $s$, we conclude
that $M^r$ is a continuous martingale and that
\begin{equation*}
    M_i^r(t)M_j^r(t)-\delta_{ij}Q_i^r(t)
\end{equation*}
is a martingale as well. Hence
\begin{equation*}
    \langle M_i^r,M_j^r\rangle(t)
    =
    \delta_{ij}Q_i^r(t),
    \qquad 0\le t\le T.
\end{equation*}

We now identify the deterministic finite-variation terms. Set
\begin{equation*}
    \chi_r
    :=
    T\wedge \inf\{t\ge0:\ |X^r(t)|\ge r\}.
\end{equation*}
Fix $t<\chi_r$. By continuity of $X^r$, the path stays a positive distance away
from the sphere $\{|x|=r\}$ on $[0,t]$. It follows that, for all sufficiently
large $n$, one has $\chi_r^{(n)}>t$. On this interval the approximating stopped
processes are not yet affected by localization, and the uniform convergence together
with continuity of $b$ and $\sigma$ gives
\begin{align*}
    D^r(t)
    &=
    \int_0^t b(s,X^r(s))\,ds, \\
    Q_i^r(t)
    &=
    \int_0^t \sigma_i^2(s,X^r(s))\,ds,
    \qquad i=1,\ldots,N.
\end{align*}

For every $\alpha\in R_+$, let
\begin{equation*}
    A^{\alpha,r}(t):=\mu^{\alpha,r}([0,t]),
    \qquad 0\le t\le T.
\end{equation*}
At every common continuity point $t$ of the finite measures
$\mu^{\alpha,r}$, $\alpha\in R_+$, the identity
\begin{equation*}
    X^{(n),r}(t)
    =
    x_0
    +
    M^{(n),r}(t)
    +
    D^{(n),r}(t)
    +
    \sum_{\alpha\in R_+}
        \alpha\,\mu_{r,T}^{(n),\alpha}([0,t])
\end{equation*}
passes to the limit and yields
\begin{equation*}
    X^r(t)
    =
    x_0
    +
    M^r(t)
    +
    D^r(t)
    +
    \sum_{\alpha\in R_+}
        \alpha\,A^{\alpha,r}(t).
\end{equation*}
The set of common continuity points is dense in $[0,T]$. Since $X^r$, $M^r$,
and $D^r$ are continuous, the vector finite-variation process
\begin{equation*}
    V^r(t)
    :=
    X^r(t)-x_0-M^r(t)-D^r(t)
\end{equation*}
is continuous. Thus $V^r$ is a continuous version of
$\sum_{\alpha\in R_+}\alpha\,A^{\alpha,r}$.

At this stage only the vector combination is known to be continuous. To conclude that
each individual process $A^{\alpha,r}$ is continuous, we use again a vector
$\rho\in\C$. Since $\langle \rho,\alpha\rangle>0$ for all $\alpha\in R_+$,
every jump of $V^r$ would satisfy
\begin{equation*}
    \langle \rho,\Delta V^r(t)\rangle
    =
    \sum_{\alpha\in R_+}
    \langle \rho,\alpha\rangle\,\Delta A^{\alpha,r}(t).
\end{equation*}
The left-hand side is zero because $V^r$ is continuous, whereas every term on the
right-hand side is nonnegative. Hence
\begin{equation*}
    \Delta A^{\alpha,r}(t)=0,
    \qquad \alpha\in R_+,
    \qquad 0\le t\le T.
\end{equation*}
So each $A^{\alpha,r}$ is continuous and increasing, and the localized identity
holds for every $t\in[0,T]$.

We now remove the localization. Choose increasing sequences
\begin{equation*}
    T_m\uparrow\infty,
    \qquad
    r_m\uparrow\infty,
    \qquad
    r_m>|x_0|.
\end{equation*}
Applying the preceding construction successively to the pairs $(T_m,r_m)$ and then
using a diagonal extraction, we obtain localized limits which are consistent on
overlapping time intervals. The consistency comes from the fact that the corresponding
stopped approximating objects agree up to the smaller stopping time, and this
agreement is preserved in the limit. We may therefore patch the localized limits
together and obtain a continuous $\Cc$-valued process $X$, a continuous local
martingale $M$, and, for every $\alpha\in R_+$, a continuous increasing process
$A^\alpha$ defined on $[0,\zeta)$, where
\begin{equation*}
    \zeta
    :=
    \lim_{m\to\infty}
    \inf\{t\ge0:\ |X(t)|\ge r_m\}
    =
    \lim_{r\to\infty}
    \inf\{t\ge0:\ |X(t)|\ge r\}.
\end{equation*}
Since $X(0)=x_0$ and $X$ is continuous, we have $\zeta>0$ almost surely.

For every $t<\zeta$, choosing $m$ so large that $t<T_m$ and
$|X(s)|<r_m$ for $0\le s\le t$, the localized equation gives
\begin{equation*}
    X(t)
    =
    x_0
    +
    M(t)
    +
    \int_0^t b(s,X(s))\,ds
    +
    \sum_{\alpha\in R_+}\alpha\,A^\alpha(t),
\end{equation*}
and similarly
\begin{equation*}
    \langle M_i,M_j\rangle(t)
    =
    \delta_{ij}
    \int_0^t \sigma_i^2(s,X(s))\,ds,
    \qquad i,j=1,\ldots,N.
\end{equation*}

It remains to rewrite the martingale $M$ as a stochastic integral with respect to a
Brownian motion. This is standard. Possibly after enlarging the probability space, one
may construct independent Brownian motions on the sets where the coefficients
$\sigma_i(s,X(s))$ vanish and divide the martingale part by $\sigma_i(s,X(s))$
elsewhere. Lévy's characterization then yields an $N$-dimensional Brownian motion
$B=(B_1,\ldots,B_N)$ such that
\begin{equation*}
    M(t)
    =
    \int_0^t \sigma(s,X(s))\,dB(s),
    \qquad t<\zeta.
\end{equation*}
This proves \eqref{eq:deg:generalized-limit} and
\eqref{eq:deg:limit-bracket}, and completes the construction of the generalized weak
limit.
\end{proof}


\subsection{Identification of the limiting singular drift}
\label{subsec:deg:interior-identification}

We now identify the part of the limiting drift which is already visible in the approximating equations. Away from the wall
$\{\langle x,\alpha\rangle=0\}$, the singular coefficient ${k_\alpha(t,x)}/{\langle x,\alpha\rangle}$ is an ordinary continuous function. Hence no loss of information can occur in the limit on regions where $\langle X,\alpha\rangle$ stays strictly positive. Any possible ambiguity in the limiting increasing process $A^\alpha$ can therefore only be supported on the set of times at which the limiting path lies on the corresponding wall. The next proposition makes this precise. It decomposes $A^\alpha$ into the absolutely continuous interior singular drift and a possible residual boundary measure. It also records that the localized uniform estimates from Proposition~\ref{prop:deg:localized-compactness} pass to the limit and give local integrability of the interior singular drift.

\begin{proposition}[Interior identification and local integrability of the limiting drift]
    \label{prop:deg:interior-ident-integrability}
    Let
    \begin{equation*}
        \left(
            X,M,\bigl(A^\alpha\bigr)_{\alpha\in R_+}
        \right)
    \end{equation*}
    be a generalized weak limit obtained in Proposition~\ref{prop:deg:skorokhod}. Then, for every $\alpha\in R_+$, there exists a continuous increasing process $\Lambda^\alpha$ on $[0,\zeta)$, with $\Lambda^\alpha(0)=0$, such that, almost surely, for every $t<\zeta$,
    \begin{equation}
        \label{eq:deg:A-decomposition-integrated}
        A^\alpha(t) = \int_0^t \frac{k_\alpha(s,X(s))}{\langle X(s),\alpha\rangle} \mathbf 1_{\{\langle X(s),\alpha\rangle>0\}}\,ds + \Lambda^\alpha(t).
    \end{equation}
    Moreover, the Stieltjes measure $d\Lambda^\alpha$ is supported on the wall $\{\langle X,\alpha\rangle=0\}$, in the sense that, almost surely, for every $t<\zeta$,
    \begin{equation}
        \label{eq:deg:Lambda-support}
        \int_0^t \mathbf 1_{\{\langle X(s),\alpha\rangle>0\}} \,d\Lambda^\alpha(s) = 0.
    \end{equation}
    Equivalently, on $[0,\zeta)$,
    \begin{equation}
        \label{eq:deg:A-decomposition} 
        dA^\alpha(t) = \frac{k_\alpha(t,X(t))}{\langle X(t),\alpha\rangle} \mathbf 1_{\{\langle X(t),\alpha\rangle>0\}}\,dt + d\Lambda^\alpha(t).
    \end{equation}
    Furthermore, for every $T,r>0$, define
    \begin{equation}
        \label{eq:def:chi-limit}
        \chi_r := T\wedge\inf\{t\in[0,\zeta):\ |X(t)|\ge r\},
    \end{equation}
    with the convention $\inf\varnothing=\infty$. Then, for every $\alpha\in R_+$,
    \begin{equation}
        \label{eq:deg:limit-int-stopped}
        \E \int_0^{\chi_r} \frac{k_\alpha(s,X(s))}{\langle X(s),\alpha\rangle} \mathbf 1_{\{\langle X(s),\alpha\rangle>0\}}\,ds \le C_{\alpha,r,T},
    \end{equation}
    where $C_{\alpha,r,T}$ is the constant from Proposition~\ref{prop:deg:localized-compactness}\textup{(ii)}. Consequently, almost surely,
    \begin{equation}
        \label{eq:deg:limit-int-pathwise}
        \int_0^t \frac{k_\alpha(s,X(s))}{\langle X(s),\alpha\rangle}\mathbf 1_{\{\langle X(s),\alpha\rangle>0\}}\,ds < \infty
    \end{equation}
    for every $t<\zeta$.
\end{proposition}


\begin{proof}
We first work on a fixed compact time interval and with a fixed localization radius. The point is to test the limiting measure only in regions where the path stays away both from the wall and from the stopping sphere. On such sets the singular density is regular, and the limit can be identified directly.

Fix $T,r>0$ and $\alpha\in R_+$. Recall the localized stopped paths and stopped drift measures \eqref{eq:def:aprox:measure}
    \begin{align*}
        X^{(n),r}(t) &:= X^{(n)}(t\wedge \chi_r^{(n)}), \\
        \mu_{r,T}^{(n),\alpha}(dt) &:= \mathbf 1_{\{t\le \chi_r^{(n)}\}} \frac{k_\alpha^{(n)}(t,X^{(n)}(t))}{\langle X^{(n)}(t),\alpha\rangle}\,dt,
    \end{align*}
together with the stopping time 
    \begin{equation*}
        \chi_r^{(n)} := T\wedge \inf\{t\in[0,\zeta^{(n)}):\ |X^{(n)}(t)|\ge r\}.
    \end{equation*}
Along the Skorokhod subsequence from Proposition~\ref{prop:deg:skorokhod}, we may assume that
    \begin{align*}
        X^{(n),r} &\longrightarrow X^r \quad \text{uniformly on } [0,T], \\
        \mu_{r,T}^{(n),\alpha} &\Longrightarrow \mu^{\alpha,r} \quad \text{weakly as finite measures on } [0,T].
    \end{align*}
Here $\mu^{\alpha,r}$ is the localized limiting measure corresponding to the root $\alpha$. We first identify $\mu^{\alpha,r}$ away from the wall and away from the stopping sphere. Let $\varphi\in C_b([0,T]\times\Cc)$ be such that, for some $\delta>0$,
    \begin{align*}
        \varphi(t,x)&=0, \qquad \text{whenever } \langle x,\alpha\rangle\le \delta \textrm{ or }|x|\ge r. 
    \end{align*}
Thus $\varphi$ only probes those portions of the paths where the denominator $\langle x,\alpha\rangle$ stays safely away from zero and where the localization has not yet reached the stopping sphere. Define
    \begin{equation*}
        F_n(t,x) := \varphi(t,x) \frac{k_\alpha^{(n)}(t,x)}{\langle x,\alpha\rangle},\quad F(t,x) := \varphi(t,x) \frac{k_\alpha(t,x)}{\langle x,\alpha\rangle}\/,
    \end{equation*}
with the obvious convention $F_n(t,x)=0$ and $F(t,x)=0$, when $\langle x,\alpha\rangle=0$. Since $\varphi$ is supported away from the wall, both $F_n$ and $F$ are bounded and continuous on the compact region relevant for the stopped paths. Moreover, since $k_\alpha^{(n)}(t,x) \to k_\alpha(t,x)$, we have $F_n\longrightarrow F$ uniformly on compact subsets of $[0,T]\times\Cc$.

For each $n$, the definition of $\mu_{r,T}^{(n),\alpha}$ gives
\begin{equation}
    \label{eq:deg:test-ident-n}
    \int_0^T
        \varphi\bigl(t,X^{(n),r}(t)\bigr)
        \,\mu_{r,T}^{(n),\alpha}(dt)
    =
    \int_0^T
        F_n\bigl(t,X^{(n),r}(t)\bigr)\,dt.
\end{equation}
Indeed, before the stopping time $\chi_r^{(n)}$ we have
$X^{(n),r}(t)=X^{(n)}(t)$, while after the stopping time the stopped path stays at a
point of norm $r$, where $\varphi$ vanishes.

We now pass to the limit in both sides of \eqref{eq:deg:test-ident-n}. Since
$X^{(n),r}\to X^r$ uniformly and $\varphi$ is bounded and continuous, we have
\begin{equation*}
    \varphi\bigl(\cdot,X^{(n),r}(\cdot)\bigr)
    \longrightarrow
    \varphi\bigl(\cdot,X^r(\cdot)\bigr)
\end{equation*}
uniformly on $[0,T]$. Therefore,
\begin{align*}
    &\left|
        \int_0^T
        \varphi\bigl(t,X^{(n),r}(t)\bigr)
        \,\mu_{r,T}^{(n),\alpha}(dt)
        -
        \int_0^T
        \varphi\bigl(t,X^r(t)\bigr)
        \,\mu^{\alpha,r}(dt)
    \right| \\
    &\le
    \sup_{t\in[0,T]}
    \left|
        \varphi\bigl(t,X^{(n),r}(t)\bigr)
        -
        \varphi\bigl(t,X^r(t)\bigr)
    \right|
    \mu_{r,T}^{(n),\alpha}([0,T]) \\
    &\quad
    +
    \left|
        \int_0^T
        \varphi\bigl(t,X^r(t)\bigr)\,\mu_{r,T}^{(n),\alpha}(dt)
        -
        \int_0^T
        \varphi\bigl(t,X^r(t)\bigr)\,\mu^{\alpha,r}(dt)
    \right|.
\end{align*}
The first term tends to zero because the integrands converge uniformly and the total
masses $\mu_{r,T}^{(n),\alpha}([0,T])$ remain bounded along the convergent
subsequence; the second tends to zero by the weak convergence of the measures. Hence
\begin{equation}
    \label{eq:deg:left-limit-measure}
    \int_0^T
        \varphi\bigl(t,X^{(n),r}(t)\bigr)
        \,\mu_{r,T}^{(n),\alpha}(dt)
    \longrightarrow
    \int_0^T
        \varphi\bigl(t,X^r(t)\bigr)
        \,\mu^{\alpha,r}(dt).
\end{equation}

On the other hand, the uniform convergence of $X^{(n),r}$ together with the uniform
convergence of $F_n$ to $F$ on compact sets yields
\begin{equation}
    \label{eq:deg:right-limit-density}
    \int_0^T
        F_n\bigl(t,X^{(n),r}(t)\bigr)\,dt
    \longrightarrow
    \int_0^T
        F\bigl(t,X^r(t)\bigr)\,dt.
\end{equation}
Combining \eqref{eq:deg:test-ident-n}, \eqref{eq:deg:left-limit-measure}, and
\eqref{eq:deg:right-limit-density}, we obtain
\begin{equation}
    \label{eq:deg:test-identity-limit}
    \int_0^T
        \varphi\bigl(t,X^r(t)\bigr)
        \,\mu^{\alpha,r}(dt)
    =
    \int_0^T
        \varphi\bigl(t,X^r(t)\bigr)
        \frac{k_\alpha(t,X^r(t))}
             {\langle X^r(t),\alpha\rangle}
        \mathbf 1_{\{\langle X^r(t),\alpha\rangle>0\}}
        \,dt.
\end{equation}
This identity already shows that the limiting measure has the expected density
wherever the path stays away from the wall and away from the stopping sphere.

We now remove the auxiliary cutoff in the space variable. Define
\begin{equation*}
    \chi_r := T\wedge \inf\{t\in[0,\zeta):\ |X(t)|\ge r\}.
\end{equation*}
and
\begin{equation*}
    O_{\alpha,r}
    :=
    \{t\in[0,T]:\ t<\chi_r,\ \langle X^r(t),\alpha\rangle>0\}.
\end{equation*}
The set $O_{\alpha,r}$ is open. Let $K\subset O_{\alpha,r}$ be compact. By
continuity of $X^r$, there exist constants $\delta>0$ and $r_0<r$ such that
\begin{align*}
    \langle X^r(t),\alpha\rangle&\ge 2\delta,\qquad |X^r(t)|\le r_0\/,
    \qquad t\in K.
\end{align*}
Choose continuous cutoff functions $\theta$ and $\eta$ satisfying
\begin{align*}
    \theta(y)&=0
    \quad \text{for } y\le \delta,
    &
    \theta(y)&=1
    \quad \text{for } y\ge 2\delta, \\
    \eta(x)&=1
    \quad \text{for } |x|\le r_0,
    &
    \eta(x)&=0
    \quad \text{for } |x|\ge r.
\end{align*}
Now let $\psi\in C_c(O_{\alpha,r})$ have support contained in $K$, and define
\begin{equation*}
    \varphi(t,x)
    :=
    \psi(t)\,\theta(\langle x,\alpha\rangle)\,\eta(x).
\end{equation*}
We extend $\psi$ by zero to all of $[0,T]$. By construction,
\begin{equation*}
    \varphi\bigl(t,X^r(t)\bigr)=\psi(t)
    \qquad \text{for } t\in \operatorname{supp}\psi.
\end{equation*}
Substituting this choice of $\varphi$ into
\eqref{eq:deg:test-identity-limit}, we obtain
\begin{equation*}
    \int_0^T \psi(t)\,\mu^{\alpha,r}(dt)
    =
    \int_0^T
        \psi(t)
        \frac{k_\alpha(t,X^r(t))}
             {\langle X^r(t),\alpha\rangle}
        \,dt.
\end{equation*}
Since continuous compactly supported functions determine finite measures on the open
set $O_{\alpha,r}$, it follows that
\begin{equation}
    \label{eq:deg:localized-interior-ident}
    \mu^{\alpha,r}(dt)
    =
    \frac{k_\alpha(t,X^r(t))}
         {\langle X^r(t),\alpha\rangle}
    \,dt
    \qquad \text{on } O_{\alpha,r}.
\end{equation}

We now return to the global limit constructed in
Proposition~\ref{prop:deg:skorokhod}. The localized limits there were chosen
consistently as $T$ and $r$ vary, so the preceding identity passes through the
diagonal construction and yields, on the random interval $[0,\zeta)$,
\begin{equation}
    \label{eq:deg:interior-restriction}
    dA^\alpha(t)
    =
    \frac{k_\alpha(t,X(t))}
         {\langle X(t),\alpha\rangle}
    \,dt
    \qquad \text{on }
    \{t<\zeta:\ \langle X(t),\alpha\rangle>0\}.
\end{equation}
At this point the decomposition is immediate. Let
    \begin{equation*}
        Z_\alpha := \{t<\zeta:\ \langle X(t),\alpha\rangle=0\}.
    \end{equation*}
Let $\lambda^\alpha$ be the measure obtained by restricting $dA^\alpha$
to $Z_\alpha$, that is,
    \begin{equation*}
         \lambda^\alpha(B):=dA^\alpha(B\cap Z_\alpha), \qquad B\subset [0,\zeta)\ \text{Borel}.
    \end{equation*}
Define
    \begin{equation*}
        \Lambda^\alpha(t):=\lambda^\alpha([0,t]).
    \end{equation*}
Then $\Lambda^\alpha$ is continuous and increasing, and $d\Lambda^\alpha$ is supported on $Z_\alpha$. Moreover,
    \begin{equation*}
        dA^\alpha = \mathbf 1_{Z_\alpha^c}\,dA^\alpha + d\Lambda^\alpha .
    \end{equation*}

By construction, the measure $dA^\alpha$ coincides with the absolutely continuous measure in \eqref{eq:deg:interior-restriction} away from the wall, and the remaining part is exactly $d\Lambda^\alpha$. This gives \eqref{eq:deg:A-decomposition}. The support property is built into the definition of $d\Lambda^\alpha$, and the integrated form \eqref{eq:deg:A-decomposition-integrated} follows by integrating \eqref{eq:deg:A-decomposition} on $[0,t]$.

\medskip
We now prove the localized expectation bound. Fix $r,T>0$ and $\alpha\in R_+$.  Recall again the stopped approximating measures \eqref{eq:def:aprox:measure}. By Proposition~\ref{prop:deg:localized-compactness}\textup{(ii)}
    \begin{equation}
        \label{eq:deg:uniform-mass-before-limit}
        \sup_{n\in\N} \E \mu_{r,T}^{(n),\alpha}([0,T]) = \sup_{n\in\N} \E A^{(n),\alpha}(\chi_r^{(n)})
    \le
    C_{\alpha,r,T}.
\end{equation}

We now pass this bound to the limit. Along the Skorokhod subsequence constructed in
Proposition~\ref{prop:deg:skorokhod}, the measures
$\mu_{r,T}^{(n),\alpha}$ converge weakly to the localized limiting measure
$\mu^{\alpha,r}$ on $[0,T]$. Since the constant function $1$ is continuous on
the compact interval $[0,T]$, weak convergence yields
\begin{equation*}
    \mu_{r,T}^{(n),\alpha}([0,T])
    \longrightarrow
    \mu^{\alpha,r}([0,T])
\end{equation*}
almost surely. Therefore Fatou's lemma together with
\eqref{eq:deg:uniform-mass-before-limit} gives
\begin{equation}
    \label{eq:deg:limit-mass-bound}
    \E \mu^{\alpha,r}([0,T])
    \le
    \liminf_{n\to\infty}
    \E \mu_{r,T}^{(n),\alpha}([0,T])
    \le
    C_{\alpha,r,T}.
\end{equation}

We next relate this localized limiting measure to the global increasing process
$A^\alpha$. By the consistency built into the diagonal construction from
Proposition~\ref{prop:deg:skorokhod}, the localized cumulative process
\begin{equation*}
    A^{\alpha,r}(t):=\mu^{\alpha,r}([0,t]),
    \qquad 0\le t\le T,
\end{equation*}
coincides with $A^\alpha(t\wedge \chi_r)$. In particular, $\mu^{\alpha,r}([0,T]) = A^\alpha(\chi_r)$. By the decomposition already proved above, we have
\begin{equation*}
    A^\alpha(t)
    =
    \int_0^t
    \frac{k_\alpha(s,X(s))}
         {\langle X(s),\alpha\rangle}
    \mathbf 1_{\{\langle X(s),\alpha\rangle>0\}}
    \,ds
    +
    \Lambda^\alpha(t),
    \qquad t<\zeta,
\end{equation*}
where $\Lambda^\alpha$ is continuous, increasing, and nonnegative. Evaluating this
identity at $t=\chi_r$, we obtain
    \begin{align*}
        \int_0^{\chi_r} \frac{k_\alpha(s,X(s))}{\langle X(s),\alpha\rangle} \mathbf 1_{\{\langle X(s),\alpha\rangle>0\}}\,ds
        &\le A^\alpha(\chi_r)  = \mu^{\alpha,r}([0,T]).
    \end{align*}
Taking expectations and using \eqref{eq:deg:limit-mass-bound} proves
\eqref{eq:deg:limit-int-stopped}.

It remains to deduce the pathwise local statement. For each integer $m>|x_0|$ and
each integer $q\ge1$, the already proved estimate with $r=m$ and $T=q$ implies
\begin{equation*}
    \int_0^{q\wedge \tau_m}
    \frac{k_\alpha(s,X(s))}
         {\langle X(s),\alpha\rangle}
    \mathbf 1_{\{\langle X(s),\alpha\rangle>0\}}
    \,ds
    <\infty
\end{equation*}
almost surely, where $\tau_m:=\inf\{t\in[0,\zeta):\ |X(t)|\ge m\}$. Since the family of pairs $(m,q)$ is countable, we may intersect these events and obtain a single event of probability one on which all such localized integrals are finite. Fix now a sample point in this event and let $t<\zeta$. Since $X$ is continuous on $[0,t]$, there exists an integer $m>|x_0|$ such that
    \begin{equation*}
        \sup_{0\le s\le t}|X(s)|<m.
    \end{equation*}
Hence $\tau_m>t$. Choosing an integer $q>t$, we get
\begin{align*}
    \int_0^t
    \frac{k_\alpha(s,X(s))}
         {\langle X(s),\alpha\rangle}
    \mathbf 1_{\{\langle X(s),\alpha\rangle>0\}}
    \,ds
    &\le
    \int_0^{q\wedge \tau_m}
    \frac{k_\alpha(s,X(s))}
         {\langle X(s),\alpha\rangle}
    \mathbf 1_{\{\langle X(s),\alpha\rangle>0\}}
    \,ds <\infty.
\end{align*}
This proves \eqref{eq:deg:limit-int-pathwise}.
\end{proof}


\subsection{Face detectors and zero local time}
\label{subsec:deg:detectors-local-time}

We now introduce a canonical detector attached to an arbitrary collision face. Its purpose is to produce a scalar quantity which is nonnegative on the chamber and vanishes exactly on the chosen face after localization. When applied to the generalized weak equation, this detector yields a one-dimensional identity in which the roots vanishing on the face appear explicitly, while all other roots are absorbed into the regular drift. This will be the basic tool in the next subsection, where we eliminate the possible boundary measures.

Fix $S\in\mathfrak S$. Recall that $F_S$ denotes the corresponding relatively open face and that, for $\eta,r>0$,
    \begin{equation*}
        \mathcal U_{S,\eta,r} := \{x\in\Cc:\ |x|\le r,\  \langle x,\beta\rangle\ge \eta \text{ for every } \beta\in R_+\setminus S\}.
    \end{equation*}
Thus $\mathcal U_{S,\eta,r}$ is a localized neighborhood of $F_S$ in which the roots outside $S$ are kept uniformly away from zero. We define the canonical face detector by
    \begin{equation*}
        u_S := \sum_{\beta\in S}\beta, \qquad \tau_S(x) := \langle u_S,x\rangle = \sum_{\beta\in S}\langle x,\beta\rangle.
    \end{equation*}
By our convention in the definition of the detector family, $u_S\in\mathfrak D(S)$. In particular, $B_S=B_{S,u_S}$ is one of the regular detector drifts to which Assumption~\ref{ass:face:sign} applies. Thus, on $\mathcal U_{S,\eta,r}$,
    \begin{equation*}
        B_S(t,x) = \langle u_S,b(t,x)\rangle + \sum_{\beta\in R_+\setminus S} \frac{k_\beta(t,x)\langle u_S,\beta\rangle}{\langle x,\beta\rangle}.
    \end{equation*}
The detector $\tau_S$ has two useful features. First, it is nonnegative on the whole
chamber and vanishes exactly when all roots in $S$ vanish. Second, it sees all roots
in $S$ simultaneously, in the sense that every such root appears with a strictly
positive coefficient in the detector equation. We also set
\begin{equation}
    \label{eq:def:q-S}
    q_S(t,x) := \sum_{j=1}^N u_{S,j}^2\,\sigma_j^2(t,x).
\end{equation}

We collect these elementary geometric facts, together with the localized detector equation needed later, in a single lemma.

\begin{lemma}[Canonical detector and its localized equation]
\label{lem:deg:canonical-detector}
Let $S\in\mathfrak S$, and let $u_S$ and $\tau_S$ be defined as above. Let $\left(X,M,\bigl(A^\alpha\bigr)_{\alpha\in R_+}\right)$ be a generalized weak limit obtained in Proposition~\ref{prop:deg:skorokhod}, and let the decomposition from Proposition~\ref{prop:deg:interior-ident-integrability} hold. Then the following statements hold.

\begin{enumerate}[label=\textup{(\roman*)}]
    \item For every $x\in\Cc$,
    \begin{equation*}
        \tau_S(x)\ge0.
    \end{equation*}
    Moreover,
    \begin{equation*}
        \tau_S(x)=0
        \qquad\Longleftrightarrow\qquad
        \langle x,\beta\rangle=0
        \quad \text{for every } \beta\in S.
    \end{equation*}
    In particular, $\tau_S$ vanishes on $F_S$, and on
    $\mathcal U_{S,\eta,r}$ one has
    \begin{equation*}
        \tau_S(x)=0
        \qquad\Longleftrightarrow\qquad
        x\in F_S.
    \end{equation*}

    \item For every $x\in\Cc$ and every $\beta\in S$,
    \begin{equation*}
        0\le \langle x,\beta\rangle\le \tau_S(x).
    \end{equation*}

    \item For every $\beta\in S$,
    \begin{equation*}
        \langle u_S,\beta\rangle>0.
    \end{equation*}
    Hence $u_S$ is an admissible detector for the face $F_S$, and $\Gamma(S,u_S)=S$. In particular, the constant $c_S:=\min_{\beta\in S}\langle u_S,\beta\rangle$ is strictly positive.

    \item We have 
    \begin{equation*}
        q_S(t,x)
        \le
        |S|\sum_{\beta\in S} a_\beta(t,x),
        \qquad
        (t,x)\in[0,\infty)\times\Cc.
    \end{equation*}
    \item Let $T,r,\eta>0$, and let $\theta$ be a stopping time such that $\theta\le T$, $X(s)\in\mathcal U_{S,\eta,r}$, $0\le s\le\theta$. Then, for every $0\le t\le T$,
    \begin{align}
        \label{eq:detector-equation}
        \tau_S\bigl(X(t\wedge\theta)\bigr) &= \tau_S(x_0) + M_S(t\wedge\theta) + \int_0^{t\wedge\theta}B_S(s,X(s))\,ds                                      \nonumber\\
    &\quad
    +
    \sum_{\beta\in S}
    \langle u_S,\beta\rangle
    \int_0^{t\wedge\theta}
    \frac{k_\beta(s,X(s))}
         {\langle X(s),\beta\rangle}
    \mathbf 1_{\{\langle X(s),\beta\rangle>0\}}\,ds       \nonumber\\
    &\quad
    +
    \sum_{\beta\in S}
    \langle u_S,\beta\rangle
    \Lambda^\beta(t\wedge\theta),
\end{align}
where $M_S(t):=\langle u_S,M(t)\rangle$.
    \item The process $M_S$ is a continuous local martingale and
\begin{equation}
    \label{eq:detector-bracket}
    \langle M_S\rangle(t)
    =
    \int_0^t q_S(s,X(s))\,ds.
\end{equation}
Consequently,
\begin{equation}
    \label{eq:detector-bracket-stopped}
    \langle M_S(\cdot\wedge\theta)\rangle(t)
    =
    \int_0^{t\wedge\theta} q_S(s,X(s))\,ds.
\end{equation}

\end{enumerate}
\end{lemma}

\begin{proof}
Since every $\beta\in S\subset R_+$ is nonnegative on $\Cc$, we have
\begin{equation*}
    \tau_S(x)
    =
    \sum_{\beta\in S}\langle x,\beta\rangle
    \ge0,
    \qquad x\in\Cc.
\end{equation*}
This proves the first assertion in \textup{(i)}. Moreover, all summands in the last display are nonnegative, so $\tau_S(x)=0$ holds if and only if every term $\langle x,\beta\rangle$, $\beta\in S$, is zero. Since every point $x\in F_S$ satisfies $\langle x,\beta\rangle=0$ for $\beta\in S$, the detector vanishes on $F_S$. On the localized set
$\mathcal U_{S,\eta,r}$, the roots outside $S$ are strictly positive, so $\tau_S(x)=0$ is equivalent to $x\in F_S$. This proves \textup{(i)}. Assertion \textup{(ii)} is immediate from the definition $\tau_S(x) = \sum_{\gamma\in S}\langle x,\gamma\rangle$ and every term in this sum is nonnegative on $\Cc$.

For \textup{(iii)}, let $R_S:=R\cap \operatorname{span}S$. Since $S=S(y)$ for some $y\in\Cb$, the point $y$ is orthogonal to every root in $S$, hence to the whole space $\operatorname{span}S$. Therefore $S=R_S\cap R_+$ so $S$ is precisely the set of positive roots of the subsystem $R_S$. Let $\Delta_S=\{\alpha_1,\dots,\alpha_m\}$ be the corresponding simple system, and set
    \begin{equation*}
        \rho_S:=\frac12\sum_{\gamma\in S}\gamma.
    \end{equation*}
Then $u_S=2\rho_S$. We first show that $\langle u_S,\alpha\rangle>0$, $\alpha\in\Delta_S$. To see that let us fix $\alpha\in\Delta_S$. The reflection $s_\alpha$ preserves the set $S\setminus\{\alpha\}$, and $s_\alpha(\alpha)=-\alpha$. Hence
    \begin{equation*}
        s_\alpha(2\rho_S) = -\alpha+\sum_{\gamma\in S\setminus\{\alpha\}}s_\alpha(\gamma) = -\alpha+\sum_{\gamma\in S\setminus\{\alpha\}}\gamma = 2\rho_S-2\alpha.
    \end{equation*}
On the other hand,
    \begin{equation*}
        s_\alpha(v)=v-\frac{2\langle v,\alpha\rangle}{\langle \alpha,\alpha\rangle}\alpha,\qquad \frac{2\langle 2\rho_S,\alpha\rangle}{\langle \alpha,\alpha\rangle}=2.
    \end{equation*}
Therefore
    \begin{equation*}
        \langle u_S,\alpha\rangle = \langle 2\rho_S,\alpha\rangle = \langle \alpha,\alpha\rangle > 0.
    \end{equation*}
Now let $\beta\in S$. Since $\beta$ is a positive root of the subsystem $R_S$, it can be written as
    \begin{equation*}
        \beta=\sum_{j=1}^m c_j\alpha_j, \qquad c_j\ge0,
    \end{equation*}
with at least one coefficient $c_j$ strictly positive. Consequently,
    \begin{equation*}
        \langle u_S,\beta\rangle = \sum_{j=1}^m c_j\langle u_S,\alpha_j\rangle > 0.
    \end{equation*}
This proves \textup{(iii)}. Finally, to prove \textup{(iv)}, write
    \begin{equation*}
        u_{S,j} = \sum_{\beta\in S}\beta_j.
    \end{equation*}
By the Cauchy-Schwarz inequality,
    \begin{equation*}
        u_{S,j}^2 = \left(\sum_{\beta\in S}\beta_j\right)^2 \le |S|\sum_{\beta\in S}\beta_j^2.
    \end{equation*}
Multiplying by $\sigma_j^2(t,x)$ and summing over $j=1,\ldots,N$, we obtain
    \begin{align*}
        q_S(t,x) &= \sum_{j=1}^N u_{S,j}^2\,\sigma_j^2(t,x)  \le |S| \sum_{j=1}^N \sum_{\beta\in S}\beta_j^2\,\sigma_j^2(t,x)  = |S| \sum_{\beta\in S} a_\beta(t,x).
    \end{align*}
This proves \textup{(iv)}.

\medskip

To deal with the remaining part note that $\tau_S$ is linear, applying it to the generalized weak equation for $X$ produces no It\^o correction. We obtain
    \begin{align*}
        \tau_S(X(t)) &= \tau_S(x_0) + M_S(t) + \int_0^t \langle u_S,b(s,X(s))\rangle\,ds + \sum_{\beta\in R_+}\langle u_S,\beta\rangle A^\beta(t).
    \end{align*}
Stopping this identity at $\theta$, with the properties stated before,  gives
    \begin{align*}
        \tau_S\bigl(X(t\wedge\theta)\bigr) &= \tau_S(x_0) + M_S(t\wedge\theta) + \int_0^{t\wedge\theta}\langle u_S,b(s,X(s))\rangle\,ds + \sum_{\beta\in R_+}\langle u_S,\beta\rangle A^\beta(t\wedge\theta).
    \end{align*}
We now split the roots into those in $S$ and those outside $S$. On the interval $[0,\theta]$, every root $\beta\in R_+\setminus S$ satisfies $\langle X(s),\beta\rangle\ge\eta$. Hence, by Proposition~\ref{prop:deg:interior-ident-integrability}, the boundary measure $d\Lambda^\beta$ carries no mass on $[0,\theta]$, and the indicator $\mathbf 1_{\{\langle X(s),\beta\rangle>0\}}$ is identically equal to one there. Therefore
\begin{equation*}
    A^\beta(t\wedge\theta)
    =
    \int_0^{t\wedge\theta}
    \frac{k_\beta(s,X(s))}
         {\langle X(s),\beta\rangle}\,ds,
    \qquad \beta\in R_+\setminus S.
\end{equation*}
Consequently, the contribution of the roots outside $S$ is
    \begin{equation*}
        \int_0^{t\wedge\theta} \sum_{\beta\in R_+\setminus S} \frac{k_\beta(s,X(s))\langle u_S,\beta\rangle}{\langle X(s),\beta\rangle}\,ds,
    \end{equation*}
which, together with the ordinary drift $\langle u_S,b(s,X(s))\rangle\,ds$, is exactly $B_S(s,X(s))\,ds$. The remaining contribution comes from the roots in $S$. For the roots $\beta\in S$, we use the decomposition from
Proposition~\ref{prop:deg:interior-ident-integrability}. Thus
\begin{align*}
    A^\beta(t\wedge\theta)
    &=
    \int_0^{t\wedge\theta}
    \frac{k_\beta(s,X(s))}
         {\langle X(s),\beta\rangle}
    \mathbf 1_{\{\langle X(s),\beta\rangle>0\}}\,ds
    +
    \Lambda^\beta(t\wedge\theta).
\end{align*}
Therefore the contribution of the roots in $S$ is
\begin{align*}
    &\sum_{\beta\in S}
    \langle u_S,\beta\rangle
    \int_0^{t\wedge\theta}
    \frac{k_\beta(s,X(s))}
         {\langle X(s),\beta\rangle}
    \mathbf 1_{\{\langle X(s),\beta\rangle>0\}}\,ds        
    +
    \sum_{\beta\in S}
    \langle u_S,\beta\rangle
    \Lambda^\beta(t\wedge\theta).
\end{align*}
Combining the contribution of the roots in $R_+\setminus S$ with the ordinary drift gives the term
\begin{equation*}
    \int_0^{t\wedge\theta} B_S(s,X(s))\,ds.
\end{equation*}
Combining this with the preceding decomposition of the roots in $S$ proves \eqref{eq:detector-equation}. Finally,
    \begin{equation*}
        M_S(t)=\sum_{j=1}^N u_{S,j}M_j(t).
    \end{equation*}
Using the diagonal bracket identity from Proposition~\ref{prop:deg:skorokhod},
    \begin{equation*}
        \langle M_i,M_j\rangle(t) = \delta_{ij}\int_0^t \sigma_i^2(s,X(s))\,ds,
    \end{equation*}
we obtain
    \begin{align*}
        \langle M_S\rangle(t) &= \sum_{j=1}^N u_{S,j}^2\,\langle M_j\rangle(t) = \int_0^t \sum_{j=1}^N u_{S,j}^2\,\sigma_j^2(s,X(s))\,ds,
    \end{align*}
which proves \eqref{eq:detector-bracket}.
\end{proof}

The next step is to show that the canonical detector has zero local time at the level zero. This is the analytic input needed for the Tanaka argument in the next subsection. The proof rests on two ingredients: first, the detector bracket is controlled by the singular drifts corresponding to roots in the face; second, individual root walls carry no quadratic-variation mass. This consequence will be proved as the first step in the proof below. Let again 
\begin{equation*}
    \left(
        X,M,\bigl(A^\alpha\bigr)_{\alpha\in R_+}
    \right)
\end{equation*}
be the generalized weak limit from Proposition~\ref{prop:deg:skorokhod}, with the decomposition from Proposition~\ref{prop:deg:interior-ident-integrability}. For later use we also introduce non-canonical detector directions. If $v\in\operatorname{span}S$, set
    \begin{equation*}
        \tau_v(x):=\langle v,x\rangle, \qquad q_{S,v}(t,x):=\sum_{j=1}^N v_j^2\sigma_j^2(t,x),
    \end{equation*}
and, on localized neighborhoods of $F_S$,
    \begin{equation*}
        B_{S,v}(t,x) := \langle v,b(t,x)\rangle + \sum_{\beta\in R_+\setminus S} \frac{k_\beta(t,x)\langle v,\beta\rangle}{\langle x,\beta\rangle}.
    \end{equation*}
Thus $q_{S,u_S}=q_S$ and $B_{S,u_S}=B_S$. Moreover, if $v=u+u_S$, with $u\in\mathfrak D(S)$, then $B_{S,v}=B_{S,u}+B_S$.

\begin{proposition}[Vanishing local time of face detectors]
    \label{prop:deg:detector-local-time}
    Assume \ref{ass:k:dom}. Let $S\in\mathfrak S$, and let $T,r,\eta>0$. Let $\theta$ be a stopping time such that
        \begin{equation*}
            \theta\le T, \qquad \theta<\zeta \quad \text{a.s.}, \qquad X(s)\in\mathcal U_{S,\eta,r}, \quad 0\le s\le\theta .
        \end{equation*}
    Let
        \begin{equation*}
            v\in \{u_S\}\cup\{u+u_S:\ u\in\mathfrak D(S)\}.
        \end{equation*}
    Then, almost surely,
        \begin{equation}
            \label{eq:face-detector-local-time-zero}
            L_{t\wedge\theta}^0\bigl(\tau_v(X)\bigr)=0, \qquad 0\le t\le T .
        \end{equation}
In particular, the conclusion holds for the canonical detector $v=u_S$.
\end{proposition}

\begin{proof}
We first record two elementary facts which will be used for the chosen direction $v$. Since $v\in\operatorname{span}S$, there exist constants $c_\beta$, $\beta\in S$, such that
    \begin{equation*}
        v=\sum_{\beta\in S}c_\beta\beta .
    \end{equation*}
Hence, by the Cauchy-Schwarz inequality, there exists a finite constant
$C_v$, depending only on $v$ and $S$, such that
    \begin{equation}
        \label{eq:q-v-control-by-root-brackets}
        q_{S,v}(t,x) \le C_v\sum_{\beta\in S}a_\beta(t,x), \qquad (t,x)\in[0,\infty)\times\Cc .
    \end{equation}
Indeed, writing $v_j=\sum_{\beta\in S}c_\beta\beta_j$, we get
    \begin{equation*}
        v_j^2 \le C_v\sum_{\beta\in S}\beta_j^2,
    \end{equation*}
and multiplication by $\sigma_j^2(t,x)$, followed by summation over $j$, gives \eqref{eq:q-v-control-by-root-brackets}. Second, $\tau_v$ is nonnegative on $\Cc$. If $v=u_S$, this follows from Lemma~\ref{lem:deg:canonical-detector}. If $v=u+u_S$, with $u\in\mathfrak D(S)$, then
    \begin{equation*}
        \tau_v=\tau_u+\tau_S,
    \end{equation*}
and both terms are nonnegative on $\Cc$. Moreover, in both cases,
    \begin{equation}
        \label{eq:tau-v-dominates-root-gaps}
        \tau_v(x)\ge \tau_S(x)\ge \langle x,\beta\rangle, \qquad x\in\Cc,\quad \beta\in S .
    \end{equation}
We next show that, for every $\beta\in R_+$,
    \begin{equation}
        \label{eq:no-bracket-mass-root-wall-general}
        \int_0^{t\wedge\theta} \mathbf 1_{\{\langle X(s),\beta\rangle=0\}} a_\beta(s,X(s))\,ds = 0 \qquad\text{a.s.}
    \end{equation}
Set $Z^\beta(t):=\langle X(t\wedge\theta),\beta\rangle$. This is a continuous semimartingale whose martingale part has quadratic variation
    \begin{equation*}
        \int_0^{t\wedge\theta}a_\beta(s,X(s))\,ds .
    \end{equation*}
The occupation-time formula applied to $Z^\beta$ gives
    \begin{equation*}
        \int_0^{t\wedge\theta} \mathbf 1_{\{Z^\beta(s)=0\}} a_\beta(s,X(s))\,ds = \int_{\mathbb R}\mathbf 1_{\{y=0\}}L_t^y(Z^\beta)\,dy = 0,
    \end{equation*}
which proves \eqref{eq:no-bracket-mass-root-wall-general}. Let $M_v(t):=\langle v,M(t)\rangle$. Using the diagonal bracket identity from Proposition~\ref{prop:deg:skorokhod}, we have
    \begin{equation*}
        \langle M_v\rangle(t) = \int_0^t q_{S,v}(s,X(s))\,ds .
    \end{equation*}
Therefore, by \eqref{eq:q-v-control-by-root-brackets},
    \begin{align*}
        \int_0^{t\wedge\theta} \mathbf 1_{\{\tau_v(X(s))>0\}} \frac{d\langle M_v\rangle(s)}{\tau_v(X(s))} &\le C_v \sum_{\beta\in S} \int_0^{t\wedge\theta} \mathbf 1_{\{\tau_v(X(s))>0\}} \frac{a_\beta(s,X(s))}{\tau_v(X(s))}\,ds .
    \end{align*}
By \eqref{eq:tau-v-dominates-root-gaps}, on the set $\{\tau_v(X(s))>0,\ \langle X(s),\beta\rangle>0\}$, we have
    \begin{equation*}
        \frac1{\tau_v(X(s))} \le \frac1{\langle X(s),\beta\rangle}.
    \end{equation*}
Moreover, by \eqref{eq:no-bracket-mass-root-wall-general}, the measure $a_\beta(s,X(s))\,ds$ does not charge the set
$\{\langle X(s),\beta\rangle=0\}$. Hence
    \begin{align}
       \label{eq:detector-energy-general}
        \int_0^{t\wedge\theta} \mathbf 1_{\{\tau_v(X(s))>0\}} \frac{d\langle M_v\rangle(s)}{\tau_v(X(s))} &\le C_v \sum_{\beta\in S}\int_0^{t\wedge\theta}\frac{a_\beta(s,X(s))}{\langle X(s),\beta\rangle} \mathbf 1_{\{\langle X(s),\beta\rangle>0\}}\,ds .
    \end{align}
We claim that the right-hand side is finite almost surely. Since $0\le s\le\theta\le T$ and $|X(s)|\le r$ on the localized interval, Assumption~\ref{ass:k:dom} gives constants $\varepsilon=\varepsilon(T,r)>0$ and $\gamma=\gamma(T,r)>0$ such that
    \begin{equation*}
        0\le \langle x,\beta\rangle\le\varepsilon \quad\Longrightarrow\quad a_\beta(s,x)\le \gamma^{-1}k_\beta(s,x),
    \end{equation*}
for $0\le s\le T$, $|x|\le r$, and $\beta\in R_+$. For fixed $\beta\in S$, split
    \begin{equation*}
        \int_0^{t\wedge\theta} \frac{a_\beta(s,X(s))}{\langle X(s),\beta\rangle} \mathbf 1_{\{\langle X(s),\beta\rangle>0\}}\,ds = I_\beta^{\mathrm{small}}+I_\beta^{\mathrm{large}},
    \end{equation*}
where $I_\beta^{\mathrm{small}}$ is the contribution of $0<\langle X(s),\beta\rangle\le\varepsilon$, and $I_\beta^{\mathrm{large}}$ is the contribution of $\langle X(s),\beta\rangle>\varepsilon$.

On the small-gap region,
    \begin{equation*}
        I_\beta^{\mathrm{small}} \le \gamma^{-1} \int_0^{t\wedge\theta} \frac{k_\beta(s,X(s))}{\langle X(s),\beta\rangle} \mathbf 1_{\{\langle X(s),\beta\rangle>0\}}\,ds .
    \end{equation*}
Let $\chi_{r+1} := T\wedge\inf\{u\in[0,\zeta): |X(u)|\ge r+1\}$ Since $X(s)\in\mathcal U_{S,\eta,r}$ for $0\le s\le\theta$, we have $t\wedge\theta\le\chi_{r+1}$. Hence
    \begin{equation*}
        I_\beta^{\mathrm{small}} \le \gamma^{-1} \int_0^{\chi_{r+1}} \frac{k_\beta(s,X(s))}{\langle X(s),\beta\rangle} \mathbf 1_{\{\langle X(s),\beta\rangle>0\}}\,ds <\infty \qquad\text{a.s.}
    \end{equation*}
by Proposition~\ref{prop:deg:interior-ident-integrability}. On the large-gap region, continuity of $\sigma$ and compactness give
    \begin{equation*}
        I_\beta^{\mathrm{large}} \le \frac{T}{\varepsilon} \sup_{\substack{0\le s\le T\\ |x|\le r}} a_\beta(s,x) <\infty .
    \end{equation*}
Thus the right-hand side of \eqref{eq:detector-energy-general} is finite almost surely, and consequently
    \begin{equation}
        \label{eq:finite-energy-general-detector}
        \int_0^{t\wedge\theta} \mathbf 1_{\{\tau_v(X(s))>0\}} \frac{d\langle M_v\rangle(s)}{\tau_v(X(s))} < \infty \qquad\text{a.s.}
    \end{equation}
Set
\begin{equation*}
    Y(s):=\tau_v(X(s)),
    \qquad 0\le s\le t\wedge\theta .
\end{equation*}
Then $Y$ is a nonnegative continuous semimartingale and its martingale part
is $M_v$. For $n\ge1$, the occupation-time formula applied to
$a\mapsto \mathbf 1_{\{a>1/n\}}/a$ gives
\begin{equation*}
    \int_{(0,\infty)}
    \frac{\mathbf 1_{\{a>1/n\}}}{a}
    L_{t\wedge\theta}^a(Y)\,da
    =
    \int_0^{t\wedge\theta}
    \mathbf 1_{\{Y(s)>1/n\}}
    \frac{d\langle M_v\rangle(s)}{Y(s)}.
\end{equation*}
Letting $n\to\infty$ and using monotone convergence yields
\begin{equation*}
    \int_{(0,\infty)}
    \frac{L_{t\wedge\theta}^a(Y)}{a}\,da
    =
    \int_0^{t\wedge\theta}
    \mathbf 1_{\{Y(s)>0\}}
    \frac{d\langle M_v\rangle(s)}{Y(s)}
    <
    \infty
    \qquad\text{a.s.}
\end{equation*}
By right-continuity of $a\mapsto L_{t\wedge\theta}^a(Y)$ at $a=0$, this
forces
\begin{equation*}
    L_{t\wedge\theta}^0(Y)=0 .
\end{equation*}
Since $Y=\tau_v(X)$, this proves \eqref{eq:face-detector-local-time-zero}
for fixed $t$. Applying the argument to rational $t\in[0,T]$ and using
monotonicity in $t$, we obtain the assertion simultaneously for all
$0\le t\le T$ on one probability-one event.
\end{proof}


\subsection{Elimination of hidden boundary measures}
We now use the canonical face detectors to eliminate the boundary measures $d\Lambda^\alpha$. The local detector argument on a fixed face combines the vanishing local time from Proposition~\ref{prop:deg:detector-local-time} with Tanaka's formula. It shows that $d\Lambda^\beta$ gives no mass to the portion of the path lying in that face. Since every boundary point belongs to one of the relatively open faces $F_S$, a covering argument then yields the global conclusion.

\begin{proposition}[Elimination and zero occupation on one face]
\label{prop:deg:one-face}
Assume \ref{ass:k:dom} and \ref{ass:face:sign}. Let $S\in\mathfrak S$, and let
$T,r,\eta>0$. Let $\theta$ be a stopping time such that
\begin{equation*}
    \theta\le T,
    \qquad
    \theta<\zeta \quad \text{a.s.},
    \qquad
    X(s)\in\mathcal U_{S,\eta,r},
    \quad 0\le s\le\theta .
\end{equation*}
Then, almost surely, for every $\beta\in S$ and every $t\ge0$,
\begin{equation}
    \label{eq:hidden-zero-on-one-face}
    \int_0^{t\wedge\theta}
    \mathbf 1_{\{\tau_S(X(s))=0\}}\,d\Lambda^\beta(s)
    =
    0 .
\end{equation}

If, in addition, Assumption~\ref{ass:nonsticky} holds, then, almost surely,
for every $t\ge0$,
\begin{equation}
    \label{eq:zero-occupation-one-face}
    \int_0^{t\wedge\theta}
    \mathbf 1_{\{\tau_S(X(s))=0\}}\,ds
    =
    0 .
\end{equation}
\end{proposition}

\begin{proof}
We first prove \eqref{eq:hidden-zero-on-one-face}. Set
\begin{equation*}
    Y(t):=\tau_S(X(t\wedge\theta)).
\end{equation*}
By Lemma~\ref{lem:deg:canonical-detector}, $Y$ is a nonnegative continuous
semimartingale. Its martingale part is $M_S(\cdot\wedge\theta)$, and its
finite-variation part is
\begin{align*}
    V(t)
    &=
    \int_0^{t\wedge\theta} B_S(s,X(s))\,ds \\
    &\quad
    +
    \sum_{\beta\in S}
    \langle u_S,\beta\rangle
    \int_0^{t\wedge\theta}
    \frac{k_\beta(s,X(s))}
         {\langle X(s),\beta\rangle}
    \mathbf 1_{\{\langle X(s),\beta\rangle>0\}}\,ds \\
    &\quad
    +
    \sum_{\beta\in S}
    \langle u_S,\beta\rangle
    \Lambda^\beta(t\wedge\theta).
\end{align*}
By Proposition~\ref{prop:deg:detector-local-time}, applied with $v=u_S$,
\begin{equation*}
    L_t^0(Y)=0,
    \qquad t\ge0 .
\end{equation*}
Since $Y\ge0$, Tanaka's formula gives
\begin{equation*}
    Y(t)
    =
    Y(0)
    +
    \int_0^t \mathbf 1_{\{Y(s)>0\}}\,dY(s)
    +
    \frac12 L_t^0(Y).
\end{equation*}
Subtracting this from $Y(t)=Y(0)+\int_0^t dY(s)$, and using
$L_t^0(Y)=0$, we get
\begin{equation*}
    \int_0^t \mathbf 1_{\{Y(s)=0\}}\,dY(s)=0 .
\end{equation*}
Writing $Y=Y(0)+N+V$, where $N=M_S(\cdot\wedge\theta)$, gives
\begin{equation*}
    \int_0^t \mathbf 1_{\{Y(s)=0\}}\,dN(s)
    +
    \int_0^t \mathbf 1_{\{Y(s)=0\}}\,dV(s)
    =
    0 .
\end{equation*}
The first term is a continuous local martingale with quadratic variation
\begin{equation*}
    \int_0^t \mathbf 1_{\{Y(s)=0\}}\,d\langle N\rangle(s).
\end{equation*}
By the occupation-time formula this quadratic variation is zero. Hence the
martingale term is identically zero, and therefore
\begin{equation*}
    \int_0^t \mathbf 1_{\{Y(s)=0\}}\,dV(s)=0 .
\end{equation*}
Since $dV$ is supported on $[0,\theta]$, we obtain
\begin{align}
    \label{eq:zero-level-canonical-detector}
    0
    &=
    \int_0^{t\wedge\theta}
    \mathbf 1_{\{\tau_S(X(s))=0\}}
    B_S(s,X(s))\,ds
    \nonumber\\
    &\quad
    +
    \sum_{\beta\in S}
    \langle u_S,\beta\rangle
    \int_0^{t\wedge\theta}
    \mathbf 1_{\{\tau_S(X(s))=0\}}
    \frac{k_\beta(s,X(s))}
         {\langle X(s),\beta\rangle}
    \mathbf 1_{\{\langle X(s),\beta\rangle>0\}}\,ds
    \nonumber\\
    &\quad
    +
    \sum_{\beta\in S}
    \langle u_S,\beta\rangle
    \int_0^{t\wedge\theta}
    \mathbf 1_{\{\tau_S(X(s))=0\}}\,d\Lambda^\beta(s).
\end{align}
On the set $\{\tau_S(X(s))=0\}$, since the path is localized in
$\mathcal U_{S,\eta,r}$, Lemma~\ref{lem:deg:canonical-detector}\textup{(i)}
implies $X(s)\in F_S$. Hence
\begin{equation*}
    \langle X(s),\beta\rangle=0,
    \qquad \beta\in S,
\end{equation*}
and the interior singular terms in
\eqref{eq:zero-level-canonical-detector} vanish. Therefore
\begin{align}
    \label{eq:zero-level-canonical-reduced}
    0
    &=
    \int_0^{t\wedge\theta}
    \mathbf 1_{\{\tau_S(X(s))=0\}}
    B_S(s,X(s))\,ds
    \nonumber\\
    &\quad
    +
    \sum_{\beta\in S}
    \langle u_S,\beta\rangle
    \int_0^{t\wedge\theta}
    \mathbf 1_{\{\tau_S(X(s))=0\}}\,d\Lambda^\beta(s).
\end{align}
By Assumption~\ref{ass:face:sign},
\begin{equation*}
    B_S(s,X(s))\ge0
    \qquad\text{on }\{\tau_S(X(s))=0\},
\end{equation*}
because this set lies in $F_S$. The measures $d\Lambda^\beta$ are
nonnegative, and
\begin{equation*}
    \langle u_S,\beta\rangle>0,
    \qquad \beta\in S,
\end{equation*}
by Lemma~\ref{lem:deg:canonical-detector}\textup{(iii)}. Thus all terms in
\eqref{eq:zero-level-canonical-reduced} are nonnegative. Since their sum is
zero, each term is zero. This proves \eqref{eq:hidden-zero-on-one-face}.

We now assume additionally \ref{ass:nonsticky} and prove
\eqref{eq:zero-occupation-one-face}. Define
\begin{equation*}
    Q_S(t,x)
    :=
    \max_{u\in\mathfrak D(S)}q_{S,u}(t,x),
    \qquad
    H_S(t,x)
    :=
    \max_{u\in\mathfrak D(S)}B_{S,u}(t,x).
\end{equation*}
On $F_S$, Assumption~\ref{ass:nonsticky} says that
\begin{equation*}
    Q_S(t,x)=0
    \quad\Longrightarrow\quad
    H_S(t,x)>0 .
\end{equation*}

First we show that the face has zero occupation on the set where
$Q_S>0$. Fix $u\in\mathfrak D(S)$. The process
$\tau_u(X(\cdot\wedge\theta))$ is a continuous semimartingale whose
martingale bracket density is $q_{S,u}(s,X(s))$. The occupation-time formula
gives
\begin{equation*}
    \int_0^{t\wedge\theta}
    \mathbf 1_{\{\tau_u(X(s))=0\}}
    q_{S,u}(s,X(s))\,ds
    =
    0 .
\end{equation*}
Since $X(s)\in F_S$ implies $\tau_u(X(s))=0$, and since on
$\mathcal U_{S,\eta,r}$ the condition $X(s)\in F_S$ is equivalent to
$\tau_S(X(s))=0$, we get
\begin{equation*}
    \int_0^{t\wedge\theta}
    \mathbf 1_{\{\tau_S(X(s))=0\}}
    q_{S,u}(s,X(s))\,ds
    =
    0 .
\end{equation*}
Therefore, for every $m\in\mathbb N$,
\begin{equation*}
    \int_0^{t\wedge\theta}
    \mathbf 1_{\{\tau_S(X(s))=0\}}
    \mathbf 1_{\{q_{S,u}(s,X(s))\ge 1/m\}}\,ds
    =
    0 .
\end{equation*}
Since $\mathfrak D(S)$ is finite, it follows that
\begin{equation}
    \label{eq:zero-occupation-q-positive}
    \int_0^{t\wedge\theta}
    \mathbf 1_{\{\tau_S(X(s))=0\}}
    \mathbf 1_{\{Q_S(s,X(s))>0\}}\,ds
    =
    0 .
\end{equation}

It remains to treat the fully degenerate part of the face:
\begin{equation*}
    \{\tau_S(X(s))=0,\ Q_S(s,X(s))=0\}.
\end{equation*}
By Assumption~\ref{ass:nonsticky}, this set is contained in
\begin{equation*}
    \bigcup_{u\in\mathfrak D(S)}
    \{ \tau_S(X(s))=0,\ B_{S,u}(s,X(s))>0 \}.
\end{equation*}
Since $\mathfrak D(S)$ is finite, it suffices to prove that for every
$u\in\mathfrak D(S)$,
\begin{equation}
    \label{eq:zero-occupation-positive-detector-drift}
    \int_0^{t\wedge\theta}
    \mathbf 1_{\{\tau_S(X(s))=0\}}
    \mathbf 1_{\{B_{S,u}(s,X(s))>0\}}\,ds
    =
    0 .
\end{equation}

Fix $u\in\mathfrak D(S)$, and set
\begin{equation*}
    v:=u+u_S .
\end{equation*}
Then
\begin{equation*}
    \tau_v=\tau_u+\tau_S .
\end{equation*}
Since both $\tau_u$ and $\tau_S$ are nonnegative on $\Cc$, and since
$\tau_S(X(s))=0$ is equivalent to $X(s)\in F_S$ on the localized interval,
we have
\begin{equation*}
    \tau_v(X(s))=0
    \qquad\Longleftrightarrow\qquad
    \tau_S(X(s))=0,
    \qquad 0\le s\le\theta .
\end{equation*}
By Proposition~\ref{prop:deg:detector-local-time}, applied to this $v$,
\begin{equation*}
    L_t^0\bigl(\tau_v(X(\cdot\wedge\theta))\bigr)=0 .
\end{equation*}
Repeating the Tanaka argument above for
\begin{equation*}
    Y_v(t):=\tau_v(X(t\wedge\theta)),
\end{equation*}
we obtain
\begin{equation}
    \label{eq:zero-level-v-detector}
    \int_0^{t\wedge\theta}
    \mathbf 1_{\{\tau_S(X(s))=0\}}
    B_{S,v}(s,X(s))\,ds
    +
    \sum_{\beta\in S}
    \langle v,\beta\rangle
    \int_0^{t\wedge\theta}
    \mathbf 1_{\{\tau_S(X(s))=0\}}\,d\Lambda^\beta(s)
    =
    0 .
\end{equation}
Indeed, on the set $\{\tau_S(X(s))=0\}$, the interior singular terms with
$\beta\in S$ vanish because
$\mathbf 1_{\{\langle X(s),\beta\rangle>0\}}=0$.

Now
\begin{equation*}
    B_{S,v}=B_{S,u}+B_S .
\end{equation*}
Both $u$ and $u_S$ belong to $\mathfrak D(S)$, and therefore
Assumption~\ref{ass:face:sign} gives
\begin{equation*}
    B_{S,u}\ge0,
    \qquad
    B_S\ge0,
    \qquad
    \text{on }F_S .
\end{equation*}
Thus
\begin{equation*}
    B_{S,v}\ge B_{S,u}\ge0
    \qquad\text{on }F_S .
\end{equation*}
Also, for $\beta\in S$,
\begin{equation*}
    \langle v,\beta\rangle
    =
    \langle u,\beta\rangle+\langle u_S,\beta\rangle
    >
    0,
\end{equation*}
because $\langle u,\beta\rangle\ge0$ by admissibility and
$\langle u_S,\beta\rangle>0$ by Lemma~\ref{lem:deg:canonical-detector}.
Therefore all terms in \eqref{eq:zero-level-v-detector} are nonnegative. Since
their sum is zero,
\begin{equation*}
    \int_0^{t\wedge\theta}
    \mathbf 1_{\{\tau_S(X(s))=0\}}
    B_{S,v}(s,X(s))\,ds
    =
    0 .
\end{equation*}
As $B_{S,v}\ge B_{S,u}\ge0$ on the face, we get
\begin{equation*}
    \int_0^{t\wedge\theta}
    \mathbf 1_{\{\tau_S(X(s))=0\}}
    B_{S,u}(s,X(s))\,ds
    =
    0 .
\end{equation*}
Consequently, for every $m\in\mathbb N$,
\begin{equation*}
    \frac1m
    \int_0^{t\wedge\theta}
    \mathbf 1_{\{\tau_S(X(s))=0\}}
    \mathbf 1_{\{B_{S,u}(s,X(s))\ge 1/m\}}\,ds
    \le
    0 .
\end{equation*}
Letting $m$ range over $\mathbb N$, we obtain
\eqref{eq:zero-occupation-positive-detector-drift}.

Combining \eqref{eq:zero-occupation-q-positive} with the preceding treatment
of the fully degenerate part yields
\begin{equation*}
    \int_0^{t\wedge\theta}
    \mathbf 1_{\{\tau_S(X(s))=0\}}\,ds
    =
    0 .
\end{equation*}
This proves \eqref{eq:zero-occupation-one-face} for fixed $t$. Applying the
argument to rational $t$ and using monotonicity gives the statement
simultaneously for all $t\ge0$ on one probability-one event.
\end{proof}

The same conclusions holds on any deterministic interval $[a,b]\subset[0,\zeta)$ on which $X(s)\in\mathcal U_{S,\eta,r}$. Indeed, the proof above is local in time and applies verbatim to the stopped shifted semimartingale $s\mapsto X((a+s)\wedge b)$. 

We can now eliminate the hidden boundary measures globally and complete the proof of the degenerate existence theorem.

\begin{proof}[Proof of Theorem~\ref{thm:degenerate-existence}]
Let
\begin{equation*}
    \left(
        X,M,\bigl(A^\alpha\bigr)_{\alpha\in R_+}
    \right)
\end{equation*}
be the generalized weak limit constructed in
Proposition~\ref{prop:deg:skorokhod}, and let
$(\Lambda^\alpha)_{\alpha\in R_+}$ be the continuous increasing processes
from Proposition~\ref{prop:deg:interior-ident-integrability}. Thus, for every
$\alpha\in R_+$,
\begin{equation*}
    A^\alpha(t)
    =
    \int_0^t
    \frac{k_\alpha(s,X(s))}
         {\langle X(s),\alpha\rangle}
    \mathbf 1_{\{\langle X(s),\alpha\rangle>0\}}\,ds
    +
    \Lambda^\alpha(t),
    \qquad t<\zeta,
\end{equation*}
and the Stieltjes measure $d\Lambda^\alpha$ is supported on the wall
\begin{equation*}
    Z_\alpha
    :=
    \{t\in[0,\zeta):\langle X(t),\alpha\rangle=0\}.
\end{equation*}
It remains to prove that
\begin{equation*}
    \Lambda^\alpha(t)=0,
    \qquad \alpha\in R_+,\quad 0\le t<\zeta .
\end{equation*}

Fix $\alpha\in R_+$. For every $S\in\mathfrak S$ such that
$\alpha\in S$, set
\begin{equation*}
    E_S
    :=
    \{t\in[0,\zeta):X(t)\in F_S\}.
\end{equation*}
Then
\begin{equation*}
    Z_\alpha
    =
    \bigcup_{\substack{S\in\mathfrak S\\ \alpha\in S}}E_S .
\end{equation*}
Indeed, if $\langle X(t),\alpha\rangle=0$, then
$S(X(t))\in\mathfrak S$, $\alpha\in S(X(t))$, and
$X(t)\in F_{S(X(t))}$. Since $\mathfrak S$ is finite, it is enough to show
that
\begin{equation*}
    d\Lambda^\alpha(E_S)=0
\end{equation*}
for every $S\in\mathfrak S$ with $\alpha\in S$.

Fix such an $S$, and let $t_0\in E_S$. Then $X(t_0)\in F_S$. Since the
roots in $R_+\setminus S$ are strictly positive at $X(t_0)$, continuity of
$X$ implies that there exist $r,\eta>0$ and an open interval
$I\subset[0,\zeta)$ containing $t_0$ such that
\begin{equation*}
    X(s)\in\mathcal U_{S,\eta,r},
    \qquad s\in I .
\end{equation*}
Choose rational numbers $a<b$ such that
\begin{equation*}
    t_0\in(a,b)\subset I .
\end{equation*}
After enlarging $r$ and shrinking $\eta$, we may take $r,\eta$ rational.
On $[a,b]$, the path stays in $\mathcal U_{S,\eta,r}$. By the interval
version of Proposition~\ref{prop:deg:one-face}, applied to the hidden-measure
conclusion \eqref{eq:hidden-zero-on-one-face},
\begin{equation*}
    \int_a^b
    \mathbf 1_{\{\tau_S(X(s))=0\}}\,d\Lambda^\alpha(s)
    =
    0 .
\end{equation*}
Since $X(s)\in\mathcal U_{S,\eta,r}$ on $[a,b]$,
Lemma~\ref{lem:deg:canonical-detector}\textup{(i)} gives
\begin{equation*}
    \tau_S(X(s))=0
    \qquad\Longleftrightarrow\qquad
    X(s)\in F_S .
\end{equation*}
Therefore
\begin{equation*}
    \int_a^b \mathbf 1_{E_S}(s)\,d\Lambda^\alpha(s)=0 .
\end{equation*}
The countable family of rational intervals $(a,b)$ and rational parameters
$(r,\eta)$ for which $X(s)\in\mathcal U_{S,\eta,r}$ on $[a,b]$ covers
$E_S$. Hence
\begin{equation*}
    d\Lambda^\alpha(E_S)=0 .
\end{equation*}
Since this holds for every $S\in\mathfrak S$ with $\alpha\in S$, we obtain
\begin{equation*}
    d\Lambda^\alpha(Z_\alpha)=0 .
\end{equation*}
But $d\Lambda^\alpha$ is supported on $Z_\alpha$. Hence
\begin{equation*}
    d\Lambda^\alpha=0
    \qquad\text{on }[0,\zeta),
\end{equation*}
and therefore
\begin{equation*}
    \Lambda^\alpha(t)=0,
    \qquad 0\le t<\zeta .
\end{equation*}

Thus, for every $\alpha\in R_+$,
\begin{equation*}
    A^\alpha(t)
    =
    \int_0^t
    \frac{k_\alpha(s,X(s))}
         {\langle X(s),\alpha\rangle}
    \mathbf 1_{\{\langle X(s),\alpha\rangle>0\}}\,ds,
    \qquad t<\zeta .
\end{equation*}
Substituting this identity into the generalized weak equation from
Proposition~\ref{prop:deg:skorokhod}, we get
\begin{equation*}
    X(t)
    =
    x_0
    +
    M(t)
    +
    \int_0^t b(s,X(s))\,ds
    +
    \sum_{\alpha\in R_+}\alpha
    \int_0^t
    \frac{k_\alpha(s,X(s))}
         {\langle X(s),\alpha\rangle}
    \mathbf 1_{\{\langle X(s),\alpha\rangle>0\}}\,ds,
    \qquad t<\zeta .
\end{equation*}
By Proposition~\ref{prop:deg:skorokhod}, possibly after enlarging the
probability space, there exists an $N$-dimensional Brownian motion $B$ such
that
\begin{equation*}
    M(t)
    =
    \int_0^t \sigma(s,X(s))\,dB(s),
    \qquad t<\zeta .
\end{equation*}
Therefore $X$ satisfies \eqref{eq:degenerate-existence-indicator}. Finally,
the singular integrals are finite on compact subintervals of $[0,\zeta)$ by
Proposition~\ref{prop:deg:interior-ident-integrability}. This completes the
proof.
\end{proof}

\begin{proof}[Proof of Theorem~\ref{thm:degenerate-remove-indicator}]
By Theorem~\ref{thm:degenerate-existence}, there exist a filtered probability
space, an $N$-dimensional Brownian motion $B$, and a continuous
$\Cc$-valued semimartingale $X$, defined up to its lifetime
$\zeta:=\life{X}>0$, such that $X(0)=x_0$ and
\begin{equation*}
    X(t)
    =
    x_0
    +
    \int_0^t \sigma(s,X(s))\,dB(s)
    +
    \int_0^t b(s,X(s))\,ds
    +
    \sum_{\alpha\in R_+}\alpha
    \int_0^t
    \frac{k_\alpha(s,X(s))}
         {\langle X(s),\alpha\rangle}
    \mathbf 1_{\{\langle X(s),\alpha\rangle>0\}}\,ds
\end{equation*}
for every $t<\zeta$. Moreover, all singular integrals in this indicator
formulation are finite on compact subintervals of $[0,\zeta)$.

It remains to prove that $X$ spends zero Lebesgue time on every root wall.
First we prove the corresponding statement for each relatively open face
$F_S$. Fix $S\in\mathfrak S$. We claim that
\begin{equation}
    \label{eq:zero-occupation-face-global}
    \int_0^t \mathbf 1_{\{X(s)\in F_S\}}\,ds
    =
    0,
    \qquad t<\zeta,
    \qquad \text{a.s.}
\end{equation}
Let $t_0<\zeta$ be such that $X(t_0)\in F_S$. Since the roots in
$R_+\setminus S$ are strictly positive at $X(t_0)$, continuity of $X$
implies that there exist $r,\eta>0$ and an open interval
$I\subset[0,\zeta)$ containing $t_0$ such that
\begin{equation*}
    X(s)\in\mathcal U_{S,\eta,r},
    \qquad s\in I .
\end{equation*}
Choose rational numbers $a<b$ such that
\begin{equation*}
    t_0\in(a,b)\subset I .
\end{equation*}
After enlarging $r$ and shrinking $\eta$, we may assume that $r,\eta$ are
rational. On $[a,b]$, the path stays in $\mathcal U_{S,\eta,r}$. By the
interval version of Proposition~\ref{prop:deg:one-face}, applied to the
zero-occupation conclusion \eqref{eq:zero-occupation-one-face},
\begin{equation*}
    \int_a^b
    \mathbf 1_{\{\tau_S(X(s))=0\}}\,ds
    =
    0 .
\end{equation*}
Since $X(s)\in\mathcal U_{S,\eta,r}$ on $[a,b]$,
Lemma~\ref{lem:deg:canonical-detector}\textup{(i)} gives
\begin{equation*}
    \tau_S(X(s))=0
    \qquad\Longleftrightarrow\qquad
    X(s)\in F_S .
\end{equation*}
Thus
\begin{equation*}
    \int_a^b \mathbf 1_{\{X(s)\in F_S\}}\,ds=0 .
\end{equation*}
The countable family of such rational intervals covers
\begin{equation*}
    \{s<\zeta:X(s)\in F_S\}.
\end{equation*}
Therefore \eqref{eq:zero-occupation-face-global} holds.

Now fix $\alpha\in R_+$. The root wall inside the closed chamber is the finite
union of relatively open faces:
\begin{equation*}
    \{x\in\Cc:\langle x,\alpha\rangle=0\}
    =
    \bigcup_{\substack{S\in\mathfrak S\\ \alpha\in S}}F_S .
\end{equation*}
Using \eqref{eq:zero-occupation-face-global} and the finiteness of
$\mathfrak S$, we obtain
\begin{equation}
    \label{eq:zero-occupation-root-wall-final}
    \int_0^t
    \mathbf 1_{\{\langle X(s),\alpha\rangle=0\}}\,ds
    =
    0,
    \qquad
    \alpha\in R_+,\quad t<\zeta,
    \qquad \text{a.s.}
\end{equation}
Hence $X$ is non-sticky in the sense of \eqref{eq:zero:leb}.

It remains only to remove the indicators from the singular drift. By
\eqref{eq:zero-occupation-root-wall-final}, for every $\alpha\in R_+$ and
every $t<\zeta$,
\begin{equation*}
    \int_0^t
    \frac{k_\alpha(s,X(s))}
         {\langle X(s),\alpha\rangle}\,ds
    =
    \int_0^t
    \frac{k_\alpha(s,X(s))}
         {\langle X(s),\alpha\rangle}
    \mathbf 1_{\{\langle X(s),\alpha\rangle>0\}}\,ds,
\end{equation*}
where the value of the quotient on the wall is immaterial, because the wall is
visited only on a Lebesgue-null set of times. The right-hand side is finite on
compact subintervals of $[0,\zeta)$ by
Theorem~\ref{thm:degenerate-existence}. Therefore the singular integrals in
the equation without indicators are well defined and locally finite.

Substituting the preceding identity into the indicator equation gives
\begin{equation*}
    X(t)
    =
    x_0
    +
    \int_0^t \sigma(s,X(s))\,dB(s)
    +
    \int_0^t b(s,X(s))\,ds
    +
    \sum_{\alpha\in R_+}\alpha
    \int_0^t
    \frac{k_\alpha(s,X(s))}
         {\langle X(s),\alpha\rangle}\,ds,
    \qquad t<\zeta .
\end{equation*}
Thus $X$ satisfies \eqref{eq:degenerate-existence-without-indicator}. The
zero-occupation property \eqref{eq:zero-occupation-root-wall-final} is exactly
the asserted non-sticking boundary behavior. This completes the proof.
\end{proof}



\section{Uniqueness}
\label{sec:uniqueness}

This section proves Theorems~\ref{thm:pathwise-uniq-indicator-nonsticky} and \ref{thm:strong:existence:uniqueness}.  Throughout the section, non-sticky solutions are understood in the sense of \eqref{eq:zero:leb}. We use the convention
    \begin{align*}
        \operatorname{sgn}(u) &=
            \begin{cases}
                1, & u>0,\\
                0, & u=0,\\
                -1, & u<0.
            \end{cases}
    \end{align*}
For $\alpha\in R_+$ and $x\in\Cc$, define
    \begin{align*}
        q_\alpha(t,x) &:=
            \begin{cases}
                \dfrac{k_\alpha(t,x)}{\langle x,\alpha\rangle}, & \langle x,\alpha\rangle>0,\\ \\[1.2ex] 
                0, & \langle x,\alpha\rangle=0.
            \end{cases}
    \end{align*}
We then define the singular force on $\Cc$ by
    \begin{align*}
        G(t,x) &:= \sum_{\alpha\in R_+} \alpha\,q_\alpha(t,x).
    \end{align*}
For non-sticky solutions this agrees, as a Lebesgue-time integrand, with the singular force in \eqref{eq:main:SDE}. 

\begin{proof}[Proof of Theorem~\ref{thm:pathwise-uniq-indicator-nonsticky}]
Let $(X,B)$ and $(\widetilde X,B)$ be two non-sticky solutions as in the statement,
driven by the same Brownian motion and with the same initial condition.  Put $Z(t):=X(t)-\widetilde X(t)$. Fix $T,r>0$ and define
    \begin{align*}
        \tau_r &:= T\wedge\life{X}\wedge\life{\widetilde X} \wedge \inf\left\{t\ge0:\ |X(t)|\vee |\widetilde X(t)|\ge r\right\}.
    \end{align*}
We also localize the finite-variation terms.  Set
    \begin{align*}
        A^r(t) &:= \int_0^{t\wedge\tau_r} \sum_{i=1}^N \bigl(|b_i(s,X(s))|+|b_i(s,\widetilde X(s))|\bigr)\,ds\\ 
        &\quad + \sum_{\alpha\in R_+} \int_0^{t\wedge\tau_r}\bigl(|q_\alpha(s,X(s))|+|q_\alpha(s,\widetilde X(s))|\bigr)\,ds,
    \end{align*}
and, for $m\in\N$, $\tau_{r,m}=\tau_r\wedge\inf\left\{t\ge0:\ A^r(t)\ge m\right\}$. The singular integrals are finite on compact subintervals of the lifetimes, and therefore $\tau_{r,m}\uparrow\tau_r$, as $m\to\infty$. If some singular integral is interpreted improperly at the initial time, the arguments below are first applied on $[\varepsilon,t\wedge\tau_{r,m}]$ and then $\varepsilon\downarrow0$.

We prove the estimate separately in the two uniqueness regimes.

\medskip

\noindent
\emph{Step 1: the Yamada-Watanabe uniqueness setting.}
Assume that alternative~\textup{(a)} in Assumption~\ref{ass:uniq:sigma} and Assumption~\ref{ass:uniq:b} hold.  On $[0,\tau_r]$, for every $i=1,\ldots,N$,
    \begin{align*}
        d\langle Z_i\rangle_t &= |\sigma_i(t,X(t))-\sigma_i(t,\widetilde X(t))|^2\,dt \le \rho_{T,r}(|Z_i(t)|)\,dt.
    \end{align*}
By the one-dimensional Yamada-Watanabe local-time criterion,
    \begin{align}
        L_{t\wedge\tau_{r,m}}^0(Z_i) &= 0, \qquad 0\le t\le T, \qquad i=1,\ldots,N.
        \label{eq:uniq:local-time-zero}
    \end{align}
Applying Tanaka's formula to $Z_i$, summing over $i$, using \eqref{eq:uniq:local-time-zero}, and taking expectations, we obtain
    \begin{align}
        \E\sum_{i=1}^N |Z_i(t\wedge\tau_{r,m})| &= \E\int_0^{t\wedge\tau_{r,m}} \sum_{i=1}^N \operatorname{sgn}(Z_i(s)) \bigl(b_i(s,X(s))-b_i(s,\widetilde X(s))\bigr)\,ds \nonumber\\
            &\quad+ \E\int_0^{t\wedge\tau_{r,m}} \sum_{i=1}^N\operatorname{sgn}(Z_i(s))\bigl(G_i(s,X(s))-G_i(s,\widetilde X(s))\bigr)\,ds.
        \label{eq:uniq:tanaka}
    \end{align}
The stochastic integral has expectation zero because the martingale integrand is bounded after stopping at $\tau_r$. By Assumption~\ref{ass:uniq:b}, the first integral in \eqref{eq:uniq:tanaka} is bounded above by
    \begin{align}
        L_{T,r} \int_0^t \E\sum_{i=1}^N |Z_i(s\wedge\tau_{r,m})|\,ds.
        \label{eq:uniq:b-bound}
    \end{align}
We next show that the singular contribution is non-positive.  Since both processes
satisfy \eqref{eq:zero:leb} and $R_+$ is finite, for Lebesgue-a.e. $s$ before
$\tau_{r,m}$ both $X(s)$ and $\widetilde X(s)$ belong to $\C$.  At such times the
indicators in the singular terms are equal to one, and we may write the singular
drift as
\begin{equation*}
    G(s,x)
    =
    \sum_{\alpha\in R_+}
    \alpha\,
    \frac{k_\alpha(s,x)}{\langle x,\alpha\rangle}.
\end{equation*}
Therefore, by Assumption~\ref{ass:uniq:k}, applied with
$x=X(s)$ and $y=\widetilde X(s)$,
\begin{align}
    \sum_{i=1}^N
    \operatorname{sgn}(Z_i(s))
    \bigl(G_i(s,X(s))-G_i(s,\widetilde X(s))\bigr)
    \le 0 .
    \label{eq:uniq:singular-l1-nonpositive}
\end{align}
Thus the second integral in \eqref{eq:uniq:tanaka} is non-positive. Combining this with \eqref{eq:uniq:b-bound}, we obtain
    \begin{align*}
        \E\sum_{i=1}^N |Z_i(t\wedge\tau_{r,m})| &\le L_{T,r} \int_0^t\E\sum_{i=1}^N |Z_i(s\wedge\tau_{r,m})|\,ds.
    \end{align*}
Gronwall's lemma gives
\begin{align}
    \E\sum_{i=1}^N |Z_i(t\wedge\tau_{r,m})|
        &=
        0,
        \qquad
        0\le t\le T.
    \label{eq:uniq:yw-zero}
\end{align}

\medskip

\noindent
\emph{Step 2: the Lipschitz uniqueness setting.}
Assume now that alternative~\textup{(b)} in Assumption~\ref{ass:uniq:sigma} holds, together with Assumptions~\ref{ass:uniq:b:euclidean} and \ref{ass:uniq:k:dissipative}. It\^o's formula gives
    \begin{align*}
        \E |Z(t\wedge\tau_{r,m})|^2 &= \E\int_0^{t\wedge\tau_{r,m}} \sum_{i=1}^N|\sigma_i(s,X(s))-\sigma_i(s,\widetilde X(s))|^2\,ds \nonumber\\
            &\quad+ 2\E\int_0^{t\wedge\tau_{r,m}} \inner{Z(s)}{b(s,X(s))-b(s,\widetilde X(s))}\,ds \nonumber\\
            &\quad+ 2\E\int_0^{t\wedge\tau_{r,m}} \inner{Z(s)}{G(s,X(s))-G(s,\widetilde X(s))}\,ds.
    \end{align*}
The stochastic integral has expectation zero because the martingale integrand is bounded after stopping at $\tau_r$. For Lebesgue-a.e. $s$ before $\tau_{r,m}$, both $X(s)$ and $\widetilde X(s)$ belong to $\C$, because both solutions satisfy \eqref{eq:zero:leb} and $R_+$ is finite.  Hence Assumption~\ref{ass:uniq:k:dissipative} gives
    \begin{equation*}
        \inner{Z(s)}{G(s,X(s))-G(s,\widetilde X(s))} \le0.
    \end{equation*}
The local Lipschitz condition on $\sigma$ and Assumption~\ref{ass:uniq:b:euclidean} therefore imply that, for a finite constant $C_{T,r}$,
    \begin{align}
        \E |Z(t\wedge\tau_{r,m})|^2 &\le C_{T,r} \int_0^t \E |Z(s\wedge\tau_{r,m})|^2\,ds. 
        \label{eq:uniq:lip-gronwall}
    \end{align}
By Gronwall's lemma,
\begin{align}
    \E |Z(t\wedge\tau_{r,m})|^2
        &=
        0,
        \qquad
        0\le t\le T.
    \label{eq:uniq:lip-zero}
\end{align}

\medskip

In either uniqueness setting, \eqref{eq:uniq:yw-zero} or \eqref{eq:uniq:lip-zero}
implies that
\begin{align}
    X(t\wedge\tau_{r,m})
        &=
        \widetilde X(t\wedge\tau_{r,m})
        \qquad\text{a.s.}
    \label{eq:uniq:stopped-equality}
\end{align}
for every fixed $t\in[0,T]$.  Intersecting the corresponding full-probability
events over rational $t\in[0,T]$ and using continuity of the paths, the equality
holds simultaneously for all $t\in[0,T]$.

Letting $m\to\infty$ gives equality up to $\tau_r$.  Letting $r\to\infty$ gives
\begin{align}
    X(t)
        &=
        \widetilde X(t),
        \qquad
        0\le t<T\wedge\life{X}\wedge\life{\widetilde X},
        \qquad
        \text{a.s.}
    \label{eq:uniq:equality-up-to-T}
\end{align}
Since $T>0$ is arbitrary, \eqref{eq:pathwise-uniq-conclusion} follows.

If the lifetimes are explosion lifetimes, equality of the paths before the minimum
lifetime implies equality of all compact exit times, and therefore
$\life{X}=\life{\widetilde X}$ a.s.  If a common deterministic lifetime is
prescribed, the equality of lifetimes is immediate.
\end{proof}

\begin{remark}
Theorem~\ref{thm:pathwise-uniq-indicator-nonsticky} does not assert that a solution exists.  It is an implication: whenever two non-sticky solutions of \eqref{eq:degenerate-existence-indicator} exist on the same filtered probability space, are driven by the same Brownian motion, and have the same initial condition, they coincide up to their common lifetime.  Since non-sticky solutions satisfy \eqref{eq:zero:leb}, they are equivalently solutions of \eqref{eq:main:SDE}, with the singular force interpreted through its interior values.  The theorem is therefore not a uniqueness statement for the broader indicator equation without the non-sticky restriction.
\end{remark}

\begin{remark}[Relation with condition \textup{(A1)} from \cite{bib:GraczykMalecki:2014}] Let us spell out what Assumption~\ref{ass:uniq:k} gives in type $\AN$. For $\alpha=e_i-e_j$, with $i<j$, the corresponding scalar force is
    \begin{align*}
        \frac{k_{e_i-e_j}(t,x)}{x_i-x_j},
    \end{align*}
where, in our chamber convention, $x_i>x_j$.  Suppose that this coefficient depends only on the two coordinates involved, namely
    \begin{align*}
        k_{e_i-e_j}(t,x) &= H_{ij}(t,x_i,x_j).
    \end{align*}
Then Assumption~\ref{ass:uniq:k} says that the force
    \begin{align*}
        \frac{H_{ij}(t,u,v)}{u-v}, \qquad u>v,
    \end{align*}
is non-increasing as the gap $u-v$ increases. Equivalently, in the chamber convention used here, for $w<x<y<z$, one obtains
    \begin{align}
        H_{ij}(t,z,w)(y-x) &\le H_{ij}(t,y,x)(z-w).
        \label{eq:uniq:ordered-particle-A1-descending}
    \end{align}
Indeed, the pair $(z,w)$ has the larger gap $z-w$, while the pair $(y,x)$ has the smaller gap $y-x$. If one rewrites the same statement in the increasing-order convention used in the ordered-particle formulation of \cite{bib:GraczykMalecki:2014}, then \eqref{eq:uniq:ordered-particle-A1-descending} becomes
    \begin{align*}
        H_{ij}(t,w,z)(y-x) &\le H_{ij}(t,x,y)(z-w), \qquad w<x<y<z,
    \end{align*}
which is precisely condition \textup{(A1)} from that setting.
\end{remark}

\begin{proof}[Proof of Theorem~\ref{thm:strong:existence:uniqueness}]
Under existence alternative~\ref{ass:exist:positive}, weak existence of a solution of \eqref{eq:main:SDE} follows from Theorem~\ref{thm:existence}.  Moreover, the singular integrals are finite on compact subintervals of the lifetime.  Since
$k_\alpha$ is strictly positive on the corresponding wall, this finiteness implies the zero-occupation property \eqref{eq:zero:leb}.  Thus the weak solution belongs to the non-sticky class.

Under existence alternative~\ref{ass:exist:degenerate}, weak existence of a solution of \eqref{eq:main:SDE} in the non-sticky class follows from Theorem~\ref{thm:degenerate-remove-indicator}.  In that case the indicator formulation is first used to construct a weak solution, and the non-sticking conclusion established in Theorem~\ref{thm:degenerate-remove-indicator} removes the indicator from the singular drift.

Pathwise uniqueness in the non-sticky class follows from Theorem~\ref{thm:pathwise-uniq-indicator-nonsticky}.  It remains to pass from weak existence and pathwise uniqueness to strong existence.  We use the Yamada-Watanabe principle in localized form.  For $r>0$ define the exit time
    \begin{align*}
        \tau_r^X &:= \inf\left\{t \ge0:\ |X(t)|\ge r\right\}.
    \end{align*}
Weak existence and pathwise uniqueness hold for the equation stopped at $\tau_r^X$.  Hence, by the Yamada-Watanabe principle, the stopped equation admits a strong solution and has uniqueness in law.  These stopped strong solutions are consistent as $r$ varies.  Since, for the explosion lifetime,
    \begin{align*}
        \life{X} &= \lim_{r\to\infty}\tau_r^X,
    \end{align*}
letting $r\to\infty$ gives a strong solution up to its lifetime, adapted to the completed filtration generated by the initial condition and the Brownian motion. The same localization argument gives uniqueness in law for weak solutions of \eqref{eq:main:SDE} in the non-sticky class with initial law $\nu$.
\end{proof}


\section{Mean-field convergence theorem}
\label{sec:mean-field-proof}

We use the notation of Section~\ref{subsec:mean-field-convergence}. We additionally define, for the short-root contribution in type $B$, the following
    \begin{align*}
        H_f(t,x) &:=
        \begin{cases}
            k(t,x)\dfrac{f'(x)}{x}, & x>0,\\[1.2ex]
            k(t,0)f''(0), & x=0.
        \end{cases}
    \end{align*}
In types $\AN$ and $\DN$ we put $H_f\equiv0$. Moreover, in the $\AN$ and $\BN$ cases we work directly with the empirical measure $\mu_t^{(N)}$ defined there.
In the $\DN$ case, it is convenient in the proof to introduce the auxiliary empirical
measure
\begin{align*}
    \widehat\mu_t^{(N)}
        &:=
        \frac1N
        \sum_{i=1}^{N}
        \delta_{|X_i^{(N)}(t)|}.
\end{align*}
This auxiliary measure differs from the empirical measure used in the statement only
by a negligible term. Indeed, for every bounded test function $f$,
\begin{align*}
    \sup_{0\le t\le T}
    \left|
        \left\langle\widehat\mu_t^{(N)},f\right\rangle
        -0                                                                 
        \left\langle\mu_t^{(N)},f\right\rangle
    \right|                                   
    &\le
    \frac{2\|f\|_\infty}{N}.
\end{align*}
Moreover, since $|X_N^{(N)}(t)|\le X_{N-1}^{(N)}(t)$ in the type $\DN$ chamber, the
moment assumption for $\mu_0^{(N)}$ implies the corresponding moment bound for
$\widehat\mu_0^{(N)}$. Hence tightness and convergence for
$\widehat\mu^{(N)}$ imply the same conclusions for $\mu^{(N)}$. Thus, throughout the proof we may write
\begin{align*}
    Y_i^{(N)}(t)
        &:=
        \begin{cases}
            X_i^{(N)}(t), & \mathfrak r=A,B,\\
            |X_i^{(N)}(t)|, & \mathfrak r=D,
        \end{cases}
\end{align*}
and then, in all three cases,
\begin{align*}
    \mu_t^{(N)}
        =
        \frac1N
        \sum_{i=1}^N
        \delta_{Y_i^{(N)}(t)}.
\end{align*}
For deterministic chamber points $x$, we similarly write $y_i=x_i$ in types $A$ and
$B$, and $y_i=|x_i|$ in type $D$.

In root-wise notation, the classical mean-field scaling is
\begin{align*}
    k_{e_i-e_j}^{(N)}(t,x)
        &:=
        \frac1N k(t,y_i,y_j),
        \qquad 1\le i<j\le N,
        \\
    k_{e_i+e_j}^{(N)}(t,x)
        &:=
        \frac1N k(t,y_i,y_j),
        \qquad 1\le i<j\le N,
        \qquad \mathfrak r=B,D,
        \\
    k_{e_i}^{(N)}(t,x)
        &:=
        k(t,y_i),
        \qquad i=1,\ldots,N,
        \qquad \mathfrak r=B.
\end{align*}
where $y_i=x_i$ in types $A$ and $B$, and $y_i=|x_i|$ in type $D$.

\begin{proof}[Proof of Theorem~\ref{thm:mf:classical}]
Fix $T>0$.  We first prove the uniform moment estimate.  Let
$p\in\{2,4,8\}$ and put
\begin{align}
    m_p^{(N)}(t)
        &:=
        \left\langle\mu_t^{(N)},|x|^p\right\rangle.
    \label{eq:mf:moment-def}
\end{align}
Applying It\^o's formula to $|x|^p$ in type $\AN$, to $x^p$ in type $\BN$, and to
$|x|^p$ in type $\DN$, localizing the singular drift if necessary, and then using the
symmetry of the pair kernel $k$, gives
\begin{align}
    \E m_p^{(N)}(t)
    &\le
    \E m_p^{(N)}(0)
    +
    C_{p,T}
    \int_0^t
    \left(1+\E m_p^{(N)}(s)\right)\,ds.
    \label{eq:mf:moment-gronwall-before}
\end{align}
Indeed, the one-particle drift and martingale terms are controlled by the growth
assumptions in \ref{ass:mf:coeff}.  For the pair terms one uses the bounds
\begin{align}
    \left|
        \frac{x^{p-1}-y^{p-1}}{x-y}
    \right|
    (1+|x|)(1+|y|)
    &\le
    C_p(1+|x|^p+|y|^p),
    \label{eq:mf:minus-moment-bound}
    \\
    \left|
        \frac{x^{p-1}+y^{p-1}}{x+y}
    \right|
    (1+x)(1+y)
    &\le
    C_p(1+x^p+y^p),
    \qquad x,y\ge0,
    \label{eq:mf:plus-moment-bound}
\end{align}
and for the short-root term in type $\BN$ one uses
\begin{align}
    x^{p-2}k(t,x)
    &\le
    C_{p,T}(1+x^p),
    \qquad x\ge0.
    \label{eq:mf:wall-moment-bound}
\end{align}
Letting the localization level tend to infinity and applying Gronwall's lemma yields
\begin{align}
    \sup_{N\ge1}
    \sup_{0\le t\le T}
    \E\left\langle\mu_t^{(N)},|x|^p\right\rangle
    &<
    \infty,
    \qquad p=2,4,8.
    \label{eq:mf:moment-bound}
\end{align}

We now derive the empirical-measure decomposition.  Fix
$f\in D$.  It\^o's formula, together with the symmetrization of
the pair roots, gives
\begin{align}
    \left\langle\mu_t^{(N)},f\right\rangle
    &=
    \left\langle\mu_0^{(N)},f\right\rangle
    +
    M_f^{(N)}(t)
    +
    A_{\sigma,f}^{(N)}(t)
    +
    \int_0^t
    \left\langle\mu_s^{(N)},b(s,\cdot)f'\right\rangle\,ds
    \notag
    \\
    &\quad+
    \eta_{\mathfrak r}^0
    \int_0^t
    \left\langle\mu_s^{(N)},H_f(s,\cdot)\right\rangle\,ds
    \notag
    \\
    &\quad+
    \frac12
    \int_0^t
    \int_{E^2}
    D_f(x,y)
    k(s,x,y)\,
    \mu_s^{(N)}(dx)\mu_s^{(N)}(dy)\,ds
    +
    R_f^{(N)}(t),
    \label{eq:mf:finite-semimartingale}
\end{align}
where
\begin{align}
    M_f^{(N)}(t)
    &:=
    \frac1{N\sqrt N}
    \sum_{i=1}^N
    \int_0^t
    f'\bigl(Y_i^{(N)}(s)\bigr)
    \sigma\bigl(s,Y_i^{(N)}(s)\bigr)\,d\widetilde B_i(s),
    \label{eq:mf:martingale-term}
    \\
    A_{\sigma,f}^{(N)}(t)
    &:=
    \frac1{2N}
    \int_0^t
    \left\langle
        \mu_s^{(N)},
        \sigma^2(s,\cdot)f''
    \right\rangle\,ds.
    \label{eq:mf:diffusion-finite-term}
\end{align}
Here $\widetilde B_i=B_i$ in types $\AN$ and $\BN$; in type $\DN$ the possible sign coming
from the map $x\mapsto |x|$ is absorbed into the Brownian motion.

The term $R_f^{(N)}$ contains the diagonal correction
\begin{align}
    R_{f,\mathrm{diag}}^{(N)}(t)
    &:=
    -
    \frac1{2N}
    \int_0^t
    \left\langle
        \mu_s^{(N)},
        D_f(x,x)k(s,x,x)
    \right\rangle\,ds.
    \label{eq:mf:diagonal-remainder}
\end{align}
This correction appears because the finite particle sums run over $i\neq j$, whereas
the product measure $\mu_s^{(N)}(dx)\mu_s^{(N)}(dy)$ contains the diagonal.  In type
$\DN$, $R_f^{(N)}$ may also contain the harmless one-coordinate correction produced by
passing from the possibly signed chamber coordinate to its absolute value in the
one-particle drift.  This additional term is of order $N^{-1}$ and is estimated in
the same way as the diagonal correction.

By the growth assumptions in \ref{ass:mf:coeff} and the moment bound
\eqref{eq:mf:moment-bound},
\begin{align}
    \E\left[
        \sup_{0\le t\le T}|M_f^{(N)}(t)|^2
    \right]
    &\le
    \frac{C_{f,T}}{N^2}
    \int_0^T
    \E\left\langle\mu_s^{(N)},1+|x|^2\right\rangle\,ds
    \longrightarrow0,
    \label{eq:mf:martingale-vanishes}
    \\
    \E\left[
        \sup_{0\le t\le T}|A_{\sigma,f}^{(N)}(t)|
    \right]
    &\le
    \frac{C_{f,T}}{N}
    \int_0^T
    \E\left\langle\mu_s^{(N)},1+|x|^2\right\rangle\,ds
    \longrightarrow0,
    \label{eq:mf:diffusion-vanishes}
    \\
    \E\left[
        \sup_{0\le t\le T}|R_f^{(N)}(t)|
    \right]
    &\le
    \frac{C_{f,T}}{N}
    \int_0^T
    \E\left\langle\mu_s^{(N)},1+|x|^2\right\rangle\,ds
    \longrightarrow0.
    \label{eq:mf:diagonal-vanishes}
\end{align}

The same decomposition, the Burkholder-Davis-Gundy inequality, H\"older's
inequality, and the eighth-moment estimate in \eqref{eq:mf:moment-bound} imply the
increment estimate
\begin{align}
    \E\left|
        \left\langle\mu_t^{(N)},f\right\rangle
        -
        \left\langle\mu_s^{(N)},f\right\rangle
    \right|^4
    &\le
    C_{f,T}|t-s|^2,
    \qquad 0\le s\le t\le T.
    \label{eq:mf:increment-estimate}
\end{align}
Together with \eqref{eq:mf:moment-bound}, this gives compact containment and tightness
of the laws of $(\mu_t^{(N)})_{0\le t\le T}$ in
$C([0,T],\mathcal P(E))$.  For instance, one may use a countable
convergence-determining family contained in $\mathcal D_{\mathfrak r}$ and apply the
standard tightness criterion for measure-valued continuous processes.

It remains to identify subsequential limits.  Let
$(\mu^{(N_m)})_{m\ge1}$ be a weakly convergent subsequence.  By Skorokhod's
representation theorem we may assume that
\begin{align}
    \mu^{(N_m)}
    &\longrightarrow
    \mu
    \qquad
    \text{a.s. in }C([0,T],\mathcal P(E)).
    \label{eq:mf:skorokhod}
\end{align}
The initial term converges by \ref{ass:mf:initial}.  The martingale term, the finite
diffusion term, and the remainder vanish by
\eqref{eq:mf:martingale-vanishes}-\eqref{eq:mf:diagonal-vanishes}.  Thus no
diffusion term survives in the limiting equation.

The ordinary drift term passes to the limit by continuity of $b$, truncation on
compact sets, and the uniform moment bound \eqref{eq:mf:moment-bound}.  The wall term
is handled in the same way.  Indeed, for $f\in\mathcal D_B$,
\begin{align}
    |H_f(t,x)|
    &\le
    C_{f,T}(1+x),
    \qquad x\ge0,\quad 0\le t\le T.
    \label{eq:mf:wall-growth}
\end{align}

For the pair interaction, note first that $D_f^{\mathfrak r}$ is bounded on
$E_{\mathfrak r}^2$.  The bound for $D_f^-$ follows from the mean-value theorem, and
for $D_f^+$ we use $f'(0)=0$:
\begin{align}
    |D_f^+(x,y)|
    &\le
    \|f''\|_\infty,
    \qquad x,y\ge0.
    \label{eq:mf:D-plus-bound}
\end{align}
Consequently, by the growth assumption on $k$,
\begin{align}
    \left|
        D_f(x,y)k(s,x,y)
    \right|
    &\le
    C_{f,T}(1+|x|)(1+|y|),
    \qquad 0\le s\le T.
    \label{eq:mf:interaction-uniform-integrability}
\end{align}
The second moments are uniformly bounded by \eqref{eq:mf:moment-bound}.  Hence
truncation and weak convergence imply convergence of the interaction terms.

Passing to the limit in \eqref{eq:mf:finite-semimartingale} gives
\eqref{eq:mf:limit-equation}.  Therefore every subsequential weak limit solves the
limiting equation.  If \eqref{eq:mf:limit-equation} has a unique solution in
$C([0,T],\mathcal P(E))$, then all subsequential limits coincide, and
therefore the full sequence converges weakly to that deterministic solution.  Since
the limit is deterministic, this convergence is equivalent to convergence in
probability.
\end{proof}


\section{Examples}
\label{sec:examples}
\subsection{Dyson-Cépa-Lépingle systems with logarithmic root potentials}
\label{subsec:examples-dyson-cepa-lepingle}

We begin with the non-degenerate logarithmic repulsion case.  This is the root-system analogue of the electrostatic particle systems of Cépa-Lépingle~\cite{bib:CepaLepingle:1997}.  Let $R$ be one of the classical root systems considered in this paper, let $R_+$ be the corresponding set of positive roots, and let $\Cc$ be the closed Weyl chamber.

Let
    \begin{equation*}
        m_\alpha:[0,\infty)\longrightarrow(0,\infty), \qquad \alpha\in R_+,
    \end{equation*}
be continuous functions.  We do not require the multiplicities $m_\alpha$ to be constant on Weyl-orbits, although this is the usual radial Dunkl specialization. Define, for $x\in\C$,
    \begin{equation*}
        h_t(x) := \prod_{\alpha\in R_+} \langle x,\alpha\rangle^{m_\alpha(t)} .
    \end{equation*}
Then
    \begin{equation*}
        \nabla\log h_t(x) = \sum_{\alpha\in R_+} m_\alpha(t) \frac{\alpha}{\langle x,\alpha\rangle}.
    \end{equation*}
We consider the chamber-valued equation
    \begin{equation}
        \label{eq:example:dyson-cepa-lepingle-root}
        dX(t) = \sigma(t,X(t))\,dB(t) + b(t,X(t))\,dt + \nabla\log h_t(X(t))\,dt, \qquad X(t)\in\Cc,
    \end{equation}
where $\sigma$ is diagonal, with entries $\sigma_1,\ldots,\sigma_N$. Equivalently, \eqref{eq:example:dyson-cepa-lepingle-root} is the root-system SDE
    \eqref{eq:main:SDE} with
    \begin{equation*}
        k_\alpha(t,x)=m_\alpha(t), \qquad \alpha\in R_+.
    \end{equation*}
Thus the singular coefficients are strictly positive on every root wall.

\begin{proposition}[Existence and uniqueness for logarithmic root-barrier systems]
\label{prop:example:dyson-cepa-lepingle}
Assume \ref{ass:cont} and \ref{ass:sigma}.  Then, for every initial condition
$x_0\in\Cc$, equation \eqref{eq:example:dyson-cepa-lepingle-root} has a weak non-sticky solution in $\Cc$, up to its lifetime. If, in addition, the growth condition \ref{ass:growth:p2} holds, then the solution is global.

Assume moreover that one of the following uniqueness regimes holds:
\begin{enumerate}[label=\textup{(\roman*)}]
    \item the Yamada-Watanabe alternative \textup{(a)} in
    Assumption~\ref{ass:uniq:sigma} holds, together with
    Assumption~\ref{ass:uniq:b};

    \item the locally Lipschitz alternative \textup{(b)} in
    Assumption~\ref{ass:uniq:sigma} holds, together with
    Assumption~\ref{ass:uniq:b:euclidean}.
\end{enumerate}
Then \eqref{eq:example:dyson-cepa-lepingle-root} has pathwise uniqueness in the non-sticky class.  Consequently, by the Yamada-Watanabe principle, it has a strong solution, unique in the non-sticky class, up to its lifetime.  Under \ref{ass:growth:p2}, the strong solution is global.
\end{proposition}

\begin{proof}
Since $k_\alpha(t,x)=m_\alpha(t)>0$, Assumption~\ref{ass:k:positive} holds.  The
weak existence statement and the integrability of the singular drifts therefore
follow from Theorem~\ref{thm:existence}.  Moreover, since the corresponding
singular coefficient is strictly positive on each wall, the finiteness of the
singular integrals implies the zero-occupation property
\begin{equation*}
    \int_0^t
    \mathbf 1_{\{\langle X(s),\alpha\rangle=0\}}\,ds
    =
    0,
    \qquad
    \alpha\in R_+,
    \qquad
    t<\life{X},
    \qquad
    \text{a.s.}
\end{equation*}
Thus the solution is non-sticky.  If Assumption~\ref{ass:growth:p2} holds, global
existence follows from Theorem~\ref{thm:non-explosion}.

It remains to check the uniqueness assumptions.  Let
\begin{equation*}
    G(t,x)
    :=
    \nabla\log h_t(x)
    =
    \sum_{\alpha\in R_+}
    m_\alpha(t)
    \frac{\alpha}{\langle x,\alpha\rangle},
    \qquad x\in\C.
\end{equation*}
We first verify Assumption~\ref{ass:uniq:k}.  Let $x,y\in\C$ and set
\begin{equation*}
    d_i:=x_i-y_i,
    \qquad
    s_i:=\operatorname{sgn}(d_i),
    \qquad i=1,\ldots,N.
\end{equation*}
For a positive root $\alpha$, the corresponding contribution to the left-hand side
of \eqref{eq:uniq:k:l1-dissipative} is
\begin{equation*}
    m_\alpha(t)
    \langle s,\alpha\rangle
    \left(
        \frac1{\langle x,\alpha\rangle}
        -
        \frac1{\langle y,\alpha\rangle}
    \right)
    =
    -
    m_\alpha(t)
    \frac{
        \langle s,\alpha\rangle\langle d,\alpha\rangle
    }
    {
        \langle x,\alpha\rangle\langle y,\alpha\rangle
    }.
\end{equation*}
The denominators are positive since $x,y\in\C$, and $m_\alpha(t)>0$.

For the classical root systems considered here, the positive roots have the forms
$e_i-e_j$, $e_i+e_j$, $e_i$, or $2e_i$, depending on the type.  For these roots,
\begin{equation*}
    \langle s,\alpha\rangle\langle d,\alpha\rangle\ge0.
\end{equation*}
Indeed,
\begin{equation*}
    (s_i-s_j)(d_i-d_j)\ge0
\end{equation*}
for $\alpha=e_i-e_j$, by monotonicity of the sign function,
\begin{equation*}
    (s_i+s_j)(d_i+d_j)\ge0
\end{equation*}
for $\alpha=e_i+e_j$, and
\begin{equation*}
    s_i d_i\ge0
\end{equation*}
for $\alpha=e_i$ or $2e_i$.  Hence each root contribution is nonpositive.  Summing
over $\alpha\in R_+$ gives
\begin{equation*}
    \sum_{i=1}^N
    \operatorname{sgn}(x_i-y_i)
    \bigl(G_i(t,x)-G_i(t,y)\bigr)
    \le0.
\end{equation*}
Thus Assumption~\ref{ass:uniq:k} holds.  Therefore the first uniqueness regime in
the statement gives pathwise uniqueness by
Theorem~\ref{thm:pathwise-uniq-indicator-nonsticky}.

For the second regime, note that $\log h_t$ is concave on $\C$, because it is a
positive linear combination of the concave functions
$x\mapsto\log\langle x,\alpha\rangle$.  Therefore its gradient is dissipative:
\begin{equation*}
    \bigl\langle x-y,\nabla\log h_t(x)-\nabla\log h_t(y)\bigr\rangle
    \le0,
    \qquad x,y\in\C.
\end{equation*}
Equivalently, Assumption~\ref{ass:uniq:k:dissipative} holds for the singular force.
The second uniqueness regime then follows again from
Theorem~\ref{thm:pathwise-uniq-indicator-nonsticky}.  Strong existence and uniqueness
follow from Theorem~\ref{thm:strong:existence:uniqueness}.
\end{proof}

\begin{remark}[More general logarithmic potentials]
The preceding proposition can be extended beyond the pure root-barrier product.
Suppose that, for each $t$, a positive differentiable function $h_t$ on $\C$
satisfies
\begin{equation}
    \label{eq:example:general-log-root-decomposition}
    \nabla\log h_t(x)
    =
    \sum_{\alpha\in R_+}
    \alpha\,
    \frac{k_\alpha(t,x)}{\langle x,\alpha\rangle},
    \qquad x\in\C,
\end{equation}
where each $k_\alpha$ extends continuously to $[0,\infty)\times\Cc$, is
nonnegative on $[0,\infty)\times\Cc$, and is strictly positive on the wall
$\{\langle x,\alpha\rangle=0\}$.  Then, under the same standing assumptions on
$\sigma$ and $b$, the existence conclusion of
Proposition~\ref{prop:example:dyson-cepa-lepingle} remains valid.

If, in addition, $\log h_t$ is concave for each $t$, then the singular force
$\nabla\log h_t$ satisfies Assumption~\ref{ass:uniq:k:dissipative}.  Hence, under
the locally Lipschitz diffusion alternative \textup{(b)} in
Assumption~\ref{ass:uniq:sigma} and the Euclidean one-sided Lipschitz condition
Assumption~\ref{ass:uniq:b:euclidean}, the same pathwise uniqueness and strong
existence conclusion holds.

For the coordinatewise Yamada-Watanabe regime, one needs instead the
$\ell^1$-dissipativity condition Assumption~\ref{ass:uniq:k}.  This condition holds
for the pure classical root-barrier product treated in
Proposition~\ref{prop:example:dyson-cepa-lepingle}, as shown above.  For a more
general logarithmic potential satisfying
\eqref{eq:example:general-log-root-decomposition}, concavity alone gives the
Euclidean dissipativity condition, but does not by itself imply the
$\ell^1$-dissipativity condition.
\end{remark}
In type $\AN$, we have
    \begin{equation*}
        R_+ = \{e_i-e_j:\ 1\le i<j\le N\}, \qquad \Cc = \{x_1\ge\cdots\ge x_N\}.
    \end{equation*}
If $m_{e_i-e_j}(t)\equiv \beta/2$, then
    \begin{equation*}
        h_t(x) = \prod_{1\le i<j\le N}(x_i-x_j)^{\beta/2},
    \end{equation*}
and \eqref{eq:example:dyson-cepa-lepingle-root} becomes
    \begin{equation*}
        dX_i(t) = \sigma_i(t,X(t))\,dB_i(t) + b_i(t,X(t))\,dt + \frac{\beta}{2} \sum_{j\ne i} \frac{dt}{X_i(t)-X_j(t)}, \qquad \beta>0.
    \end{equation*}
Thus the Dyson-Cépa-Lépingle electrostatic repulsion is recovered, but with the diagonal martingale distortion allowed by Assumption~\ref{ass:sigma}.  If $\sigma_i(t,x)=\sigma(t,x_i)$ and $b_i(t,x)=b(t,x_i)$, then the Yamada-Watanabe regime applies under the usual one-dimensional modulus condition on $\sigma$ and the one-sided Lipschitz condition on $b$, because the logarithmic Dyson force satisfies Assumption~\ref{ass:uniq:k}.  If the coefficients depend on the full configuration and are locally Lipschitz, the Euclidean regime applies instead, using the concavity of $\log h_t$.


\subsection{Generalized particle systems with weak repulsion}
\label{subsec:examples-graczyk-malecki-H}

We next consider type $\AN$ pair-repulsion systems of the form studied by Graczyk and Ma{\l}ecki~\cite{bib:GraczykMalecki:2014}, under the boundary-variance compatibility required by the present root-system construction.  The point of the present formulation is that we do not impose the additional assumptions which force particles never to collide.  Weak repulsion may allow collisions, but the singular drift is still integrable and, under the non-sticking condition below, the solutions spend zero Lebesgue time on collision hyperplanes.

Let $R_+ = \{e_i-e_j:\ 1\le i<j\le N\}$ and $\Cc  =  \{x_1\ge x_2\ge\cdots\ge x_N\}$. For $1\le i\le N$, let $\sigma_i,b_i:[0,\infty)\times\R\to\R$ be continuous.  For $1\le i<j\le N$, let $H_{ij}:[0,\infty)\times\R^2\to[0,\infty)$ be continuous.  We set
    \begin{equation*}
        H_{ji}(t,y,x):=H_{ij}(t,x,y), \qquad i<j,
    \end{equation*}
and define the type $\AN$ root coefficients by
    \begin{equation}
        \label{eq:example:GM-k}
        k_{e_i-e_j}(t,x) := H_{ij}(t,x_i,x_j), \qquad 1\le i<j\le N.
    \end{equation}
The corresponding ordered-particle equation is
    \begin{equation}
        \label{eq:example:GM-H-system}
        dX_i(t) = \sigma_i(t,X_i(t))\,dB_i(t) + b_i(t,X_i(t))\,dt + \sum_{j\ne i} \frac{H_{ij}(t,X_i(t),X_j(t))}{X_i(t)-X_j(t)}\,dt .
    \end{equation}
When collisions are possible, the degenerate construction first gives the indicator version
    \begin{equation}
        \label{eq:example:GM-H-indicator}
        dX_i(t) = \sigma_i(t,X_i(t))\,dB_i(t) + b_i(t,X_i(t))\,dt + \sum_{j\ne i} \frac{H_{ij}(t,X_i(t),X_j(t))}{X_i(t)-X_j(t)} \mathbf 1_{\{X_i(t)\ne X_j(t)\}}\,dt .
    \end{equation}
We impose the following compatibility condition on the diffusion coefficients:
    \begin{equation}
        \label{eq:example:GM-sigma-compatibility}
        \sigma_i^2(t,z)=\sigma_j^2(t,z), \qquad t\ge0,\quad z\in\R,\quad 1\le i,j\le N.
    \end{equation}

\begin{remark}[Diffusion compatibility at collision walls]

Condition~\eqref{eq:example:GM-sigma-compatibility} is not a non-collision assumption.  It is the type $\AN$ form of Assumption~\ref{ass:sigma}.  When two particles collide, the two rank positions become indistinguishable on the chamber boundary.  The invariant-polynomial construction therefore requires the martingale variances attached to these two positions to agree on the collision wall.  Equivalently, the diffusion coefficient must not distinguish between two particles which occupy the same spatial position. This condition is automatic in the exchangeable case
    \begin{equation*}
        \sigma_i(t,x)=\sigma(t,x), \qquad i=1,\ldots,N,
    \end{equation*}
and more generally whenever
    \begin{equation*}
        \sigma_i^2(t,z)=\sigma_j^2(t,z) \qquad \text{for all }i,j,t,z.
    \end{equation*}
It does exclude rank-dependent models for which two particles have different variances at the same collision point.  Such models belong to the broader Graczyk-Ma{\l}ecki framework when their assumptions preventing collisions are imposed, but they are not covered by the present root-system theorem from arbitrary boundary initial data. 
\end{remark}

The normal dominance condition becomes the following pairwise estimate.  For every $T,R>0$ there exist constants $\varepsilon=\varepsilon(T,R)>0$ and $\gamma=\gamma(T,R)>0$ such that, for all $1\le i<j\le N$,
    \begin{equation}
        \label{eq:example:GM-dominance}
        |x-y|\le\varepsilon,\quad |x|\vee |y|\le R,\quad 0\le t\le T \quad\Longrightarrow\quad H_{ij}(t,x,y) \ge \gamma\bigl(\sigma_i^2(t,x)+\sigma_j^2(t,y)\bigr).
    \end{equation}
Here $\gamma$ is only required to be strictly positive.  No threshold such as $\gamma\ge1$ is imposed.  Indeed, for the root $\alpha=e_i-e_j$,
    \begin{equation*}
        \sum_{\ell=1}^N\alpha_\ell^2\sigma_\ell^2(t,x) = \sigma_i^2(t,x_i)+\sigma_j^2(t,x_j),
    \end{equation*}
so \eqref{eq:example:GM-dominance} is exactly Assumption~\ref{ass:k:dom} near the collision wall $x_i=x_j$.

It remains to formulate the face conditions.  Let
    \begin{equation*}
        I=\{a,\ldots,a+k-1\}, \qquad k\ge2,
    \end{equation*}
be a maximal collision block on a relatively open face, and let $x_p=z$, $p\in I$. For a cut $r=1,\ldots,k-1$, put
    \begin{equation*}
        I_L:=\{a,\ldots,a+r-1\}, \qquad I_R:=\{a+r,\ldots,a+k-1\},
    \end{equation*}
and use the detector
    \begin{equation*}
        u_{I,r} := \frac1r\sum_{p\in I_L}e_p - \frac1{k-r}\sum_{p\in I_R}e_p.
    \end{equation*}
The quadratic-variation density of this detector is
    \begin{equation*}
        q_{I,r}(t,x) := \sum_{p\in I_L}\frac{\sigma_p^2(t,z)}{r^2} + \sum_{p\in I_R}\frac{\sigma_p^2(t,z)}{(k-r)^2}.
    \end{equation*}
The corresponding regular detector drift on the face is
    \begin{align*}
        B_{I,r}(t,x) ={}& \frac1r\sum_{p\in I_L}b_p(t,z) - \frac1{k-r}\sum_{p\in I_R}b_p(t,z) +
            \sum_{\ell<a} \frac{ \displaystyle \frac1{k-r}\sum_{p\in I_R} H_{\ell p}(t,x_\ell,z) - \frac1r\sum_{p\in I_L} H_{\ell p}(t,x_\ell,z)}{x_\ell-z} \\
            &+ \sum_{q>a+k-1} \frac{ \displaystyle\frac1r\sum_{p\in I_L} H_{p q}(t,z,x_q)-\frac1{k-r}\sum_{p\in I_R} H_{p q}(t,z,x_q)}{z-x_q}.
    \end{align*}
The denominators are strictly positive on the relatively open face, because the block $I$ is maximal.

We assume the following inward condition:
    \begin{equation}
        \label{eq:example:GM-face-sign}
        B_{I,r}(t,x)\ge0
    \end{equation}
for every collision block $I$, every cut $1\le r\le |I|-1$, every $t\ge0$, and every point $x$ of the corresponding face.  This is precisely Assumption~\ref{ass:face:sign}, written in block coordinates, when the detector family contains the cut detectors $u_{I,r}$.

Finally, to obtain the equation without indicators, we impose the block escape condition at fully degenerate collision points:
    \begin{equation}
        \label{eq:example:GM-block-escape}
        \max_{I,r} q_{I,r}(t,x)=0 \quad\Longrightarrow\quad \max_{I,r} B_{I,r}(t,x)>0
    \end{equation}
on every collision face.  The maxima are taken over all collision blocks of the face and all their cuts.  This is the concrete type $\AN$ form of Assumption~\ref{ass:nonsticky}.  It says that if all detector martingale variations vanish at a degenerate collision point, then some deterministic detector drift must push the process away from the face.

\begin{proposition}[Weak and strong solutions for Generalized particle systems]
    \label{prop:example:GM-H}
    Assume \eqref{eq:example:GM-sigma-compatibility}, \eqref{eq:example:GM-dominance}, and \eqref{eq:example:GM-face-sign}.  Then, for every initial condition $x_0\in\Cc$, there exists a weak solution of the indicator equation \eqref{eq:example:GM-H-indicator}, up to its lifetime.  Moreover, all singular integrals are finite on compact subintervals of the lifetime.

    If, in addition, the block escape condition \eqref{eq:example:GM-block-escape} holds, then there exists a weak solution of the genuine equation \eqref{eq:example:GM-H-system}, and this solution is non-sticky.

    Assume moreover that one of the two uniqueness regimes in Theorem~\ref{thm:pathwise-uniq-indicator-nonsticky} holds.  Then the solution is pathwise unique in the non-sticky class.  Consequently, by Theorem~\ref{thm:strong:existence:uniqueness}, the equation has a strong solution, unique in the non-sticky class, up to its lifetime.  If the radial growth condition \ref{ass:growth:p2} holds, then the solution is global.
\end{proposition}

\begin{proof}
The coefficient definitions \eqref{eq:example:GM-k} put \eqref{eq:example:GM-H-indicator} into the type $\AN$ root-system form. The compatibility condition \eqref{eq:example:GM-sigma-compatibility} gives Assumption~\ref{ass:sigma}, and \eqref{eq:example:GM-dominance} gives Assumption~\ref{ass:k:dom}.  The block condition \eqref{eq:example:GM-face-sign} is exactly Assumption~\ref{ass:face:sign} for the cut detector family.  Therefore Theorem~\ref{thm:degenerate-existence} gives a weak solution of the indicator equation.

If \eqref{eq:example:GM-block-escape} holds, then Assumption~\ref{ass:nonsticky} is satisfied.  Theorem~\ref{thm:degenerate-remove-indicator} therefore removes the indicator and gives a non-sticky weak solution of \eqref{eq:example:GM-H-system}.  Pathwise uniqueness and strong existence follow from Theorems~\ref{thm:pathwise-uniq-indicator-nonsticky} and \ref{thm:strong:existence:uniqueness}.  The final assertion follows from Theorem~\ref{thm:non-explosion}.
\end{proof}

\begin{remark}[Comparison with the non-collision assumptions]

The estimate \eqref{eq:example:GM-dominance} is much weaker than the usual conditions used to exclude collisions.  It only says that, close to the diagonal, the repulsion coefficient dominates the normal martingale variation by some positive constant.  Thus beta-Dyson type systems with small beta are included. Collisions may occur, but the singular drift remains integrable and the collision hyperplanes have zero Lebesgue occupation time once the non-sticking condition is verified.
\end{remark}

\begin{remark}[The exchangeable case]

A particularly important special case is
\begin{equation*}
    \sigma_i(t,x)=\sigma(t,x),
    \qquad
    b_i(t,x)=b(t,x),
    \qquad
    H_{ij}(t,x,y)=H(t,x,y),
\end{equation*}
where $H(t,x,y)=H(t,y,x)$.  Then the compatibility condition
\eqref{eq:example:GM-sigma-compatibility} is automatic, and the face sign condition
\eqref{eq:example:GM-face-sign} holds with equality:
\begin{equation*}
    B_{I,r}(t,x)=0
\end{equation*}
for every maximal block and every cut.  Hence
Theorem~\ref{thm:degenerate-existence} gives a weak solution of the indicator
equation under the dominance estimate
\begin{equation*}
    H(t,x,y)
    \ge
    \gamma\bigl(\sigma^2(t,x)+\sigma^2(t,y)\bigr)
\end{equation*}
near the diagonal, for any $\gamma>0$.

The block escape condition is more delicate.  For a block at level $z$,
\begin{equation*}
    q_{I,r}(t,x)
    =
    \sigma^2(t,z)
    \left(
        \frac1r+\frac1{k-r}
    \right),
\end{equation*}
whereas $B_{I,r}(t,x)=0$ by symmetry.  Thus the degenerate non-sticking condition is
automatic at collision levels with $\sigma(t,z)>0$, but it fails at symmetric
collision levels with $\sigma(t,z)=0$ unless an additional asymmetric inward drift
is present.

If $H(t,z,z)>0$ on the relevant diagonal, then this wall belongs to the strictly
positive repulsion regime and one should apply Theorem~\ref{thm:existence} rather
than the degenerate indicator theorem.  The genuinely delicate case is
\begin{equation*}
    \sigma(t,z)=0,
    \qquad
    H(t,z,z)=0.
\end{equation*}
At such a fully degenerate collision point the detector quadratic variation
vanishes, the singular repulsion coefficient vanishes, and the symmetric regular
detector drift is zero.  The indicator equation may therefore admit sticky block
solutions.
\end{remark}

\begin{remark}[Uniqueness conditions in pair coordinates]
In the Yamada-Watanabe regime, the rootwise monotonicity assumption
\ref{ass:uniq:k} becomes the following condition: for every pair $i<j$,
\begin{equation*}
    \bigl((x_i-x_j)-(y_i-y_j)\bigr)
    \left(
        \frac{H_{ij}(t,x_i,x_j)}{x_i-x_j}
        -
        \frac{H_{ij}(t,y_i,y_j)}{y_i-y_j}
    \right)
    \le0
\end{equation*}
whenever $x_i>x_j$ and $y_i>y_j$.  In the locally Lipschitz regime, this rootwise
condition may be replaced by the global dissipativity assumption
\ref{ass:uniq:k:dissipative} for the full singular force.
\end{remark}

\begin{proposition}[Mean-field limit for exchangeable pair-repulsion systems]
\label{prop:example:GM-mean-field}
Assume that
\begin{equation*}
    \sigma_i(t,x)=\sigma(t,x),
    \qquad
    b_i(t,x)=b(t,x),
    \qquad
    H_{ij}(t,x,y)=H(t,x,y)=H(t,y,x),
\end{equation*}
where $\sigma,b$ and $H$ are continuous.  For each $N$, consider
\begin{equation}
    \label{eq:example:GM-mean-field-system}
    dX_i^{(N)}(t)
    =
    \frac1{\sqrt N}\sigma\bigl(t,X_i^{(N)}(t)\bigr)\,dB_i(t)
    +
    b\bigl(t,X_i^{(N)}(t)\bigr)\,dt
    +
    \frac1N
    \sum_{j\ne i}
    \frac{H\bigl(t,X_i^{(N)}(t),X_j^{(N)}(t)\bigr)}
         {X_i^{(N)}(t)-X_j^{(N)}(t)}\,dt .
\end{equation}
Suppose that the finite systems satisfy the existence and integrability assumptions
of Theorem~\ref{thm:mf:classical}; for example, this follows from
Proposition~\ref{prop:example:GM-H} whenever the corresponding dominance, face sign,
and non-sticking conditions hold.  Assume also the coefficient growth and initial
moment hypotheses \ref{ass:mf:coeff}-\ref{ass:mf:initial}.  Then the empirical
measure processes
\begin{equation*}
    \mu_t^{(N)}
    =
    \frac1N\sum_{i=1}^N\delta_{X_i^{(N)}(t)}
\end{equation*}
are tight in $C([0,T],\mathcal P(\R))$ for every $T>0$.  Every subsequential limit
$\mu=(\mu_t)_{0\le t\le T}$ satisfies, for every $f\in C_b^2(\R)$,
\begin{align}
    \langle\mu_t,f\rangle
    &=
    \langle\mu_0,f\rangle
    +
    \int_0^t
    \langle\mu_s,b(s,\cdot)f'\rangle\,ds
     +
    \frac12
    \int_0^t
    \int_{\R^2}
    D_f^{-}(x,y)H(s,x,y)\,
    \mu_s(dx)\mu_s(dy)\,ds.
    \label{eq:example:GM-mean-field-limit}
\end{align}

If the limiting equation \eqref{eq:example:GM-mean-field-limit} has a unique solution
in $C([0,T],\mathcal P(\R))$, then the whole sequence $\mu^{(N)}$ converges to that
solution.
\end{proposition}

\begin{proof}
This is the type $\AN$ specialization of Theorem~\ref{thm:mf:classical}, with
\begin{equation*}
    k(t,x,y)=H(t,x,y).
\end{equation*}
The factor $1/2$ appears because the unordered pair interaction is symmetrized in
the empirical-measure equation.
\end{proof}

\begin{remark}[Relation with random matrix flow limits]
Proposition~\ref{prop:example:GM-mean-field} is in the same spirit as the
empirical-measure limits for general random matrix flows studied by
Ma{\l}ecki-P{\'e}rez~\cite{bib:MaleckiPerez:2022}.  The present statement is
formulated directly at the particle level and allows a time-dependent pair kernel
$H(t,x,y)$, not necessarily coming from a matrix flow.  It also applies to weakly
repulsive finite systems constructed by the degenerate theorem, provided the
non-sticking condition is verified.  For Wishart-type high-dimensional limits, see
also Song-Yao-Yuan~\cite{bib:SongYaoYuan:2020}.
\end{remark}

\subsection{Squared Bessel and Wishart particle systems}
\label{subsec:examples-multivariate-squared-bessel}

We now discuss squared Bessel particle systems and beta-Wishart systems.  This is the main family of examples in which the repulsion coefficient vanishes at a boundary point together with the normal martingale variation.  The degenerate theory is therefore needed, and the formulation with indicators is the natural starting point.

\subsubsection*{The nonnegative $B_N$ formulation}

Let
\begin{equation*}
    \Cc = \{x_1\ge x_2\ge\cdots\ge x_N\ge0\}.
\end{equation*}
Fix parameters
\begin{equation*}
    \beta>0, \qquad \delta\in\R, \qquad \gamma\in\R.
\end{equation*}
The beta-Wishart particle system, written in our decreasing order convention, is
\begin{equation}
    \label{eq:example:MSB-system}
    d\lambda_i(t)
    =
    2\sqrt{\lambda_i(t)}\,dB_i(t)
    +
    \left(
        \delta
        -
        2\gamma\lambda_i(t)
        +
        \beta
        \sum_{j\ne i}
        \frac{\lambda_i(t)+\lambda_j(t)}
             {\lambda_i(t)-\lambda_j(t)}
    \right)dt,
    \qquad
    i=1,\ldots,N.
\end{equation}
For $\gamma=0$ this is the multivariate squared Bessel particle system.  For $\gamma>0$ it is the mean-reverting Laguerre, or beta-Wishart, version.  The sign convention for $\gamma$ is chosen so that positive $\gamma$ gives mean reversion.

Although \eqref{eq:example:MSB-system} contains only the singular differences $\lambda_i-\lambda_j$, it is useful to embed it into the root system $B_N$,
\begin{equation*}
    R_+
    =
    \{e_i:\ 1\le i\le N\}
    \cup
    \{e_i-e_j,\ e_i+e_j:\ 1\le i<j\le N\}.
\end{equation*}
The roots $e_i$ and $e_i+e_j$ are auxiliary: they rewrite continuous drift contributions in a form suitable for the degenerate root-system theorem.  Choose constants
\begin{equation*}
    \theta_0>0,
    \qquad
    \theta_+>0,
\end{equation*}
and define, for $x\in\Cc_B$,
\begin{align}
    \sigma_i(t,x)
    &:={2\sqrt{x_i}},
    \label{eq:example:MSB-sigma}
    \\
    k_{e_i-e_j}(t,x)
    &:=
    \beta(x_i+x_j),
    \qquad 1\le i<j\le N,
    \label{eq:example:MSB-k-minus}
    \\
    k_{e_i+e_j}(t,x)
    &:=
    \theta_+(x_i+x_j),
    \qquad 1\le i<j\le N,
    \label{eq:example:MSB-k-plus}
    \\
    k_{e_i}(t,x)
    &:=
    \theta_0 x_i,
    \qquad 1\le i\le N,
    \label{eq:example:MSB-k-zero}
    \\
    b_i(t,x)
    &:=
    \delta
    -
    2\gamma x_i
    -
    \theta_0
    -
    (N-1)\theta_+ .
    \label{eq:example:MSB-b}
\end{align}
At interior points of $\Cc_B$, the corresponding $B_N$ root-system equation is exactly \eqref{eq:example:MSB-system}.  Indeed,
\begin{equation*}
    \frac{k_{e_i}(t,x)}{x_i}=\theta_0,
    \qquad
    \frac{k_{e_i+e_j}(t,x)}{x_i+x_j}=\theta_+,
\end{equation*}
and these continuous contributions cancel the constants subtracted in \eqref{eq:example:MSB-b}.

Let us verify the degenerate assumptions.  Assumption~\ref{ass:sigma} holds.  For roots $e_i-e_j$,
\begin{equation*}
    \sigma_i^2(t,x)-\sigma_j^2(t,x)=4(x_i-x_j),
\end{equation*}
which vanishes on the collision wall $x_i=x_j$.  For roots $e_i+e_j$, the wall $x_i+x_j=0$ can occur in $\Cc_B$ only when $x_i=x_j=0$, and therefore $\sigma_i^2(t,x)-\sigma_j^2(t,x)$ again vanishes there. The normal dominance condition holds globally, since
\begin{align}
    \sum_{\ell=1}^N(e_i-e_j)_\ell^2\sigma_\ell^2(t,x)
    &=
    4x_i+4x_j,
    \label{eq:example:MSB-a-minus}
    \\
    \sum_{\ell=1}^N(e_i+e_j)_\ell^2\sigma_\ell^2(t,x)
    &=
    4x_i+4x_j,
    \label{eq:example:MSB-a-plus}
    \\
    \sum_{\ell=1}^N(e_i)_\ell^2\sigma_\ell^2(t,x)
    &=
    4x_i.
    \label{eq:example:MSB-a-zero}
\end{align}
Thus
\begin{equation*}
    k_\alpha(t,x)
    \ge
    c_*
    \sum_{\ell=1}^N\alpha_\ell^2\sigma_\ell^2(t,x),
    \qquad
    \alpha\in R_+,
    \qquad
    c_*=\frac{1}{4}\min\{\beta,\theta_+,\theta_0\}>0.
\end{equation*}

Positive collision blocks away from zero are locally symmetric in the colliding coordinates.  For cut detectors inside such a block the regular detector drift vanishes by symmetry, and the detector quadratic variation is strictly positive because the common collision level is positive.  Hence these faces do not cause sticking.

The only delicate faces are zero blocks.  Let
\begin{equation*}
    J_r:=\{N-r+1,\ldots,N\},
    \qquad
    1\le r\le N,
\end{equation*}
be a zero block, so that $x_j=0$ for $j\in J_r$ and, if $r<N$, $x_{N-r}>0$.  For the block-center detector
\begin{equation*}
    \bar u_{J_r}:=\frac1r\sum_{j\in J_r}e_j,
\end{equation*}
the detector quadratic variation is
\begin{equation}
    \label{eq:example:MSB-zero-block-q}
    q_{\bar u_{J_r}}(t,x)
    =
    \sum_{j\in J_r}\frac1{r^2}\sigma_j^2(t,x)
    =0
\end{equation}
on the zero-block face.  The regular detector drift is
\begin{equation}
    \label{eq:example:MSB-zero-block-B}
    B_{J_r}(t,x)
    :=
    B_{S,\bar u_{J_r}}(t,x)
    =
    \delta-\theta_0-(N-r)\beta-(r-1)\theta_+ .
\end{equation}
Indeed, the ordinary drift contribution equals $\delta-\theta_0-(N-1)\theta_+$, while each outside positive coordinate contributes $\theta_+-\beta$.  The cut detectors inside the zero block have zero regular detector drift by symmetry.  Thus the block-center detector determines whether the zero block is pushed away from the boundary.

Set
\begin{equation}
    \label{eq:example:MSB-Cr}
    C_r:=\delta-\theta_0-(N-r)\beta-(r-1)\theta_+,
    \qquad r=1,\ldots,N.
\end{equation}
The face sign condition at zero blocks is $C_r\ge0$ for $r=1,\ldots,N$, while the non-sticking condition is verified if
\begin{equation}
    \label{eq:example:MSB-nonsticky-condition}
    C_r>0,
    \qquad r=1,\ldots,N.
\end{equation}
A simple sufficient choice is
\begin{equation}
    \label{eq:example:MSB-parameter-condition}
    0<\theta_+\le\beta,
    \qquad
    0<\theta_0<\delta-\beta(N-1).
\end{equation}
Consequently, the condition
\begin{equation}
    \label{eq:example:MSB-delta-condition}
    \delta>\beta(N-1)
\end{equation}
allows us to choose $\theta_0,\theta_+$ so that all inequalities in \eqref{eq:example:MSB-nonsticky-condition} hold.

We also verify the uniqueness assumptions.  The diffusion coefficient satisfies the Yamada-Watanabe condition, because
\begin{equation*}
    |2\sqrt{u}-2\sqrt{v}|^2\le4|u-v|,
    \qquad u,v\ge0.
\end{equation*}
The regular drift satisfies Assumption~\ref{ass:uniq:b}, since
\begin{equation*}
    \sum_{i=1}^N\operatorname{sgn}(x_i-y_i)
    \bigl(b_i(t,x)-b_i(t,y)\bigr)
    =
    -2\gamma\sum_{i=1}^N|x_i-y_i|
    \le
    2|\gamma|\sum_{i=1}^N|x_i-y_i|.
\end{equation*}
It remains to check Assumption~\ref{ass:uniq:k}.  The auxiliary $e_i$ and $e_i+e_j$ terms give only constants in the interior force and therefore disappear in $G(t,x)-G(t,y)$.  For the Wishart pair force define
\begin{equation*}
    \Psi(u,v):=\frac{u+v}{u-v},
    \qquad u>v\ge0.
\end{equation*}
Then $\Psi$ is non-increasing in $u$ and non-decreasing in $v$.  Hence, for $d_i=x_i-y_i$ and $s_i=\operatorname{sgn}(d_i)$, the contribution of a pair $i<j$ to the $\ell^1$ singular-force estimate is
\begin{equation*}
    \beta(s_i-s_j)
    \bigl(\Psi(x_i,x_j)-\Psi(y_i,y_j)\bigr)
    \le0.
\end{equation*}
Summing over pairs gives Assumption~\ref{ass:uniq:k}.

\begin{proposition}[Beta-Wishart particles in type $B_N$]
\label{prop:example:MSB-existence}
Let $\beta>0$, $\gamma\in\R$, and assume
\begin{equation*}
    \delta>\beta(N-1).
\end{equation*}
Then, for every initial condition $\lambda(0)\in\Cc_B$, the system \eqref{eq:example:MSB-system} has a global strong solution.  The solution is non-sticky, namely
\begin{equation*}
    \int_0^t
    \mathbf 1_{\{\lambda_i(s)=\lambda_j(s)\}}\,ds
    =0,
    \qquad
    1\le i<j\le N,
    \qquad
    t<\life{\lambda},
    \qquad
    \text{a.s.},
\end{equation*}
and
\begin{equation*}
    \int_0^t
    \mathbf 1_{\{\lambda_i(s)=0\}}\,ds
    =0,
    \qquad
    1\le i\le N,
    \qquad
    t<\life{\lambda},
    \qquad
    \text{a.s.}
\end{equation*}
Moreover, pathwise uniqueness and uniqueness in law hold in the class of non-sticky solutions.
\end{proposition}

\begin{proof}
Choose $\theta_0,\theta_+$ satisfying \eqref{eq:example:MSB-parameter-condition}.  The coefficient family \eqref{eq:example:MSB-sigma}-\eqref{eq:example:MSB-b} satisfies Assumptions~\ref{ass:cont}, \ref{ass:sigma}, \ref{ass:k:dom}, \ref{ass:face:sign}, and \ref{ass:nonsticky}, by the computations above.  Theorem~\ref{thm:degenerate-remove-indicator} gives a weak non-sticky solution of the genuine equation \eqref{eq:example:MSB-system}.  The Yamada-Watanabe uniqueness regime, with Assumptions~\ref{ass:uniq:sigma}\textup{(a)}, \ref{ass:uniq:b}, and \ref{ass:uniq:k}, holds by the preceding paragraph.  Therefore Theorems~\ref{thm:pathwise-uniq-indicator-nonsticky} and \ref{thm:strong:existence:uniqueness} give strong existence and uniqueness in the non-sticky class.  Finally, the radial growth condition \ref{ass:growth:p2} is immediate from the linear growth of $b$, the estimate $\sigma_i^2(t,x)=4x_i$, and the linear growth of the coefficients $k_\alpha$ on $\Cc_B$; hence Theorem~\ref{thm:non-explosion} gives global existence.
\end{proof}

\begin{remark}[Relation with Jourdain-Kahn]
In the increasing-order convention used by Jourdain-Kahn~\cite{bib:JourdainKahn:2022}, their parameter $\alpha$ corresponds to our $\delta$.  They prove strong existence and pathwise uniqueness up to the first multiple collision at zero and give sharp conditions for such collisions to be avoided.  In particular, for a zero block of size $r$, the relevant threshold is
\begin{equation*}
    r\bigl(\delta-(N-r)\beta\bigr)\ge2.
\end{equation*}
For $r=2$ this is equivalent to
\begin{equation*}
    \delta-\beta(N-1)\ge1-\beta.
\end{equation*}
Thus, when $0<\beta<1$, the interval
\begin{equation*}
    0<\delta-\beta(N-1)<1-\beta
\end{equation*}
is a regime in which a multiple zero collision may occur.  Proposition~\ref{prop:example:MSB-existence} nevertheless gives a global continuation through such contacts, and it is pathwise unique in the non-sticky class.  This is the sense in which the present theorem improves the continuation result available from the collision-avoidance approach.  The comparison should not be overstated: Jourdain-Kahn give a sharper analysis of which collisions occur or are avoided, while the present root-system argument gives a robust non-sticky continuation once the boundary is reached.
\end{remark}

\begin{remark}[Auxiliary coefficients]
The constants $\theta_0$ and $\theta_+$ are auxiliary.  They do not appear in the final SDE \eqref{eq:example:MSB-system}.  They are introduced only to embed the equation into the degenerate $B_N$ root-system framework and to verify the boundary conditions at zero blocks.  Different choices satisfying \eqref{eq:example:MSB-nonsticky-condition} lead to the same particle system \eqref{eq:example:MSB-system}.
\end{remark}

\subsubsection*{A type $A_{N-1}$ absolute-value analogue}

There is also a natural type $\AN$ analogue in which the state space is the whole ordered line,
\begin{equation*}
    \Cc=
    \{x_1\ge x_2\ge\cdots\ge x_N\},
\end{equation*}
and the square-root coefficient is interpreted with an absolute value.  Consider
\begin{equation}
    \label{eq:example:ABS-A-indicator}
    dX_i(t)
    =
    2\sqrt{|X_i(t)|}\,dB_i(t)
    +
    \bigl(\delta-2\gamma X_i(t)\bigr)dt
    +
    \beta
    \sum_{j\ne i}
    \frac{|X_i(t)|+|X_j(t)|}{X_i(t)-X_j(t)}
    \mathbf 1_{\{X_i(t)\ne X_j(t)\}}\,dt.
\end{equation}
This is not the nonnegative beta-Wishart system.  It is a type $\AN$ ordered-particle model with squared-Bessel type noise and an interaction kernel that degenerates only when colliding particles meet at zero.

In root-system notation, \eqref{eq:example:ABS-A-indicator} corresponds to
\begin{equation*}
    R_+=\{e_i-e_j:\ 1\le i<j\le N\},
    \qquad
    \sigma_i(t,x)=2\sqrt{|x_i|},
    \qquad
    b_i(t,x)=\delta-2\gamma x_i,
\end{equation*}
and
\begin{equation}
    \label{eq:example:ABS-A-k}
    k_{e_i-e_j}(t,x)=\beta\bigl(|x_i|+|x_j|\bigr),
    \qquad 1\le i<j\le N.
\end{equation}
The compatibility condition Assumption~\ref{ass:sigma} holds because, on the wall $x_i=x_j$, one has
\begin{equation*}
    \sigma_i^2(t,x)=4|x_i|=4|x_j|=\sigma_j^2(t,x).
\end{equation*}
The normal dominance condition holds globally:
\begin{equation*}
    k_{e_i-e_j}(t,x)
    =
    \frac{\beta}{4}
    \bigl(\sigma_i^2(t,x)+\sigma_j^2(t,x)\bigr).
\end{equation*}
The face sign condition holds with equality.  Indeed, on a collision block all colliding coordinates have the same value $z$, hence the regular drift $\delta-2\gamma z$ is the same for all particles in the block, and the interactions with particles outside the block are symmetric.  Therefore every cut detector has regular detector drift equal to zero.

The uniqueness assumptions are also satisfied in the Yamada-Watanabe regime.  The diffusion coefficient obeys
\begin{equation*}
    |2\sqrt{|u|}-2\sqrt{|v|}|^2\le4|u-v|,
    \qquad u,v\in\R,
\end{equation*}
and the regular drift satisfies Assumption~\ref{ass:uniq:b}.  For the singular force, set
\begin{equation*}
    \Psi(u,v)=\frac{|u|+|v|}{u-v},
    \qquad u>v.
\end{equation*}
This function is non-increasing in $u$ and non-decreasing in $v$ on the region $u>v$.  Hence the pairwise sign argument used above gives Assumption~\ref{ass:uniq:k}.

\begin{proposition}[Type $\AN$ absolute-value squared Bessel particles]
\label{prop:example:ABS-A}
Let $\beta>0$ and $\delta,\gamma\in\R$.  For every initial condition $X(0)\in\Cc_A$, the indicator equation \eqref{eq:example:ABS-A-indicator} has a global weak solution in $\Cc_A$, and all singular drift integrals are finite on compact time intervals.  Moreover, pathwise uniqueness holds among non-sticky solutions, that is, among solutions satisfying
\begin{equation*}
    \int_0^t
    \mathbf 1_{\{X_i(s)=X_j(s)\}}\,ds
    =0,
    \qquad
    1\le i<j\le N,
    \qquad
    t<\life{X},
    \qquad
    \text{a.s.}
\end{equation*}
\end{proposition}

\begin{proof}
The verification above gives Assumptions~\ref{ass:cont}, \ref{ass:sigma}, \ref{ass:k:dom}, and \ref{ass:face:sign}.  Therefore Theorem~\ref{thm:degenerate-existence} gives a weak solution of the indicator equation.  The radial growth condition \ref{ass:growth:p2} follows from the estimates
\begin{equation*}
    \sigma_i^2(t,x)=4|x_i|,
    \qquad
    |b_i(t,x)|\le C(1+|x_i|),
    \qquad
    k_{e_i-e_j}(t,x)\le\beta(|x_i|+|x_j|),
\end{equation*}
so Theorem~\ref{thm:non-explosion} gives global existence.  The Yamada-Watanabe uniqueness assumptions \ref{ass:uniq:sigma}\textup{(a)}, \ref{ass:uniq:b}, and \ref{ass:uniq:k} have also been checked above.  Theorem~\ref{thm:pathwise-uniq-indicator-nonsticky} therefore gives pathwise uniqueness in the non-sticky class.
\end{proof}

\begin{remark}[Zero blocks and possible sticky behavior]
For the type $\AN$ absolute-value model, the fully degenerate collision faces occur when a collision block sits at level zero.  At such a face all cut-detector quadratic variations vanish, and by symmetry the regular detector drifts also vanish.  Thus Assumption~\ref{ass:nonsticky} is not obtained from the present general criterion, and Proposition~\ref{prop:example:ABS-A} deliberately states existence for the indicator equation only, together with uniqueness once the class is restricted to non-sticky solutions.

A detailed study of which zero blocks are sticky, how non-sticky continuations can be selected, and how the indicator can be removed would require additional arguments.  This model should be approachable by methods similar in spirit to the analysis of squared Bessel particle systems in Graczyk-Małecki~\cite{bib:GraczykMalecki:2019}.
\end{remark}


\subsection{Noncolliding Brownian bridges with time-dependent drift}
\label{subsec:examples-katori-bridges}

One advantage of allowing the coefficients to depend on time is that bridge-type
particle systems can be treated directly at the level of SDEs.  In the work of
Izumi-Katori~\cite{bib:IzumiKatori:2011}, such processes are introduced through
transition probability densities.  Here we record a simple representative example:
the type $\AN$ noncolliding Brownian bridge with terminal collapse at the origin.

Let
\begin{equation*}
    \Cc
    =
    \{x_1\ge x_2\ge\cdots\ge x_N\}
\end{equation*}
and fix $T>0$.  We consider, for $0\le t<T$, the system
\begin{equation}
    \label{eq:example:katori-A-bridge}
    dX_i(t)
    =
    dB_i(t)
    +
    \sum_{j\ne i}
    \frac{dt}{X_i(t)-X_j(t)}
    -
    \frac{X_i(t)}{T-t}\,dt,
    \qquad
    i=1,\ldots,N.
\end{equation}
The last term is the Brownian bridge drift.  The singular drift is the usual
logarithmic repulsion in type $\AN$.

\begin{proposition}[Type $\AN$ noncolliding Brownian bridge]
\label{prop:example:katori-A-bridge}
For every $x_0\in\Cc$, equation \eqref{eq:example:katori-A-bridge} has a strong
solution on $[0,T)$.  This solution is pathwise unique in the non-sticky class.
In particular, the singular drift integrals are finite on compact subintervals of
$[0,T)$.
\end{proposition}

\begin{proof}
We write \eqref{eq:example:katori-A-bridge} in the type $\AN$ root-system form with
\begin{equation*}
    R_+
    =
    \{e_i-e_j:\ 1\le i<j\le N\},
    \qquad
    \sigma_i(t,x)=1,
    \qquad
    k_{e_i-e_j}(t,x)=1,
\end{equation*}
and
\begin{equation*}
    b_i(t,x)
    =
    -\frac{x_i}{T-t}.
\end{equation*}
Fix $T_0<T$.  On $[0,T_0]$ the coefficients are continuous, and
Assumption~\ref{ass:k:positive} holds because $k_\alpha\equiv1$ for every positive
root.  Hence Theorem~\ref{thm:existence} gives weak existence and finiteness of the
singular drift integrals on compact time intervals.  In particular, the solution is non-sticky.

The uniqueness assumptions are also immediate on $[0,T_0]$.  The diffusion
coefficient is constant, the regular drift is locally Lipschitz in $x$, and the
singular force is
\begin{equation*}
    G(x)
    =
    \sum_{1\le i<j\le N}
    \frac{e_i-e_j}{x_i-x_j}
    =
    \nabla\log h_A(x),
    \qquad
    h_A(x)
    =
    \prod_{1\le i<j\le N}(x_i-x_j).
\end{equation*}
Since $\log h_A$ is concave on the chamber interior, Assumption~\ref{ass:uniq:k:dissipative}
holds.  Therefore Theorems~\ref{thm:pathwise-uniq-indicator-nonsticky} and
\ref{thm:strong:existence:uniqueness} give strong existence and pathwise uniqueness
in the non-sticky class on $[0,T_0]$.

Finally,
\begin{equation*}
    \sum_{i=1}^N x_i b_i(t,x)
    =
    -\frac{|x|^2}{T-t}
    \le0
\end{equation*}
on $[0,T_0]$, and the remaining terms in Assumption~\ref{ass:growth:p2} are
constant.  Hence Theorem~\ref{thm:non-explosion} excludes explosion before $T_0$.
Since $T_0<T$ is arbitrary, the solution is defined on the whole interval $[0,T)$.
\end{proof}

\begin{remark}[Other Katori bridge systems]
The same SDE-first viewpoint applies to several other finite-duration
noncolliding systems considered by Izumi-Katori~\cite{bib:IzumiKatori:2011}.  For
example, in type $B_N$ one obtains bridge versions of noncolliding
three-dimensional Bessel processes by adding the time-dependent bridge drift to
the usual type $B_N$ logarithmic repulsion:
\begin{align}
    dY_i(t)
    &=
    dB_i(t)
    +
    \frac{dt}{Y_i(t)}
    +
    \sum_{j\ne i}
    \left(
        \frac{1}{Y_i(t)-Y_j(t)}
        +
        \frac{1}{Y_i(t)+Y_j(t)}
    \right)dt
    -
    \frac{Y_i(t)}{T-t}\,dt,
    \nonumber\\
    &\hspace{8cm}
    i=1,\ldots,N.
    \label{eq:example:katori-B-bridge}
\end{align}
This is again covered by the strictly positive repulsion regime on every compact
subinterval of $[0,T)$.  Thus our framework allows one to start from the SDE and
derive existence and uniqueness directly, instead of first constructing the process
from transition densities.
\end{remark}

\begin{remark}[Affine Katori bridges]
Katori also studies bridge systems associated with affine root systems and
Macdonald denominators~\cite{bib:Katori:2019Macdonald}.  These models should not
all be described as ordinary interval models with the same boundary behavior.  The
affine type $A_{N-1}$ model is periodic and lives on a circle.  The remaining
reduced affine types live on an interval $[0,\pi r]$: the types $B_N$,
$C_N^\vee$, and $BC_N$ have absorbing boundary at $0$ and reflecting boundary at
$\pi r$; the types $B_N^\vee$ and $C_N$ have absorbing boundary at both endpoints;
and the type $D_N$ has reflecting boundary at both endpoints.

The absorbing endpoints correspond to singular logarithmic drifts which prevent
hitting of the endpoint, while reflecting endpoints require local-time reflection
terms.  Therefore these affine models are not literally examples of the present
finite conical Weyl-chamber theorem.  Nevertheless, our techniques should apply to
the corresponding interior SDEs up to the first hitting time of an affine boundary
wall.  A full treatment of the affine Katori systems would require a mild extension
of the present arguments from finite Weyl chambers to affine alcoves, and in the
reflecting cases also the inclusion of reflection terms at the endpoint walls.
\end{remark}


\section{Acknowledgments}

The author would like to express his gratitude to Piotr Graczyk and Makoto Katori for many important and fruitful discussions on particle systems, which eventually led to significant improvements of the paper. The author is also very grateful to Prof. Michael Voit for pointing out the missing arguments in the construction strategy presented in \cite{bib:GraczykMalecki:2014}.

\section{Appendix}


\begin{proposition}
\label{prop:vandermonde}
	For fixed $k=1,\ldots,N+1$ and $x=(x_1,\ldots,x_N)$ let
	$$
		V_N^k(x) = \det
		\begin{bmatrix}
			1 & x_1 & \ldots & x_1^{k-1} & x_1^{k+1} & \ldots & x_1^{N} \\
			\vdots & \vdots & \ddots & \vdots & \vdots & \ddots& \vdots\\
			1 & x_N & \ldots & x_N^{k-1} & x_N^{k+1} & \ldots & x_N^{N}
		\end{bmatrix}\/.
	$$
	Then 
	\begin{equation}
		\label{eq:vandk}
		V_N^k(x) = \left(\prod_{1\leq i<j\leq N} (x_j-x_i)\right)e_{N-k+1}(x)\/.
	\end{equation}
\end{proposition}
\begin{proof}
	Note that the matrix appearing in the definition of $V_N^k(x)$ is the Vandermonde matrix, where we removed the $k$-th column and add the last column of the form $[x_1^N, \ldots, x_N^N]^T$. Obviously $V_N^{N+1}(x)$ is just the Vandermonde determinant and we simply have
	$$
		V_N^{N+1}(x) = \left(\prod_{1\leq i<j\leq N} (x_j-x_i)\right) = \left(\prod_{1\leq i<j\leq N} (x_j-x_i)\right) e_0(x)
	$$
	and \eqref{eq:vandk} holds in this case. From the other side we also simply get that 
		$$
		V_N^{1}(x) = x_1\cdot\ldots\cdot x_N\left(\prod_{1\leq i<j\leq N} (x_j-x_i)\right) = \left(\prod_{1\leq i<j\leq N} (x_j-x_i)\right) e_{N}(x)\/,
	$$
	which also agrees with \eqref{eq:vandk}. Finally, for $k=1,\ldots,N+1$, by simple manipulations on the rows of the matrix, we can get 
	$$
		V_N^k(x_1,\ldots,x_N) = \prod_{i=1}^{N-1}(x_N-x_i)\left[V_{N-1}^{k-1}(x_1,\ldots,x_{N-1})+x_N V_{N-1}^{k}(x_1,\ldots,x_{N-1})\right]\/.
	$$
	Thus the induction with respect to $N$ based on the above-given formula and the relation $e_k(x_1,\ldots,x_N)=e_{k}^{\overline{x_N}}(x_1,\ldots,x_N)+x_N e_{k-1}(x_1,\ldots,x_{N-1})$, leads to \eqref{eq:vandk}.
\end{proof}


\newcommand{\Bii}{\mathcal{B}_i}
\newcommand{\Bjj}{\mathcal{B}_{j}}
\newcommand{\Bjjj}{\mathcal{B}_{j+1}}

\newcommand{\eBii}[1]{e_{#1}(\Bii)}
\newcommand{\eBjj}[1]{e_{#1}(\Bjj)}
\newcommand{\eBjjj}[1]{e_{#1}(\Bjjj)}

\begin{proposition}
    \label{prop:xipk}
    Consider $p=(p_1,\ldots,p_N):\C^{\AN}\longrightarrow \R^N$ defined as
    $$
	   p_k(x) = \frac{1}{k}\sum_{i=1}^N x_i^k\/,\quad k=1,\ldots,N\/.
    $$
    Then the inverse function $x(p)=(x_1(p),\ldots,x_N(p)):p[\C^{\AN}]\longrightarrow \C^{\AN}$ is smooth and 
    \begin{equation}
        \label{eq:xipk}
				\dfrac{\partial x_i}{\partial p_k} = (-1)^{k+1}\dfrac{e_{N-k}^{\overline{x_i}}(x)}{\eBii{N-1}}
    \end{equation}
\end{proposition}

\begin{proof}
	For every $x\in \C^{\AN}$ and $i,j =1,\ldots,N$ we simply have
	$$
		\dfrac{\partial x_i}{\partial x_j} = \sum_{k=1}^N \dfrac{\partial x_i}{\partial p_k} \dfrac{\partial p_k}{\partial x_j}  = \sum_{k=1}^N  x_j^{k-1}\dfrac{\partial x_i}{\partial p_k} = \delta_{i,j}\/,
	$$
	where $\delta_{i,j}=1$ for $i=j$ and $\delta_{i,j}=0$ otherwise. Consequently, for fixed $i$ we obtain system of linear equations with variables $(\frac{\partial x_i}{\partial p_1},\ldots,\frac{\partial x_i}{\partial p_k})$ and the corresponding coefficient matrix is just the Vandermonde matrix $[x_j^{k-1}]_{j,k=1,\ldots,N}$. Since the Vandermonde determinant is non-zero on $C_+^{A_{N-1}}$, we obtain by the Cramer's rule that
	$$
		\dfrac{\partial x_i}{\partial p_k} = (-1)^{i+k}\dfrac{V^{i,k}_{N-1}(x)}{\eB{N(N-1)/2}}\/,
	$$
	where $V^{i,k}_{N-1}(x)$ is the determinant of $(N-1)\times(N-1)$-matrix obtain from the Vandermonde matrix by removing the $i$-th row and $k$-th column. Using \eqref{eq:vandk} from Proposition \ref{prop:vandermonde} ends the proof of \eqref{eq:xipk}. 
\end{proof}


In the next proposition we show that, at boundary points where the relevant root (or isolated pair of roots) is separated from its neighbors, the corresponding inverse coordinate (respectively the pair-sum) depends smoothly on the power sums. As a consequence, the first and higher derivatives with respect to $(p_1,\dots,p_N)$ admit continuous extensions at such points. In the proposition below we remove the normalization $1/k$ in the definition of the power sums $p_k$ to make the proof easier to read. Obviously, this has no effect on the presented result, i.e. the claim is exactly the same for the power sums with normalization $1/k$. Moreover, it is also true for any other ordering of the variables $x_1,\ldots,x_N$ and corresponding $F$ function.

\begin{proposition}
    \label{prop:root:smooth}
    Let $\Omega=\{x=(x_1,\dots,x_N)\in\mathbb R^N:\ x_1<\cdots<x_N\}$ and define the mapping $F:\overline\Omega\to\mathbb R^N$ by
        \begin{equation*}
            F(x)=(p_1,\dots,p_N),\qquad p_k=\sum_{j=1}^N x_j^k,\quad k=1,\dots,N.
        \end{equation*}
    Fix $x^*\in\overline\Omega$ and set $p^*=F(x^*)$.
    \begin{enumerate}[label=\emph{(\roman*)}, ref=\roman*]
    \item \label{prop:xip:item_single} Assume that for some $i\in\{1,\dots,N\}$ one has
        \begin{equation*}
            x^*_{i-1}<x^*_i<x^*_{i+1},
        \end{equation*}
    with the obvious interpretation when $i=1$ or $i=N$. Then there exists a neighborhood $U\subset\mathbb R^N$ of $p^*$ and a unique function
        \begin{equation*}
            \varphi:U\to\R
        \end{equation*}
    of class $\mathcal{C}^\infty$ such that for all $p\in U\cap F(\overline\Omega)$ we have $\varphi(p)=x_i(p)$, where
    $x(p)=(x_1(p)\le\cdots\le x_N(p))$ denotes the ordered root-vector inverse of $F$ on $F(\overline\Omega)$.
    In particular, the coordinate $p\mapsto x_i(p)$ is $\mathcal{C}^\infty$ at $p^*$.

    \item \label{prop:xip:item_sum} Assume that for some $i\in\{1,\dots,N-1\}$ one has the \emph{isolation condition}
        \begin{equation}
            \label{eq:isolate:merged}
            x^*_{i-1}<x^*_i\le x^*_{i+1}<x^*_{i+2},
        \end{equation}
    with the obvious endpoint modifications when $i=1$ or $i=N-1$. Define
        \begin{equation*}
            S(p):=x_i(p)+x_{i+1}(p).
        \end{equation*}
    Then there exists a neighborhood $U\subset\R^N$ of $p^*$ and a unique function
        \begin{equation*}
            \psi:U\to\R
        \end{equation*}
    of class $\mathcal{C}^\infty$ such that
        \begin{equation*}
            \psi(p)=S(p)\qquad\text{for all }p\in U\cap F(\overline\Omega).
        \end{equation*}
    In particular, the sum $p\mapsto x_i(p)+x_{i+1}(p)$ is $\mathcal{C}^\infty$ at $p^*$.
    \end{enumerate}
\end{proposition}

\begin{proof}
    For $p=(p_1,\dots,p_N)$, denote by $e_0(p):=1$ and $e_1(p),\dots,e_N(p)$ the elementary symmetric polynomials, expressed via Newton identities
        \begin{equation*}
            k\,e_k(p)=\sum_{j=1}^k (-1)^{j-1}e_{k-j}(p)\,p_j,\qquad k=1,\dots,N.
        \end{equation*}
    Each $e_k$ is a polynomial in $(p_1,\dots,p_k)$ and, in particular, $p\mapsto e_k(p)$ is $\mathcal{C}^\infty$. Define the monic polynomial
        \begin{equation}
            \label{eq:qpt:defn:merged}
            q_p(t):=t^N-e_1(p)t^{N-1}+e_2(p)t^{N-2}-\cdots+(-1)^N e_N(p).
        \end{equation}
    If $p=F(x)$ for some $x\in\overline\Omega$, then the multiset of roots of $q_p$ (counted with multiplicity) is precisely $\{x_1,\dots,x_N\}$.

    \medskip

    We begin with the proof of the first claim. We have $q_{p^*}(x_i^*)=0$. Since $x^*\in\overline\Omega$ is non-decreasing and $x^*_{i-1}<x^*_i<x^*_{i+1}$, the value $x_i^*$ is distinct
    from every other coordinate $x_m^*$ for $m\neq i$. Hence
        \begin{equation*}
            q'_{p^*}(x_i^*)=\prod_{m\neq i}(x_i^*-x_m^*)\neq 0,
        \end{equation*}
    so $t=x_i^*$ is a simple root of $q_{p^*}$. Consider $\Phi(p,t):=q_p(t)$. This map is $\mathcal{C}^\infty$ in $(p,t)$ and satisfies
        \begin{equation*}
            \Phi(p^*,x_i^*)=0,\qquad \partial_t\Phi(p^*,x_i^*)=q'_{p^*}(x_i^*)\neq 0.
        \end{equation*}
    By the implicit function theorem, there exist neighborhoods $U$ of $p^*$ and $I$ of $x_i^*$, and a unique $\mathcal{C}^\infty$ function
    $\varphi:U\to I$ such that
        \begin{equation*}
            q_p(\varphi(p))=0\quad\text{for all }p\in U,\qquad \varphi(p^*)=x_i^*.
        \end{equation*}
    Because $x_i^*$ is strictly separated from $x^*_{i-1}$ and $x^*_{i+1}$, we can choose $\varepsilon>0$ and (if necessary) shrink $U$ so that for every
    $p\in U\cap F(\overline\Omega)$, the polynomial $q_p$ has exactly one real root in the interval
    $\bigl(x^*_{i-1}+\varepsilon,\ x^*_{i+1}-\varepsilon\bigr)$, and this root is $\varphi(p)$. But the ordered root vector $x(p)$ has $x_i(p)$ as the unique root
    in that interval, hence $\varphi(p)=x_i(p)$ for all $p\in U\cap F(\overline\Omega)$.

    \medskip

    To deal with the second part let
        \begin{equation*}
            S^*:=x_i^*+x_{i+1}^*,\qquad P^*:=x_i^*x_{i+1}^*,
        \end{equation*}
    and define the monic quadratic polynomial
        \begin{equation*}
            Q^*(t):=t^2-S^*t+P^*=(t-x_i^*)(t-x_{i+1}^*).
        \end{equation*}
    Let $R^*(t)$ be the monic polynomial of degree $N-2$ defined by the factorization
        \begin{equation*}
            q_{p^*}(t)=Q^*(t)\,R^*(t).
        \end{equation*}
    The isolation condition \eqref{eq:isolate:merged} implies that the roots of $Q^*$ are separated from all other roots of $q_{p^*}$, hence $Q^*$ and $R^*$ have
    no common root. In particular,
        \begin{equation}
            \label{eq:coprime:merged}
            \gcd(Q^*,R^*)=1.
        \end{equation}
    Define $u\in\mathbb R^N$ by
        \begin{equation*}
            u=(S,P,r_1,\dots,r_{N-2}),
        \end{equation*}
    where $S,P\in\mathbb R$ and
        \begin{equation*}
            R(t):=t^{N-2}+r_1 t^{N-3}+\cdots+r_{N-2}
        \end{equation*}
    is monic. Introduce the map $H:\mathbb R^N\times\mathbb R^N\to\mathbb R^N$ by
        \begin{equation*}
            H(u,p):=\mathrm{coeff}\Big((t^2-St+P)R(t)-q_p(t)\Big),
        \end{equation*}
    where $\mathrm{coeff}$ denotes the coefficient vector (e.g.\ from $t^{N-1}$ down to $t^0$). Then $H$ is $\mathcal{C}^\infty$ in $(u,p)$, and for
        \begin{equation*}
            u^*:=(S^*,P^*,r_1^*,\dots,r_{N-2}^*)
        \end{equation*}
    (the coefficients of $Q^*$ and $R^*$) we have $H(u^*,p^*)=0$.

    We claim that the Jacobian $\partial_u H(u^*,p^*)$ is invertible. Indeed, suppose $\delta u=(\delta S,\delta P,\delta r_1,\dots,\delta r_{N-2})$ satisfies
    $\partial_u H(u^*,p^*)\,\delta u=0$. Differentiating the product $(t^2-St+P)R(t)$ in the $u$-variables gives the polynomial identity
        \begin{equation}
            \label{eq:variation:merged}
            \delta Q(t)\,R^*(t)+Q^*(t)\,\delta R(t)\equiv 0,
        \end{equation}
    where
        \begin{equation*}
            \delta Q(t):=-\delta S\,t+\delta P,
            \qquad
            \delta R(t):=\delta r_1 t^{N-3}+\cdots+\delta r_{N-2}.
        \end{equation*}
    Evaluating \eqref{eq:variation:merged} at any root $\alpha$ of $Q^*$ yields $\delta Q(\alpha)R^*(\alpha)=0$. By \eqref{eq:coprime:merged}, one has
    $R^*(\alpha)\neq 0$ for each root $\alpha$ of $Q^*$, hence $\delta Q(\alpha)=0$ at both roots of $Q^*$. If $x_i^*<x_{i+1}^*$, these are two distinct points,
    so $\delta Q\equiv 0$ (since $\deg \delta Q\le 1$). If $x_i^*=x_{i+1}^*=a$, then evaluating \eqref{eq:variation:merged} at $t=a$ gives $\delta Q(a)=0$, and
    differentiating \eqref{eq:variation:merged} in $t$ and evaluating at $t=a$ gives $\delta Q'(a)=0$ (since $Q^*(a)=Q^{*\prime}(a)=0$ for a double root),
    hence again $\delta Q\equiv 0$. In either case, $\delta S=\delta P=0$, and then \eqref{eq:variation:merged} implies $Q^*(t)\delta R(t)\equiv 0$, so
    $\delta R\equiv 0$. Therefore $\delta u=0$, proving $\ker(\partial_u H(u^*,p^*))=\{0\}$ and hence $\partial_u H(u^*,p^*)$ is invertible.

    By the implicit function theorem, there exist a neighborhood $U\subset\mathbb R^N$ of $p^*$ and a unique $\mathcal{C}^\infty$ map
        \begin{equation*}
            u(p)=(S(p),P(p),r_1(p),\dots,r_{N-2}(p))
        \end{equation*}
    defined on $U$ such that $u(p^*)=u^*$ and
        \begin{equation*}
            (t^2-S(p)t+P(p))\,R_p(t)=q_p(t)\qquad\text{for all }p\in U,
        \end{equation*}
    where $R_p(t):=t^{N-2}+r_1(p)t^{N-3}+\cdots+r_{N-2}(p)$. In particular, $S(p)$ is $\mathcal{C}^\infty$ on $U$.

    Finally, by \eqref{eq:isolate:merged} and continuity of roots with respect to coefficients, after possibly shrinking $U$ the polynomial $q_p$ has exactly two roots
    (counted with multiplicity) in a small neighborhood of $\{x_i^*,x_{i+1}^*\}$, and these two roots are precisely the roots of the quadratic factor
    $Q_p(t):=t^2-S(p)t+P(p)$. For $p\in U\cap F(\overline\Omega)$, these are exactly the ordered roots $x_i(p)$ and $x_{i+1}(p)$, hence
        \begin{equation*}
            S(p)=x_i(p)+x_{i+1}(p)\qquad\text{for all }p\in U\cap F(\overline\Omega).
        \end{equation*}
    Setting $\psi:=S$ concludes the proof of (\ref{prop:xip:item_sum}).
\end{proof}


\begin{proposition}
    \label{prop:refl:comute}
    Let $\abQ$ and $\gamma=\binva$. Then $\ainvb = \gamma$ and $\cinva=\beta$, $\cinvb = \alpha$. Consequently, if $\delta \in \{\alpha,\beta,\gamma\}$, then $\{\ainvd,\binvd,\cinvd\} = \{\alpha,\beta,\gamma\}$.
\end{proposition}

\begin{proof}
    Since $\abQ$ gives $\ab=\pm1$ and $|\alpha|^2=|\beta|^2=2$ we have
    $$
        \ainvb = \pm\left(\alpha-\frac{2\ab \beta}{|\beta|^2}\right) = \pm \ab\left(\ab \alpha-\beta\right) = \mp \ab \left(\beta-\frac{2\ab \alpha}{|\alpha|^2}\right) = \binva.
    $$
    We also have
    \begin{eqnarray*}
        \ac &=& \langle\alpha,\binva\rangle = \pm\left(\ab - \frac{2\ab |\alpha|^2}{|\alpha|^2}\right) = \mp \ab,\\
        \bc &=& \pm\left(|\beta|^2-\frac{2\ab^2}{|\alpha|^2}\right) = \pm1
    \end{eqnarray*}
    and consequently
    \begin{eqnarray*}
        \cinva &=& \pm\left(\gamma-\frac{2\ac \alpha}{|\alpha|^2}\right) = \pm\left(\beta-\frac{2\ab \alpha}{|\alpha|^2}+\frac{2\ab \alpha}{|\alpha|^2}\right) = \beta,\\
        \cinvb &=& \pm\left(\gamma-\frac{2\bc \beta}{|\beta|^2}\right) = \pm\left(\beta - \frac{2\ab \alpha}{|\alpha|^2}-\beta\right) = \alpha.
    \end{eqnarray*}
\end{proof}    


\begin{proposition}
    \label{prop:refl:deltagamma}
    Let $\abQ$ and $\delta\in R_+$ such that $|\delta|^2=2$. Then 
    \begin{itemize}
        \item[\textrm{(a)}] If $\ad=\bd=0$, then $\binvad=0$.
        \item[\textrm{(b)}] If $|\ad\bd| = 1$, then $\binvad=0$.
        \item[\textrm{(c)}] If $|\ad\bd| = 2$, then $|\binvad| = 1$.
    \end{itemize}
\end{proposition}
\begin{proof}
    The proof follows from the definition of $\binva$, which together with the assumptions that $|\alpha|=|\beta|=|\delta|=\sqrt{2}$ give
    \begin{equation*}
        \binvad = \pm(\bd-\frac{2\ab\ad}{|\alpha|^2}) = \pm(\bd-\ab\ad)\/.
    \end{equation*}
    The first part follows directly. To justify the second claim, note that since all $\bd$, $\ab$, $\ad$ are equal to $\pm 1$, then $\alpha,\beta,\delta \in R_+^i$ for some $i=1,\ldots,N$ and
    $$
        \bd-\ab\ad = \beta_i\delta_i - \alpha_i^2\beta_i\delta_i = 0.
    $$
    Finally, $|\ad\bd| = 2$ means that $\delta=\alpha$ or $\delta=\beta$. In both cases, we obtain
    \begin{eqnarray*}
        |\binvad| &=& |\langle\binva,\alpha\rangle| = \left|\ab-\frac{2\ab|\alpha|^2}{|\alpha|^2}\right| = |\ab| =1\/,\quad \textrm{ for }\delta = \alpha,\\
        |\binvad| &=& |\langle\binva,\beta\rangle| = \left||\beta|^2-\frac{2\ab^2}{|\alpha|^2}\right| = 2-1 = 1\/,\quad \textrm{ for }\delta = \beta.
    \end{eqnarray*}
\end{proof}

\bibliographystyle{amsalpha}
\bibliography{bibtex}

\end{document}